\documentclass[english,leqno]{article}

\usepackage{latexsym}
\usepackage{linearb}
\usepackage[LGR, OT1]{fontenc}
\usepackage{amsmath,amsthm}
\usepackage{amssymb}
\usepackage{MnSymbol}
\usepackage{graphicx}
\usepackage{marvosym}
\usepackage{eufrak}
\usepackage[margin=1.25in,dvips]{geometry}
\usepackage{color}
\usepackage{comment}
\usepackage{mathrsfs}\usepackage[LGR,T1]{fontenc}
\usepackage[latin9]{inputenc}
\usepackage{geometry}
\usepackage{babel}
\usepackage{enumitem}
\usepackage{amsthm}
\usepackage{amstext}
\usepackage{amssymb}
\usepackage{esint}
\usepackage[unicode=true,pdfusetitle,
 bookmarks=true,bookmarksnumbered=false,bookmarksopen=false,
 breaklinks=false,pdfborder={0 0 1},backref=false,colorlinks=false]
 {hyperref}
\usepackage{breakurl}
\makeatletter

\DeclareRobustCommand{\greektext}{%
  \fontencoding{LGR}\selectfont\def\encodingdefault{LGR}}
\DeclareRobustCommand{\textgreek}[1]{\leavevmode{\greektext #1}}
\DeclareFontEncoding{LGR}{}{}
\DeclareTextSymbol{\~}{LGR}{126}
\newcommand{\lyxmathsym}[1]{\ifmmode\begingroup\def\b@ld{bold}
  \text{\ifx\math@version\b@ld\bfseries\fi#1}\endgroup\else#1\fi}

\numberwithin{equation}{section}
\numberwithin{figure}{section}
\usepackage{enumitem}		
\newcommand{\lyxaddress}[1]{
\par {\raggedright #1
\vspace{1.4em}
\noindent\par}
}
  \theoremstyle{plain}
  \newtheorem*{thm*}{\protect\theoremname}
\theoremstyle{plain}
\newtheorem{thm}{\protect\theoremname}[section]
  \theoremstyle{remark}
  \newtheorem*{rem*}{\protect\remarkname}
  \theoremstyle{plain}
  \newtheorem{cor}[thm]{\protect\corollaryname}
 \theoremstyle{definition}
 \newtheorem*{defn*}{\protect\definitionname}
  \theoremstyle{definition}
  \newtheorem*{example*}{\protect\examplename}
  \theoremstyle{plain}
  \newtheorem{lem}[thm]{\protect\lemmaname}
  \theoremstyle{plain}
  \newtheorem{prop}[thm]{\protect\propositionname}
 \newlist{casenv}{enumerate}{4}
 \setlist[casenv]{leftmargin=*,align=left,widest={iiii}}
 \setlist[casenv,1]{label={{\itshape\ \casename} \arabic*.},ref=\arabic*}
 \setlist[casenv,2]{label={{\itshape\ \casename} \roman*.},ref=\roman*}
 \setlist[casenv,3]{label={{\itshape\ \casename\ \alph*.}},ref=\alph*}
 \setlist[casenv,4]{label={{\itshape\ \casename} \arabic*.},ref=\arabic*}

\makeatother

  \providecommand{\corollaryname}{Corollary}
  \providecommand{\definitionname}{Definition}
  \providecommand{\examplename}{Example}
  \providecommand{\lemmaname}{Lemma}
  \providecommand{\propositionname}{Proposition}
  \providecommand{\remarkname}{Remark}
  \providecommand{\theoremname}{Theorem}
 \providecommand{\casename}{Case}
\providecommand{\theoremname}{Theorem}

\begin{document}
\title{Logarithmic local energy decay for scalar waves\\ on a general class of asymptotically flat spacetimes}

\author{Georgios Moschidis}

\maketitle

\lyxaddress{Princeton University, Department of Mathematics, Fine Hall, Washington
Road, Princeton, NJ 08544, United States, \tt gm6@math.princeton.edu}
\begin{abstract}
This paper establishes that on the domain of outer communications
of a general class of stationary and asymptotically flat Lorentzian
manifolds of dimension $d+1$, $d\ge3$, the local energy of solutions
to the scalar wave equation $\square_{g}\text{\textgreek{y}}=0$ decays
at least with an inverse logarithmic rate. This class of Lorentzian
manifolds includes (non-extremal) black hole spacetimes with no restriction
on the nature of the trapped set. Spacetimes in this class are moreover
allowed to have a small ergoregion but are required to satisfy an
energy boundedness statement. Without making further assumptions,
this logarithmic decay rate is shown to be sharp. Our results can
be viewed as a generalisation of a result of Burq, dealing with the
case of the wave equation on flat space outside compact obstacles,
and the results of Rodnianski and Tao for asymptotically conic product
Lorentzian manifolds. The proof will bridge ideas from Rodnianski
and Tao (see \cite{Rodnianski2011}) with techniques developed in
the black hole setting by Dafermos and Rodnianski (see \cite{DafRod6},
\cite{DafRod5}). As a soft corollary of our results, we will infer
an asymptotic completeness statement for the wave equation on the
spacetimes considered, in the case where no ergoregion is present.
\end{abstract}
\tableofcontents{}

\section{Introduction}

Recent progress in understanding the behaviour of solutions to the
scalar wave equation 
\begin{equation}
\square_{g}\text{\textgreek{y}}=0\label{eq:WaveEquation}
\end{equation}
 on various general relativistic backgrounds has been astonishing:
In the case of the Schwarzschild exterior background, boundedness
and decay results were established in \cite{KayWald,DafRod1,DafRod2,DafRod4,BlueSof1,BlueSterb}.
In the case of the Kerr family, similar results in the very slowly
rotating case (i.\,e. for angular momentum $a$ and mass $M$ satisfying
$|a|\ll M$) were obtained in \cite{DafRod5,DafRod6,DafRod9,TatToh1,AndBlue1}.
The full subextremal range $|a|<M$ was finally treated in \cite{DafRod8,Shlap,DafRodSchlap}.
See also \cite{Civ} for the Kerr--Newman case. Various refinements
of the earlier decay results on these spacetimes appear in \cite{Luk2010b,Luk2,Schlue2013}.
This picture changes dramatically when one switches attention to extremal
black hole spacetimes, where instability results have been established
in \cite{Aretakis2011,Aretakis2011a,Aretakis,Aretakis2012a}. 

All the preceding examples dealt with asymptotically flat spacetimes,
but a plethora of relevant results have also been proven for spacetimes
with different asymptotic structure: See \cite{DafRod3,Dyatlov2,Dyatlov2011,Melrose2008,Vasy2013}
for the case of Schwarzschild- and Kerr-de\,Sitter spacetimes, and
\cite{Holzegel2012,Holzegel2013,Holzegel2013a} for the case of  Kerr-AdS
spacetimes. 

Given the amount of technical machinery that has been developed by
the aforementioned authors, it is now feasible to move out of the
realm of backgrounds that are algebraically special solutions to the
Einstein equations or perturbations thereof and address questions
regarding the behaviour of scalar waves on more general Lorentzian
manifolds $(\mathcal{M}^{d+1},g)$. A demanding first question in
this direction is the following:

\bigskip{}

\emph{What are the most general types of spacetimes $(\mathcal{M}^{d+1},g)$
on which boundedness and decay of solutions to the scalar wave equation
(\ref{eq:WaveEquation}) can be obtained and studied? Or, from a different
perspective, what are the possible obstructions to stability for (\ref{eq:WaveEquation})
on general backgrounds?}

\bigskip{}

This paper aims to make a step towards providing answers to the above
question, by establishing the following general decay result:
\begin{thm*}
Let $(\mathcal{M}^{d+1},g)$, $d\ge3$, be a globally hyperbolic spacetime,
which is stationary and asymptotically flat, and which can possibly
contain black holes with a non degenerate horizon and a small ergoregion.
Moreover, suppose that an energy boundedness statement is true for
solutions $\text{\textgreek{y}}$ to (\ref{eq:WaveEquation}) on the
domain of outer communications $\mathcal{D}$ of $(\mathcal{M},g)$.
Then the local energy of $\text{\textgreek{y}}$ on $\mathcal{D}$
decays at least with a logarithmic rate: 
\begin{equation}
E_{loc}(t)\lesssim_{m}\frac{1}{\big\{\log(2+t)\big\}^{2m}}E_{w}^{(m)}(0),\label{eq:FirstStatementTheorem}
\end{equation}
 where $t$ is a suitable time function on $\mathcal{D}$ and $E_{w}^{(m)}(0)$
is a weighted initial energy of the first $m$ derivatives of $\text{\textgreek{y}}$.
\end{thm*}
A more detailed and rigorous statement of this theorem and the assumptions
on the spacetimes under consideration will be presented in Sections
\ref{sub:Assumptions} and \ref{sub:Theorem}. 

As an application of our results, we will deduce quantitative decay
rates for solutions to (\ref{eq:WaveEquation}) on a number of vacuum
(and other) spacetime backgrounds which appear in the literature,
but whose trapping structure is not yet completely understood. These
examples include, for instance, axisymmetric scalar waves on the Emparan--Reall
black rings (see \cite{Emparan2002}) and the Elvang--Figueras black
Saturn (see \cite{Elvang2007}). Furthermore, the results of the current
paper will be used to rigorously establish the so called Friedman
ergosphere instability (introduced and supported heuristically in
\cite{Friedman1978}). That is to say, we will establish by contradiction
that no non-degenerate energy boundedness statement can hold on stationary
spacetimes with ergoregion and no event horizon; see our forthcoming
\cite{Moschidis}.

As a corollary of the above theorem, using also the results of \cite{Moschidisc},
we will infer that the energy flux of solutions $\text{\textgreek{y}}$
to (\ref{eq:WaveEquation}) through a foliation of hyperboloidal hypersurfaces
decays logarithmically in time. This result will in turn yield an
asymptotic completeness statement for the wave equation on spacetimes
$(\mathcal{M},g)$ satisfying the assumptions of the above theorem
but without ergoregion. See Sections \ref{sec:Theorem-Sketch}, \ref{sec:Proof-of-corollary}
and \ref{sec:Proof-of-Corollary2}.

In the context of non-asymptotically flat spacetimes, we should note
the results of Vasy \cite{Vasy2013}, dealing with a class of Lorentzian
manifolds generalising Kerr-de\,Sitter spacetime. For \cite{Vasy2013},
however, the structure of the trapped null geodesics of the underlying
manifold and the de\,Sitter asymptotics play a crucial role in establishing
exponential decay rates for solutions to the wave equation $\square_{g}\text{\textgreek{y}}=0$.
In particular, the trapped set is required to resemble closely that
of Kerr-de\,Sitter spacetime. In contrast, in our setting no structural
condition is placed on the trapped set, which leads to an inverse
logarithmic decay rate (\ref{eq:FirstStatementTheorem}) for the local
energy of scalar waves which is sharp for some of the spacetimes in
the class under consideration. Notice also that the asymptotic flatness
of the spacetimes considered here would prohibit establishing faster
than polynomial decay rates for solutions to (\ref{eq:WaveEquation}).

Before stating more precisely the main result of the current paper,
we will first examine two well understood examples: Solutions of (\ref{eq:WaveEquation})
on flat space outside a compact obstacle and solutions to (\ref{eq:WaveEquation})
on asymptotically conic manifolds of product type. It is the results
of Burq \cite{Burq1998}, in the obstacle case, and of Rodnianski--Tao
\cite{Rodnianski2011}, in the product case, that our main theorem
generalises.

\subsection{The wave equation on $\mathbb{R}^{d}$ with obstacles and a result
of Burq}

Let $\mathcal{O}$ be a compact subset of $\mathbb{R}^{d}$, which
is the closure of a finite number of domains with smooth boundary.
Let also $\text{\textgreek{y}}=\text{\textgreek{y}}(t,x)\in C^{\infty}\big(\mathbb{R}\times(\mathbb{R}^{n}\backslash\mathcal{O})\big)$
solve 
\begin{equation}
(\partial_{t}^{2}-\text{\textgreek{D}})\text{\textgreek{y}}=0
\end{equation}
 on $\mathbb{R}\times(\mathbb{R}^{d}\backslash\mathcal{O})$, with
$\text{\textgreek{y}}\equiv0$ on $\partial\mathcal{O}$ and with
$(\text{\textgreek{y}},\partial_{t}\text{\textgreek{y}})$ compactly
supported (or at least suitably decaying) on $\{t=0\}$. In this case,
one immediately sees that there exists a positive-definite conserved
energy 
\begin{equation}
E(t)=\int_{\mathbb{R}^{d}\backslash\mathcal{O}}|\partial\text{\textgreek{y}}(t,x)|^{2}\, dx
\end{equation}
 (corresponding to the Killing field $\partial_{t}$ on $\mathbb{R}^{d+1}\backslash(\mathbb{R}\times\mathcal{O})$),
which allows to easily handle boundedness issues.

Starting from the boundedness of the energy $E(t)$, ``soft'' arguments
can be used to infer that for any $R>0$, the local energy $E_{R}(t)$
(i.\,e.~the energy contained in a ball $B_{R}$ of $\mathbb{R}^{d}$
of fixed radius $R$) tends to $0$ as $t\rightarrow+\infty$. Attempting
to study the precise rate at which the local energy decays, it is
inevitable that the nature of trapping%
\footnote{Line rays that are ``reflected'' on the obstacle's surface and remain
in a bounded region of space for arbitrarily long time are said to
be \emph{trapped}. An obstacle is called \emph{non-trapping }if it
does not give rise to a trapped ray.%
} will come into play. This subject has been extensively studied during
the last 50 years, but here we will refer to only a few indicative
results. In the case the obstacle $\mathcal{O}\subseteq\mathbb{R}^{d}$
is \emph{non-trapping}, Morawetz, Ralston and Strauss (see \cite{Morawetz1977})
showed that for any function $\text{\textgreek{y}}$ solving $(\partial_{t}^{2}-\text{\textgreek{D}})\text{\textgreek{y}}=0$
on $\mathbb{R}\times(\mathbb{R}^{d}\backslash\mathcal{O})$, with
$\text{\textgreek{y}}=0$ on $\partial\mathcal{O}$ and $(\text{\textgreek{y}},\partial_{t}\text{\textgreek{y}})|_{t=0}$
compactly supported, the local energy 
\[
E_{R}(t)\doteq\int_{B_{R}\backslash\mathcal{O}}|\partial\text{\textgreek{y}}(t,x)|^{2}\, dx
\]
decays at least polynomially in time $t$, and in fact this decay
rate becomes exponential if the space dimension $d$ is odd. In order
to establish this decay rate, it was first shown that 
\begin{equation}
\int_{0}^{\infty}E_{R}(t)\, dt\le C(R)\cdot E(0)\label{eq:ILEDNonTrapping}
\end{equation}
where $E(0)=\int_{\mathbb{R}^{d}\backslash\mathcal{O}}|\partial\text{\textgreek{y}}(0,x)|^{2}\, dx$
is the initial energy of the wave. An inequality of the form (\ref{eq:ILEDNonTrapping})
is usually referred to as an \emph{integrated local energy decay estimate}.

Due to a result by Ralston (see \cite{Ralston1969}), however, the
presence of even a single trapped ray is inconsistent with any quantitative
decay rate for the local energy of waves in terms only of their initial
energy. Moreover, no statement of the form (\ref{eq:ILEDNonTrapping})
can be valid. This fact has been recently generalised to the case
of globally hyperbolic Lorentzian manifolds with trapped null geodesics
by Sbierski in \cite{Sbierski2013}. 

If one is willing to ``sacrifice'' some initial regularity on the
right hand side of (\ref{eq:ILEDNonTrapping}), so as to establish
a quantitative decay estimate for the local energy of a wave $\text{\textgreek{y}}$
in terms of the initial energy of higher order derivatives of $\text{\textgreek{y}}$,
then one can still obtain various types of decay rates, which are
expected to become faster as the trapping becomes more unstable (compare
for instance \cite{Ikawa1988} and \cite{Ralston1971}). It is remarkable,
therefore, that in \cite{Burq1998} Burq was able to prove \emph{without
any assumptions regarding the nature of trapping and the form of the
obstacle} $\mathcal{O}$, and with the hypotheses on $\text{\textgreek{y}}$
as before, that one has 
\begin{equation}
E_{R}(t)\lesssim\big\{\log(2+t)\big\}^{-2m}\cdot E^{(m)}(0),
\end{equation}
where $E^{(m)}(0)$ denotes the energy of the first $m$ derivatives
of $\text{\textgreek{y}}$ at $\{t=0\}$. In the preceding $\lesssim$
notation, the implicit constant depends on $R$ and on the size of
the compact support of the initial data for $\text{\textgreek{y}}$
on $\{t=0\}$. 

In fact, Burq's result is a bit more general, as it allowed for \textgreek{y}
satisfying an equation of the form 
\begin{equation}
\partial_{t}^{2}\text{\textgreek{y}}-\partial_{i}(a^{ij}\partial_{j})\text{\textgreek{y}}=0
\end{equation}
where $a^{ij}=a^{ij}(x)$ is a smooth positive definite matrix which
is equal to the identity outside some large compact set. However,
the fact that this operator is identical to $\partial_{t}^{2}-\text{\textgreek{D}}$
outside a large spatial ball was used in an essential way in Burq's
argument, which relied on a decomposition of $\text{\textgreek{y}}$
into spherical harmonics outside a large ball and on using explicit
representation formulas to study the asymptotics of the resulting
Hankel functions. Hence, his proof does not immediately generalise
to the case where $\{a^{ij}\}-Id$ is not compactly supported.

\subsection{A result of Rodnianski--Tao for product Lorentzian manifolds}

The next step after studying the wave equation on $\mathbb{R}^{d}$
outside a compact obstacle is addressing the wave equation on product
Lorentzian manifolds with a more general asymptotically flat structure.

Given an arbitrary Riemannian manifold $(\mathcal{N}^{d},\bar{g})$,
quantitative estimates for the resolvent $R(\text{\textgreek{l}}+i\text{\textgreek{e}};\text{\textgreek{D}}_{\bar{g}})$
of the Laplacian $\text{\textgreek{D}}_{\bar{g}}$ on $(\mathcal{N},g)$
are closely connected to estimates for solutions to the wave equation
$\square_{g}\text{\textgreek{y}}=0$ on the product Lorentzian manifold
$(\mathbb{R}\times\mathcal{N},g)$ with $g=-dt^{2}+\bar{g}$ ($t$
being the projection to the first factor of $\mathbb{R}\times\mathcal{N}$).%
\footnote{Notice that on such a Lorentzian manifold, the Killing field $\partial_{t}$
immediately gives rise to a bounded (in fact preserved) energy $\int_{t=const}J_{\text{\textgreek{m}}}^{\partial_{t}}(\text{\textgreek{y}})n^{\text{\textgreek{m}}}$
for solutions $\text{\textgreek{y}}$ to the wave equation $\square_{g}\text{\textgreek{y}}=0$.%
} Therefore, relevant to our question on the behaviour of scalar waves
on general Lorentzian manifolds are the results of Rodnianski and
Tao (see \cite{Rodnianski2011}).%
\footnote{For results similar to the ones in \cite{Rodnianski2011}, but restricted
to high frequencies $\text{\textgreek{l}}\gg1$ and obtained by different
methods, see also the work of Cardoso and Vodev, \cite{Cardoso2002,Cardoso2004}.%
} In particular, in \cite{Rodnianski2011} the authors have established
(among other results) quantitative bounds for such resolvents for
a general class of asymptotically conic Riemannian manifolds $\mathcal{N}$
of dimension $d\ge3$. This enabled them to obtain, as a corollary
of their resolvent estimates, a decay without a rate result for solutions
to the wave equation on $(\mathbb{R}\times\mathcal{N},g)$.

Though not explicitly stated, the results established in \cite{Rodnianski2011}
are also sufficient to obtain a logarithmic decay statement for scalar
waves on $(\mathbb{R}\times\mathcal{N},g)$: The exponential bounds
on the resolvent $R(\text{\textgreek{l}}+i\text{\textgreek{e}};\text{\textgreek{D}}_{\bar{g}})$
as $\text{\textgreek{l}}\rightarrow\infty$, as well as the bounds
for the resolvent near $\text{\textgreek{l}}=0$, suffice to establish
that on the product manifold $(\mathbb{R}\times\mathcal{N},g)$, the
local energy of solutions $\text{\textgreek{y}}$ to the wave equation
decays with a rate of the form:
\[
E_{R}(\text{\textgreek{t}})\lesssim\big\{\log(2+\text{\textgreek{t}})\big\}^{-2m}\cdot E_{w}^{(m)}(0).
\]
In the above, $E_{R}(t)=\int_{\{t=\text{\textgreek{t}}\}\cap\{r\le R\}}\big(|\partial_{t}\text{\textgreek{y}}|^{2}+|\nabla^{(\mathcal{N})}\text{\textgreek{y}}|^{2}\big)$
is the local energy contained in the region $\{r\le R\}$ at time
$\text{\textgreek{t}}$, for some function $r\ge0$ tending to $+\infty$
towards the asymptotically conic end of $\mathcal{N}$, and $E_{w}^{(m)}(0)$
is a suitable weighted energy of the first $m$ derivatives of $\text{\textgreek{y}}$
at $\{t=0\}$. Moreover, the implicit constant in the preceding $\lesssim$
notation depends only on the geometry of $(\mathcal{N},\bar{g})$,
and not on the size of the support of the intial data for $\text{\textgreek{y}}$. 

Thus, the results of \cite{Rodnianski2011} generalise the results
of Burq in \cite{Burq1998} to the case of product asymptotically
conic Lorentzian manifods, that are not necessarily identical to $(\mathbb{R}^{d+1},\text{\textgreek{h}})$
outside a bounded cylinder, $\text{\textgreek{h}}$ being the usual
Minkowski metric on $\mathbb{R}^{d+1}$.

We will now proceed to state the main result of this paper, which
should be viewed as a generalisation of the results of \cite{Rodnianski2011}
to not necessarily product Lorentzian manifolds, which are moreover
allowed to contain black hole regions and a small ergoregion. We remark
already that many of the ideas of the proof of the main result of
the current paper actually originate in \cite{Rodnianski2011}. The
importance of the ideas and techniques of \cite{Rodnianski2011} in
our setting will become evident in the subsequent sections.

\subsection{The main result: Logarithmic decay in a suitably general asymptotically
flat setting}

The aforementioned results of Burq and Rodnianski--Tao indicate that
if one wants to study a more general class of asymptotically flat
spacetimes, possibly containing a black hole and a small ergoregion,
and with few further restrictions on the geometry of the region between
the horizon (if non empty) and ``infinity'', one can only hope to
prove a logarithmic decay rate for the local energy of solutions to
(\ref{eq:WaveEquation}). This is exactly the result that we will
prove here. 

In particular, we will be concerned with globally hyperbolic spacetimes
$(\mathcal{M}^{d+1},g)$, $d\ge3$, satisfying the following assumptions:
\begin{enumerate}
\item \emph{\label{enu:As1}\bf{\emph{Asymptotic flatness and stationarity:}}}
The metric $\bar{g}$ and the second fundamental form $k$ induced
on a Cauchy hypersurface $\text{\textgreek{S}}$ of $(\mathcal{M},g)$
must form an asymptotically flat triad $(\text{\textgreek{S}},\bar{g},k)$.
Moreover, there must exist a Killing field $T$ (called the \emph{stationary}
Killing field) with complete orbits on the \emph{domain of outer communications}
$\mathcal{D}$ of $(\mathcal{M},g)$, extending also to the boundary
of $\mathcal{D}$, such that $T$ becomes strictly timelike in the
asymptotically flat region of $(\mathcal{M},g)$. See Section \ref{sub:Assumptions}
for a precise definition of these notions.
\item \emph{\label{enu:As2}\bf{\emph{Killing horizon with positive surface gravity:}}}
If $(\mathcal{M},g)$ is not identical to the domain of outer communications
$\mathcal{D}$, then the boundary $\mathcal{H}$ of $\mathcal{D}$
(called the \emph{event horizon} of the spacetime) must be a non degenerate
Killing horizon (possibly for a different Killing vector field than
$T$), with positive surface gravity.
\item \emph{\label{enu:As3}\bf{\emph{Smallness of the ergoregion:}}} The
subset $\{g(T,T)\ge0\}$ of the domain of outer communications $\mathcal{D}$
must be quantitatively small, as should be the maximum positive value
of $g(T,T)$. If $\mathcal{H}=\emptyset$, then this requirement reduces
to the assumption that the ergoregion being empty.%
\footnote{For the energy instability associated to the existence of an ergoregion
in the absence of an event horizon, see already our forthcoming \cite{Moschidis}.
We will return to this issue at the discussion in the next section.%
}
\item \emph{\label{enu:As4}\bf{\emph{Boundedness of the energy of waves:}}}
The energy of solutions $\text{\textgreek{y}}$ to the scalar wave
equation $\square_{g}\text{\textgreek{y}}=0$ with respect to a $T$-invariant,
globally timelike vector field $N$ and a suitable foliation of the
domain of outer communications $\mathcal{D}$ must remain bounded
by a constant times the energy of $\text{\textgreek{y}}$ initially.
\end{enumerate}
These assumptions will be stated more precisely in Section \ref{sub:Assumptions}.
The main result of this paper can be stated as follows:
\begin{thm}
\label{thm:TheoremIntroduction}Let $(\mathcal{M}^{d+1},g)$, $d\ge3$,
be a spacetime satisfying the hypotheses \ref{enu:As1}--\ref{enu:As4},
and let $\mathcal{D}$ be its domain of outer communications. Then
any smooth solution $\text{\textgreek{y}}$ to $\square_{g}\text{\textgreek{y}}=0$
on $\mathcal{D}$ with suitably decaying initial data on a Cauchy
hypersurface $\text{\textgreek{S}}$ of $\mathcal{M}$ satisfies for
any integer $m>0$ and any $0<\text{\textgreek{d}}_{0}\le1$:

\begin{equation}
E_{loc}(t)\le\frac{C_{m}}{\big\{\log(2+t)\big\}^{2m}}E^{(m)}(0)+\frac{C}{t^{\text{\textgreek{d}}_{0}}}E_{w,\text{\textgreek{d}}_{0}}(0).\label{eq:CariqatureInequality}
\end{equation}
 In the above, $t\ge0$ is a suitable time function on $J^{+}(\text{\textgreek{S}})\cap\mathcal{D}$
($J^{+}(\text{\textgreek{S}})$ being the causal future of $\text{\textgreek{S}}$)
with $\{t=0\}\equiv\text{\textgreek{S}}\cap\mathcal{D}$, and $E_{loc}(t)$
is the local energy of the wave at time $t$. $E^{(m)}(0)$ is the
energy of the first $m$ derivatives of $\text{\textgreek{y}}$ at
$\{t=0\}$, while $E_{w,\text{\textgreek{d}}_{0}}(0)$ is a suitable
weighted energy of $\text{\textgreek{y}}$ at $\{t=0\}$. The constant
$C$ on the right hand side depends on the size of the ball where
the local energy $E_{loc}$ is measured and on the geometry of $(\mathcal{D},g)$,
while in addition to that, $C_{m}$ also depends on the number $m$
of derivatives of $\text{\textgreek{y}}$ in $E^{(m)}(0)$.
\end{thm}
A more precise statement of the above result will be given in Section
\ref{sub:Theorem}, after a more detailed discussion of the geometric
hypotheses on $(\mathcal{M},g)$.

This result can be immediately applied to many interesting black hole
spacetimes, on many of which no decay result had been established
so far. These examples will be discussed in Section \ref{sub:Examples}.

Let us also mention here that our result is a proper generalisation
of the theorem of Burq (at least in dimensions $d\ge3$): While not
explicitly mentioned later in this paper, the proof of Theorem \ref{thm:TheoremIntroduction}
can also be repeated in the case of the wave equation on $\mathbb{R}^{d}$
with obstacles, relaxing the assumption of $a^{ij}$ in the perturbed
Laplacian $\partial_{i}(a^{ij}\partial_{j})$ being equal to $Id$
outside a compact set (as required by Burq) to just demanding that
$(a^{ij})\rightarrow Id$ as $r\rightarrow\infty$ in a quantitative
manner. Moreover, the constants in the final inequality do not depend
on the size of the support of the initial data for $\text{\textgreek{y}}$
(which in particular needs not be compactly supported initially);
only the finiteness of a suitable weighted norm of an initial higher
order energy of $\text{\textgreek{y}}$ is necessary.

\subsection{Necessity of Assumptions 1--4}

At this point, a few words should be said regarding the necessity
of Assumptions \ref{enu:As1}-\ref{enu:As4} on the spacetimes under
consideration. We will see that while each of these assumptions can
possibly be relaxed, discarding completely any of them (without adding
any extra hypothesis in its place) appears to lead to spacetimes where
Theorem \ref{thm:TheoremIntroduction} does not hold.

Some assumption of global character, such as Assumption \ref{enu:As1}
on stationarity and asymptotic flatness, seems to be necessary for
a first statement of the generality we are after. In particular, Assumption
\ref{enu:As1} seems to be the most appropriate for applications in
cases of isolated self gravitating relativistic systems satisfying
the Einstein equations (with suitable matter fields) with zero cosmological
constant. 

Replacing the asymptotic flatness assumption with other asymptotics
can lead to instability results contradicting Theorem \ref{thm:TheoremIntroduction}:
This is the case for asymptotically AdS spacetimes, for instance (with
suitable boundary conditions at the timelike conformal boundary considered
as in \cite{Holzegel2012}). In fact, on Anti-deSitter space itself,
while the usual $\partial_{t}$ Killing vector field gives rise to
a conserved non degenerate energy current, there exist time periodic
solutions to the wave equation, and this fact prohibits the proof
of any sort of uniform decay rate (see \cite{Holzegel2013a} and references
therein). 
\begin{rem*}
The presence of a non-empty non-degenerate horizon $\mathcal{H}^{+}$
in an asymptotically AdS spacetime might be sufficient to provide
stability results for scalar waves, at least in the case where a boundedness
statement for the energy of solutions to $\square\text{\textgreek{y}}=0$
is true, as the results on Kerr-AdS suggest: In \cite{Holzegel2013},
Holzegel and Smulevici show that under certain bounds on the value
of the cosmological constant $\text{\textgreek{L}}$, the angular
momentum $a$, the mass $M$ of the Kerr-AdS spacetime and the mass
$m$, implying among other things the existence of a globally causal
Killing field, any solution $\text{\textgreek{y}}$ to the massive
wave equation 
\begin{equation}
\square_{g}\text{\textgreek{y}}+m\text{\textgreek{y}}=0\label{eq:MassiveEq}
\end{equation}
decays logarithmically in time (provided, of course, that some weighted
energy of the initial data is finite). And in fact, this decay rate
is optimal (see \cite{Holzegel2013a}, and also \cite{Gannot2014}).
Hence, it would be of particular interest to explore the minimal geometric
conditions that need to be imposed on a general asymptotically AdS
spacetime with a non-degenerate event horizon, in order for a logarithmic
decay estimate for solutions to (\ref{eq:MassiveEq}) to hold. Notice,
of course, that such a decay estimate can not hold on all asymptotically
AdS spacetimes containing a non-degenerate horizon, since there exist
examples of such spacetimes on which (\ref{eq:MassiveEq}) is actually
unstable: In the recent \cite{Dold2015}, Dold constructed exponentially
growing solutions to the massive wave equation on Kerr-AdS spacetimes
violating the Hawking--Reall bound.
\end{rem*}
Concerning Assumption \ref{enu:As2} on the non-degeneracy of the
horizon $\mathcal{H}^{+}$, we notice that the instability results
of Aretakis \cite{Aretakis2011,Aretakis2011a,Aretakis,Aretakis2012a}
on extremal black holes are rooted in the absence of a red shift type
vector field near $\mathcal{H}^{+}$ in the sense of \cite{DafRod2}.
This fact suggests that dispensing with Assumption \ref{enu:As2}
might lead to spacetimes on which Theorem \ref{thm:TheoremIntroduction}
is not true. Notice, however, that Aretakis in \cite{Aretakis2011a,Aretakis2011}
was able to exploit additional geometric features of the horizon of
extremal Reissner--Nordstr\"om spacetime in order to obtain stability
results for the first order energy of a scalar wave $\text{\textgreek{y}}$.
Thus, it may be possible to replace Assumption \ref{enu:As2} by some
weaker assumption regarding the geometry of $\mathcal{H}$, that allows
interesting extremal examples.%
\footnote{Even in that case, however, it may be necessary to weaken Theorem
\ref{thm:TheoremIntroduction} to a statement about some degenerate
energy current.%
}

Assumption \ref{enu:As3} on the smallness of the ergoregion, if violated,
will contradict any decay statement for the local energy of scalar
waves: On a spacetime with a spatially compact ergoregion and no horizon,
for instance, a solution $\text{\textgreek{y}}$ to $\square\text{\textgreek{y}}=0$
with initial $T$ energy $\int_{t=0}J_{\text{\textgreek{m}}}^{T}(\text{\textgreek{y}})n^{\text{\textgreek{m}}}=-1$
will satisfy $\int_{t=\text{\textgreek{t}}}J_{\text{\textgreek{m}}}^{T}(\text{\textgreek{y}})n^{\text{\textgreek{m}}}=-1$
for all times $\text{\textgreek{t}}\ge0$. This means that $\int_{\{t=\text{\textgreek{t}}\}\cap\{ergoregion\}}J_{\text{\textgreek{m}}}^{T}(\text{\textgreek{y}})n^{\text{\textgreek{m}}}=-1$
for any $\text{\textgreek{t}}\ge0$, and hence the local energy of
$\text{\textgreek{y}}$ over any compact spatial subset containing
the ergoregion will not decay at all.

In addition to the violation of any decay statement, the presence
of a large ergoregion could even preclude the boundedness of the energy
of waves: According to the heuristics of Friedman (\cite{Friedman1978}),
in the case of a stationary, asymptotically flat spacetime with no
horizon but with non empty ergoregion, one expects no uniform boundedness
statement to hold regarding the energy of solutions to the scalar
wave equation. Using the machinery of the proof of Theorem \ref{thm:Theorem},
we were able to provide a rigorous proof of this instability result,
which will be presented in our forthcoming \cite{Moschidis}. It is
an open problem, however, to understand whether Theorem \ref{thm:TheoremIntroduction}
is true in the case Assumption \ref{enu:As3} is violated but energy
boundedness (namely Assumption \ref{enu:As4}) still holds.

Finally, even if one assumes the existence of a non empty horizon
with positive surface gravity and some form of smallness of the ergoregion,
superradiance (i.\,e.\ the fact that the domain of outer communications
$\mathcal{D}$ of $(\mathcal{M},g)$ admits no globally non spacelike
Killing vector field) may pose serious risks for the boundedness of
the energy of waves: Shlapentokh--Rothman (see \cite{Shlapentokh-Rothman2013})
has constructed exponentially growing solutions to the Klein--Gordon
equation $\square_{g}\text{\textgreek{y}}-m\text{\textgreek{y}}=0$
on a slowly rotating Kerr background, for arbitrarily small values
of the mass $m$. For the case of the scalar wave equation on Kerr,
energy boundedness can also be violated if a compactly supported real
and positive potential is added; see our forthcoming \cite{Moschidisa}.
These facts suggest that Assumption \ref{enu:As4} on the boundedness
of the energy of scalar waves, which plays a fundamental role in the
proof of Theorem \ref{thm:TheoremIntroduction}, can not be inferred
from the rest of the Assumptions \ref{enu:As1}--\ref{enu:As3}, and
thus should be indisposable if one is not willing to further restrict
the class of spacetimes under consideration.

\subsection{\label{sub:Examples}Examples}

Beyond the class of asymptotically flat Lorentzian manifolds of product
type $(\mathbb{R}\times\mathcal{N},-dt^{2}+\bar{g})$, which were
already treated by Rodnianski and Tao in \cite{Rodnianski2011}, there
is a large number of other interesting classes of spacetimes on which
Theorem \ref{thm:Theorem} readily applies to give a novel logarithmic
decay statement for the local energy of waves:
\begin{itemize}
\item The general class of asymptotically flat non-extremal black hole spacetimes
without ergoregion, for which a non degenerate energy boundedness
statement has been established by Dafermos and Rodnianski in Section
7 of \cite{DafRod6}. This class includes the spherically symmetric
black hole spacetimes studied in \cite{Masarik2011}, and in particular
includes the black hole solutions of the $SU(2)$ Einstein--Yang Mills
equations constructed in \cite{Smoller1993}.
\item When restricted to axisymmetric solutions of $\square_{g}\text{\textgreek{y}}=0$,
Theorem \ref{thm:Theorem} also applies on the domain of outer communications
of the higher dimensional Emparan--Reall black rings (see \cite{Emparan2002,Emparan2002a})
and the Elvang--Figueras black Saturn (see \cite{Elvang2007}). Restriction
to axisymmetry in this case is necessitated by the fact that no energy
boundedness statement (in the spirit of Assumption \ref{enu:As4})
for general solutions to $\square_{g}\text{\textgreek{y}}=0$ has
been established so far on these spacetimes.
\item The class of $C^{1}$ perturbations of the slowly rotating Kerr spacetime
exterior for which a non degenerate energy boundedness statement was
proven by Dafermos and Rodnianski in \cite{DafRod5}. Notice that
since these perturbations are only close to Kerr exterior in a suitable
$C^{1}$ norm, the structure of trapping (which would be controlled
by the smallness of a $C^{2}$ norm) can differ substantially on these
spacetimes compared to trapping on Kerr exterior. Moreover, note that
these spacetimes can possibly contain an ergoregion.
\item Keir's class of static spherically symmetric spacetimes, including
a class of ultracomact neutron stars (see \cite{Keir2014}). Of course,
on these spacetimes \cite{Keir2014} has already established a logarithmic
decay bound, showing also that it is sharp. Thus, these spacetimes
also serve to exhibit the sharpness of Theorem \ref{thm:Theorem}
on the class of spacetimes under consideration.
\end{itemize}
The above examples will be discussed in more detail in Section \ref{sub:Examples-Assumptions},
after a more precise statement of Assumptions \ref{enu:As1}--\ref{enu:As4}
in Section \ref{sub:Assumptions}.

\subsection{Acknowledgements}

I would like to thank my advisor Mihalis Dafermos for suggesting this
problem to me and providing me with comments, ideas and assistance
while this paper was being written. I would also like to thank Igor
Rodnianski for his invaluable ideas and suggestions. I would finally
like to thank Yakov Shlapentokh-Rothman and Stefanos Aretakis for
many insightful conversations and important comments on preliminary
versions of this paper.

\section{\label{sec:Theorem-Sketch}Statement of the theorem and sketch of
the proof}

We will now proceed to rigorously state and explain the various geometric
assumptions on the class of spacetimes under consideration, and provide
the precise statement of the main result of this paper. We will then
briefly sketch the proof of this result, highlighting the important
ideas involved in the detailed proof occupying the subsequent sections
of the paper.

\subsection{\label{sub:Assumptions}Geometric assumptions on the background spacetime}

For the following, let $(\mathcal{M}^{d+1},g)$ be a smooth, time
oriented, globally hyperbolic $d+1$ dimensional Lorentzian manifold
for $d\ge3$. Let also $\text{\textgreek{S}}$ be a smooth spacelike
Cauchy hypersurface of $(\mathcal{M},g)$. We will now formulate assumptions
\ref{enu:As1}--\ref{enu:As4} in more detail in terms of the geometry
of $(\mathcal{M},g)$. For a detailed description of the notational
conventions adopted in what follows, see Section \ref{sec:Notational-conventions}.

Since in our formulation of Assumptions \hyperref[Assumption 1]{1}--\hyperref[Assumption 4]{4}
we have strived to include a broad class of globally hyperbolic Lorentzian
manifolds with as little structure as possible, we will need to perform
from scratch some geometric constructions that are trivial in many
concrete examples of spacetimes that appear in the literature. To
this end, the reader who is only interested in applications of the
results of this paper in specific examples of spacetimes should feel
free to skip many of the details of the explanation of the constructions
performed in the next paragraphs.

\subsubsection{\emph{Assumption \ref{enu:As1}}\label{Assumption 1} \emph{(Asymptotic
flatness and stationarity).}}

We assume that there exists a smooth Killing vector field $T$ on
$\mathcal{M}$, which has complete orbits%
\footnote{namely starting from any point of $\mathcal{M}$ we can follow the
flow of $T$ for infinite time in both directions%
} and which, when restricted on $\text{\textgreek{S}}$, becomes future
directed and timelike outside a compact subset of $\text{\textgreek{S}}$.
We will refer to $T$ as the \emph{stationary Killing field}.

We also assume that there exists a compact subset $K$ of $\text{\textgreek{S}}$
such that $\text{\textgreek{S}}\backslash K$ has a finite number
of connected components $\{\mathcal{S}_{i}\}_{i=1}^{m}$, and each
connected component $\mathcal{S}_{i}$ of $\text{\textgreek{S}}\backslash K$
can be diffeomorphically mapped to $\mathbb{R}^{d}\backslash B_{R}(0)$
(where $B_{R}(0)$ is the ball of $\mathbb{R}^{d}$ of radius $R$
centered at the origin). We will set 
\begin{equation}
\mathcal{S}\doteq\text{\textgreek{S}}\backslash K=\cup_{i=1}^{m}\mathcal{S}_{i}.
\end{equation}
In the resulting chart $(x^{1},x^{2},\ldots,x^{d})$ on each of the
$\mathcal{S}_{i}$'s, we let $r=\sqrt{(x^{1})^{2}+(x^{2})^{2}+\ldots+(x^{d})^{2}}$
be the pullback of the polar distance on $\mathbb{R}^{d}$. See Section
\ref{sec:Notational-conventions} for the $(r,\text{\textgreek{sv}})$
notation used for the polar coordinate chart on each connected component
of $\mathcal{S}$.

Without loss of generality, we assume that the compact subset $K$
was chosen large enough so that the restriction of $T$ on $\mathcal{S}$
is strictly timelike. Then the integral curves of $T$ would intersect
$\mathcal{S}$ transversally. Letting $\mathcal{S}_{\text{\textgreek{t}}}$
be the image of $\mathcal{S}$ under the flow of $T$ for time $\lyxmathsym{\textgreek{t}}$,
we deduce that the set 
\begin{equation}
\mathcal{I}_{as}\doteq\cup_{\text{\textgreek{t}}\in\mathbb{R}}\mathcal{S}_{\text{\textgreek{t}}}\simeq\mathbb{R}\times\mathcal{S}
\end{equation}
 is an open subset of $(\mathcal{M},g)$. 

We define the coordinate function $t:\mathcal{I}_{as}\rightarrow\mathbb{R}$
by the relation 
\begin{equation}
t|_{\mathcal{S}}=0,\, T(t)=1
\end{equation}
 and we note that if we extend the coordinate functions $(r,\text{\textgreek{sv}}):\mathcal{S}\rightarrow\mathbb{R}_{+}\times\mathbb{S}^{d-1}$
(see Section \ref{sec:Notational-conventions} for the $\text{\textgreek{sv}}$
notation) in the whole of $\mathcal{I}_{as}$ by the conditions $Tr=0,T\text{\textgreek{sv}}=0$,
then $(t,r,\text{\textgreek{sv}})$ is a valid coordinate chart on
each connected component of $\mathcal{I}_{as}$. In this chart, $\partial_{t}=T$.

\paragraph*{\noindent \emph{Quantitative asympotics for the metric:} }

\noindent In the $(t,r,\text{\textgreek{sv}})$ coordinates oneach
connected component of $\mathcal{I}_{as}$, there exists an $a\in(0,1]$
such that the metric $g$ takes the asymptotically flat form

\begin{equation}
g=-\Big(1-\frac{2M}{r}+h_{1}(r,\text{\textgreek{sv}})\Big)dt^{2}+\Big(1+\frac{2M}{r}+h_{2}(r,\text{\textgreek{sv}})\Big)dr^{2}+r^{2}\cdot\Big(g_{\mathbb{S}^{d-1}}+h_{3}(r,\text{\textgreek{sv}})\Big)+h_{4}(r,\text{\textgreek{sv}})dtd\text{\textgreek{sv}}\label{eq:metric}
\end{equation}
 where:
\begin{itemize}
\item $h_{1},h_{2}$ are smooth $O_{4}(r^{-1-a})$ functions. See Section
\ref{sec:Notational-conventions} for the $O_{k}(\cdot)$ notation.
\item $h_{4}$ is a smooth $O_{4}(r^{-a})$ function.
\item For every $r$, $h_{3}(r,\cdot)$ is a smooth symmetric $(0,2)$-tensor
field defined on $\mathbb{S}^{d-1}$, and satisfies the bound $h_{3}=O_{4}(r^{-1-a})$.
\end{itemize}
In the above expression for the metric in the asymptotically flat
region, we will refer to the constant $M$ as the mass of the spacetime
$(\mathcal{M},g)$, although it will only coincide with the ADM mass
of the spacetime if the space dimension is $d=3$.

Notice that (\ref{eq:metric}) implies that $T$ has been normalised
so that $g(T,T)\rightarrow-1$ as $r\rightarrow\infty$.%
\footnote{One could point out that a general asymptotically flat metric should
also contain a $dt\cdot dr$ component, but if this component has
$O_{4}(r^{-1-a})$ asymptotics, it can always be annihilated for $\{r\gg1\}$
by a transformation of the form $t\rightarrow t+f(r,\text{\textgreek{sv}})$,
which would of course change our choice of Cauchy hypersurface \textgreek{S}
corresponding to $\{t=0\}$. The same is true for the $dr\cdot d\text{\textgreek{sv}}$
terms that might possibly exist in a more general expression of an
asymptotically flat metric: If the $dr\cdot d\text{\textgreek{sv}}$
components are $O_{5}(r^{\frac{1-a}{2}})$, then one can choose a
new spherical parametrization $\text{\textgreek{sv}}_{new}=\text{\textgreek{sv}}_{new}(r,\text{\textgreek{sv}})$,
so that in $(t,r,\text{\textgreek{sv}}_{new})$ coordinates and for
$r\gg1$ the metric has the form \ref{eq:metric}.%
}

\paragraph*{\noindent \emph{Some special subsets of $(\mathcal{M},g)$: }}

\noindent We will define a few subsets of $(\mathcal{M},g)$ which
are tied to the causal structure of the spacetime, and will be frequently
referred to throughout the paper. 

We define the \emph{domain of outer communications} of the asymptotically
flat end $\mathcal{S}$ of $\text{\textgreek{S}}$ to be the subset
\begin{equation}
\mathcal{D}\doteq clos\big(J^{-}(\mathcal{I}_{as})\cap J^{+}(\mathcal{I}_{as})\big)
\end{equation}
 of $\mathcal{M}$. Here, $J^{+}(A)$ denotes the causal future of
the set $A$,%
\footnote{i.\,e. the set of points of $\mathcal{M}$ which can be the endpoints
of future directed causal curves in $\mathcal{M}$ starting from $A$%
} and $J^{-}(A)$ the causal past. If $\mathcal{D}$ is not the whole
of $\mathcal{M}$, then the set 
\begin{equation}
\mathcal{H}\doteq\partial\mathcal{D}
\end{equation}
 will be non empty, and we will call this set the \emph{horizon} of
$(\mathcal{M},g)$. The horizon $\mathcal{H}$ can be split naturally
as the union $\mathcal{H}^{+}\cup\mathcal{H}^{-}$, where $\mathcal{H}^{+}\doteq J^{+}(\mathcal{I}_{as})\cap\partial\big(J^{-}(\mathcal{I}_{as})\big)$
and $\mathcal{H}^{-}\doteq J^{-}(\mathcal{I}_{as})\cap\partial\big(J^{+}(\mathcal{I}_{as})\big)$.
Note that $\mathcal{H}^{-}$ lies in the past of $\mathcal{H}^{+}$. 

If non empty, we will assume that $\mathcal{H}^{+},\mathcal{H}^{-}$
are smooth null hypersurfaces of $(\mathcal{M},g)$, possibly with
boundary $\mathcal{H}^{+}\cap\mathcal{H}^{-}$. Since $T$ is a Killing
field of $\mathcal{M}$, and $\mathcal{I}_{as}$ is by definition
invariant under the flow of $T$, we conclude that $\mathcal{D}$,
$\mathcal{H}^{+}$, $\mathcal{H}^{-}$, $\mathcal{H}^{+}\cap\mathcal{H}^{-}$
and $\mathcal{H}$ must be invariant under the flow of $T$. Since
$\mathcal{H}^{+},\mathcal{H}^{-}$ were assumed to be smooth hypersurfaces
with boundary, this means that $T$ is tangent to them, and tangent
to their boundary.

Notice also that $\mathcal{D}$ might have more than one connected
comonents, but it is possible for more than one connected components
of $\mathcal{I}_{as}$ to lie in the same component of $\mathcal{D}$.

\paragraph*{\noindent \emph{Requirements regarding the Cauchy hypersurface \textgreek{S}:} }

\noindent We will need to guarantee that the Cauchy hypersurface $\text{\textgreek{S}}$
is in the right position so that we can uniquely solve the wave equation
$\square_{g}\text{\textgreek{y}}=0$ on $J^{+}(\text{\textgreek{S}})\cap\mathcal{D}$,
given initial data on $\text{\textgreek{S}}\cap\mathcal{D}$. This
will inevitably lead to a few extra hypotheses regarding the part
of the Cauchy hypersurface that does not lie in the asymptotically
flat region $\mathcal{I}_{as}$. 

We assume that, by altering $\text{\textgreek{S}}$ if necessary (but
keeping $\text{\textgreek{S}}$ as before in the region $\{r\gg1\}$),
we can arrange so that $\mathcal{H}^{-}\subset I^{-}(\text{\textgreek{S}}\cap\mathcal{D})$%
\footnote{i.\,e. the timelike past of $\text{\textgreek{S}}\cap\mathcal{D}$.%
}. With this assumption, for any $p\in J^{+}(\text{\textgreek{S}})\cap\mathcal{D}$
we have $J^{-}(p)\cap\text{\textgreek{S}}\subseteq\text{\textgreek{S}}\cap\mathcal{D}$
(see Lemma \ref{lem:-for-,Inclusion} for a proof of this fact), and
this inclusion implies, according to the arguments of \cite{Ringstroem2009},
that any smooth solution to $\square_{g}\text{\textgreek{y}}=0$ on
$J^{+}(\text{\textgreek{S}})\cap\mathcal{D}$ is uniquely specified
by $(\text{\textgreek{y}},n_{\text{\textgreek{S}}}\text{\textgreek{y}})|_{\text{\textgreek{S}}\cap\mathcal{D}}$,
$n_{\text{\textgreek{S}}}$ denoting the future directed unit normal
to $\text{\textgreek{S}}$.

\paragraph*{\noindent \emph{Transversality assumptions on the stationary vector
field $T$:} }

\noindent It would be convenient for us to work in coordinate charts
where $T$ is a coordinate vector field, and for this reason we have
to state two more assumptions regarding the orbits of $T$. 

First, we assume that $T$ is transversal to $\text{\textgreek{S}}\cap\mathcal{D}$.%
\footnote{Note that for this to be true in the Schwarzschild spacetime, for
instance, we must choose $\text{\textgreek{S}}$ so that it does not
intersect the bifurcate sphere.%
} This implies that $T$ points to the future of $\text{\textgreek{S}}\cap\mathcal{D}$,%
\footnote{Note that T may fail to be timelike!%
} in the sense that following the integral curves of $T$ for some
arbitrarily small positive time, starting from apoint on $\text{\textgreek{S}}\cap\mathcal{D}$,
we are led to a point in $I^{+}(\text{\textgreek{S}}\cap\mathcal{D})$.
\footnote{Proof of the last claim: if this claim was false, and $A\subseteq\text{\textgreek{S}}\cap\mathcal{D}$
was the set of points of $\text{\textgreek{S}}\cap\mathcal{D}$ starting
from which the flow of $T$ did not lead to $I^{+}(\text{\textgreek{S}})$
for small times, then $\partial A\neq\emptyset$ (since $A$ was assumed
to be non empty, but $A$ does not contain any point in the asymptotically
flat region $\{r\gg1\}$). But then, due to the smoothness of $T$,
at the points of $\partial A$ $T$ should be tangent to $\text{\textgreek{S}}$,
which contradicts the assumption of the transversality of $\mathcal{D}\cap\text{\textgreek{S}}$
and $T$.%
}

Therefore, the image of $\text{\textgreek{S}}\cap\mathcal{D}$ under
the flow of $T$ for positive times always lies in $J^{+}(\text{\textgreek{S}})\cap\mathcal{D}$.
At this point, we also need to introduce a second assumption on $T$,
namely that these images actually cover the whole of $J^{+}(\text{\textgreek{S}})\cap\mathcal{D}$.
Given this, one can extend the function $t$ defined before on $\mathcal{I}_{as}$
to the whole of $J^{+}(\text{\textgreek{S}})\cap\mathcal{D}$, by
demanding as before that $t|_{\text{\textgreek{S}}}=0$ and $T(t)=1$. 

By following the flow of $T$ for negative times, $t$ is also defined
on the past images of $\text{\textgreek{S}}\cap\mathcal{D}$ under
the flow of $T$. In this way, $t$ will be defined only on $\mathcal{D}\backslash\mathcal{H}^{-}$
and not on the whole of $\mathcal{D}$ (this will not be a problem
to us, since we will perform our analysis only on $J^{+}(\text{\textgreek{S}})\cap\mathcal{D}$).
We will also denote $\text{\textgreek{S}}_{\text{\textgreek{t}}}\doteq\{t=\text{\textgreek{t}}\}$,
and with this notation we will have $\text{\textgreek{S}}_{0}=\text{\textgreek{S}}\cap\mathcal{D}$.

With these assumptions on $T$, if $x=(x^{1},x^{2},\ldots,x^{d})$
is a local chart on a subset $\mathcal{V}$ of $\text{\textgreek{S}}\cap\mathcal{D}$,
then extending its coordinates functions $x^{i}$ by the requirement
$T(x^{i})=0$ one can construct a local coordinate chart $(t,x^{1},x^{2},\ldots,x^{d})$
on $\mathcal{D}$. In such a local chart, one has $\partial_{t}=T$.
We will mostly work in local charts of this form.

\paragraph*{\noindent \emph{Implications for $\text{\textgreek{S}}\cap\mathcal{H}$: }}

\noindent Due to our previous assumption that $\mathcal{H}^{-}\subset I^{-}(\text{\textgreek{S}}\cap\mathcal{D})$,
we deduce that $(\text{\textgreek{S}}\cap\mathcal{D})\cap\mathcal{H}^{-}=\emptyset$
and hence also $(\text{\textgreek{S}}\cap\mathcal{D})\cap(\mathcal{H}^{+}\cap\mathcal{H}^{-})=\emptyset$.
This, together with the fact that $\text{\textgreek{S}}$ intersects
$\mathcal{H}^{+}$ transversally (since $T$ is tangent to $\mathcal{H}^{+}$
and transversal to $\text{\textgreek{S}}$) implies that $\text{\textgreek{S}}\cap\mathcal{H}^{+}$
is a smooth submanifold of $\text{\textgreek{S}}$. Since $\mathcal{H}^{+}$
is closed (due to its definition), $\text{\textgreek{S}}\cap\mathcal{H}^{+}$
is a closed subset of $\text{\textgreek{S}}\cap\mathcal{D}$. In fact,
$\text{\textgreek{S}}\cap\mathcal{H}^{+}$ is the boundary of the
manifold-with-boundary $\text{\textgreek{S}}\cap\mathcal{D}$. Moreover,
$\mathcal{H}$ can not intersect the asymptotically flat region $\mathcal{I}_{as}$,%
\footnote{\noindent since $\mathcal{H}=\partial\mathcal{D}$ and $\mathcal{D}=J^{-}(\mathcal{I}_{as})\cap J^{+}(\mathcal{I}_{as})$%
} and hence $\text{\textgreek{S}}\cap\mathcal{H}^{+}$ lies in a compact
subset of $\text{\textgreek{S}}$. Therefore, we conclude $\text{\textgreek{S}}\cap\mathcal{H}^{+}$
is a compact submanifold of $\text{\textgreek{S}}$. 

Since $\mathcal{H}^{+}$ is $T$-invariant and the future translates
of $\text{\textgreek{S}}\cap\mathcal{D}$ cover $J^{+}(\text{\textgreek{S}})\cap\mathcal{D}$,
we deduce that the future translates of $\mathcal{H}^{+}\cap\text{\textgreek{S}}$
by $T$ cover $J^{+}(\text{\textgreek{S}})\cap\mathcal{H}=J^{+}(\text{\textgreek{S}})\cap\mathcal{H}^{+}$.
Since $T$ is transversal to $\text{\textgreek{S}}\cap\mathcal{H}^{+}$,
this implies that $J^{+}(\text{\textgreek{S}})\cap\mathcal{H}\simeq\mathbb{R}\times(\text{\textgreek{S}}\cap\mathcal{H}^{+})$
is a smooth submanifold of $\mathcal{D}$ with boundary $\text{\textgreek{S}}\cap\mathcal{H}^{+}$.
Since $\mathcal{H}^{+}\backslash(\mathcal{H}^{+}\cap\mathcal{H}^{-})$
was assumed to be a null hypersurface, $J^{+}(\text{\textgreek{S}})\cap\mathcal{H}$
must also be a null hypersurface if non empty.

\paragraph*{\noindent \emph{Extension of the function r:} }

\noindent We can extend the function $r$ defined before on $\mathcal{S}$
(as the pullback of the polar distance on $\mathbb{R}^{d}$ on each
component of $\mathcal{S}$) to the whole of $\text{\textgreek{S}}\cap\mathcal{D}$,
by simply demanding that it is a Morse function on $\text{\textgreek{S}}\cap\mathcal{D}$
, satisfying $r\ge0$ on $\text{\textgreek{S}}\cap\mathcal{D}$ and
with $r=0$ only on $\text{\textgreek{S}}\cap\mathcal{H}^{+}$ (this
is possible since we showed that $\text{\textgreek{S}}\cap\mathcal{H}^{+}$is
a smooth compact submanifold of $\text{\textgreek{S}}$, and at the
same time it is the boundary of the manifold with boundary $\text{\textgreek{S}}\cap\mathcal{D}$).
Moreover, we demand that $dr|_{\text{\textgreek{S}}\cap\mathcal{H}^{+}}\neq0$,
so that $r$ can also be used as a coordinate function close to the
horizon. 

We can then extend $r$ to the whole of $\mathcal{D}$ by just demanding
that $Tr=0$. Of course, we might not be able to make $r$ a coordinate
function, since if $\text{\textgreek{S}}$ has complicated topology,
Morse theory implies that $\nabla r$ should necessarily vanish somewhere
on $\text{\textgreek{S}}$.

\paragraph*{\noindent \emph{Hyperboloidal hypersurfaces terminating at $\mathcal{I}^{+}$:} }

\noindent We will need to define a special class of hypersurfaces
that will be used frequently throughout this paper.

Let $S$ be a spacelike hypersurface of $\mathcal{D}$. We will say
that $S$ is a hyperboloidal hypersurface terminating at $\mathcal{I}^{+}$,
if in each component of the asymptotically flat region $\mathcal{I}_{as}$
of $(\mathcal{D},g)$, equipped with the coordinate chart $(t,r,\text{\textgreek{sv}})\in\mathbb{R}\times(\mathbb{R}^{d}\backslash B_{R})$,
$S$ is the set $\{(t,r,\text{\textgreek{sv}})|t=f(r,\text{\textgreek{sv}})\}$
for some function $f:\,\mathbb{R}^{d}\backslash B_{R}\rightarrow\mathbb{R}$
such that 
\begin{equation}
\sup_{(r,\text{\textgreek{sv}})\in\mathbb{R}^{d}\backslash B_{R}}\big(|f(r,\text{\textgreek{sv}})-r-2M\log(r-2M)|\big)<+\infty.\label{eq:HyperboloidTerminatingAtNullInfinity}
\end{equation}

\subsubsection{\emph{Assumption \ref{enu:As2}}\label{Assumption 2} \emph{(Killing
horizon with positive surface gravity)}.}

In the case $J^{+}(\text{\textgreek{S}})\cap\mathcal{H}^{+}\neq\emptyset$,
we assume that there exists a non-zero vector field $V$ defined on
$J^{+}(\text{\textgreek{S}})\cap\mathcal{H}^{+}$, which is parallel
to the null generators of $J^{+}(\text{\textgreek{S}})\cap\mathcal{H}^{+}$,
and its flow preserves the induced (degenerate) metric on $I^{+}(\text{\textgreek{S}})\cap\mathcal{H}^{+}$.
We also assume that $V$ commutes with $T$ on $J^{+}(\text{\textgreek{S}})\cap\mathcal{H}$.
Moreover, we assume that there exists a $T$-invariant strictly timelike
vector field $N$ on $\mathcal{D}$, which, when restricted on $J^{+}(\text{\textgreek{S}})\cap\mathcal{H}$,
satisfies 
\begin{equation}
K^{N}(\text{\textgreek{y}})\ge cJ_{\text{\textgreek{m}}}^{N}(\text{\textgreek{y}})N^{\text{\textgreek{m}}}\label{eq:PositiveKN}
\end{equation}
 for any $\text{\textgreek{y}}\in C^{\infty}(\mathcal{D})$, where
$c>0$ is independent of $\text{\textgreek{y}}$. For the notations
on currents, see Section \ref{sub:Currents}.

We will call the vector field $N$ the \emph{red shift} vector field.
The reason for this name is that a vector field of that form was shown
to exist for a general class of Killing horizons with positive surface
gravity by Dafermos and Rodnianski in \cite{DafRod2}. However, here
we will just assume the existence of such a vector field without specifying
the geometric orgin of it. 

Note that we can modify the vector field $N$ away from the horizon,
so that in the region $r\gg1$ it coincides with $T$, and still retain
the bound (\ref{eq:PositiveKN}) on $J^{+}(\text{\textgreek{S}})\cap\mathcal{H}$.%
\footnote{The convexity of the cone of the future timelike vectors over each
point of $\mathcal{D}$ is used in this argument%
} We will hence assume without loss of generality that $N$ has been
chosen so that $N\equiv T$ away from the horizon. 

Due to the smoothness of $N$, there exists an $r_{0}>0$, such that
(\ref{eq:PositiveKN}) also holds (possibly with a smaller constant
$c$ on the right hand side) in a neighborhood of $\mathcal{H}^{+}\backslash(\mathcal{H}^{+}\cap\mathcal{H}^{-})$
in $\mathcal{D}$ of the form $\{r\le r_{0}\}$. For $r\gg1$, of
course, since $N\equiv T$ there, we have $K^{T}(\text{\textgreek{y}})\equiv0$.
Hence, due to the $T-$invariance of $N$ and the compactness of the
sets of the form $\{r\le R\}\cap\text{\textgreek{S}}$, there is also
a (possibly large) constant $C>0$ such that 
\begin{equation}
|K^{N}(\text{\textgreek{y}})|\le C\cdot J_{\text{\textgreek{m}}}^{N}(\text{\textgreek{y}})N^{\text{\textgreek{m}}}
\end{equation}
 everywhere on $\mathcal{D}$ for any $\text{\textgreek{y}}\in C^{\infty}(\mathcal{D})$.

Without loss of generality, we also assume that $r_{0}$ is small
enough so that $dr\neq0$ on $\{r\le3r_{0}\}$. This is possible,
since we have assumed that $dr|_{\text{\textgreek{S}}\cap\mathcal{H}^{+}}\neq0$
and that $\text{\textgreek{S}}\cap\mathcal{H}^{+}$ is compact.

\subsubsection{\emph{Assumption \ref{enu:As3}}\label{Assumption 3} \emph{(Smallness
of the ergoregion).}}

In the case $\mathcal{H}^{+}\neq\emptyset$, we assume that $g(T,T)<0$
in the region $\{r>\frac{r_{0}}{2}\}$ (where $r_{0}$ was defined
before so that $K^{N}(\text{\textgreek{y}})\ge C\cdot J_{\text{\textgreek{m}}}^{N}(\text{\textgreek{y}})N^{\text{\textgreek{m}}}$
in $\{r\le r_{0}\}$). Moreover, we assume that 
\begin{equation}
\sup\{g(T,T)\}<\text{\textgreek{e}}\label{eq:Epsilon}
\end{equation}
 for a fixed small positive constant $\text{\textgreek{e}}$, the
value of which will be specified exactly later on in terms of the
rest of the geometry of our spacetime.%
\footnote{Informally, this means that the ergoregion is small enough so that
the ``red shift'' effect can ``protect'' us from the various difficulties
tied to coping with waves inside the ergoregion, such us those stemming
from the superradiance effect.%
} Note that due to the asymptotic flatness assumption, on $\{r\ge\frac{3r_{0}}{4}\}$
the quantity $g(T,T)$ is uniformly bounded away from zero. 

Since $\{g(T,T)>0\}$ in $\{r\ge r_{0}\}$, one could have adjusted
the previous construction of the vector field $N$ so that it coincides
with $T$ on $\{r\ge2r_{0}\}$. Hence, without loss of generality,
from now on we will assume that $N\equiv T$ on $\{r\ge2r_{0}\}$.

In the case $\mathcal{H}^{+}=\emptyset$, we assume that $T$ is everywhere
timelike on $\mathcal{D}$. 
\begin{rem*}
We should remark again that, in view of Friedman's ergospher instability
(see \cite{Friedman1978}), the existence of a non empty ergoregion
in the case $\mathcal{H}^{+}=\emptyset$ would lead to a violation
of Assumption \hyperref[Assumption 4]{4}. For a rigorous proof of
this result, see our forthcoming \cite{Moschidis}.
\end{rem*}

\subsubsection{\emph{Assumption \ref{enu:As4}}\label{Assumption 4}\emph{ (Boundedness
of the energy of scalar waves).}}

Recall that the leaves of the foliation $\{\text{\textgreek{S}}_{t}|t\ge0\}$
of $\mathcal{D}\cap J^{+}(\text{\textgreek{S}})$ are the level sets
of the function $t$ defined under Assumption~\hyperref[Assumption 1]{1}.
We assume that there exists a positive constant $C>0$ such that for
any smooth solution \textgreek{y} to the wave equation $\square_{g}\text{\textgreek{y}}=0$
on $\mathcal{D}\cap J^{+}(\text{\textgreek{S}})$, and for any $0\le t\le t'$,
we can bound
\begin{equation}
\int_{\text{\textgreek{S}}_{t'}}J_{\text{\textgreek{m}}}^{N}(\text{\textgreek{y}})n^{\text{\textgreek{m}}}\le C\cdot\int_{\text{\textgreek{S}}_{t}}J_{\text{\textgreek{m}}}^{N}(\text{\textgreek{y}})n^{\text{\textgreek{m}}},\label{eq:Boundedness-1}
\end{equation}
 where $n^{\text{\textgreek{m}}}$ is the unit normal to the leaves
of the foliation $\{\text{\textgreek{S}}_{\text{\textgreek{t}}}\}_{\text{\textgreek{t}}\ge0}$
and the integrals are taken with respect to the volume form of the
induced metric.

This assumption will be used as a ``black box'' in the subsequent
sections, in the sense that the mechanisms from which boundedness
stems will not be addressed. Notice that we have not included in our
assumption the boundedness of the horizon flux $\int_{\mathcal{H}^{+}}J_{\text{\textgreek{m}}}^{N}(\text{\textgreek{y}})n_{\mathcal{H}}^{\text{\textgreek{m}}}$
(for our conventions on integration over null hypersurfaces, see section
\ref{sec:Notational-conventions}).
\begin{rem*}
While the boundedness assumption \hyperref[Assumption 4]{4} was stated
only for the level sets of the function $t$, it actually holds for
the level sets of any function of the form $t+\text{\textgreek{q}}$
with achronal level sets, $\text{\textgreek{q}}:\mathcal{D}\rightarrow\mathbb{R}$
being any function supported in $r\gg1$ and satisfying $T\text{\textgreek{q}}=0$.
This fact can be easily deduced by the conservation of the $J_{\text{\textgreek{m}}}^{T}$
current and the fact that $T\equiv N$ is timelike for $r\gg1$. 
\end{rem*}

\subsection{\label{sub:Examples-Assumptions}Examples of spacetimes satisfying
Assumptions \hyperref[Assumption 1]{1}--\hyperref[Assumption 4]{4}}

A first example of a broad class of spacetimes satisfying Assumptions
\hyperref[Assumption 1]{1}--\hyperref[Assumption 4]{4} is the class
of asymptotically flat Lorentzian manifolds of product type $(\mathbb{R}\times\mathcal{N},-dt^{2}+\bar{g})$,
which was already treated by Rodnianski and Tao in \cite{Rodnianski2011}. 

In order to verify that Assumptions \hyperref[Assumption 1]{1}-\hyperref[Assumption 4]{4}
are indeed satisfied on these spacetimes, one has to notice that Assumption
\hyperref[Assumption 1]{1} on stationarity and asymptotic flatness
follows readily from the asymptotic flatness of the Riemannian manifold
$\mathcal{N}$, while Assumptions \hyperref[Assumption 2]{2} and
\hyperref[Assumption 3]{3} are trivially satisfied by the absence
of a black hole region and an ergoregion ($g(\partial_{t},\partial_{t})$
being identical to $-1$ everywhere) on product Lorentzian manifolds.
Finally, since the vector field $\partial_{t}$ gives rise to the
conserved positive definite energy $\int_{\{t=const\}}J_{\text{\textgreek{m}}}^{\partial_{t}}(\text{\textgreek{y}})n_{\{t=const\}}^{\text{\textgreek{m}}}$
for solutions $\text{\textgreek{y}}$ to the wave equation, Assumption
\hyperref[Assumption 4]{4} is also satisfied for any Lorentzian manifold
of product type. 

In addition to the class of asymptotically flat Lorentzian manifolds
of product type, some other classes of spacetimes satisfying the assumptions
of the previous Section, and on which Theorem \ref{thm:Theorem} readily
applies to give a logarithmic decay statement for the local energy
of waves, are the following:

\subsubsection{Black hole spacetimes without ergoregion}

In Section 7 of \cite{DafRod6}, Dafermos and Rodnianski consider
a large class of stationary spacetimes without ergoregion, with horizons
exhibiting positive surface gravity, and with minimal further geometric
and topological assumptions. For these spacetimes, they establish
a boundedness estimate for the energy of solutions to the wave equation
$\square_{g}\text{\textgreek{y}}=0$ with respect to a timelike vector
field and a suitable foliation. Therefore, Assumption \hyperref[Assumption 4]{4}
on energy boundedness holds for spacetimes in this class.

Restricting our attention to the asymptotically flat spacetimes contained
in this class, namely the spacetimes satisfying Assumption \hyperref[Assumption 1]{1},
the reader can easily verify that they also satisfy Assumptions \hyperref[Assumption 2]{2}-\hyperref[Assumption 3]{3}
due to the non degeneracy of the horizons and the absence of an ergoregion.
Thus, Theorem \ref{thm:Theorem} applies in this case to give a logarithmic
decay estimate for the local energy of solutions to $\square_{g}\text{\textgreek{y}}=0$.
This decay rate for the local energy can also be upgraded to a pointwise
logarithmic decay statement for $\text{\textgreek{y}}$, using the
commutation techniques explained in \cite{DafRod6}, thus improving
the results of \cite{DafRod6} on those spacetimes.

This class includes the spherically symmetric and asymptotically flat
black hole spacetimes studied in \cite{Masarik2011}. An example of
spacetimes in the class studied by \cite{Masarik2011} are the black
hole solutions of the $SU(2)$ Einstein-Yang Mills equations constructed
in \cite{Smoller1993}. On these spacetimes, our Theorem \ref{thm:Theorem}
improves the qualitative results of \cite{Masarik2011}.

\subsubsection{Axisymmetric waves on black hole rings and black Saturn}

The topological censorship theorem, established by Friedman, Schleich
and Witt in \cite{Friedman1993}, implies, through the work of Chru\'sciel
and Wald \cite{Chrusciel1994} and Jacobson and Venkataramani \cite{Jacobson1995},
that in the realm of $3+1$ dimensional vacuum spacetimes, spatial
cross sections of black hole horizons must have spherical topology.
In higher dimensions, however, horizons are allowed to become more
complicated. In $4+1$ dimensions, for instance, the intersection
of the black hole horizon with a Cauchy hypersurface in a vacuum spacetime
is allowed to have components diffeomorphic to $\mathbb{S}^{1}\times\mathbb{S}^{2}$.

The first exact solutions of the vacuum Einstein equations exhibiting
horizons with non spherical topology were the black ring solution
constructed by Emparan and Reall in \cite{Emparan2002,Emparan2002a}.
These are $4+1$ dimensional, stationary and axisymmetric asymptotically
flat solutions to the vacuum Einstein equations, with non degenerate
horizons having $\mathbb{S}^{1}\times\mathbb{S}^{2}$ spatial cross
sections. A variation of these spacetimes are the black hole rings
rotating in two different directions presented by Pomeransky and Sen'kov
in \cite{Pomeransky2006}.

The Emparan--Reall spacetimes necessarily posses an ergoregion, and
this fact implies an energy boundedness statement on their domain
of outer communications $\mathcal{D}$ can not be inferred as a corollary
of the general result of Dafermos and Rodnianski in \cite{DafRod6},
discussed in the previous paragraph. However, since in $\mathcal{D}$
the span of the stationary Killing field $T$ and the axisymmetric
Killing field $\text{\textgreek{F}}$ always contains a timelike vector,
axisymmetric functions on these spacetimes have positive $T-$energy
with respect to any spacelike foliation of $\mathcal{D}$. Thus, axisymmetric
solutions to the wave equation $\square\text{\textgreek{y}}=0$ on
the domain of outer communications of the aforementioned spacetimes
do not ``see'' the ergoregion, and for them the results of Dafermos
and Rodnianski discussed in Section 7 of \cite{DafRod6} applies to
establish a non degenerate energy boundedness statement.

Therefore, restricting our attention only to axisymmetric waves on
$\mathcal{D}$, we can easily verify that all Assumptions \hyperref[Assumption 1]{1}--\hyperref[Assumption 4]{4}
(reduced, of course, to the case of axisymmetric waves) are satisfied
in this case. Furthermore, in view of the fact that $\text{\textgreek{F}}$
is a Killing field, the proof of Theorem \ref{thm:Theorem} can be
restricted on the class of axisymmetric functions on $\mathcal{D}$.
Thus, we infer that the local energy of axisymmetric solutions to
the wave equation $\square_{g}\text{\textgreek{y}}=0$ on the domain
of outer communications of the Emparan--Reall spacetimes decay at
least logarithmically in time (and this decay rate can also be made
pointwise, after implementing as before the commutation techniques
explained in \cite{DafRod6}). 

The situation is similar on the ``black Saturn'' solution to the
vacuum Einstein equations, constructed by Elvang and Figueras in \cite{Elvang2007}.
The Elvang--Figueras spacetime is a $4+1$ dimensional, stationary,
axisymmetric and asymptotically flat vacuum spacetime, possessing
a black hole region with non degenerate horizon. The horizon has spatial
cross sections resembling Saturn: one component of the horizon is
homeomorphic to $\mathbb{S}^{3}$, while the other one is homeomorphic
to $\mathbb{S}^{1}\times\mathbb{S}^{2}$. Arguing exactly as before,
the results of Dafermos and Rodnianski presented in \cite{DafRod6}
establish a non degenerate energy boundedness statement for axisymmetric
solutions to the wave equation $\square_{g}\text{\textgreek{y}}=0$
on the domain of outer communications $\mathcal{D}_{s}$ of this spacetime.
Assumptions \hyperref[Assumption 1]{1}--\hyperref[Assumption 4]{4}
are also satisfied, therefore, in this case as before, and we deduce
that axisymmetric solutions to the wave equation $\square\text{\textgreek{y}}=0$
on $\mathcal{D}_{s}$ decay at least logarithmically in time. 

It would be interesting to study in greater detail the axisymmetric
trapping on the Elvang--Figueras spacetime (which appears to be a
more challenging task due to the horiozn having two components), so
as to infer better decay results in this case. Moreover, we should
remark that establishing an energy boundedness statement for those
of the Emparan--Reall and Elvang--Figueras spacetimes possessing a
small ergoregion (which could in principle be achieved by the strategy
followed in \cite{DafRod5} for slowly rotating Kerr spacetimes, if
supperadiance and trapping can be shown to not ``overlap'') would
remove our restriction on axisymmetric waves, making Theorem \ref{thm:Theorem}
applicable to all solutions of the wave equation on these spacetimes.

\subsubsection{$C^{1}$ perturbations of slowly rotating Kerr spacetimes}

Of special interest are also the $C^{1}$ perturbations of the exterior
of slowly rotating Kerr spacetimes, discussed by Dafermos and Rodnianski
in \cite{DafRod5}. In particular, these spacetimes $(\mathcal{M},g)$
possess an atlas, in which the expression for the metric $g$ is $C^{1}$
close to the metric $g_{M}$ expressed on a fixed atlas on the domain
of outer communications (including the horizon) of Schwarzschild spacetime
$(\mathcal{M}_{M},g_{M})$ with mass $M$. In addition, all spacetimes
$(\mathcal{M},g)$ in this class are required to be stationary and
axisymmetric. These spacetimes include the exterior of very slowly
rotating Kerr spacetimes (i.\,e.~with $|a|\ll M$). 

On spacetimes belonging to the above class, it was established in
\cite{DafRod5} that solutions to the wave equation $\square_{g}\text{\textgreek{y}}=0$
have bounded energy with respect to a globally timelike $T-$invariant
vector field $N$ (where $T$ is the stationary Killing field of $(\mathcal{M},g)$)
and a suitable spacelike foliation. Therefore, it can be easily checked
that Assumptions \hyperref[Assumption 1]{1}-\hyperref[Assumption 4]{4}
are satisfied for spacetimes belonging to this class, and thus our
Theorem \ref{thm:Theorem} applies in this case to provide a logarithmic
decay rate for the local energy of waves (which, again, can also be
upgraded to a pointwise logarithmic decay rate). 

It is important to remark that on this class of spacetimes, the nature
of the trapped set can be very different from the one in Schwarzschild
or slowly rotating Kerr spacetimes, because the metric $g$ of $\mathcal{M}$
was required to be only $C^{1}$ close to that of $(\mathcal{M}_{M},g_{M})$.%
\footnote{That is to say, since it is (at least) the $C^{2}$ norm of the difference
of the metrics that would control the change in trapping, and this
norm is not assumed to be small, the nature of the trapped set in
a spacetime $(\mathcal{M},g)$ of this class can be fundamentally
different from the one in Schwarzschild spacetime.%
} Hence, methods relying on the particular form of trapping on Schwarzschild
or Kerr exterior backgrounds would not be applicable in order to obtain
quantitative decay estimates in this class of spacetimes.

\subsubsection{Keir's class of spherically symmetric spacetimes}

Another class of spacetimes which satisfy Assumptions \hyperref[Assumption 1]{1}--\hyperref[Assumption 4]{4}
are the static spherically symmetric spacetimes (including spherically
symmetric ultracompact neutron stars) studied by Keir in \cite{Keir2014}.
Of course, \cite{Keir2014} already establishes a logarithmic decay
estimate for solutions to the wave equation $\square_{g}\text{\textgreek{y}}=0$,
and thus Theorem \ref{thm:Theorem} in this case just recovers some
of the results of \cite{Keir2014}. Since it was shown in \cite{Keir2014}
that in this class of spacetimes the logarithmic decay rates are optimal,
this class also exhibits the optimality of Theorem \ref{thm:Theorem}.
See Section \ref{sub:Sharpness} for more on this.

\subsection{\label{sub:Theorem}Statement of the decay result}

We are now able to state the main result of the current paper:
\begin{thm}
\label{thm:Theorem}Let $(\mathcal{M}^{d+1},g)$, $d\ge3$, be a globally
hyperbolic spacetime with a Cauchy hypersurface $\text{\textgreek{S}}$
satisfying Assumptions \hyperref[Assumption 1]{1}--\hyperref[Assumption 4]{4}
as described in the previous section. Let $\mathcal{D}$ be the domain
of outer communications of $(\mathcal{M},g)$, and let $t:J^{+}(\text{\textgreek{S}})\cap\mathcal{D}\rightarrow[0,+\infty)$
be the time function defined according to Assumption \hyperref[Assumption 1]{1}.
Assume, moreover, that, in case $\mathcal{H}^{+}\neq\emptyset$, the
$\text{\textgreek{e}}>0$ appearing in (\ref{eq:Epsilon}) is sufficiently
small in terms of the geometry of $J^{+}(\text{\textgreek{S}})\cap\mathcal{D}\backslash\{g(T,T)>0\}$.
Then for any $R_{1}>0$, any $0<\text{\textgreek{d}}_{0}\le1$ and
any $m\in\mathbb{N}$, there exists a positive constant $C_{m,\text{\textgreek{d}}_{0}}(R_{1})>0$,
such that every smooth solution $\text{\textgreek{y}}:J^{+}(\text{\textgreek{S}})\cap\mathcal{D}\rightarrow\mathbb{C}$
to the wave equation $\square_{g}\text{\textgreek{y}}=0$ satisfies
the following estimate for any $\text{\textgreek{t}}>0$:

\begin{equation}
\int_{\{t=\text{\textgreek{t}}\}\cap\{r\le R_{1}\}}J_{\text{\textgreek{m}}}^{N}(\text{\textgreek{y}})n^{\text{\textgreek{m}}}\le\frac{C_{m,\text{\textgreek{d}}_{0}}(R_{1})}{\{\log(2+\text{\textgreek{t}})\}^{2m}}\Big(\sum_{j=0}^{m}\int_{\{t=0\}}J_{\text{\textgreek{m}}}^{N}(T^{j}\text{\textgreek{y}})n^{\text{\textgreek{m}}}\Big)+\frac{C_{m}(R_{1})}{\text{\textgreek{t}}^{\text{\textgreek{d}}_{0}}}\int_{\{t=0\}}(1+r)^{\text{\textgreek{d}}_{0}}\cdot J_{\text{\textgreek{m}}}^{N}(\text{\textgreek{y}})n^{\text{\textgreek{m}}},\label{eq:BasicTheorem}
\end{equation}
 $n^{\text{\textgreek{m}}}$ being the future directed normal to the
leaves of the foliation $\{t=const\}$ of $J^{+}(\text{\textgreek{S}})\cap\mathcal{D}$.
See Section \ref{sub:Currents} for the notations on currents.
\end{thm}

\subsection{Corollaries of Theorem \ref{thm:Theorem}}

As a first corollary of Theorem \ref{thm:Theorem}, using \cite{Moschidisc},
we will upgrade the statement of Theorem \ref{thm:Theorem} itself
to a logarithmic decay estimate for the energy of a solution $\text{\textgreek{y}}$
to (\ref{eq:WaveEquation}) with respect to hyperboloidal hypersurfaces
terminating at $\mathcal{I}^{+}$:
\begin{cor}
\label{cor:Corollary}\textbf{(Logarithmic decay of the energy through a hyperboloidal foliation).}
Let $(\mathcal{M},g)$, $\text{\textgreek{S}}$, $\mathcal{D}$ and
$t$ be as in the statement of Theorem \ref{thm:Theorem}. Let $S_{0}$
be any smooth hyperboloidal hypersurface terminating at $\mathcal{I}^{+}$
(according to the definition under Assumption \hyperref[Assumption 1]{1}),
intersecting transversally $\mathcal{H}^{+}\cap J^{+}(\text{\textgreek{S}})$
(if non empty), such that $S_{0}\subset J^{+}(\{t=0\})$. Let also
$S_{t}$ denote the image of $S_{0}$ under the flow of $T$ for time
$t>0$. Then, for any $m\in\mathbb{N}$, there exists a positive constant
$C_{m}>0$ such that every smooth solution $\text{\textgreek{y}}:J^{+}(\text{\textgreek{S}})\cap\mathcal{D}\rightarrow\mathbb{C}$
to the wave equation $\square_{g}\text{\textgreek{y}}=0$ satisfies
the following estimate for any $\text{\textgreek{t}}>0$:

\begin{equation}
\int_{S_{\text{\textgreek{t}}}}J_{\text{\textgreek{m}}}^{N}(\text{\textgreek{y}})n_{S}^{\text{\textgreek{m}}}\le\frac{C_{m}}{\{\log(2+\text{\textgreek{t}})\}^{2m}}\Big(\int_{\{t=0\}}r^{\text{\textgreek{d}}_{0}}J_{\text{\textgreek{m}}}^{N}(\text{\textgreek{y}})n^{\text{\textgreek{m}}}+\sum_{j=0}^{m}\int_{\{t=0\}}J_{\text{\textgreek{m}}}^{N}(T^{j}\text{\textgreek{y}})n^{\text{\textgreek{m}}}\Big),\label{eq:BasicCorollary}
\end{equation}
 $n_{S}^{\text{\textgreek{m}}}$ being the future directed normal
to the leaves of the foliation $S_{\text{\textgreek{t}}}$ and $n^{\text{\textgreek{m}}}$
being the future directed normal to the leaves of the foliation $\{t=const\}$.
\end{cor}
The proof of this corollary is a direct consequence of Theorem \ref{thm:Theorem},
combined with the $r^{p}$-weighted energy method of \cite{Moschidisc}.
It is presented in Section \ref{sec:Proof-of-corollary}.

The quantitative decay rate (\ref{eq:BasicCorollary}) also allows
us to infer that, for solutions $\text{\textgreek{y}}$ to the wave
equation which have initial data with only finite initial $J^{N}-$energy
(and thus with initial data with not necessarily enough decay in $r$
for Theorem \ref{thm:Theorem} to apply), the $J^{N}-$energy of $\text{\textgreek{y}}$
with respect to a foliation of hyperboloidal hypersurfaces terminating
at $\mathcal{I}^{+}$ still decays to $0$, albeit in a non quantitative
manner. In particular, we can establish the following result:
\begin{cor}
\label{cor:CorollaryQualititative}\textbf{(Non-quantitative hyperboloidal decay in the energy class).}
Let $(\mathcal{M},g)$, $\text{\textgreek{S}}$, $\mathcal{D}$ and
$t$ be as in the statement of Theorem \ref{thm:Theorem}. Let $S_{0}$
be any smooth hyperboloidal hypersurface terminating at $\mathcal{I}^{+}$
(according to the definition under Assumption \hyperref[Assumption 1]{1}),
intersecting transversally $\mathcal{H}^{+}\cap J^{+}(\text{\textgreek{S}})$
(if non-empty), such that $S_{0}\subset J^{+}(\{t=0\})$. Let also
$S_{t}$ denote the image of $S_{0}$ under the flow of $T$ for time
$t>0$. Then, for every solution $\text{\textgreek{y}}:J^{+}(\text{\textgreek{S}})\cap\mathcal{D}\rightarrow\mathbb{C}$
to the wave equation $\square_{g}\text{\textgreek{y}}=0$ on $(\mathcal{D},g)$
with $\int_{t=0}J_{\text{\textgreek{m}}}^{N}(\text{\textgreek{y}})n^{\text{\textgreek{m}}}<\infty$,
$n^{\text{\textgreek{m}}}$ being the future directed normal to the
leaves of the foliation $\{t=const\}$, the following is true:

\begin{equation}
\lim_{\text{\textgreek{t}}\rightarrow+\infty}\int_{S_{\text{\textgreek{t}}}}J_{\text{\textgreek{m}}}^{N}(\text{\textgreek{y}})n_{S}^{\text{\textgreek{m}}}=0,\label{eq:Corollary2}
\end{equation}
 $n_{S}^{\text{\textgreek{m}}}$ being the future directed normal
to the leaves of the foliation $S_{\text{\textgreek{t}}}$.

If the stationary vector field $T$ is non spacelike in $\mathcal{D}$
and timelike away from the future event horizon $\mathcal{H}^{+}$,
then 
\begin{equation}
\lim_{\text{\textgreek{t}}\rightarrow+\infty}\int_{S_{\text{\textgreek{t}}}}J_{\text{\textgreek{m}}}^{T}(\text{\textgreek{y}})n_{S}^{\text{\textgreek{m}}}=0\label{eq:TenergyQualitativedecayrate}
\end{equation}
 also holds for solutions $\text{\textgreek{y}}$ to $\square_{g}\text{\textgreek{y}}=0$
satisfying merely $\int_{t=0}J_{\text{\textgreek{m}}}^{T}(\text{\textgreek{y}})n^{\text{\textgreek{m}}}<\infty$. 
\end{cor}
The proof of Corollary \ref{cor:CorollaryQualititative} is presented
in Section \ref{sec:Proof-of-Corollary2}.

In spacetimes $(\mathcal{M},g)$ where the stationary vector field
$T$ is everywhere non spacelike on the domain of outer communications
$\mathcal{D}$, and timelike away from the horizon $\mathcal{H}$,%
\footnote{In this case, the $J^{T}-$energy is coercive.%
} the qualitative decay statement (\ref{eq:TenergyQualitativedecayrate})
for solutions to the wave equation with finite initial $J^{T}$-energy
readily leads to a few interesting scattering results. Before stating
these results, we will first define the notion of the \emph{future
radiation field }on the spacetimes under consideration.

Repeating the proof of Theorem 7.1 of \cite{Moschidisc} for spacetimes
$(\mathcal{M},g)$ satisfying Assumption \hyperref[Assumption 1]{1}
on asymptotic flatness,%
\footnote{The only difference between spacetimes satisfying Assumption \hyperref[Assumption 1]{1}
and the ones considered in the statement of Theorem 7.1 of \cite{Moschidisc}
is in the regularity of the metric $g$ near $\mathcal{I}^{+}$.%
} any smooth solution $\text{\textgreek{y}}$ to the wave equation
$\square_{g}\text{\textgreek{y}}=0$ on $\mathcal{D}$ (the domain
of outer communications of an asymptotically flat end of $\mathcal{M}$)
with compactly supported initial data on $\text{\textgreek{S}}\cap\mathcal{D}$
($\text{\textgreek{S}}$ being a Cauchy hypersurface of $\mathcal{M}$
according to Assumption \hyperref[Assumption 1]{1}) gives rise to
a well defined radiation field $\text{\textgreek{Y}}_{\mathcal{I}^{+}}$
on the future null infinity $\mathcal{I}^{+}$ of $(\mathcal{D},g)$.
The future radiation field $\text{\textgreek{Y}}_{\mathcal{I}^{+}}$
is defined as the limit in the $(u,r,\text{\textgreek{sv}})$ coordinate
system in the region $\{r\gg1\}$ (see Section \ref{sec:ConstructionOfTheDoubleNullCoordinateSystem}
for how to construct this coordinate system in a spacetime with the
asymptotics (\ref{eq:metric}) for the metric) as follows: 
\begin{defn*}
For any smooth solution $\text{\textgreek{y}}$ to the wave equation
$\square_{g}\text{\textgreek{y}}=0$ on $\mathcal{D}$ with compactly
supported initial data on $\text{\textgreek{S}}\cap\mathcal{D}$,
the future radiation field $\text{\textgreek{Y}}_{\mathcal{I}^{+}}$
on $\mathcal{I}^{+}$ is defined as the limit (in each connected component
of $\mathcal{I}_{as}$) 
\begin{equation}
\text{\textgreek{Y}}_{\mathcal{I}^{+}}(u,\text{\textgreek{sv}})=\lim_{r\rightarrow+\infty}r^{\frac{d-1}{2}}\text{\textgreek{y}}(u,r,\text{\textgreek{sv}}),
\end{equation}
 where the convergence refers to the $H_{loc}^{1}(\mathbb{R}\times\mathbb{S}^{d-1},du^{2}+g_{\mathbb{S}^{d-1}})$
topology. 
\end{defn*}
The $J^{T}$-energy flux of $\text{\textgreek{y}}$ on $\mathcal{I}^{+}$
is then defined in the following way: 
\begin{equation}
\int_{\mathcal{I}^{+}}J_{\text{\textgreek{m}}}^{T}(\text{\textgreek{y}})n_{\mathcal{I}^{+}}^{\text{\textgreek{m}}}\doteq\sum_{components\, of\,\mathcal{I}^{+}}\int_{-\infty}^{\infty}\int_{\mathbb{S}^{d-1}}\big(\partial_{u}\text{\textgreek{Y}}_{\mathcal{I}^{+}}(u,\text{\textgreek{sv}})\big)^{2}\, d\text{\textgreek{sv}}du.\label{eq:EnergyI+}
\end{equation}

A priori, the energy (\ref{eq:EnergyI+}) might be infinite. However,
if the vector field $T$ is everywhere non spacelike on $\mathcal{D}$,
and timelike away from the horizon $\mathcal{H}$, then the conservation
of the $J^{T}$-current easily implies that 
\begin{equation}
\int_{\mathcal{I}^{+}}J_{\text{\textgreek{m}}}^{T}(\text{\textgreek{y}})n_{\mathcal{I}^{+}}^{\text{\textgreek{m}}}\le\int_{t=0}J_{\text{\textgreek{m}}}^{T}(\text{\textgreek{y}})n^{\text{\textgreek{m}}}.\label{eq:BoundI+}
\end{equation}
Through the conservation of the $J^{T}$-current one can bound the
energy flux on the future horizon by the initial energy of $\text{\textgreek{y}}$.
In particular, one can bound: 
\begin{equation}
\int_{\mathcal{H}^{+}\cap\{t\ge0\}}J_{\text{\textgreek{m}}}^{T}(\text{\textgreek{y}})n_{\mathcal{H}^{+}}^{\text{\textgreek{m}}}+\int_{\mathcal{I}^{+}}J_{\text{\textgreek{m}}}^{T}(\text{\textgreek{y}})n_{\mathcal{I}^{+}}^{\text{\textgreek{m}}}\le\int_{t=0}J_{\text{\textgreek{m}}}^{T}(\text{\textgreek{y}})n^{\text{\textgreek{m}}}.\label{eq:BoundHorizon}
\end{equation}
 (see Section \ref{sec:Notational-conventions} on integration over
null hypersurfaces).

By a standard density argument, one can define the future radiation
field $\text{\textgreek{Y}}_{\mathcal{I}^{+}}$ for solutions $\text{\textgreek{y}}$
of $\square_{g}\text{\textgreek{y}}=0$ satisfying merely $\int_{t=0}J_{\text{\textgreek{m}}}^{T}(\text{\textgreek{y}})n^{\text{\textgreek{m}}}<\infty$
as follows:
\begin{defn*}
For any solution $\text{\textgreek{y}}$ to the wave equation $\square_{g}\text{\textgreek{y}}=0$
on $\mathcal{D}$ with $H_{loc}^{1}$ initial data on $\text{\textgreek{S}}\cap\mathcal{D}$
satisfying $\int_{t=0}J_{\text{\textgreek{m}}}^{T}(\text{\textgreek{y}})n^{\text{\textgreek{m}}}<\infty$,
the future radiation field $\text{\textgreek{Y}}_{\mathcal{I}^{+}}$
is defined as the limit (in the topology defined by the energy norm
(\ref{eq:EnergyI+})) 
\begin{equation}
\text{\textgreek{Y}}_{\mathcal{I}^{+}}=\lim_{n}\text{\textgreek{Y}}_{n,\mathcal{I}^{+}},
\end{equation}
 where $\text{\textgreek{Y}}_{n,\mathcal{I}^{+}}$ is the future radiation
field of a sequence of smooth functions $\text{\textgreek{y}}_{n}$
solving $\square_{g}\text{\textgreek{y}}_{n}=0$ on $\mathcal{D}$
with $(\text{\textgreek{y}}_{n},T\text{\textgreek{y}}_{n})|_{\text{\textgreek{S}}\cap\mathcal{D}}$
compactly supported and approximating $(\text{\textgreek{y}},T\text{\textgreek{y}})|_{\text{\textgreek{S}}\cap\mathcal{D}}$
in the $\int_{t=0}J_{\text{\textgreek{m}}}^{T}(\cdot)n^{\text{\textgreek{m}}}$
norm as $n\rightarrow+\infty$. 
\end{defn*}
The existence and uniqueness of $\text{\textgreek{Y}}_{\mathcal{I}^{+}}$
in this case follows by applying (\ref{eq:BoundHorizon}) for the
sequence $\text{\textgreek{y}}_{n+1}-\text{\textgreek{y}}_{n}$.

The qualitative decay statement for the $J^{T}$-energy on hyperboloids
provided by Corollary \ref{cor:CorollaryQualititative} leads to the
following scattering result:
\begin{cor}
\label{cor:CorollaryScattering}\textbf{(Asymptotic completeness for spacetimes without ergoregion).}
Let $(\mathcal{M},g)$, $\text{\textgreek{S}}$, $\mathcal{D}$ and
$t$ be as in the statement of Theorem \ref{thm:Theorem}, and moreover
suppose that the stationary vector field $T$ on $\mathcal{D}$ is
everywhere non spacelike, and timelike away from the horizon $\mathcal{H}$.
Then, for every solution $\text{\textgreek{y}}:J^{+}(\text{\textgreek{S}})\cap\mathcal{D}\rightarrow\mathbb{R}$
to the wave equation $\square_{g}\text{\textgreek{y}}=0$ on $(\mathcal{D},g)$
with $\int_{t=0}J_{\text{\textgreek{m}}}^{T}(\text{\textgreek{y}})n^{\text{\textgreek{m}}}<\infty$,
$n^{\text{\textgreek{m}}}$ being the future directed normal to the
leaves of the foliation $\{t=const\}$, inequality (\ref{eq:BoundHorizon})
is actually an equality:

\begin{equation}
\int_{\mathcal{H}^{+}\cap\{t\ge0\}}J_{\text{\textgreek{m}}}^{T}(\text{\textgreek{y}})n_{\mathcal{H}^{+}}^{\text{\textgreek{m}}}+\int_{\mathcal{I}^{+}}J_{\text{\textgreek{m}}}^{T}(\text{\textgreek{y}})n_{\mathcal{I}^{+}}^{\text{\textgreek{m}}}=\int_{t=0}J_{\text{\textgreek{m}}}^{T}(\text{\textgreek{y}})n^{\text{\textgreek{m}}}.\label{eq:CorollaryScattering}
\end{equation}

Furthermore, for any scattering data set $(\text{\textgreek{y}}|_{\mathcal{H}^{+}},\text{\textgreek{Y}}_{\mathcal{I}^{+}})$
with $\int_{\mathcal{H}^{+}}J_{\text{\textgreek{m}}}^{T}(\text{\textgreek{y}})n_{\mathcal{H}^{+}}^{\text{\textgreek{m}}}+\int_{\mathcal{I}^{+}}J_{\text{\textgreek{m}}}^{T}(\text{\textgreek{y}})n_{\mathcal{I}^{+}}^{\text{\textgreek{m}}}<\infty$,
one can find a solution $\text{\textgreek{y}}$ to the wave equation
on $\mathcal{D}$ with $\int_{t=0}J_{\text{\textgreek{m}}}^{T}(\text{\textgreek{y}})n^{\text{\textgreek{m}}}<\infty$,
for which the restriction on $\mathcal{H}^{+}$ and the resulting
future radiation field coincide with the given scattering data.

Thus, any solution $\text{\textgreek{y}}$ to $\square_{g}\text{\textgreek{y}}=0$
with $\int_{t=0}J_{\text{\textgreek{m}}}^{T}(\text{\textgreek{y}})n^{\text{\textgreek{m}}}<\infty$
on $\mathcal{D}\cap J^{+}(\text{\textgreek{S}})$ is uniquely determined
by its scattering data on $\mathcal{H}^{+},\mathcal{I}^{+}$ (i.\,e.
by $\text{\textgreek{y}}|_{\mathcal{H}^{+}}$, $\text{\textgreek{Y}}_{\mathcal{I}^{+}}$).\end{cor}
\begin{rem*}
Let $(\mathcal{M},g)$ and $\mathcal{D}$ be as in the statement of
Theorem \ref{thm:Theorem} (now $(\mathcal{D},g)$ is allowed to have
a non-empty ergoregion), and assume that $(\mathcal{M},g)$ is also
axisymmetric with axisymmetric Killing field $\text{\textgreek{F}}$.
If the span of $\{T,\text{\textgreek{F}}\}$ is timelike everywhere
on $\mathcal{D}\backslash\mathcal{H}$, then Corollary \ref{cor:CorollaryScattering}
is valid when restricted to the class of axisymmetric solutions to
(\ref{eq:WaveEquation}) on $(\mathcal{M},g)$.
\end{rem*}
The proof of Corollary \ref{cor:CorollaryScattering} will be furnished
in Section \ref{sec:Proof-of-Corollary2}.

Let us notice that Corollary \ref{cor:CorollaryScattering} reproves
the classical asymptotic completeness result on Schwarzschild spacetime
(see \cite{Dimock1985,Dimock1986,Dimock1987,Nicolas1995}), and applies
on a large class of asymptotically flat spacetimes, including, for
example, the black hole exterior spacetimes constructed in Section
7 of \cite{DafRod6}, the black hole solutions of the $SU(2)$ Einstein-Yang
Mills equations of \cite{Smoller1993} and the spherically symmetric
spacetimes studied by Keir in \cite{Keir2014}. Furthermore, restricted
on the class of axisymmetric solutions of equation (\ref{eq:WaveEquation})
(see the remark below Corollary \ref{cor:CorollaryScattering}), it
also applies on the subsextremal Kerr exterior, the Emparan--Reall
black ring (see \cite{Emparan2002,Emparan2002a}) and the Elvang--Figueras
black Saturn (see \cite{Elvang2007}). These spacetimes were discussed
in Section \ref{sub:Examples-Assumptions}. 

Notice, however, that Corollary \ref{cor:CorollaryScattering} does
not apply on the general class of solutions to (\ref{eq:WaveEquation})
on spacetimes with an ergoregion, and, thus, it can not reprove the
results of \cite{Dafermos2014} on the subextremal Kerr family (with
Kerr parameters in the full subextremal range $|a|<M$) without axisymmetry.

\subsection{\label{sub:SketchOfProof}Sketch of the proof of Theorem \ref{thm:Theorem}}

In order to make the reasoning behind the subsequent proof of Theorem
\ref{thm:Theorem} more transparent, we will first present a sketch
of this proof on a heuristic level. The final step of the proof of
Theorem \ref{thm:Theorem} is the implementation of a frequency interpolation
scheme, similar to the one appearing already in \cite{Holzegel2013}. 

\medskip{}

\textbf{\emph{Frequency decomposition:}} Given any solution $\text{\textgreek{y}}$
to the wave equation $\square_{g}\text{\textgreek{y}}=0$ as in the
statement of Theorem \ref{thm:Theorem}, one starts by splitting it
into two parts, $\text{\textgreek{y}}_{\le\text{\textgreek{w}}_{+}}$
and $\text{\textgreek{y}}_{\ge\text{\textgreek{w}}_{+}}$, where $\text{\textgreek{y}}_{\le\text{\textgreek{w}}_{+}}$
has frequency support with respect to the $t$ coordinate in $\{|\text{\textgreek{w}}|\le\text{\textgreek{w}}_{+}\}$,
and similarly $\text{\textgreek{y}}_{\ge\text{\textgreek{w}}_{+}}$
has frequency support contained in $\{|\text{\textgreek{w}}|\ge\text{\textgreek{w}}_{+}\}$.
The parameter $\text{\textgreek{w}}_{+}>0$ will be determined later. 

Since the Killing vector field $T$ is not necessarily timelike near
$\mathcal{H}^{+}$ and we have not assumed that the red-shift vector
field $N$ satisfies some ``nice'' commutation properties with the
wave operator $\square_{g}$, the boundedness assumption for the energy
of $\text{\textgreek{y}}$ alone is not even enough to exclude the
pointwise exponential growth of \textgreek{y}. Thus, we will need
to perform a suitable cut-off in physical space in order for the Fourier
transform of $\text{\textgreek{y}}$ in $t$ to make sense. Such a
cut-off gives rise to well understood error terms in all the subsequent
estimates (these issues being conceptually identical to the ones appearing
in \cite{DafRod5}), and the treatment of these terms will be carried
out in a fashion similar to \cite{DafRod5} in Section \ref{sub:Frequency-cut-off}.
The boundedness assumption \hyperref[Assumption 4]{4} will play a
crucial role in estimating these error terms. But for the sake of
simplicity of the current sketch of the proof, let us momentarily
ignore the need for a physical space cut-off with all its corresponding
implications. 

\medskip{}

\textbf{\emph{The interpolation argument:}} The components $\text{\textgreek{y}}_{\ge\text{\textgreek{w}}_{+}}$
and $\text{\textgreek{y}}_{\le\text{\textgreek{w}}_{+}}$ of $\text{\textgreek{y}}$
also satisfy the wave equation, due to the stationarity of the metric
and the fact that a frequency cut-off corresponds to convolution with
a function of $t$ in physical space. One can hence use the boundedness
assumption \hyperref[Assumption 4]{4} for the high frequency part
$\text{\textgreek{y}}_{\ge\text{\textgreek{w}}_{+}}$ to conclude
that for any $\text{\textgreek{t}}>0$: 
\begin{equation}
\int_{\{t=\text{\textgreek{t}}\}\cap\{r\le R_{1}\}}J_{\text{\textgreek{m}}}^{N}(\text{\textgreek{y}}_{\ge\text{\textgreek{w}}_{+}})n^{\text{\textgreek{m}}}\le C\cdot\int_{\{t=0\}}J_{\text{\textgreek{m}}}^{N}(\text{\textgreek{y}})n^{\text{\textgreek{m}}}.\label{eq:SketchBoundedness}
\end{equation}
(Notice that the right hand side of (\ref{eq:SketchBoundedness})
contains the energy of $\text{\textgreek{y}}$ and not that of $\text{\textgreek{y}}_{\ge\text{\textgreek{w}}_{+}}$,
which is a result of the technical issues concerning the physical
space cut-off and the frequency decomposition that we will supress
in this section). Since the vector field $T$ is Killing, (\ref{eq:SketchBoundedness})
also holds for the pair $T^{m}\text{\textgreek{y}}_{\ge\text{\textgreek{w}}_{+}},T^{m}\text{\textgreek{y}}$
in place of $\text{\textgreek{y}}_{\ge\text{\textgreek{w}}_{+}},\text{\textgreek{y}}$,
and thus, due to the fact that the frequency support of $\text{\textgreek{y}}_{\ge\text{\textgreek{w}}_{+}}$
is contained in $\{|\text{\textgreek{w}}|\ge\text{\textgreek{w}}_{+}\}$,
one can prove a statement of the form

\begin{equation}
\int_{\{t=\text{\textgreek{t}}\}\cap\{r\le R_{1}\}}J_{\text{\textgreek{m}}}^{N}(\text{\textgreek{y}}_{\ge\text{\textgreek{w}}_{+}})n^{\text{\textgreek{m}}}\le\frac{C}{\text{\textgreek{w}}_{+}^{2m}}\sum_{j=0}^{m}\int_{\{t=0\}}J_{\text{\textgreek{m}}}^{N}(T^{j}\text{\textgreek{y}})n^{\text{\textgreek{m}}}.\label{eq:BoundednessHighFrequencies}
\end{equation}

As for the low frequency part $\text{\textgreek{y}}_{\le\text{\textgreek{w}}_{+}}$,
we will show that its local energy decays in time at a slow polynomial
rate:

\begin{equation}
\int_{\{t=\text{\textgreek{t}}\}\cap\{r\le R_{1}\}}J_{\text{\textgreek{m}}}^{N}(\text{\textgreek{y}}_{\le\text{\textgreek{w}}_{+}})n^{\text{\textgreek{m}}}\le\frac{C\cdot e^{C\text{\textgreek{w}}_{+}}}{\text{\textgreek{t}}^{\text{\textgreek{d}}_{0}}}\int_{\{t=0\}}J_{\text{\textgreek{m}}}^{N}(\text{\textgreek{y}})n^{\text{\textgreek{m}}}+\frac{C}{\text{\textgreek{t}}^{\text{\textgreek{d}}_{0}}}\int_{\{t=0\}}r^{\text{\textgreek{d}}_{0}}\cdot J_{\text{\textgreek{m}}}^{N}(\text{\textgreek{y}})n^{\text{\textgreek{m}}}.\label{eq:DecayLowFrequencies}
\end{equation}
 Notice already the exponential dependence of the constant of the
first term of the right hand side of (\ref{eq:DecayLowFrequencies})
on $\text{\textgreek{w}}_{+}$. Assuming for a moment that (\ref{eq:DecayLowFrequencies})
holds, the logarithmic decay of the local energy of $\text{\textgreek{y}}$
is deduced as follows: Combining (\ref{eq:BoundednessHighFrequencies})
with (\ref{eq:DecayLowFrequencies}), and recalling that $\text{\textgreek{y}}=\text{\textgreek{y}}_{\le\text{\textgreek{w}}_{+}}+\text{\textgreek{y}}_{\ge\text{\textgreek{w}}_{+}}$,
we obtain for any $\text{\textgreek{t}}>0$

\begin{equation}
\int_{\{t=\text{\textgreek{t}}\}\cap\{r\le R_{1}\}}J_{\text{\textgreek{m}}}^{N}(\text{\textgreek{y}})n^{\text{\textgreek{m}}}\le\frac{C\cdot e^{C\text{\textgreek{w}}_{+}}}{\text{\textgreek{t}}^{\text{\textgreek{d}}_{0}}}\int_{\{t=0\}}J_{\text{\textgreek{m}}}^{N}(\text{\textgreek{y}})n^{\text{\textgreek{m}}}+\frac{C}{\text{\textgreek{t}}^{\text{\textgreek{d}}_{0}}}\int_{\{t=0\}}r\cdot J_{\text{\textgreek{m}}}^{N}(\text{\textgreek{y}})n^{\text{\textgreek{m}}}+\frac{C}{\text{\textgreek{w}}_{+}^{2m}}\sum_{j=0}^{m}\int_{\{t=0\}}J_{\text{\textgreek{m}}}^{N}(T^{j}\text{\textgreek{y}})n^{\text{\textgreek{m}}}.\label{eq:InterpolationHeuristics}
\end{equation}
Thus, choosing $\text{\textgreek{w}}_{+}\sim\frac{1}{2C}\log(2+\text{\textgreek{t}})$
yields (if $\text{\textgreek{t}}\ge1)$ the desired logarithmic decay
estimate:

\begin{equation}
\int_{\{t=\text{\textgreek{t}}\}\cap\{r\le R_{1}\}}J_{\text{\textgreek{m}}}^{N}(\text{\textgreek{y}})n^{\text{\textgreek{m}}}\le\frac{C}{\{\log(2+\text{\textgreek{t}})\}^{2m}}\Big(\sum_{j=0}^{m}\int_{t=0}J_{\text{\textgreek{m}}}^{N}(T^{j}\text{\textgreek{y}})n^{\text{\textgreek{m}}}\Big)+\frac{C}{\text{\textgreek{t}}^{\text{\textgreek{d}}_{0}}}\int_{t=0}r^{\text{\textgreek{d}}_{0}}\cdot J_{\text{\textgreek{m}}}^{N}(\text{\textgreek{y}})n^{\text{\textgreek{m}}}.
\end{equation}
 
\begin{rem*}
In this interpolation procedure it is evident that, exactly as expected,
the ``slow'' log decay is caused by the high frequency part of $\text{\textgreek{y}}$,
since the low frequency one decays polynomially. Such an interpolation
scheme was also used in \cite{Holzegel2013} (see also \cite{Burq1998}).

\medskip{}

\end{rem*}
Thus, the proof will be complete after establishing a uniform decay
statement of the form (\ref{eq:DecayLowFrequencies}) for the low
frequency part $\text{\textgreek{y}}_{\le\text{\textgreek{w}}_{+}}$.

\smallskip{}

\textbf{\emph{Polynomial decay for $\text{\textgreek{y}}_{\le\text{\textgreek{w}}_{+}}$:}}
In order to obtain an estimate of the form (\ref{eq:DecayLowFrequencies}),
we will use the results of \cite{Moschidisc}, which generalise the
$r^{p}$-weighted energy method of Dafermos and Rodnianski \cite{DafRod7}
to a more general setting of asymptotically flat backgrounds. Here,
the asymptotic form of the metric (\ref{eq:metric}) comes into play.
Using the results of \cite{Moschidisc}, the problem of proving (\ref{eq:DecayLowFrequencies})
reduces to proving an integrated local energy decay statement of the
form

\begin{equation}
\int_{0}^{\infty}\Big\{\int_{\{t=\text{\textgreek{t}}\}\cap\{r\le R\}}J_{\text{\textgreek{m}}}^{N}(\text{\textgreek{y}}_{\le\text{\textgreek{w}}_{+}})n^{\text{\textgreek{m}}}\Big\}\, d\text{\textgreek{t}}\le C(R)\cdot e^{C(R)\cdot\text{\textgreek{w}}_{+}}\int_{\{t=0\}}J_{\text{\textgreek{m}}}^{N}(\text{\textgreek{y}})n^{\text{\textgreek{m}}}.\label{eq:ILEDLowFrequencies}
\end{equation}
Thus, the rest of the proof is centered around establishing the estimate
(\ref{eq:ILEDLowFrequencies}).

Heuristically, the reason that (\ref{eq:ILEDLowFrequencies}) can
be obtained lies in the fact that $\text{\textgreek{y}}_{\le\text{\textgreek{w}}_{+}}$
has bounded frequency support, and hence it does not experience trapping:
Its energy will eventually ``leak'' (at a rate shrinking exponentially
in $\text{\textgreek{w}}_{+}$ as $\text{\textgreek{w}}_{+}\rightarrow\infty$)
to $\mathcal{I}^{+}$ or through $\mathcal{H}^{+}$. This phenomenon
is captured by a Carleman-type inequality which is roughly of the
form
\begin{equation}
\begin{split}\int_{\mathcal{R}(0,T)\cap\{r_{1}\le r\le R\}}e^{s\cdot\text{\textgreek{w}}_{+}w} & \Big\{|\text{\textgreek{y}}_{\le\text{\textgreek{w}}_{+}}|^{2}+J_{\text{\textgreek{m}}}^{N}(\text{\textgreek{y}}_{\le\text{\textgreek{w}}_{+}})n^{\text{\textgreek{m}}}\Big\}\le\\
\le & C(R)(1+(s\text{\textgreek{w}}_{+})^{2})\int_{\mathcal{R}(0,T)\cap\{\frac{r_{1}}{2}\le r\le r_{1}\})}e^{s\cdot\text{\textgreek{w}}_{+}w}\big\{|\text{\textgreek{y}}_{\le\text{\textgreek{w}}_{+}}|^{2}+J_{\text{\textgreek{m}}}^{N}(\text{\textgreek{y}}_{\le\text{\textgreek{w}}_{+}})n^{\text{\textgreek{m}}}\big\}+\\
 & +C(R)(1+(s\text{\textgreek{w}}_{+})^{2})\int_{\mathcal{R}(0,T)\cap\{R\le r\le R+1\})}e^{s\cdot\text{\textgreek{w}}_{+}w}\big\{|\text{\textgreek{y}}_{\le\text{\textgreek{w}}_{+}}|^{2}+J_{\text{\textgreek{m}}}^{N}(\text{\textgreek{y}}_{\le\text{\textgreek{w}}_{+}})n^{\text{\textgreek{m}}}\big\}+\\
 & +C(R)\cdot e^{C(R)\cdot s\text{\textgreek{w}}_{+}}\int_{\{t=0\}}J_{\text{\textgreek{m}}}^{N}(\text{\textgreek{y}})n^{\text{\textgreek{m}}}.
\end{split}
\label{eq:CarlemannForLowFrequencies}
\end{equation}
In (\ref{eq:CarlemannForLowFrequencies}), $r_{1}>0$ should be considered
small and $R>0$ large, $\mathcal{R}(0,T)$ denotes the region $\{0\le t\le T\}$
for some arbitrary $T>0$, and $w$ is a suitable function on $\mathcal{D}$
which for the purpose of this discussion can be assumed to be a strictly
increasing (but not bounded!)~positive function of $r$. Finally,
$s>0$ is a large constant. 

The extraction of (\ref{eq:CarlemannForLowFrequencies}) follows closely
the derivation of a similar Carleman type inequality in \cite{Rodnianski2011},
using the multiplier method. The boundedness assumption \hyperref[Assumption 4]{4}
is used in an essential way to control boundary terms at $\{t=T\}$
appearing in this procedure. 

Notice that due to the fact that $w$ has been chosen strictly increasing
in $r$ (at least near the horizon and near null infinity), the boundary
term near the horizon in the right hand side of (\ref{eq:CarlemannForLowFrequencies})
(namely the integral over $r\sim r_{1}$) carries a weight $e^{s\text{\textgreek{w}}_{+}\cdot w}$
that is small in comparison to the weights $e^{s\text{\textgreek{w}}_{+}\cdot w}$
in the bulk integral of the left hand side. The opposite happens with
the boundary term near spacelike infinity ($r\sim R$), where the
$e^{s\text{\textgreek{w}}_{+}w}$ weights should be considered large
in comparison to the weights in the integral of the left hand side.

In order to attain the ILED statement (\ref{eq:ILEDLowFrequencies}),
we would like to absorb the boundary terms of the right hand side
of (\ref{eq:CarlemannForLowFrequencies}) by the left hand side, in
a process producing error terms that can be bounded by the initial
energy of $\text{\textgreek{y}}$. 

\medskip{}

\noindent \emph{Treatment of the boundary terms near the horizon:}
In order to dispense with the first term of the right hand side of
(\ref{eq:CarlemannForLowFrequencies}), one uses the positivity properties
of the red shift vector field $N$ near the horizon (according to
Assumption \hyperref[Assumption 2]{2}), as well as the fact that
the weight $e^{s\text{\textgreek{w}}_{+}w}$ in the $r\sim r_{1}$
region has much smaller values than in the $\{r\gg r_{1}\}$ region
if $s$ is chosen sufficiently large. Let us remark at this point
that Assumption \hyperref[Assumption 3]{3} is also used here: The
proof of the Carleman inequality (\ref{eq:CarlemannForLowFrequencies})
requires that the set $\{r_{1}\le r\le R\}$ does not intersect the
ergoregion, while, in order to use Assumption \hyperref[Assumption 2]{2}
in order to deal with the first term of the right hand side of (\ref{eq:CarlemannForLowFrequencies}),
it is necessary that $r_{0}\gtrsim r_{1}$ (where $r_{0}$ is the
parameter appearing in Assumption \hyperref[Assumption 2]{2}). Thus,
it is essential in this step that the ergoregion is contained in the
set $\{r\lesssim r_{0}\}$, and this fact is guaranteed by Assumption
\hyperref[Assumption 3]{3}.%
\footnote{\noindent We should also note that in the case where $\mathcal{H}^{+}=\emptyset$,
the first term of the right hand side of (\ref{eq:CarlemannForLowFrequencies})
can be dropped, and the integral in the left hand side of (\ref{eq:CarlemannForLowFrequencies})
is over the whole region $\mathcal{R}(0,T)\cap\{r\le R\}$.%
}

In particular, according to (\ref{eq:PositiveKN}), the $K^{N}$ current
is positive near the horizon and controls the first term of the right
hand side of (\ref{eq:CarlemannForLowFrequencies}), while it becomes
non positive only in the $r_{0}\le r\le2r_{0}$ region, where $r_{1}<r_{0}$.
Hence, the first term in the right hand side of (\ref{eq:CarlemannForLowFrequencies})
can be bounded as: 
\begin{equation}
\begin{split}\int_{\mathcal{R}(0,T)\cap\{\frac{r_{1}}{2}\le r\le r_{1}\})}(1 & +(s\text{\textgreek{w}}_{+})^{2})e^{s\cdot\text{\textgreek{w}}_{+}w}\big\{|\text{\textgreek{y}}_{\le\text{\textgreek{w}}_{+}}|^{2}+J_{\text{\textgreek{m}}}^{N}(\text{\textgreek{y}}_{\le\text{\textgreek{w}}_{+}})n^{\text{\textgreek{m}}}\big\}\le\\
\le C\cdot & (1+(s\text{\textgreek{w}}_{+})^{2})\Big(\sup_{r\le r_{1}}e^{s\text{\textgreek{w}}_{+}w}\Big)\int_{\mathcal{R}(0,T)\cap\{r_{0}\le r\le2r_{0}\})}\big\{|\text{\textgreek{y}}_{\le\text{\textgreek{w}}_{+}}|^{2}+J_{\text{\textgreek{m}}}^{N}(\text{\textgreek{y}}_{\le\text{\textgreek{w}}_{+}})n^{\text{\textgreek{m}}}\big\}+\\
 & +C\cdot(1+(s\text{\textgreek{w}}_{+})^{2})\Big(\sup_{r\le r_{1}}e^{s\text{\textgreek{w}}_{+}w}\Big)\cdot\int_{\{t=0\}}J_{\text{\textgreek{m}}}^{N}(\text{\textgreek{y}})n^{\text{\textgreek{m}}}.
\end{split}
\label{eq:NewTermRHS}
\end{equation}
 The second term of the right hand side of (\ref{eq:NewTermRHS})
is what we would like to end up with in the right hand side of (\ref{eq:CarlemannForLowFrequencies}),
so we still have to dispense with the first term of (\ref{eq:NewTermRHS}).
Notice that since $r_{1}<r_{0}$ and $w$ is a strictly increasing
function of $r$, we can bound 
\begin{equation}
\sup_{r\le r_{1}}e^{s\text{\textgreek{w}}_{+}w}\le e^{-c\cdot\text{\textgreek{w}}_{+}s}\cdot\inf_{r_{0}\le r\le2r_{0}}e^{s\text{\textgreek{w}}_{+}w}
\end{equation}
for some $c>0$ depending on the precise choice of $r_{1},r_{0}$.
Thus, if $s\gg1$, we can etimate the first term of the right hand
side of (\ref{eq:NewTermRHS}) by a small constant times the left
hand side of (\ref{eq:CarlemannForLowFrequencies}) (recall $\text{\textgreek{w}}_{+}\ge1$):

\begin{multline}
(1+(s\text{\textgreek{w}}_{+})^{2})\Big(\sup_{r\le r_{1}}e^{s\text{\textgreek{w}}_{+}w}\Big)\int_{\mathcal{R}(0,T)\cap\{r_{0}\le r\le2r_{0}\})}\big\{|\text{\textgreek{y}}_{\le\text{\textgreek{w}}_{+}}|^{2}+J_{\text{\textgreek{m}}}^{N}(\text{\textgreek{y}}_{\le\text{\textgreek{w}}_{+}})n^{\text{\textgreek{m}}}\big\}\le\\
\le e^{-c\cdot\text{\textgreek{w}}_{+}s}\cdot\int_{\mathcal{R}(0,T)\cap\{r_{1}\le r\le R\}}e^{s\cdot\text{\textgreek{w}}_{+}w}\Big\{|\text{\textgreek{y}}_{\le\text{\textgreek{w}}_{+}}|^{2}+J_{\text{\textgreek{m}}}^{N}(\text{\textgreek{y}}_{\le\text{\textgreek{w}}_{+}})n^{\text{\textgreek{m}}}\Big\}.\label{eq:ABoundInSketch}
\end{multline}
Thus, the first term of the right hand side of (\ref{eq:NewTermRHS})
can be absorbed into the left hand side of (\ref{eq:CarlemannForLowFrequencies})
if $s$ is chosen sufficiently large in terms of the geometry of $(\mathcal{M},g)$.
Thus, (\ref{eq:CarlemannForLowFrequencies}), (\ref{eq:NewTermRHS})
and (\ref{eq:ABoundInSketch}) yield:
\begin{equation}
\begin{split}\int_{\mathcal{R}(0,T)\cap\{r\le R\}}e^{s\cdot\text{\textgreek{w}}_{+}w} & \Big\{|\text{\textgreek{y}}_{\le\text{\textgreek{w}}_{+}}|^{2}+J_{\text{\textgreek{m}}}^{N}(\text{\textgreek{y}}_{\le\text{\textgreek{w}}_{+}})n^{\text{\textgreek{m}}}\Big\}\le\\
\le & C(R)\int_{\mathcal{R}(0,T)\cap\{R\le r\le R+1\})}(1+(s\text{\textgreek{w}}_{+})^{2})e^{s\cdot\text{\textgreek{w}}_{+}w}\big\{|\text{\textgreek{y}}_{\le\text{\textgreek{w}}_{+}}|^{2}+J_{\text{\textgreek{m}}}^{N}(\text{\textgreek{y}}_{\le\text{\textgreek{w}}_{+}})n^{\text{\textgreek{m}}}\big\}+\\
 & +C(R)\cdot e^{C(R)\cdot s\text{\textgreek{w}}_{+}}\int_{\{t=0\}}J_{\text{\textgreek{m}}}^{N}(\text{\textgreek{y}})n^{\text{\textgreek{m}}}.
\end{split}
\label{eq:CarlemannForLowFrequencies-1}
\end{equation}

\medskip{}

\noindent \emph{Treatment of the boundary terms near infinity:} Unfortunately,
the previous trick does not apply in order to deal with the first
term of the right hand side of (\ref{eq:CarlemannForLowFrequencies-1}),
i.\,e. the $r\sim R$ boundary term: In this case, the weight $e^{s\text{\textgreek{w}}_{+}w}$
is now much larger in the region $r\sim R$ than in the region $r<R$.
Hence we can not hope to directly absorb the first term of the right
hand side of (\ref{eq:CarlemannForLowFrequencies-1}) into the left
hand side. 

In order to deal with this term, therefore, a much more delicate analysis
is necessary, and it is at this point that the results of \cite{Rodnianski2011}
are used in a fundamental way. In particular, we will employ an ODE
lemma proven in Sections 9 and 10 of \cite{Rodnianski2011}, and we
will apply it on a system of ordinary differential inequalities satisfied
by the mass of $\text{\textgreek{f}}$ and its derivatives on the
cylinders $\{r=\text{\textgreek{r}}\}$ (and varying eith $\text{\textgreek{r}}$).
This lemma will imply that, if $\text{\textgreek{f}}$ solves the
wave equation on $(\mathcal{D},g)$ and has frequency support in $\{|\text{\textgreek{w}}|\ge\text{\textgreek{w}}_{1}\}$,
then for any large $R_{f}$ and for any given $C_{2}>0$, for $R$
taking values in any interval of the form $[R_{f},C\cdot R_{f}]$
the quantity $\int_{\mathcal{R}(0,T)\cap\{R\le r\le R+1\})}\big\{|\text{\textgreek{f}}|^{2}+J_{\text{\textgreek{m}}}^{N}(\text{\textgreek{f}})n^{\text{\textgreek{m}}}\big\}$
will either decay as a function of $R$ with a $e^{-C_{2}\text{\textgreek{w}}_{1}R}$
rate, or it will be bounded by $C(C_{2})$ times the initial energy
of $\text{\textgreek{f}}$. 

Assume, for a moment, that our solution $\text{\textgreek{y}}_{\le\text{\textgreek{w}}_{+}}$
actually has frequency support only in the regime $|\text{\textgreek{w}}|\sim\text{\textgreek{w}}_{+}$.
Then the ODE lemma of Rodnianski and Tao would yield that, for a suitable
choice of $C_{2}\gg1$, the first term in the right hand side of (\ref{eq:CarlemannForLowFrequencies-1})
can either be absorbed by the left hand side, or can be bounded by
the initial energy of $\text{\textgreek{y}}$ (since, again, in our
heuristic setting the energy of $\text{\textgreek{y}}_{\le\text{\textgreek{w}}_{+}}$
can be bounded by the energy of $\text{\textgreek{y}}$). Thus, in
the case where $\text{\textgreek{y}}_{\le\text{\textgreek{w}}_{+}}$
has frequency support only in the regime $|\text{\textgreek{w}}|\sim\text{\textgreek{w}}_{+}$,
(\ref{eq:CarlemannForLowFrequencies-1}) readily implies the ILED
statement:

\begin{equation}
\int_{\mathcal{R}(0,T)\cap\{r\le R\}}e^{s\cdot\text{\textgreek{w}}_{+}w}\Big\{|\text{\textgreek{y}}_{\le\text{\textgreek{w}}_{+}}|^{2}+J_{\text{\textgreek{m}}}^{N}(\text{\textgreek{y}}_{\le\text{\textgreek{w}}_{+}})n^{\text{\textgreek{m}}}\Big\}\le C(R)\cdot e^{C(R)\cdot s\text{\textgreek{w}}_{+}}\int_{\{t=0\}}J_{\text{\textgreek{m}}}^{N}(\text{\textgreek{y}})n^{\text{\textgreek{m}}}.\label{eq:CarlemannForLowFrequencies-1-1}
\end{equation}

In general, however, $\text{\textgreek{y}}_{\le\text{\textgreek{w}}_{+}}$
will not have frequency support only in the region $|\text{\textgreek{w}}|\sim\text{\textgreek{w}}_{+}$.
For this reason, we need to decompose $\text{\textgreek{y}}_{\le\text{\textgreek{w}}_{+}}$
further into $\text{\textgreek{y}}_{k}$ pieces, which have frequency
supports in small intervals $[\text{\textgreek{w}}_{k},\text{\textgreek{w}}_{k+1}]$
with comparable endpoints, except for $\text{\textgreek{y}}_{0}$,
which has frequency support in a small neighborhood around $0$. For
each of the frequency decomposed components $\text{\textgreek{y}}_{k}$,
for $k\neq0$, the previous heuristics work exactly as presented for
$\text{\textgreek{y}}_{\le\text{\textgreek{w}}_{+}}$, since the Carleman
inequality (\ref{eq:CarlemannForLowFrequencies}) and the ODE lemma
of Rodnianski and Tao hold for $\text{\textgreek{y}}_{k}$ as well.
Thus, after applying this procedure and then summing the $\text{\textgreek{y}}_{k}$'s
for $k\neq0$, this line of arguments establishes the integrated local
energy decay statement (\ref{eq:ILEDLowFrequencies}) for $\text{\textgreek{y}}_{\le\text{\textgreek{w}}_{+}}-\text{\textgreek{y}}_{0}$
in place of $\text{\textgreek{y}}_{\le\text{\textgreek{w}}_{+}}$. 

\medskip{}

\noindent \emph{Treatment of the very low frequency component $\text{\textgreek{y}}_{0}$:}
For $\text{\textgreek{y}}_{0}$, a different argument needs to be
furnished in order to reach the full integrated local energy decay
statement (\ref{eq:ILEDLowFrequencies}). This is accomplished in
Section \ref{sub:LowFrequencies}, with the use of a special vector
field current, coupled with a locally improved Morawetz ILED statement.
At this point, Assumption \hyperref[Assumption 3]{3} on the smallness
of the ergoregion plays a crucial role.

\subsection{\noindent \label{sub:Remark}Remark on the proof of Theorem \ref{thm:Theorem}.\emph{ }}

\noindent We will only be concerned with proving Theorem \ref{thm:Theorem}
in the case where future horizon $\mathcal{H}^{+}$ is non-empty.
In the case $\mathcal{H}^{+}=\emptyset$, Assumption \hyperref[Assumption 2]{2}
can be completely dropped, and Assumption \hyperref[Assumption 3]{3}
degenerates to the condition that $g(T,T)<0$ everywhere on $\mathcal{D}$.
Therefore, in the case $\mathcal{H}^{+}=\emptyset$, Theorem \ref{thm:Theorem}
can be established using almost the same arguments as in the case
$\mathcal{H}^{+}\neq\emptyset$, and in fact the proof of the results
of Sections \ref{sub:LowFrequencies} and \ref{sec:ILEDlow} are simplified.
Thus, from now on we will assume without loss of generality that $\mathcal{H}^{+}\neq\emptyset$
(the differences of the proof in the cases $\mathcal{H}^{+}\neq\emptyset$
and $\mathcal{H}^{+}=\emptyset$ will be highlighted in the footnotes
of Sections \ref{sub:LowFrequencies} and \ref{sec:ILEDlow}).

In order to avoid confusion when doing estimates in each connected
component of the asymptotically flat region $\mathcal{I}_{as}$ of
$\mathcal{D}$, we will assume without loss of generality that $\mathcal{I}_{as}$
has only one connected component (i.\,e.~that $(\mathcal{M},g)$
has only one asymptotically flat end). This assumption will simplify
the derivation of the estimates in the region $\{r\gg1\}$ appearing
in Sections \ref{sec:ILEDlow}--\ref{sec:Polynomial-decay}, in view
of the fact that $\{r\gg1\}$ will be covered by a single polar coordinate
chart. In the general case where $\mathcal{I}_{as}$ has more than
one components, the same estimates will follow after repeating the
same proves on each component seperately and then adding the resulting
estimates. The reader is advised to keep in mind this simplifying
assumption when reading Sections \ref{sec:ILEDlow}--\ref{sec:Polynomial-decay}.

It will also be convenient for the proof of Theorem \ref{thm:Theorem}
to assume without loss of generality that $\text{\textgreek{y}}$
has compact initial data. This will not pose any further restriction
for the statement of the theorem, since the full statement can then
follow by a usual density argument. However, compact initial data
for $\text{\textgreek{y}}$ imply, due to the domain of dependence
property, that the restriction of $\text{\textgreek{y}}$ on any $\{t=const\}$
hypersurface is also compactly supported, and in particular, $\text{\textgreek{y}}$
must be supported in a set of the form $\{r\lesssim R_{sup}+|t|\}$
for some large $R_{sup}$ depending on $\text{\textgreek{y}}$.%
\footnote{Of course no constant in the proof must be allowed to depend on $R_{sup}$%
}

\subsection{Outline of the paper}

The proof of Theorem \ref{thm:Theorem} is presented in Sections \ref{sec:FreqDecomposition}--\ref{sec:The-interpolation-argument}.
While reading these sections, one is advised to keep in mind the above
heuristics, since despite their simplicity, the main arguments are
often blurred by technicalities (mainly because of the need to carefully
perform physical space as well as frequency space cut-offs, and also
because of the very general assumptions on the geometry). 

In Section \ref{sec:FreqDecomposition}, we carry out in detail the
cut-off procedure in both physical and frequency space, and we establish
all the required lemmas regarding the behaviour of $\text{\textgreek{y}}_{\le\text{\textgreek{w}}_{+}},\text{\textgreek{y}}_{\ge\text{\textgreek{w}}_{+}}$
and $\text{\textgreek{y}}_{k}$. 

In Section \ref{sub:EstimatesPsiKAsymptoticallyFlat}, we obtain estimates
for the behaviour of $\text{\textgreek{y}}_{k},\text{\textgreek{y}}_{\le\text{\textgreek{w}}_{+}},\text{\textgreek{y}}_{\ge\text{\textgreek{w}}_{+}}$
in the asymtotically flat region $\{r\gg1\}$. This is achieved by
applying the new method of Dafermos and Rodnianski, originally appearing
in \cite{DafRod7}, and generalised in \cite{Moschidisc} to include
a broader class of asymptotically flat manifolds. In particular, in
Section \ref{sub:EstimatesPsiKAsymptoticallyFlat} the results of
\cite{Moschidisc} are specialized for the case of the frequency decomposed
components $\text{\textgreek{y}}_{k},\text{\textgreek{y}}_{\le\text{\textgreek{w}}_{+}},\text{\textgreek{y}}_{\ge\text{\textgreek{w}}_{+}}$
of $\text{\textgreek{y}}$.

In Section \ref{sub:LowFrequencies}, we provide the proof of integrated
local energy decay for the very low frequency part $\text{\textgreek{y}}_{0}$. 

The proof of the integrated local energy decay of $\text{\textgreek{y}}_{\le\text{\textgreek{w}}_{+}}$
occupies Section \ref{sec:ILEDlow}. There, we establish a Carleman
type inequality for $\text{\textgreek{y}}_{k}$, $k\neq0$, which
is then upgraded to an integrated local energy decay statement for
$\text{\textgreek{y}}_{\le\text{\textgreek{w}}_{+}}$ with the use
of a technical lemma of \cite{Rodnianski2011}. With this last statement
at our disposal, and with the use of the generalised version of the
$r^{p}$-weighted energy hierarchy of Dafermos and Rodnianski (from
Section \ref{sub:BoundsForPsiK}), we then infer in Section \ref{sec:Polynomial-decay}
that the local energy of $\text{\textgreek{y}}_{\le\text{\textgreek{w}}_{+}}$
decays polynomially in time. 

The proof of Theorem \ref{thm:Theorem} is completed in Section \ref{sec:The-interpolation-argument},
where the interpolation scheme (\ref{eq:InterpolationHeuristics})
is rigorously formulated.

The proof of Corollary \ref{cor:Corollary} is presented in Section
\ref{sec:Proof-of-corollary}, while the proof of Corollaries \ref{cor:CorollaryQualititative}
and \ref{cor:CorollaryScattering} are presented in Section \ref{sec:Proof-of-Corollary2}.
Finally, we show that the result of Theorem \ref{thm:Theorem} is
optimal in this generality, by establishing the sharpness of the logarithmic
decay rate in Section \ref{sub:Sharpness}.

\section{\label{sec:Notational-conventions}Notational conventions and Hardy
inequalities}

In this paper, we will adopt the same notational conventions as we
did in \cite{Moschidisc}. We will now proceed to describe in more
detail these conventions.

\subsection{Constants and parameters}

We will adopt the following convention for denoting constants appearing
in inequalities, as is done in \cite{DafRod5}: Capital letters (e.\,g.~$C$)
will be used to denote ``large'' constants, typically appearing
on the right hand side of inequalities. (Such constants can be ``freely''
replaced by larger ones without rendering the inequality invalid.)
We will use lower case letters (e.\,g. $c$) to denote ``small''
constants (which can similarly freely be replaced by smaller ones).
The same characters will be frequently used to denote different constants. 

We will assume that all non-explicit constants will depend on the
specific geometric aspects of our spacetime (e.\,g the topology and
the exact form of the metric) and we will not keep track of this dependence,
except for some very specific cases. However, since we will need to
define a plethora of parameters in the following sections, we will
always keep track of the dependence of all constants on each of these
parameters. Once a parameter is fixed (which will be clearly stated
in the text), the dependence of constants on it will be dropped. 

The parameter $R_{1}$, appearing in the statement of Theorem \ref{thm:Theorem},
will be considered fixed, and hence we will drop any explicit referrence
to it from constants depending on its choice, unless there is a need
to emphasize this dependence. Since we can always increase $R_{1}$without
affecting the statement of Theorem \ref{thm:Theorem}, we will assume
without loss of generality that $R_{1}$ is large enough so that the
region $\{r\ge R_{1}\}$ is contained in the chart where the metric
$g$ takes the form (\ref{eq:metric}).

\subsection{Inequality symbols}

We will use the notation $f_{1}\lesssim f_{2}$ for two real functions
$f_{1},f_{2}$ as usual to imply that there exists some $C>0$, such
that $f_{1}\le C\cdot f_{2}$. This constant $C$ might depend on
free parameters, and these parameters will be stated clearly in each
case. If nothing is stated regarding the dependence of this constant
on parameters, it should be assumed that it only depends on the geometry
of $(\mathcal{D},g)$, and on the parameter $R_{1}$ appearing in
the statement of Theorem \ref{thm:Theorem}. 

We will write $f_{1}\sim f_{2}$ when we can bound $f_{1}\lesssim f_{2}$
and $f_{2}\lesssim f_{1}$. The notation $f_{1}\ll f_{2}$ will be
equivalent to the statement that $\frac{|f_{1}|}{|f_{2}|}$ can be
bounded by some sufficiently small positive constant, the magnitude
and the dependence of which on variable parameters will be clear in
each case from the context. For any function $f:\mathcal{M}\rightarrow[0,+\infty)$,
$\{f\gg1\}$ will denote the subset $\{f\ge C\}$ of $\mathcal{M}$
for some constant $C\gg1$.

For functions $f_{1},f_{2}:[x_{0},+\infty)\rightarrow\mathbb{R}$,
the notation $f_{1}=o(f_{2})$ will mean that $\frac{|f_{1}|}{|f_{2}|}$
can be bounded by some continuous function $h:[x_{0},+\infty)\rightarrow(0,+\infty)$
such that $h(x)\rightarrow0$ as $x\rightarrow+\infty$. This bound
$h$ might deppend on free parameters, and this fact will be clear
in each case from the context.

\subsection{\label{sub:CoordinateCharts}Coordinate charts on $\mathcal{D}$
and subsets of $\mathcal{D}$ associated to $t$}

In this paper, we will identify $\mathcal{D}\backslash\mathcal{H}^{-}$
with $\mathbb{R}\times(\text{\textgreek{S}}\cap\mathcal{D})$ by setting
$\{t\}\times(\text{\textgreek{S}}\cap\mathcal{D})$ to be equal to
the image of $\text{\textgreek{S}}\cap\mathcal{D}$ under the flow
of the vector field $T$ for time $t$ (see the remarks in Assumption
\hyperref[Assumption 1]{1}. 

When performing calculations in specific coordinate charts on regions
of $\mathcal{D}\backslash\mathcal{H}^{-}$, we will always pick coordinate
charts of the following form: For any local coordinate chart $x=(x^{1},x^{2},\ldots,x^{d})$
on an open subset $\mathcal{V}$ of $\text{\textgreek{S}}\cap\mathcal{D}$,
we will extend the functions $x^{i}$ on the whole of $\mathbb{R}\times\mathcal{V}\subset\mathcal{D}\backslash\mathcal{H}^{-}$
by the requirement $T(x^{i})=0$, and we will use the coordinate chart
$(t,x^{1},\ldots,x^{d})$ on $\mathbb{R}\times\mathcal{V}$. Notice
that in such a coordinate chart we have $\partial_{t}\equiv T$. In
the expression of any tensor $\mathfrak{w}$ in such a coordinate
chart, components with indices ranging from $1$ to $d$ will correspond
to the components of the tensor associated to $\partial_{x^{1}},\ldots,\partial_{x^{d}}$
(or $dx^{1},\ldots,dx^{d}$) respectively.

For any $t_{1}\le t_{2}$, we will denote 
\begin{equation}
\mathcal{R}(t_{1},t_{2})\doteq\{t_{1}\le t\le t_{2}\}\subset\mathcal{D}.
\end{equation}
Furthermore, for any $t_{0}\in\mathbb{R}$, we will denote: 
\begin{equation}
\text{\textgreek{S}}_{t_{0}}\doteq\{t=t_{0}\}.
\end{equation}

\subsection{\label{sub:Connections-and-volume}Connections and volume forms}

We will usually denote the natural connection of a pseudo-Riemannian
manifold $(\mathcal{N},h_{\mathcal{N}})$ as $\nabla^{h_{\mathcal{N}}}$
or $\nabla_{h_{\mathcal{N}}}$ (or simply $\nabla_{\mathcal{N}}$
when there is no ambiguity about the metric $h_{\mathcal{N}}$). The
associated volume form will be denoted as $dh_{\mathcal{N}}$. If
$h_{\mathcal{N}}$ is Riemannian, $\big|\cdot\big|_{h_{\mathcal{N}}}$
will denote the associated norm on the tensor bundle of $\mathcal{N}$. 

For any integer $l\ge0$, $\big(\nabla^{h_{\mathcal{N}}}\big)^{l}$
or $\nabla_{h_{\mathcal{N}}}^{l}$ will denote the higher order operator
\begin{equation}
\underbrace{\nabla_{h_{\mathcal{N}}}\cdots\nabla_{h_{\mathcal{N}}}}_{l\mbox{ times}}.\label{eq:ProductDerivatives}
\end{equation}
 We should remark that the product (\ref{eq:ProductDerivatives})
is not symmetrised. We will also adopt the convention that we will
always use Latin characters to denote such powers of covariant derivative
operators. On the other hand, Greek characters will be used for the
indices of a tensor in an abstract index notation. 

Let us state an example of the above convention: Let $k$ be a $(n_{1},n_{2})$-tensor
and $\text{\textgreek{f}}:(\mathcal{N},h_{\mathcal{N}})\rightarrow\mathbb{C}$
be a smooth function. Then the quantity 
\begin{equation}
k_{\text{\textgreek{b}}_{1}\ldots\text{\textgreek{b}}_{n_{2}}}^{\text{\textgreek{a}}_{1}\ldots\text{\textgreek{a}}_{n_{1}}}\cdot\big(\nabla_{h_{\mathcal{N}}}^{n_{1}+n_{2}}\big)_{\hphantom{\text{\textgreek{b}}_{1}\ldots\text{\textgreek{b}}_{m}}\text{\textgreek{a}}_{1}\ldots\text{\textgreek{a}}_{n_{1}}}^{\text{\textgreek{b}}_{1}\ldots\text{\textgreek{b}}_{n_{2}}}\text{\textgreek{f}}\label{eq:ExampleIndices}
\end{equation}
denotes the contraction of the $n_{1}+n_{2}$ order derivative $\nabla_{h_{\mathcal{N}}}^{n_{1}+n_{2}}\text{\textgreek{f}}$
of $\text{\textgreek{f}}$ with the tensor $k$. In the above, the
metric $h_{\mathcal{N}}$ was used to raise the first $n_{2}$ indices
of $\nabla_{h_{\mathcal{N}}}^{n_{1}+n_{2}}\text{\textgreek{f}}$.
Notice that in (\ref{eq:ExampleIndices}) we have used the abstract
index notation, and hence the indices in (\ref{eq:ExampleIndices})
are not associated to any fixed local coordinate chart.

\subsection{Integration over domains and hypersurfaces}

In the cases where we use the natural volume form $\text{\textgreek{w}}$
associated to the metric $g$ of a Lorentzian manifold $(\mathcal{M},g)$
in order to integrate over open subsets of $\mathcal{M}$, the volume
form will be often dropped in the expression for the integral. Recall
that in any local coordinate chart $(x^{0},x^{1},x^{2},\ldots x^{d})$,
$\text{\textgreek{w}}$ is expressed as 
\[
\text{\textgreek{w}}=\sqrt{-det(g)}dx^{0}\cdots dx^{d}.
\]
We will apply the same rule when integrating over any spacelike hypersurface
$\mathcal{S}$ of $(\mathcal{M},g)$ using the natural volume form
of its induced (Riemannian) metric.

In the case of a smooth null hypersurface $\mathscr{H}$, the volume
form with which integration will be considered will as usual depend
on the choice of a future directed null generator $n_{\mathcal{\mathscr{H}}}$
for $\mathcal{\mathscr{H}}$. For any such choice of $n_{\mathcal{\mathscr{H}}}$,
selecting an arbitrary vecor field $X$ on $T_{\mathcal{\mathscr{H}}}M$
such that $g(X,n_{\mathscr{H}})=-1$ enables the construction of a
non degenerate $d$-form on $\mathcal{\mathscr{H}}$: $dvol_{n_{\mathscr{H}}}\doteq i_{X}\text{\textgreek{w}}$,
which depends on the on the precise choice of $n_{\mathcal{\mathscr{H}}}$,
but not on the choice for $X$. In that case, $dvol_{n_{\mathcal{\mathscr{H}}}}$
will be the volume form on $\mathscr{H}$ associated with $n_{\mathcal{\mathscr{H}}}$.

\subsection{Notations for derivatives on $\mathbb{S}^{d-1}$}

Since, in the present paper, we will frequently work in polar coordinates
in the asymptotically flat region of $(\mathcal{D},g)$, we will introduce
some convenient shorthand notation regarding iterated derivatives
on the unit sphere $\mathbb{S}^{d-1}$, $d\ge3$.

The usual round metric on the sphere $\mathbb{S}^{d-1}$ will be denoted
as $g_{\mathbb{S}^{d-1}}$. This is simply the induced metric on the
unit sphere of $\mathbb{R}^{d}$. We will also denote with $g_{\mathbb{S}^{d-1}}$
the natural extension of the round metric to an inner product on the
space of tensors over $\mathbb{S}^{d-1}$. According to the conventions
of Section \ref{sub:Connections-and-volume}, for any tensor field
$\mathfrak{w}$ on $\mathbb{S}^{d-1}$, $|\mathfrak{w}|_{g_{\mathbb{S}^{d-1}}}$
will denote the norm of $\mathfrak{w}$ with respect to $g_{\mathbb{S}^{d-1}}$
and $\nabla^{\mathbb{S}^{d-1}}$ (or $\nabla_{\mathbb{S}^{d-1}}$)
will denote the covariant derivative associated with $g_{\mathbb{S}^{d-1}}$.
Furthermore, for any smooth $(n_{1},n_{2})$-tensor field $\mathfrak{w}$
on $\mathbb{S}^{d-1}$, $\big(\nabla^{\mathbb{S}^{d-1}}\big)^{k}\mathfrak{w}$
(or $\nabla_{\mathbb{S}^{d-1}}^{k}\mathfrak{w}$) will denote the
$(n_{1},n_{2}+k)$-tensor field on $\mathbb{S}^{d-1}$ obtained after
applying the operator $\nabla^{\mathbb{S}^{d-1}}$ on $\mathfrak{w}$
$k$ times. The Laplace--Beltrami operator on $(\mathbb{S}^{d-1},g_{\mathbb{S}^{d-1}})$
will be denoted as $\text{\textgreek{D}}_{g_{\mathbb{S}^{d-1}}}$. 

Frequently, we will work on regions $\mathcal{U}$ of a spacetime
$\mathcal{M}^{d+1}$ such that $\mathcal{U}$ can be mapped diffeomorphically,
through a coordinate ``chart'', onto $\mathbb{R}_{+}\times\mathbb{R}_{+}\times\mathbb{S}^{d-1}$.
In any such a coordinate ``chart'', $\text{\textgreek{sv}}$ will
denote the projection $\text{\textgreek{sv}}:\mathcal{U}\rightarrow\mathbb{S}^{d-1}$.
We remark that, for any $x\in\mathcal{M}$, $\text{\textgreek{sv}}(x)$
is a point on $\mathbb{S}^{d-1}$ and not just the coordinates of
this point in a coordinate chart on $\mathbb{S}^{d-1}$. The same
$\text{\textgreek{sv}}$ notation will also be used for the spherical
variable of a polar coordinate ``chart'' on codimension $1$ submanifolds
of $\mathcal{M}$ (in this case, the range of such a ``chart'' will
be simply $\mathbb{R}_{+}\times\mathbb{S}^{d-1}$). For example, we
will use the notation $(r,\text{\textgreek{sv}}):\{x^{0}=0\}\rightarrow\mathbb{R}_{+}\times\mathbb{S}^{d-1}$
for the usual polar coordinate chart on the hyperplane $\{x^{0}=0\}$
of $\mathbb{R}^{d+1}$. 

On a subset $\mathcal{U}$ of a spacetime $\mathcal{M}$ covered by
a polar coordinate chart $(u_{1},u_{2},\text{\textgreek{sv}}):\mathcal{U}\rightarrow\mathbb{R}_{+}\times\mathbb{R}_{+}\times\mathbb{S}^{d-1}$,
for any function $h:\mathcal{U}\rightarrow\mathbb{C}$ and any $\text{\textgreek{b}}_{1},\text{\textgreek{b}}_{2}\in\mathbb{R}_{+}$,
$h(\text{\textgreek{b}}_{1},\text{\textgreek{b}}_{2},\cdot)$ defines
a function on $\mathbb{S}^{d-1}$. Under this correspondence, the
$\nabla^{\mathbb{S}^{d-1}}$ differential operator on $\mathbb{S}^{d-1}$
is extended to a tangential differential operator on the hypersurfaces
$\{u_{1},u_{2}=const\}\subset\mathcal{U}$. This operator is, of course,
related to the specific choice of the polar coordinate chart $(u_{1},u_{2},\text{\textgreek{sv}})$.

The following schematic notation for derivatives on $\mathbb{S}^{d-1}$
(and the associated tangential operators on the hypersurfaces $\{u_{1},u_{2}=const\}$
in a $(u_{1},u_{2},\text{\textgreek{sv}})$ coordinate chart on a
spacetime $\mathcal{M}$) will be frequently used in the present paper:
For any function $h:\mathbb{S}^{d-1}\rightarrow\mathbb{C}$ and any
$l\in\mathbb{N}$, we will denote the $l$-th order derivative $\nabla_{\mathbb{S}^{d-1}}^{l}h$
as $\partial_{\text{\textgreek{sv}}}^{l}h$. The norm of this tensor
will be denoted as: 
\begin{equation}
|\partial_{\text{\textgreek{sv}}}^{l}h|\doteq\big|\nabla_{\mathbb{S}^{d-1}}^{l}h\big|_{\mathbb{S}^{d-1}}.
\end{equation}
 Furthermore, for any \underline{symmetric} $(l,0)$-tensor $b$
on $\mathbb{S}^{d-1}$, the following schematic notation for the contraction
of $\big(\nabla^{\mathbb{S}^{d-1}}\big)^{l}h$ with $b$ will be frequently
used:
\begin{equation}
b\cdot\partial_{\text{\textgreek{sv}}}^{l}h\doteq b^{\text{\textgreek{i}}_{1}\ldots\text{\textgreek{i}}_{l}}(\nabla_{\mathbb{S}^{d-1}}^{l})_{\text{\textgreek{i}}_{1}\ldots\text{\textgreek{i}}_{l}}h\label{eq:ContractionOneFunction}
\end{equation}
(for the notations on powers of covariant derivatives and the abstract
index notation, see Section \ref{sub:Connections-and-volume}). The
same notation will be used for the contraction of the product of derivatives
of multiple functions: For any family of $m$ functions $h_{1},\ldots,h_{m}:\mathbb{S}^{d-1}\rightarrow\mathbb{C}$
and any set $(j_{1},\ldots j_{m})$ of non-negative integers, for
any $(\sum_{k=1}^{n}j_{k},0$)-tensor $b$ on $\mathbb{S}^{d-1}$
which is symmetric in any pair of indices lying in the same one of
the intervals $I_{n}=\Big(\sum_{k=1}^{n-1}j_{k}+1,\sum_{k=1}^{n}j_{k}\Big)$
for each $n\in\{1,\ldots m\}$, we will adopt the notation 
\begin{equation}
b\cdot\partial_{\text{\textgreek{sv}}}^{j_{1}}h_{1}\cdots\partial_{\text{\textgreek{sv}}}^{j_{m}}h_{n}\doteq b^{\text{\textgreek{i}}_{1}\ldots\text{\textgreek{i}}_{\sum_{k=1}^{m}j_{k}}}\cdot\big(\nabla_{\mathbb{S}^{d-1}}^{j_{1}}\big)_{\text{\textgreek{i}}_{1}\ldots\text{\textgreek{i}}_{j_{1}}}h_{1}\cdots\big(\nabla_{\mathbb{S}^{d-1}}^{j_{m}}\big)_{\text{\textgreek{i}}_{\sum_{k=1}^{m-1}j_{k}+1}\ldots\text{\textgreek{i}}_{\sum_{k=1}^{m}j_{k}}}h_{m}.\label{eq:ContractionProductFunctions}
\end{equation}
Notice that the tensor $b$ used in the notation (\ref{eq:ContractionProductFunctions})
will not necessarily be symmetric in pairs of indices lying in seperate
pairs of the $I_{n}$ intervals.

We will also use the same notation (i.\,e.~(\ref{eq:ContractionOneFunction})
and (\ref{eq:ContractionProductFunctions})) when $h$, $h_{1},\ldots,h_{m}$
are tensor fields on $\mathbb{S}^{d-1}$.

Notice also that, working in a polar coordinate ``chart'' $(u_{1},u_{2},\text{\textgreek{sv}}):\mathcal{U}\rightarrow\mathbb{R}_{+}\times\mathbb{R}_{+}\times\mathbb{S}^{d-1}$,
the following commutation relation holds for any function $h$ on
$\mathcal{U}$: 
\begin{equation}
\big[\mathcal{L}_{\partial_{u_{i}}}\nabla^{\mathbb{S}^{d-1}},\nabla^{\mathbb{S}^{d-1}}\partial_{u_{i}}\big]h=0,
\end{equation}
where $\partial_{u_{i}}$ is the coordinate vector field associated
to the coordinate function $u_{i}$ (for $i=1,2$). Therefore, we
will frequently denote 
\begin{equation}
\mathcal{L}_{\partial_{u_{i}}}\nabla^{\mathbb{S}^{d-1}}h\doteq\partial_{u_{i}}\partial_{\text{\textgreek{sv}}}h
\end{equation}
and this notation will allow commuting $\partial_{u_{i}}$ with $\partial_{\text{\textgreek{sv}}}$,
as if $\partial_{\text{\textgreek{sv}}}$ was a regular coordinate
vector field. 

The notation $d\text{\textgreek{sv}}$ will be used in two different
ways, depending on the context: it will denote either the usual volume
form on $(\mathbb{S}^{d-1},g_{\mathbb{S}^{d-1}}$) or a $1$-form
on $\mathbb{S}^{d-1}$ satisfying for $0\le k\le4$ the bound $\big|\big(\nabla^{\mathbb{S}^{d-1}}\big)^{k}d\text{\textgreek{sv}}\big|_{g_{\mathbb{S}^{d-1}}}\le1$.
Similarly, $d\text{\textgreek{sv}}d\text{\textgreek{sv}}$ will denote
a symmetric $(2,0)$-tensor on $\mathbb{S}^{d-1}$ satisfying for
any $0\le k\le4$ the bound $\big|\big(\nabla^{\mathbb{S}^{d-1}}\big)^{k}(d\text{\textgreek{sv}}d\text{\textgreek{sv}})\big|_{g_{\mathbb{S}^{d-1}}}\le1$.
\begin{example*}
For any function $f$ and any tensor $b$ on $\mathbb{S}^{d-1}$ with
the aforementioned symmetries, the above notation will allow us to
perform the following integration by parts procedure: 
\begin{equation}
\int_{\mathbb{S}^{d-1}}b\cdot\partial_{\text{\textgreek{sv}}}f\cdot\partial_{\text{\textgreek{sv}}}\partial_{\text{\textgreek{sv}}}f\, d\text{\textgreek{sv}}=-\frac{1}{2}\int_{\mathbb{S}^{d-1}}(c_{1}\partial_{\text{\textgreek{sv}}}b+c_{2}b)\cdot\partial_{\text{\textgreek{sv}}}f\cdot\partial_{\text{\textgreek{sv}}}f\, d\text{\textgreek{sv}},\label{eq:Integrationbypartsspheregeneralrule-1}
\end{equation}
 for some smooth contracting tensors $c_{1}$, $c_{2}$ which are
bounded with bounds depending only on the tensor type of $b$. 
\end{example*}
The notation (\ref{eq:ContractionOneFunction}) and (\ref{eq:ContractionProductFunctions})
will be frequently used in cases where we lack an explicit form for
the contracting tensor $b$, but we have bounds for the norm of $b$
and its derivatives. This is the reason motivating our choice of a
notation which apparently loses information regarding the structure
of the underlying expression.

\subsection{The $O_{k}(\cdot)$ notation}

For any integer $k\ge0$ and any $b\in\mathbb{R}$, the notation $h=O_{k}(r^{b})$
for some smooth function $h:\mathcal{M}\rightarrow\mathbb{C}$ will
be used to denote that in the $(t,r,\text{\textgreek{sv}})$ polar
coordinate chart in the region $\{r\gg1\}$ of $\mathcal{D}$ (see
Assumption \hyperref[Assumption 1]{1}): 
\begin{equation}
\sum_{j=0}^{k}\sum_{j_{1}+j_{2}+j_{3}=j}r^{j_{1}+j_{2}}|\partial_{t}^{j_{1}}\partial_{r}^{j_{2}}\partial_{\text{\textgreek{sv}}}^{j_{3}}h|\le C\cdot r^{b}
\end{equation}
for some constant $C>0$ dependng on $k$ and $h$. The same notation
(omitting the $\partial_{t}$ derivatives) will also be used for functions
on regions of manifolds cover by an $(r,\text{\textgreek{sv}})$ polar
coordinate chart.

\subsection{\noindent \label{sub:Currents}Vector field multipliers and currents:}

In the present paper, we will frequently use the language of currents
and vector field multipliers in order to establish the desired estimates.
On any Lorentzian manifold $(\mathcal{M},g)$, associated to the wave
operator $\square_{g}=\frac{1}{\sqrt{-det(g)}}\partial_{\text{\textgreek{m}}}\Big(\sqrt{-det(g)}\cdot g^{\text{\textgreek{m}\textgreek{n}}}\partial_{\text{\textgreek{n}}}\Big)$
is a $(0,2)$-tensor called the \emph{energy momentum tensor} $T$.
For any smooth function $\text{\textgreek{y}}:\mathcal{M}\rightarrow\mathbb{C}$,
the energy momentum tensor takes the form

\begin{equation}
T_{\text{\textgreek{m}\textgreek{n}}}(\text{\textgreek{y}})=\frac{1}{2}\Big(\partial_{\text{\textgreek{m}}}\text{\textgreek{y}}\cdot\partial_{\text{\textgreek{n}}}\bar{\text{\textgreek{y}}}+\partial_{\text{\textgreek{m}}}\bar{\text{\textgreek{y}}}\cdot\partial_{\text{\textgreek{n}}}\text{\textgreek{y}}\Big)-\frac{1}{2}\big(\partial^{\text{\textgreek{l}}}\text{\textgreek{y}}\cdot\partial_{\text{\textgreek{l}}}\bar{\text{\textgreek{y}}}\big)g_{\text{\textgreek{m}\textgreek{n}}}.
\end{equation}

For any continuous and piecewise $C^{1}$ vector field $X$ on $\mathcal{M}$,
the following associated currents can be defined almost everywhere:

\begin{equation}
J_{\text{\textgreek{m}}}^{X}(\text{\textgreek{y}})=T_{\text{\textgreek{m}\textgreek{n}}}(\text{\textgreek{y}})X^{\text{\textgreek{m}}},
\end{equation}
\begin{equation}
K^{X}(\text{\textgreek{y}})=T_{\text{\textgreek{m}\textgreek{n}}}(\text{\textgreek{y}})\nabla^{\text{\textgreek{m}}}X^{\text{\textgreek{n}}}.
\end{equation}
 The following divergence identity then holds almost everywhere:

\begin{equation}
\nabla^{\text{\textgreek{m}}}J_{\text{\textgreek{m}}}^{X}(\text{\textgreek{y}})=K^{X}(\text{\textgreek{y}})+Re\Big\{(\square_{g}\text{\textgreek{y}})\cdot X\bar{\text{\textgreek{y}}}\Big\}.
\end{equation}

\subsection{Hardy-type inequalities}

Frequently throughout this paper, we will need to control the weighted
$L^{2}$ norm of some function $u$ by some weighted $L^{2}$ norm
of its derivative $\nabla u$. This will always be accomplished with
the use of some variant of the following Hardy-type inequality on
$\mathbb{R}^{d}$ (which is true for $d\ge1$, although we will only
need it for $d\ge3$):
\begin{lem}
For any $a>0$, there exists some $C_{a}>0$ such that for any smooth
and compactly supported function $u:\mathbb{R}^{d}\rightarrow\mathbb{C}$
and any $R>0$ we can bound
\begin{equation}
\int_{\mathbb{R}^{d}\cap\{r\ge R\}}r^{-d+a}\cdot|u|^{2}\, dx+\int_{\{r=R\}}R^{-(d-1)+a}\cdot|u|^{2}\, dg_{\{r=R\}}\le C_{a}\int_{\mathbb{R}^{d}\cap\{r\ge R\}}r^{-(d-2)+a}|\partial_{r}u|^{2}\, dx\label{eq:GeneralHardyBound}
\end{equation}
 In the above, $r$ is the polar distance on $\mathbb{R}^{d}$, $dx$
is the usual volume form on $\mathbb{R}^{d}$ and $dg_{\{r=R\}}$
is the volume form of the induced metric on the sphere $\{r=R\}\subset\mathbb{R}^{d}$.\end{lem}
\begin{proof}
Following the usual steps for proving a Hardy inequality, we can calculate
in polar coordinates:
\begin{align}
\int_{\mathbb{R}^{d}\cap\{r\ge R\}}r^{-d+a}\cdot|u|^{2}\, dx & =\int_{\mathbb{S}^{d-1}}\Big(\int_{R}^{\infty}r^{-d+a}\cdot|u|^{2}\, r^{d-1}dr\Big)dg_{\mathbb{S}^{d-1}}=\label{eq:GeneralHardyIntermediate}\\
 & =\int_{\mathbb{S}^{d-1}}\Big(\int_{R}^{\infty}r^{-1+a}\cdot|u|^{2}\, dr\Big)dg_{\mathbb{S}^{d-1}}=\nonumber \\
 & =\frac{1}{a}\int_{\mathbb{S}^{d-1}}\Big(\int_{R}^{\infty}\partial_{r}(r^{a})\cdot|u|^{2}\, dr\Big)dg_{\mathbb{S}^{d-1}}=\nonumber \\
 & =\frac{1}{a}\Big\{-\int_{\mathbb{S}^{d-1}}\big(r^{a}\cdot|u|^{2}\big)|_{r=R}\, dg_{\mathbb{S}^{d-1}}+2\int_{\mathbb{S}^{d-1}}\Big(\int_{R}^{\infty}r^{a}\cdot Re(\partial_{r}u\cdot\bar{u})\, dr\Big)dg_{\mathbb{S}^{d-1}}\Big\},\nonumber 
\end{align}
 where, in order to integrate by parts, we have used the fact that
$u$ was compactly supported.

Moving the first term of the right hand side of (\ref{eq:GeneralHardyIntermediate})
to the left hand side, and using the fact that in polar coordinates
$dg_{\{r=R\}}=R^{d-1}dg_{\mathbb{S}^{d-1}}$ and $dx=r^{d-1}drdg_{\mathbb{S}^{d-1}}$,
we obtain:
\begin{multline}
\int_{\mathbb{R}^{d}\cap\{r\ge R\}}r^{-d+a}\cdot|u|^{2}\, dx+\frac{1}{a}\int_{\{r=R\}}R^{-(d-1)+a}\cdot|u|^{2}\, dg_{\{r=R\}}=\frac{2}{a}\int_{\mathbb{R}^{d}\cap\{r\ge R\}}r^{-(d-1)+a}\cdot Re(\partial_{r}u\cdot\bar{u})\, dx\le\\
\le\frac{2}{a}\Big(\int_{\mathbb{R}^{d}\cap\{r\ge R\}}r^{-(d-2)+a}\cdot|\partial_{r}u|^{2}\, dx\Big)^{1/2}\cdot\Big(\int_{\mathbb{R}^{d}\cap\{r\ge R\}}r^{-d+a}\cdot|u|^{2}\, dx\Big)^{1/2}.\label{eq:Generalhardy2}
\end{multline}
 The desired inequality (\ref{eq:GeneralHardyBound}) now readily
follows after absorbing the second factor of the right hand side of
(\ref{eq:Generalhardy2}) into the left hand side.
\end{proof}

\section{\label{sec:FreqDecomposition}Construction of the frequency decomposed
components of $\text{\textgreek{y}}$}

In this section, we will assume that we are given a smooth function
$\text{\textgreek{y}}:\mathcal{D}\rightarrow\mathbb{C}$ as in the
statement of Theorem \ref{thm:Theorem} solving $\square_{g}\text{\textgreek{y}}=0$
on $J^{+}(\text{\textgreek{S}})\cap\mathcal{D}$ with compactly supported
initial data on $\text{\textgreek{S}}$. We will introduce the parameters
$t^{*}>0$, $\text{\textgreek{w}}_{+}>1$ and $0<\text{\textgreek{w}}_{0}<1$,
and we will decompose the function $\text{\textgreek{y}}$ into components
with localised frequency support in the $t$ variable. We will always
identify $\mathcal{D}\backslash\mathcal{H}^{-}$ with $\mathbb{R}\times\text{\textgreek{S}}_{0}=\mathbb{R}\times(\text{\textgreek{S}}\cap\mathcal{D})$
under the flow of $T$ as explained in Section \ref{sub:CoordinateCharts}. 

In order to be able to apply the Fourier transform in the time variable,
we will first need to multiply $\text{\textgreek{y}}$ with a suitable
cut-off function in time, so that the resulting function has compact
support in $t$. This cut off procedure will be similar to the one
followed by Dafermos--Rodnianski in \cite{DafRod5}. We will then
establish various estimates for the frequency decomposed components
of (the cut-off of) $\text{\textgreek{y}}$. We will now proceed with
the details.

\subsection{\label{sub:Frequency-cut-off}Frequency cut-off }

Let $\text{\textgreek{q}}_{1}:\mathbb{R}_{+}\rightarrow[0,1]$ be
a smooth function such that $\text{\textgreek{q}}_{1}\equiv0$ on
$x\le R_{1}$ and $\text{\textgreek{q}}_{1}\equiv1$ on $x\ge R_{1}+1$,
for the given $R_{1}$ in the statement of Theorem \ref{thm:Theorem}.
Since we can always increase $R_{1}$without affecting the statement
of Theorem \ref{thm:Theorem}, whenever needed we will assume without
loss of generality that $R_{1}$ is large enough in terms of the geometry
of $(\mathcal{D},g)$.

We define the following ``distorted'' time functions on $\mathcal{D}$:

\begin{equation}
t_{+}=t-\frac{1}{2}\text{\textgreek{q}}_{1}(r)(r-R_{1}),\label{eq:DistortedTime+}
\end{equation}

\begin{equation}
t_{-}=t+\frac{1}{2}\text{\textgreek{q}}_{1}(r)(r-R_{1}).\label{eq:DistortedTime-}
\end{equation}
Note that the level sets of these time functions are indeed spacelike
hypersurfaces (provided $R_{1}\gg1$),%
\footnote{terminating at spacelike infinity%
} and they agree with the level sets of $t$ on $\{r\le R_{1}\}$.
Moreover: 
\begin{equation}
\{t_{+}=0\}\subset J^{+}(\{t=0\})\subset J^{+}(\{t_{-}=0\}).
\end{equation}
The reason for introducing these distorted time functions will become
apparent in the next section, when we will need to exchange decay
in time with decay in $r$. 

We will now describe the cut off of procedure in the time variable:
Let $\text{\textgreek{q}}_{2}:\mathbb{R}\rightarrow[0,1]$ be a smooth
function such that $\text{\textgreek{q}}_{2}\equiv1$ on $[1,+\infty)$
and $\text{\textgreek{q}}_{2}\equiv0$ on $(-\infty,0]$. Then for
any given $t^{*}>0$, we define the smooth cut-off function $h_{t^{*}}:\mathcal{D}\rightarrow[0,1]$:

\begin{equation}
h_{t^{*}}=h_{t^{*}}(t,r)=\text{\textgreek{q}}_{2}(t_{-}(t,r))\text{\textgreek{q}}_{2}(t^{*}-t_{+}(t,r)).\label{eq:PhysicalSpaceCutOff}
\end{equation}
We observe that $supp(h_{t^{*}})\subseteq\{t^{+}\le t^{*}\}\cap\{t^{-}\ge0\}$
and $supp(\nabla h_{t^{*}})\subseteq\{0\le t_{-}\le1\}\cup\{t^{*}-1\le t_{+}\le t^{*}\}$.
It is also readily verified that $\sup|\nabla h_{t^{*}}|,|\nabla^{2}h_{t^{*}}|\le C$
for a constant $C$ independent of $t^{*}$. 

Given a positive real number $t^{*}$ as before, we will define the
time cut-off $\text{\textgreek{y}}_{t^{*}}$ of \textgreek{y} as 

\begin{equation}
\text{\textgreek{y}}_{t^{*}}=h_{t^{*}}\cdot\text{\textgreek{y}}.\label{eq:Cut-OffDefinition}
\end{equation}
 Observe that $\text{\textgreek{y}}_{t^{*}}$ satisfies the equation
\begin{equation}
\square_{g}\text{\textgreek{y}}_{t^{*}}=F,
\end{equation}
where 
\begin{equation}
F=\partial^{\text{\textgreek{m}}}h_{t^{*}}\cdot\partial_{\text{\textgreek{m}}}\text{\textgreek{y}}+(\square_{g}h_{t^{*}})\cdot\text{\textgreek{y}}\label{eq:SourceTerm}
\end{equation}
is supported in $\{0\le t_{-}\le1\}\cup\{t^{*}-1\le t_{+}\le t^{*}\}$. 

In view of the assumption that the initial data for $\text{\textgreek{y}}$
on $\{t=0\}$ are supported in a set of the form $\{r\le R_{sup}\}$
(see Section \ref{sub:Remark}), we infer that $\text{\textgreek{y}}_{t^{*}}$
is supported in a cylinder of the form $\{r\lesssim R_{sup}+t^{*}\}$.
This fact will serve to show that some spacetime integrals of $\text{\textgreek{y}}_{t^{*}}$
are well defined at various points throughout the proof.

Note that for any $\text{\textgreek{t}}\in[0,t^{*}]$:

\begin{align}
\int_{t=\text{\textgreek{t}}}J_{\text{\textgreek{m}}}^{N}(\text{\textgreek{y}}_{t^{*}})n^{\text{\textgreek{m}}} & \le C\cdot\int_{t=\text{\textgreek{t}}}\big(h_{t^{*}}J_{\text{\textgreek{m}}}^{N}(\text{\textgreek{y}})n^{\text{\textgreek{m}}}+|\nabla h_{t^{*}}|\cdot|\text{\textgreek{y}}|^{2}\big)\\
 & \le C\cdot\big(\int_{t=\text{\textgreek{t}}}h_{t^{*}}J_{\text{\textgreek{m}}}^{N}(\text{\textgreek{y}})n^{\text{\textgreek{m}}}+\int_{\{t=\text{\textgreek{t}}\}\cap\{r\le2R_{1}\}}|\text{\textgreek{y}}|^{2}\big)\nonumber \\
 & \le C\cdot\int_{t=\text{\textgreek{t}}}J_{\text{\textgreek{m}}}^{N}(\text{\textgreek{y}})n^{\text{\textgreek{m}}}\le C\cdot\int_{t=0}J_{\text{\textgreek{m}}}^{N}(\text{\textgreek{y}})n^{\text{\textgreek{m}}}.\nonumber 
\end{align}

We will also need to perform a cut-off procedure in the frequency
domain: Let $\text{\textgreek{w}}_{0}>0$ be a (small) positive constant,
and $\text{\textgreek{w}}_{+}\gg\text{\textgreek{w}}_{0}$ a (large)
positive constant. We decompose the interval $[\text{\textgreek{w}}_{0},\text{\textgreek{w}}_{+}]$
into a finite number of closed intervals $\{[\text{\textgreek{w}}_{k},\text{\textgreek{w}}_{k+1}]\}_{k=0}^{n-1}$
such that $\text{\textgreek{w}}_{n}=\text{\textgreek{w}}_{+}$ and
$\frac{1}{4}\text{\textgreek{w}}_{0}<\text{\textgreek{w}}_{k+1}-\text{\textgreek{w}}_{k}<\frac{1}{2}\text{\textgreek{w}}_{0}$.
Note that $n\sim\frac{\text{\textgreek{w}}_{+}}{\text{\textgreek{w}}_{0}}$.
We will also set for $-n\le k\le-1$: $\text{\textgreek{w}}_{k}=-\text{\textgreek{w}}_{-k}$.

Fix a smooth function $\text{\textgreek{q}}_{3}:\mathbb{R}\rightarrow[0,1]$
such that $\text{\textgreek{q}}_{3}\equiv1$ on $[-1,1]$ and $\text{\textgreek{q}}_{3}\equiv0$
outside $(-\frac{9}{8},\frac{9}{8})$, and define the smooth cut-off
function $\text{\textgreek{z}}_{\le\text{\textgreek{w}}_{+}}:\mathbb{R}\rightarrow[0,1]$,
\begin{equation}
\text{\textgreek{z}}_{\le\text{\textgreek{w}}_{+}}(\text{\textgreek{w}})=\text{\textgreek{q}}_{3}(\frac{\text{\textgreek{w}}}{\text{\textgreek{w}}_{+}}).
\end{equation}
We also set $\text{\textgreek{z}}_{\ge\text{\textgreek{w}}_{+}}=1-\text{\textgreek{z}}_{\le\text{\textgreek{w}}_{+}}$.
These two functions will be used to split $\text{\textgreek{y}}_{t^{*}}$
into a high frequency component $\text{\textgreek{y}}_{\ge\text{\textgreek{w}}_{+}}$
and a low frequency component $\text{\textgreek{y}}_{\le\text{\textgreek{w}}_{+}}$. 

As we remarked in Section \ref{sub:SketchOfProof}, we will need to
perform a finer frequency cut-off on the low frequency component $\text{\textgreek{y}}_{\le\text{\textgreek{w}}_{+}}$.
To this end, we also define the following cut-off functions on $\mathbb{R}$:
\begin{equation}
\tilde{\text{\textgreek{z}}}_{0}(\text{\textgreek{w}})=\text{\textgreek{q}}_{3}(\frac{\text{\textgreek{w}}}{\text{\textgreek{w}}_{0}})
\end{equation}
and for $1\le k\le n$: 
\begin{equation}
\tilde{\text{\textgreek{z}}}_{k}(\text{\textgreek{w}})=\text{\textgreek{q}}_{3}(\frac{\text{\textgreek{w}}-\frac{\text{\textgreek{w}}_{k-1}+\text{\textgreek{w}}_{k}}{2}}{\frac{1}{2}(\text{\textgreek{w}}_{k}-\text{\textgreek{w}}_{k-1})}).
\end{equation}
For $-n\le k\le-1$, we set $\tilde{\text{\textgreek{z}}}_{k}(\text{\textgreek{w}})\doteq\tilde{\text{\textgreek{z}}}_{-k}(-\text{\textgreek{w}})$.
Finally, for $-n\le k\le n$, we define the functions 
\begin{equation}
\text{\textgreek{z}}_{k}(\text{\textgreek{w}})=\cfrac{\tilde{\text{\textgreek{z}}}_{k}(\text{\textgreek{w}})}{\sum_{i=-n}^{n}\tilde{\text{\textgreek{z}}}_{i}(\text{\textgreek{w}})}.\label{eq:zetak}
\end{equation}

The properties of the functions $\text{\textgreek{z}}_{k}$ that we
will need are the following:

\medskip{}

\begin{itemize}
\item $supp(\text{\textgreek{z}}_{\le\text{\textgreek{w}}_{+}})\subseteq[-\frac{9}{8}\text{\textgreek{w}}_{+},\frac{9}{8}\text{\textgreek{w}}_{+}]\mbox{, }supp(\text{\textgreek{z}}_{\ge\text{\textgreek{w}}_{+}})\subseteq(-\infty,-\text{\textgreek{w}}_{+}]\cup[\text{\textgreek{w}}_{+},+\infty)$
\item $supp(\text{\textgreek{z}}_{0})\subseteq[-\frac{9}{8}\text{\textgreek{w}}_{0},\frac{9}{8}\text{\textgreek{w}}_{0}]$
\item For $1\le k\le n$: \\
$supp(\text{\textgreek{z}}_{k})\subseteq[\text{\textgreek{w}}_{k-1}-\frac{\text{\textgreek{w}}_{0}}{8},\text{\textgreek{w}}_{k}+\frac{\text{\textgreek{w}}_{0}}{8}]$ 
\item For $-n\le k\le-1$:\\
$supp(\text{\textgreek{z}}_{k})\subseteq[-\text{\textgreek{w}}_{k}-\frac{\text{\textgreek{w}}_{0}}{8},-\text{\textgreek{w}}_{k-1}+\frac{\text{\textgreek{w}}_{0}}{8}]$
\item $\sum_{j=-n}^{n}\text{\textgreek{z}}_{j}\equiv1$ on $[-\text{\textgreek{w}}_{+},\text{\textgreek{w}}_{+}]$.
\end{itemize}
\medskip{}

For any smooth function $\text{\textgreek{Y}}:\mathcal{D}\rightarrow\mathbb{C}$
such that $\text{\textgreek{Y}}(t,x)$ has compact support in $t$
for any fixed $x\in\text{\textgreek{S}}\cap\mathcal{D}$ (having used
the identification of $\mathcal{D}\backslash\mathcal{H}^{-}$ with
$\mathbb{R}\times(\text{\textgreek{S}}\cap\mathcal{D})$, see Section
\ref{sub:CoordinateCharts}), we will denote the Fourier transform
of $\text{\textgreek{Y}}$ in the $t$ coordinate by $\hat{\text{\textgreek{Y}}}$,
and the inverse Fourier transform in $t$ with $\check{\text{\textgreek{Y}}}$. 

With this notation, we define the following partition of $\text{\textgreek{y}}_{t^{*}}$
(defined as (\ref{eq:Cut-OffDefinition})) into frequency decomposed
components:
\begin{itemize}
\item $\text{\textgreek{y}}_{\le\text{\textgreek{w}}_{+}}(t,\cdot)\doteq\int_{-\infty}^{\infty}\text{\textgreek{z}}_{\le\text{\textgreek{w}}_{+}}(\text{\textgreek{w}})\cdot e^{i\text{\textgreek{w}}t}\hat{\text{\textgreek{y}}}_{t^{*}}(\text{\textgreek{w}},\cdot)\, d\text{\textgreek{w}}$
\item $\text{\textgreek{y}}_{\ge\text{\textgreek{w}}_{+}}(t,\cdot)\doteq\int_{-\infty}^{\infty}\text{\textgreek{z}}_{\ge\text{\textgreek{w}}_{+}}(\text{\textgreek{w}})\cdot e^{i\text{\textgreek{w}}t}\hat{\text{\textgreek{y}}}_{t^{*}}(\text{\textgreek{w}},\cdot)\, d\text{\textgreek{w}}=\text{\textgreek{y}}_{t^{*}}(t,\cdot)-\text{\textgreek{y}}_{\le\text{\textgreek{w}}_{+}}(t,\cdot)$
\end{itemize}
\noindent and we decompose $\text{\textgreek{y}}_{\le\text{\textgreek{w}}_{+}}$
further:
\begin{itemize}
\item $\text{\textgreek{y}}_{k}(t,\cdot)\doteq\int_{-\infty}^{\infty}\text{\textgreek{z}}_{k}(\text{\textgreek{w}})\cdot e^{i\text{\textgreek{w}}t}\hat{\text{\textgreek{y}}}_{\le\text{\textgreek{w}}_{+}}(\text{\textgreek{w}},\cdot)\, d\text{\textgreek{w}}=\int_{-\infty}^{\infty}\text{\textgreek{z}}_{\le\text{\textgreek{w}}_{+}}(\text{\textgreek{w}})\cdot\text{\textgreek{z}}_{k}(\text{\textgreek{w}})\cdot e^{i\text{\textgreek{w}}t}\hat{\text{\textgreek{y}}}_{t^{*}}(\text{\textgreek{w}},\cdot)\, d\text{\textgreek{w}}$
for $-n\le k\le n$.
\end{itemize}
\noindent Note that $\text{\textgreek{y}}_{\le\text{\textgreek{w}}_{+}}+\text{\textgreek{y}}_{\ge\text{\textgreek{w}}_{+}}=\text{\textgreek{y}}_{t^{*}}$
and $\sum_{k=-n}^{n}\text{\textgreek{y}}_{k}=\text{\textgreek{y}}_{\le\text{\textgreek{w}}_{+}}$.

In the same way, we will decompose $F$ (defined as (\ref{eq:SourceTerm}))
in $F_{k}$, $F_{\le\text{\textgreek{w}}_{+}}$, $F_{\ge\text{\textgreek{w}}_{+}}$.
Due to the linearity of the cut-off operators, we have $\square_{g}\text{\textgreek{y}}_{k}=F_{k}$,
$\square_{g}\text{\textgreek{y}}_{\le\text{\textgreek{w}}_{+}}=F_{\le\text{\textgreek{w}}_{+}}$,
$\square_{g}\text{\textgreek{y}}_{\ge\text{\textgreek{w}}_{+}}=F_{\ge\text{\textgreek{w}}_{+}}$.

Note also that since $\text{\textgreek{y}}_{t^{*}}$ is supported
in the cylinder $\{r\lesssim R_{sup}+t^{*}\}$, the same is also true
for the functions $\text{\textgreek{y}}_{k},\text{\textgreek{y}}_{\le\text{\textgreek{w}}_{+}},\text{\textgreek{y}}_{\ge\text{\textgreek{w}}_{+}}$.

\subsection{\label{sub:BoundsForPsiK}Bounds for the frequency-decomposed components}

In this section, we will establish some useful estimates for the energy
of $\text{\textgreek{y}}_{k},\text{\textgreek{y}}_{\le\text{\textgreek{w}}_{+}},\text{\textgreek{y}}_{\ge\text{\textgreek{w}}_{+}}$
(as well as for the ``error'' terms $F_{k},F_{\le\text{\textgreek{w}}_{+}},F_{\ge\text{\textgreek{w}}_{+}}$)
in terms of the initial energy of $\text{\textgreek{y}}$. 

We will start by producing some basic estimates for the projection
operators $\text{\textgreek{z}}_{k},\text{\textgreek{z}}_{\le\text{\textgreek{w}}_{+}}$.
Since the functions $\text{\textgreek{z}}_{\le\text{\textgreek{w}}_{+}}$
and $\text{\textgreek{z}}_{k}\cdot\text{\textgreek{z}}_{\le\text{\textgreek{w}}_{+}}$
are smooth with compact support, their inverse Fourier transforms
$h_{\le\text{\textgreek{w}}_{+}}$ and $h_{k}$ are Schwartz functions.
The following lemma establishes some Schwartz bounds for $h_{\le\text{\textgreek{w}}_{+}},h_{k}$.
In view of the fact that 
\begin{equation}
\text{\textgreek{y}}_{\le\text{\textgreek{w}}_{+}}(t,\cdot)=\int_{-\infty}^{\infty}h_{\le\text{\textgreek{w}}_{+}}(t-s)\cdot\text{\textgreek{y}}_{t^{*}}(s,\cdot)\, ds,
\end{equation}
 and for $-n\le k\le n$: 
\begin{equation}
\text{\textgreek{y}}_{k}(t,\cdot)=\int_{-\infty}^{\infty}h_{k}(t-s)\cdot\text{\textgreek{y}}_{t^{*}}(s,\cdot)\, ds,
\end{equation}
 the Schwartz bounds for $h_{\le\text{\textgreek{w}}_{+}}$ and $h_{k}$
will then be used to establish useful estimates for the functions
$\text{\textgreek{y}}_{k},\text{\textgreek{y}}_{\le\text{\textgreek{w}}_{+}},\text{\textgreek{y}}_{\ge\text{\textgreek{w}}_{+}}$.
\begin{lem}
\label{lem:SchwartzBoundsForProjections} For $\text{\textgreek{w}}_{0}\le1$
and $\text{\textgreek{w}}_{+}\ge1$, the convolution kernels $h_{\le\text{\textgreek{w}}_{+}}$
and $h_{k}$ (for $-n\le k\le n$) satisfy:
\end{lem}
\begin{equation}
\sup_{t}\big(\text{\textgreek{w}}_{+}^{-1}(1+|\text{\textgreek{w}}_{+}t|)^{q}|h_{\le\text{\textgreek{w}}_{+}}(t)|\big)\le C_{q}\label{eq:SchwartzBounds1}
\end{equation}
 and

\begin{equation}
\sup_{t}\big(\text{\textgreek{w}}_{0}^{-1}(1+|\text{\textgreek{w}}_{0}t|)^{q}|h_{k}(t)|\big)\le C_{q}.\label{eq:SchwartzBounds2}
\end{equation}

\begin{proof}
From the definition $\text{\textgreek{z}}_{\le\text{\textgreek{w}}_{+}}(\text{\textgreek{w}})=\text{\textgreek{q}}_{3}(\frac{\text{\textgreek{w}}}{\text{\textgreek{w}}_{+}})$,
we compute that $h_{\le\text{\textgreek{w}}_{+}}(t)=\text{\textgreek{w}}_{+}\cdot\check{\text{\textgreek{q}}}_{3}(\text{\textgreek{w}}_{+}t)$,
where $\check{\text{\textgreek{q}}}_{3}$ is a Schwartz function being
the inverse Fourier transform of the compactly supported $\text{\textgreek{q}}_{3}$.
Hence, we can bound for each $q\in\mathbb{N}$:

\begin{equation}
\sup_{t}\big(\text{\textgreek{w}}_{+}^{-1}(1+|\text{\textgreek{w}}_{+}t|)^{q}|h_{\le\text{\textgreek{w}}_{+}}(t)|\big)\le C_{q}.\label{eq:SchwartzBounds1-1}
\end{equation}

Similarly, due to the definition of $\text{\textgreek{z}}_{k}$ (\ref{eq:zetak})
and the bounds $\frac{1}{4}\text{\textgreek{w}}_{0}<\text{\textgreek{w}}_{k+1}-\text{\textgreek{w}}_{k}<\frac{1}{2}\text{\textgreek{w}}_{0}$,
we also have for each $q\in\mathbb{N}$: 
\begin{equation}
\int_{-\infty}^{\infty}\Big|\frac{d^{q}}{d\text{\textgreek{w}}^{q}}\{\text{\textgreek{z}}_{k}\cdot\text{\textgreek{z}}_{\le\text{\textgreek{w}}_{+}}\}(\text{\textgreek{w}})\Big|\, d\text{\textgreek{w}}\le C_{q}\cdot\sum_{i=0}^{q}\int_{-\infty}^{\infty}|\text{\textgreek{z}}_{k}^{(i)}(\text{\textgreek{w}})|\cdot|\text{\textgreek{z}}_{\le\text{\textgreek{w}}_{+}}^{(q-i)}(\text{\textgreek{w}})|\, d\text{\textgreek{w}}\le C_{q}\text{\textgreek{w}}_{0}^{1-q}.
\end{equation}
 In the above, the constant $C_{q}$ does not depend on $k$, since
$\sup_{\text{\textgreek{w}}}|\text{\textgreek{z}}_{\le\text{\textgreek{w}}_{+}}^{(l)}(\text{\textgreek{w}})|\le C_{l}\text{\textgreek{w}}_{+}^{-l}\le C_{l}$
and $\int_{-\infty}^{\infty}|\text{\textgreek{z}}_{k}^{(l)}(\text{\textgreek{w}})|\, d\text{\textgreek{w}}\le C_{l}\cdot\text{\textgreek{w}}_{0}^{-l+1}$. 

Thus, since $h_{k}=\check{(\text{\textgreek{z}}_{k}\cdot\text{\textgreek{z}}_{\le\text{\textgreek{w}}_{+}})}$,
we can bound for each $q\in\mathbb{N}$, $-n\le k\le n$:
\[
\sup_{t}|(\text{\textgreek{w}}_{0}t){}^{q}\cdot h_{k}(t)|\le C\cdot\int_{-\infty}^{\infty}\text{\textgreek{w}}_{0}^{q}\Big|\frac{d^{q}}{d\text{\textgreek{w}}^{q}}\{\text{\textgreek{z}}_{k}\cdot\text{\textgreek{z}}_{\le\text{\textgreek{w}}_{+}}\}(\text{\textgreek{w}})\Big|\, d\text{\textgreek{w}}\le C_{q}\text{\textgreek{w}}_{0}^{1-q+q}=C_{q}\text{\textgreek{w}}_{0}
\]
 and hence for any $q\in\mathbb{N}$ we obtain the desired estimate:

\begin{equation}
\sup_{t}\big(\text{\textgreek{w}}_{0}^{-1}(1+|\text{\textgreek{w}}_{0}t|)^{q}|h_{k}(t)|\big)\le C_{q}.\label{eq:SchwartzBounds2-1}
\end{equation}

\end{proof}
\medskip{}

Let us also state a straightforward lemma that will be used frequently
throughout this paper:
\begin{lem}
\label{lem:ComparisonTT-}For any $q\ge0$, there exist constants
$c_{q},C_{q}$, such that for any solution $\text{\textgreek{y}}$
to the wave equation $\square_{g}\text{\textgreek{y}}=0$ on $\mathcal{D}$
and any $\text{\textgreek{t}}\in\mathbb{R}$ we can bound
\begin{equation}
c_{q}\cdot\int_{\{t_{-}=\text{\textgreek{t}}\}}r^{q}\cdot J_{\text{\textgreek{m}}}^{N}(\text{\textgreek{y}})n^{\text{\textgreek{m}}}\le\int_{\{t=\text{\textgreek{t}}\}}r^{q}\cdot J_{\text{\textgreek{m}}}^{N}(\text{\textgreek{y}})n^{\text{\textgreek{m}}}\le C_{q}\cdot\int_{\{t_{-}=\text{\textgreek{t}}\}}r^{q}\cdot J_{\text{\textgreek{m}}}^{N}(\text{\textgreek{y}})n^{\text{\textgreek{m}}}.\label{eq:ComparisonEnergy}
\end{equation}

In paticular, for $q=0$ we have the identity
\begin{equation}
\int_{\{t_{-}=\text{\textgreek{t}}\}}J_{\text{\textgreek{m}}}^{N}(\text{\textgreek{y}})n^{\text{\textgreek{m}}}=\int_{\{t=\text{\textgreek{t}}\}}J_{\text{\textgreek{m}}}^{N}(\text{\textgreek{y}})n^{\text{\textgreek{m}}}.\label{eq:EnergyIdentity}
\end{equation}
\end{lem}
\begin{proof}
As we remarked earlier, $\{t_{-}=\text{\textgreek{t}}\}\cap\{r\le R_{1}\}\equiv\{t=\text{\textgreek{t}}\}\cap\{r\le R_{1}\}$,
and $N\equiv T$ for $r\ge R_{1}$. Hence, since $T$ is Killing and
$\text{\textgreek{y}}$ solves the equation $\square_{g}\text{\textgreek{y}}=0$,
the current $J_{\text{\textgreek{m}}}^{N}(\text{\textgreek{y}})=T_{\text{\textgreek{m}\textgreek{n}}}(\text{\textgreek{y}})N^{\text{\textgreek{n}}}$
is divergence free for $r\ge R_{1}$. Integrating, therefore, $\nabla^{\text{\textgreek{m}}}J_{\text{\textgreek{m}}}^{N}(\text{\textgreek{y}})=0$
in the domain bounded by $\{t_{-}=\text{\textgreek{t}}\}$ and $\{t=\text{\textgreek{t}}\}$
which lies entirely in the region $\{r\ge R_{1}\}$), identity (\ref{eq:EnergyIdentity})
follows immediately.

In order to show (\ref{eq:ComparisonEnergy}), we define for $l\in\mathbb{Z}$
the intervals $I_{l}=[2^{l},2^{l+1}]\subseteq\mathbb{R}_{+}$, for
$l\ge1$, $I_{0}=[0,2]$ for $l=0$, and $I_{l}=\emptyset$ for $l<0$.
We then compute that 
\begin{equation}
\int_{t=\text{\textgreek{t}}}r^{q}\cdot J_{\text{\textgreek{m}}}^{N}(\text{\textgreek{y}})n^{\text{\textgreek{m}}}\sim_{q}\,\sum_{l=0}^{\infty}2^{q\cdot l}\cdot\int_{\{t=\text{\textgreek{t}}\}\cap\{r\in I_{l}\}}J_{\text{\textgreek{m}}}^{N}(\text{\textgreek{y}})n^{\text{\textgreek{m}}},\label{eq:DyadicDecompositionEnergy}
\end{equation}
and similarly 
\begin{equation}
\int_{t_{-}=\text{\textgreek{t}}}r^{q}\cdot J_{\text{\textgreek{m}}}^{N}(\text{\textgreek{y}})n^{\text{\textgreek{m}}}\sim_{q}\,\sum_{l=0}^{\infty}2^{q\cdot l}\cdot\int_{\{t_{-}=\text{\textgreek{t}}\}\cap\{r\in I_{l}\}}J_{\text{\textgreek{m}}}^{N}(\text{\textgreek{y}})n^{\text{\textgreek{m}}}.\label{eq:DyadicDecompositionEnergy-1}
\end{equation}

Since $t_{-}=t+\frac{1}{2}\text{\textgreek{q}}_{1}(r)(r-R_{1})$,
we deduce that there exists an integer $Z>0$ such that for every
$\text{\textgreek{t}}\in\mathbb{R}$ and every $l\in\mathbb{N}$:
\begin{equation}
J^{+}\big(\{t_{-}=\text{\textgreek{t}}\}\cap\{r\in I_{l}\}\big)\cap\{t=\text{\textgreek{t}}\}\subseteq\{t=\text{\textgreek{t}}\}\cap\{r\in\cup_{i=l-Z}^{l+Z}I_{i}\}\label{eq:Relation1}
\end{equation}
 and
\begin{equation}
J^{-}\big(\{t=\text{\textgreek{t}}\}\cap\{r\in I_{l}\}\big)\cap\{t_{-}=\text{\textgreek{t}}\}\subseteq\{t_{-}=\text{\textgreek{t}}\}\cap\{r\in\cup_{i=l-Z}^{l+Z}I_{i}\}.\label{eq:Relation2}
\end{equation}

Since $\{t_{-}=\text{\textgreek{t}}\}\equiv\{t=\text{\textgreek{t}}\}$
for $r\le R_{1}$ and $T\equiv N$ for $r\ge R_{1}$, the inclusions
(\ref{eq:Relation1}) and (\ref{eq:Relation2}) imply, after integrating
the identity $\nabla^{\text{\textgreek{m}}}J_{\text{\textgreek{m}}}^{T}=0$
on $J^{-}\big(\{t=\text{\textgreek{t}}\}\cap\{r\in I_{l}\}\big)\cap J^{+}\big(\{t_{-}=\text{\textgreek{t}}\}\big)$,
that for any $l$:
\begin{equation}
\int_{\{t=\text{\textgreek{t}}\}\cap\{r\in I_{l}\}}J_{\text{\textgreek{m}}}^{N}(\text{\textgreek{y}})n^{\text{\textgreek{m}}}\le\int_{\{t_{-}=\text{\textgreek{t}}\}\cap\{r\in\cup_{i=l-Z}^{l+Z}I_{i}\}}J_{\text{\textgreek{m}}}^{N}(\text{\textgreek{y}})n^{\text{\textgreek{m}}}\label{eq:Relation3}
\end{equation}
 and similarly, after an integration on $J^{+}\big(\{t_{-}=\text{\textgreek{t}}\}\cap\{r\in I_{l}\}\big)\cap J^{-}\big(\{t=\text{\textgreek{t}}\}\big)$:
\begin{equation}
\int_{\{t_{-}=\text{\textgreek{t}}\}\cap\{r\in I_{l}\}}J_{\text{\textgreek{m}}}^{N}(\text{\textgreek{y}})n^{\text{\textgreek{m}}}\le\int_{\{t=\text{\textgreek{t}}\}\cap\{r\in\cup_{i=l-Z}^{l+Z}I_{i}\}}J_{\text{\textgreek{m}}}^{N}(\text{\textgreek{y}})n^{\text{\textgreek{m}}}.\label{eq:Relation4}
\end{equation}

Thus, in view of (\ref{eq:DyadicDecompositionEnergy}), (\ref{eq:DyadicDecompositionEnergy-1}),
(\ref{eq:Relation3}) and (\ref{eq:Relation4}) we can bound:

\begin{align}
\int_{t=\text{\textgreek{t}}}r^{q}\cdot J_{\text{\textgreek{m}}}^{N}(\text{\textgreek{y}})n^{\text{\textgreek{m}}} & \le C_{q}\cdot\sum_{l=0}^{\infty}2^{ql}\cdot\int_{\{t=\text{\textgreek{t}}\}\cap\{r\in I_{l}\}}J_{\text{\textgreek{m}}}^{N}(\text{\textgreek{y}})n^{\text{\textgreek{m}}}\\
 & \leq C_{q}\cdot\sum_{l=0}^{\infty}2^{ql}\cdot\int_{\{t_{-}=\text{\textgreek{t}}\}\cap\{r\in\cup_{i=l-Z}^{l+Z}I_{i}\}}J_{\text{\textgreek{m}}}^{N}(\text{\textgreek{y}})n^{\text{\textgreek{m}}}\nonumber \\
 & \le2ZC_{q}\cdot\sum_{l=0}^{\infty}2^{q(l+Z+1)}\cdot\int_{\{t_{-}=\text{\textgreek{t}}\}\cap\{r\in I_{l}\}}J_{\text{\textgreek{m}}}^{N}(\text{\textgreek{y}})n^{\text{\textgreek{m}}}\nonumber \\
 & =C_{q}\cdot2Z\cdot2^{q(Z+1)}\sum_{l=0}^{\infty}2^{ql}\cdot\int_{\{t_{-}=\text{\textgreek{t}}\}\cap\{r\in I_{l}\}}J_{\text{\textgreek{m}}}^{N}(\text{\textgreek{y}})n^{\text{\textgreek{m}}}\nonumber \\
 & \le C_{q}\cdot\int_{t_{-}=\text{\textgreek{t}}}r^{q}\cdot J_{\text{\textgreek{m}}}^{N}(\text{\textgreek{y}})n^{\text{\textgreek{m}}}\nonumber 
\end{align}
 and similarly

\begin{align}
\int_{t_{-}=\text{\textgreek{t}}}r^{q}\cdot J_{\text{\textgreek{m}}}^{N}(\text{\textgreek{y}})n^{\text{\textgreek{m}}} & \le C_{q}\cdot\sum_{l=0}^{\infty}2^{ql}\cdot\int_{\{t_{-}=\text{\textgreek{t}}\}\cap\{r\in I_{l}\}}J_{\text{\textgreek{m}}}^{N}(\text{\textgreek{y}})n^{\text{\textgreek{m}}}\\
 & \leq C_{q}\cdot\sum_{l=0}^{\infty}2^{ql}\cdot\int_{\{t=\text{\textgreek{t}}\}\cap\{r\in\cup_{i=l-Z}^{l+Z}I_{i}\}}J_{\text{\textgreek{m}}}^{N}(\text{\textgreek{y}})n^{\text{\textgreek{m}}}\nonumber \\
 & \le2ZC_{q}\cdot\sum_{l=0}^{\infty}2^{q(l+Z+1)}\cdot\int_{\{t=\text{\textgreek{t}}\}\cap\{r\in I_{l}\}}J_{\text{\textgreek{m}}}^{N}(\text{\textgreek{y}})n^{\text{\textgreek{m}}}\nonumber \\
 & \le C_{q}\cdot\int_{t=\text{\textgreek{t}}}r^{q}\cdot J_{\text{\textgreek{m}}}^{N}(\text{\textgreek{y}})n^{\text{\textgreek{m}}},\nonumber 
\end{align}
 thus reaching the desired inequality

\begin{equation}
c_{q}\cdot\int_{t_{-}=\text{\textgreek{t}}}r^{q}\cdot J_{\text{\textgreek{m}}}^{N}(\text{\textgreek{y}})n^{\text{\textgreek{m}}}\le\int_{t=\text{\textgreek{t}}}r^{q}\cdot J_{\text{\textgreek{m}}}^{N}(\text{\textgreek{y}})n^{\text{\textgreek{m}}}\le C_{q}\cdot\int_{t_{-}=\text{\textgreek{t}}}r^{q}\cdot J_{\text{\textgreek{m}}}^{N}(\text{\textgreek{y}})n^{\text{\textgreek{m}}}.
\end{equation}

\end{proof}
We will now establish some estimates for the error terms $F_{k},F_{\le\text{\textgreek{w}}_{+}},F_{\ge\text{\textgreek{w}}_{+}}$.
Note that from now on we will always assume without loss of generality
that $\text{\textgreek{w}}_{0}\le1$ and $\text{\textgreek{w}}_{+}\ge1$.
\begin{lem}
\label{lem:BoundF}We can bound for any $q,q'\in\mathbb{N}$ and any
$0\le t_{1}\le t_{2}\le t^{*}$: 
\begin{equation}
\int_{\mathcal{R}(t_{1},t_{2})}r^{q}|F_{k}|^{2}\le C_{q,q'}(\text{\textgreek{w}}_{0})\cdot\big((1+t_{1})^{-q'}+(1+t^{*}-t_{2})^{-q'}\big)\int_{t=0}J_{\text{\textgreek{m}}}^{N}(\text{\textgreek{y}})n^{\text{\textgreek{m}}}
\end{equation}
 for $-n\le k\le n$. The same inequality also holds for $F_{\le\text{\textgreek{w}}_{+}},F_{\ge\text{\textgreek{w}}_{+}}$
in place of $F_{k}$.\end{lem}
\begin{proof}
Since $F_{k}(t,\cdot)=\int_{-\infty}^{\infty}h_{k}(t-s)\cdot F(s,\cdot)\, ds$,
we can bound ($dx$ denoting in the next lines the $dg_{\text{\textgreek{S}}}$
integration measure):

\begin{align}
\int_{\mathcal{R}(t_{1},t_{2})}r^{q}|F_{k}|^{2} & \le C\cdot\int_{\text{\textgreek{S}}_{0}}\int_{t_{1}}^{t_{2}}r^{q}\Big|\int_{-\infty}^{\infty}h_{k}(t-s)\cdot F(s,x)\, ds\Big|^{2}\, dtdx\label{eq:FInequality}\\
 & \le C\cdot R_{1}^{q}\cdot\Big(\int_{\text{\textgreek{S}}_{0}\cap\{r\le R_{1}\}}\int_{t_{1}}^{t_{2}}\Big|\int_{-\infty}^{\infty}h_{k}(t-s)\cdot F(s,x)\, ds\Big|^{2}\, dtdx\Big)+\nonumber \\
 & \hphantom{\le C}+\int_{\text{\textgreek{S}}_{0}\cap\{r\ge R_{1}\}}r^{q}\int_{t_{1}}^{t_{2}}\Big|\int_{-\infty}^{\infty}h_{k}(t-s)\cdot F(s,x)\, ds\Big|^{2}\, dtdx.\nonumber 
\end{align}

Because in the region $\{r\le R_{1}\}$, $F=\partial^{\text{\textgreek{m}}}h_{t^{*}}\cdot\partial_{\text{\textgreek{m}}}\text{\textgreek{y}}+(\square_{g}h_{t^{*}})\cdot\text{\textgreek{y}}$
is supported in $\{0\le t\le1\}\cup\{t^{*}-1\le t\le t^{*}\}$ and
$\sup_{t}|\text{\textgreek{w}}_{0}^{-1}(1+|\text{\textgreek{w}}_{0}t|)^{q'+1}h_{k}(t)|\le C_{q'+1}$
due to (\ref{eq:SchwartzBounds2}), the first term of the right hand
side of (\ref{eq:FInequality}) can be bounded by
\begin{equation}
\begin{split}\int_{\text{\textgreek{S}}_{0}\cap\{r\le R_{1}\}}\int_{t_{1}}^{t_{2}}\Big|\int_{-\infty}^{\infty} & h_{k}(t-s)\cdot F(s,x)\, ds\Big|^{2}\, dtdx\le\\
\le & \int_{\text{\textgreek{S}}_{0}\cap\{r\le R_{1}\}}\int_{t_{1}}^{t_{2}}\Big|\int_{[0,1]\cup[t^{*}-1,t^{*}]}\frac{C_{q'}\text{\textgreek{w}}_{0}}{(1+\text{\textgreek{w}}_{0}|t-s|)^{q'+1}}\cdot F(s,x)\, ds\Big|^{2}\, dtdx\\
\le & C_{q'}(\text{\textgreek{w}}_{0})\int_{\text{\textgreek{S}}_{0}\cap\{r\le R_{1}\}}\int_{t_{1}}^{t_{2}}\Big\{(1+|t|)^{-q'-1}\int_{[0,1]}|F(s,x)|^{2}\, ds+(1+|t^{*}-t|)^{-q'-1}\int_{[t^{*}-1,t^{*}]}|F(s,x)|^{2}\, ds\Big\}\, dtdx\\
\le & C_{q'}(\text{\textgreek{w}}_{0})\cdot\big((1+t_{1})^{-q'}+(1+t^{*}-t_{2})^{-q'}\big)\int_{t=0}J_{\text{\textgreek{m}}}^{N}(\text{\textgreek{y}})n^{\text{\textgreek{m}}},
\end{split}
\label{eq:BoundFirstTermFIneq}
\end{equation}
the last inequality being a consequence of the boundedness assumption
\hyperref[Assumption 4]{4}, the fact that $F=\partial^{\text{\textgreek{m}}}h_{t^{*}}\cdot\partial_{\text{\textgreek{m}}}\text{\textgreek{y}}+(\square_{g}h_{t^{*}})\cdot\text{\textgreek{y}}$,
as well as the Hardy inequality 
\begin{equation}
\int_{\{t=\text{\textgreek{t}}\}\cap\{r\le R_{1}\}}|\text{\textgreek{y}}|^{2}\le CR_{1}^{2}\cdot\int_{t=\text{\textgreek{t}}}J_{\text{\textgreek{m}}}^{N}(\text{\textgreek{y}})n^{\text{\textgreek{m}}}
\end{equation}
 following from (\ref{eq:GeneralHardyBound}). 

For the second term of (\ref{eq:FInequality}), we will use the definition
of $t_{+},t_{-}$ and the support of $F$ to conclude that for $r\ge R_{1}$,
$l\in\mathbb{N}$ and $t\in(t_{1},t_{2})$:

\begin{align}
\int_{-\infty}^{\infty}|h_{k}(t-s)|\cdot|F(s,r,\text{\textgreek{sv}})|\, ds & =\int_{-\frac{1}{2}\text{\textgreek{q}}_{1}\cdot(r-R_{1})}^{1-\frac{1}{2}\text{\textgreek{q}}_{1}\cdot(r-R_{1})}|h_{k}(t-s)|\cdot|F(s,r,\text{\textgreek{sv}})|\, ds+\\
 & \hphantom{=C}+\int_{t^{*}-1+\frac{1}{2}\text{\textgreek{q}}_{1}\cdot(r-R_{1})}^{t^{*}+\frac{1}{2}\text{\textgreek{q}}_{1}\cdot(r-R_{1})}|h_{k}(t-s)|\cdot|F(s,r,\text{\textgreek{sv}})|\, ds\nonumber \\
 & \le C_{l}(\text{\textgreek{w}}_{0})\Big(\frac{1}{(1+|t+\frac{1}{2}(r-R_{1})|)^{l}}\int_{-\frac{1}{2}\text{\textgreek{q}}_{1}\cdot(r-R_{1})}^{1-\frac{1}{2}\text{\textgreek{q}}_{1}\cdot(r-R_{1})}|F(s,r,\text{\textgreek{sv}})|\, ds\nonumber \\
 & \hphantom{\le C_{l}(\text{\textgreek{w}}_{0})\Big(}+\frac{1}{(1+|t-t^{*}-\frac{1}{2}(r-R_{1})|)^{l}}\int_{t^{*}-1+\frac{1}{2}\text{\textgreek{q}}_{1}\cdot(r-R_{1}))^{1/2}}^{t^{*}+\frac{1}{2}\text{\textgreek{q}}_{1}\cdot(r-R_{1})}|F(s,r,\text{\textgreek{sv}})|\, ds\Big).\nonumber 
\end{align}
Therefore, by choosing $l$ large enough (with respect to $q,q'$),
we conclude
\begin{equation}
\begin{split}\int_{\text{\textgreek{S}}_{0}\cap\{r\ge R_{1}\}}\int_{t_{1}}^{t_{2}}r^{q}\Big| & \int_{-\infty}^{\infty}h_{k}(t-s)\cdot F(s,x)\, ds\Big|^{2}\, dtdx\le\\
\le & C_{l}(\text{\textgreek{w}}_{0})\cdot\int_{\text{\textgreek{S}}_{0}\cap\{r\ge R_{1}\}}\int_{t_{1}}^{t_{2}}\frac{r^{q}}{(1+|t+\frac{1}{2}(r-R_{1})|)^{2l}}\Big(\int_{-\frac{1}{2}\text{\textgreek{q}}_{1}\cdot(r-R_{1})}^{1-\frac{1}{2}\text{\textgreek{q}}_{1}\cdot(r-R_{1})}|F(s,x)|^{2}\, ds\Big)\, dtdx+\\
 & \hphantom{C_{q,q'}(}+C_{l}(\text{\textgreek{w}}_{0})\cdot\int_{\text{\textgreek{S}}_{0}\cap\{r\ge R_{1}\}}\int_{t_{1}}^{t_{2}}\frac{r^{q}}{(1+|t-t^{*}-\frac{1}{2}(r-R_{1})|)^{2l}}\Big(\int_{t^{*}-1+\frac{1}{2}\text{\textgreek{q}}_{1}\cdot(r-R_{1})}^{t^{*}+\frac{1}{2}\text{\textgreek{q}}_{1}\cdot(r-R_{1})}|F(s,x)|^{2}\, ds\Big)\, dtdx\le\\
\le & C_{l}(\text{\textgreek{w}}_{0})\cdot\int_{\text{\textgreek{S}}_{0}\cap\{r\ge R_{1}\}}\big(\int_{t_{1}}^{t_{2}}\frac{r^{q}}{(1+|t+\frac{1}{2}(r-R_{1})|)^{2l}}\, dt\big)\cdot\Big(\int_{\frac{1}{2}\text{\textgreek{q}}_{1}\cdot(r-R_{1})}^{1-\frac{1}{2}\text{\textgreek{q}}_{1}\cdot(r-R_{1})}\big(J_{\text{\textgreek{m}}}^{N}(\text{\textgreek{y}})n^{\text{\textgreek{m}}}+|\text{\textgreek{y}}|^{2}\big)(s,x)\, ds\Big)\, dx+\\
 & \hphantom{C_{q,q'}(}+C_{l}(\text{\textgreek{w}}_{0})\cdot\int_{\text{\textgreek{S}}_{0}\cap\{r\ge R_{1}\}}\big(\int_{t_{1}}^{t_{2}}\frac{r^{q}}{(1+|t-t^{*}-\frac{1}{2}(r-R_{1})|)^{2l}}\, dt\big)\cdot\Big(\int_{t^{*}-1+\frac{1}{2}\text{\textgreek{q}}_{1}\cdot(r-R_{1})}^{t^{*}+\frac{1}{2}\text{\textgreek{q}}_{1}\cdot(r-R_{1})}\big(J_{\text{\textgreek{m}}}^{N}(\text{\textgreek{y}})n^{\text{\textgreek{m}}}+|\text{\textgreek{y}}|^{2}\big)(s,x)\, ds\Big)\, dx\le
\end{split}
\label{eq:BoundSecondTermFIneq}
\end{equation}
\[
\begin{split}\hphantom{\int_{\text{\textgreek{S}}_{0}\cap\{r\ge R_{1}\}}\int_{t_{1}}^{t_{2}}r}\le & C_{q,q'}(\text{\textgreek{w}}_{0})\cdot(1+t_{1})^{-q'}\int_{\text{\textgreek{S}}_{0}\cap\{r\ge R_{1}\}}\frac{1}{1+r^{4}}\Big(\int_{-\frac{1}{2}\text{\textgreek{q}}_{1}\cdot(r-R_{1})}^{1-\frac{1}{2}\text{\textgreek{q}}_{1}\cdot(r-R_{1})}\big(J_{\text{\textgreek{m}}}^{N}(\text{\textgreek{y}})n^{\text{\textgreek{m}}}+|\text{\textgreek{y}}|^{2}\big)(s,x)\, ds\Big)\, dx+\\
 & \hphantom{C_{q,q'}(}+C_{q,q'}(\text{\textgreek{w}}_{0})\cdot(1+t^{*}-t_{2})^{-q'}\int_{\text{\textgreek{S}}_{0}\cap\{r\ge R_{1}\}}\frac{1}{1+r^{4}}\Big(\int_{t^{*}-1+\frac{1}{2}\text{\textgreek{q}}_{1}\cdot(r-R_{1})}^{t^{*}+\frac{1}{2}\text{\textgreek{q}}_{1}\cdot(r-R_{1})}\big(J_{\text{\textgreek{m}}}^{N}(\text{\textgreek{y}})n^{\text{\textgreek{m}}}+|\text{\textgreek{y}}|^{2}\big)(s,x)\, ds\Big)\, dx\le\\
\le & C_{q,q'}(\text{\textgreek{w}}_{0})(1+t_{1})^{-q'}\cdot\int_{0}^{\infty}\frac{1}{1+\text{\textgreek{t}}^{2}}\int_{\{t^{-}=\text{\textgreek{t}}\}\cap\{R_{1}\lesssim r\lesssim R_{1}+\text{\textgreek{t}}\}}\big(J_{\text{\textgreek{m}}}^{N}(\text{\textgreek{y}})n^{\text{\textgreek{m}}}+\frac{1}{r^{2}}|\text{\textgreek{y}}|^{2}\big)\, dxd\text{\textgreek{t}}+\\
 & \hphantom{C_{q,q'}(}+C_{q,q'}(\text{\textgreek{w}}_{0})(1+t^{*}-t_{2})^{-q'}\int_{0}^{1}\int_{\{t^{-}=\text{\textgreek{t}}\}\cap\{r\ge R_{1}\}}\frac{1}{r^{2}}\big(J_{\text{\textgreek{m}}}^{N}(\text{\textgreek{y}})n^{\text{\textgreek{m}}}+\frac{1}{r^{2}}|\text{\textgreek{y}}|^{2}\big)\, dxd\text{\textgreek{t}}\le\\
\le & C_{q,q'}(\text{\textgreek{w}}_{0})\big((1+t_{1})^{-q'}+(1+t^{*}-t_{2})^{-q'}\big)\int_{t=0}J_{\text{\textgreek{m}}}^{N}(\text{\textgreek{y}})n^{\text{\textgreek{m}}},
\end{split}
\]
the last inequality being a consequence of the boundedness assumption
\hyperref[Assumption 4]{4} and Lemma \ref{lem:ComparisonTT-} (and
a Hardy inequality for the $\frac{1}{r^{2}}|\text{\textgreek{y}}|^{2}$
term). Thus, from (\ref{eq:FInequality}), (\ref{eq:BoundFirstTermFIneq})
and (\ref{eq:BoundSecondTermFIneq}) we obtain the desired bound:

\begin{equation}
\int_{\mathcal{R}(t_{1},t_{2})}r^{q}|F_{k}|^{2}\le C_{q,q'}(\text{\textgreek{w}}_{0})\big((1+t_{1})^{-q'}+(1+t^{*}-t_{2})^{-q'}\big)\int_{t=0}J_{\text{\textgreek{m}}}^{N}(\text{\textgreek{y}})n^{\text{\textgreek{m}}}.
\end{equation}

In the same way, one can show the same inequality for $F_{\le\text{\textgreek{w}}_{+}},F_{\ge\text{\textgreek{w}}_{+}}=F-F_{\le\text{\textgreek{w}}_{+}}$.
\end{proof}
We will als\textgreek{o} need to bound the energy of the frequency-decomposed
components of $\text{\textgreek{y}}$ in terms of the energy of $\text{\textgreek{y}}$
itself:
\begin{lem}
\label{lem:BoundednessPsiK}There exists a positive constant $C(\text{\textgreek{w}}_{0})$
such that for any $-n\le k\le n$, and any $\text{\textgreek{t}}\in[0,t^{*}]$:

\begin{equation}
\int_{t=\text{\textgreek{t}}}J_{\text{\textgreek{m}}}^{N}(\text{\textgreek{y}}_{k})n^{\text{\textgreek{m}}}\le C(\text{\textgreek{w}}_{0})\cdot\int_{t=0}J_{\text{\textgreek{m}}}^{N}(\text{\textgreek{y}})n^{\text{\textgreek{m}}}.\label{eq:A Bound for psiK}
\end{equation}
 The same estimate holds for $\text{\textgreek{y}}_{\le\text{\textgreek{w}}_{+}},\text{\textgreek{y}}_{\ge\text{\textgreek{w}}_{+}}$
in place of $\text{\textgreek{y}}_{k}$ (and in that case the constant
in (\ref{eq:Boundedness1}) does not depend on $\text{\textgreek{w}}_{0}$).\end{lem}
\begin{proof}
The proof will be similar to the proof of the previous lemma. Since
$\nabla_{\text{\textgreek{m}}}\text{\textgreek{y}}_{k}(t,\cdot)=\int_{-\infty}^{\infty}h_{k}(t-s)\cdot\nabla_{\text{\textgreek{m}}}\text{\textgreek{y}}_{t^{*}}(s,\cdot)\, ds$,
we can bound:

\begin{equation}
\int_{t=\text{\textgreek{t}}}J_{\text{\textgreek{m}}}^{N}(\text{\textgreek{y}}_{k})n^{\text{\textgreek{m}}}\le C\cdot\int_{t=\text{\textgreek{t}}}\big(|T\text{\textgreek{y}}_{k}|^{2}+|\nabla_{\text{\textgreek{S}}_{\text{\textgreek{t}}}}\text{\textgreek{y}}_{k}|_{g_{t}}^{2}\big),
\end{equation}
 where $g_{t}$ is the induced Riemannian metric on $\text{\textgreek{S}}_{\text{\textgreek{t}}}$.
We can also estimate for $q$ large enough
\begin{align}
\int_{t=\text{\textgreek{t}}}|T\text{\textgreek{y}}_{k}|^{2} & =\int_{t=\text{\textgreek{t}}}\Big|\int_{-\infty}^{\infty}h_{k}(t-s)\cdot T\text{\textgreek{y}}_{t^{*}}(s,\cdot)\, ds\Big|^{2}\label{eq:Boundedness1}\\
 & \le C_{q}\cdot\text{\textgreek{w}}_{0}^{2}\int_{t=\text{\textgreek{t}}}\Big(\int_{-\infty}^{\infty}\frac{1}{1+(\text{\textgreek{w}}_{0}|t-s|)^{q}}|T\text{\textgreek{y}}_{t^{*}}(s,\cdot)|\, ds\Big)^{2}\nonumber \\
 & \le C_{q}\cdot\text{\textgreek{w}}_{0}\int_{t=\text{\textgreek{t}}}\Big(\int_{-\infty}^{\infty}\frac{1}{1+(\text{\textgreek{w}}_{0}|t-s|)^{q}}|T\text{\textgreek{y}}_{t^{*}}(s,\cdot)|^{2}\, ds\Big)\nonumber \\
 & \le C_{q}\cdot\text{\textgreek{w}}_{0}\int_{-\infty}^{\infty}\frac{1}{1+(\text{\textgreek{w}}_{0}|\text{\textgreek{t}}-s|)^{q}}\Big(\int_{t=s}|T\text{\textgreek{y}}_{t^{*}}|^{2}\Big)\, ds.\nonumber 
\end{align}

We now recall that $T\text{\textgreek{y}}_{t^{*}}=h_{t^{*}}\cdot T\text{\textgreek{y}}+Th_{t^{*}}\cdot\text{\textgreek{y}}$.
We note that $h_{t^{*}}$ is supported only in $\{t_{-}\ge0\}$, and
hence
\begin{equation}
\int_{\{t=s\}}|h_{t^{*}}\cdot T\text{\textgreek{y}}|^{2}=\int_{\{t=s\}\cap\{t_{-}\ge0\}}|h_{t^{*}}\cdot T\text{\textgreek{y}}|^{2}.
\end{equation}
Therefore, for $s\ge0$, the boundedness assumption \hyperref[Assumption 4]{4}
implies that
\begin{equation}
\int_{\{t=s\}}|h_{t^{*}}\cdot T\text{\textgreek{y}}|^{2}\le C\cdot\int_{\{t=0\}}J_{\text{\textgreek{m}}}^{N}(\text{\textgreek{y}})n^{\text{\textgreek{m}}}.\label{eq:FirstBoundTDerivative}
\end{equation}
For $s<0$, integrating $\nabla^{\text{\textgreek{m}}}J_{\text{\textgreek{m}}}^{T}$
in the domain bounded by $\{t=s\}\cap\{t_{-}\ge0\}$ and $\{t_{-}=0\}$
we obtain
\begin{equation}
\int_{\{t=s\}\cap\{t_{-}\ge0\}}|h_{t^{*}}\cdot T\text{\textgreek{y}}|^{2}\le C\cdot\int_{t_{-}=0}J_{\text{\textgreek{m}}}^{N}(\text{\textgreek{y}})n^{\text{\textgreek{m}}}.
\end{equation}
Thus, applying in this case Lemma \ref{lem:ComparisonTT-}, we deduce
that for any $s\in\mathbb{R}$ we can bound
\begin{equation}
\int_{t=s}|h_{t^{*}}\cdot T\text{\textgreek{y}}|^{2}\le C\cdot\int_{t=0}J_{\text{\textgreek{m}}}^{N}(\text{\textgreek{y}})n^{\text{\textgreek{m}}}.\label{eq:SecondBoundTDerivative}
\end{equation}

Moreover, $grad(h_{t^{*}})$ is only supported in $\{0\le t_{-}\le1\}\cup\{t^{*}-1\le t_{+}\le t^{*}\}$,
and thus we can bound through a Hardy inequality (in view of the definition
(\ref{eq:DistortedTime+}) and (\ref{eq:DistortedTime-}) of $t_{+}$,$t_{-}$)
for any $s\in\mathbb{R}$
\begin{equation}
\int_{t=s}|Th_{t^{*}}\cdot\text{\textgreek{y}}|^{2}\le C\cdot\big(1+dist\{s,[0,t^{*}]\}\big)^{2}\cdot\int_{\{t=s\}\cap\{t_{-}\ge0\}}J_{\text{\textgreek{m}}}^{N}(\text{\textgreek{y}})n^{\text{\textgreek{m}}}.\label{eq:AnIntermediateInequality}
\end{equation}
 For $s\ge0$, Assumption \hyperref[Assumption 4]{4} states that
$\int_{t=s}J_{\text{\textgreek{m}}}^{N}(\text{\textgreek{y}})n^{\text{\textgreek{m}}}\le C\cdot\int_{t=0}J_{\text{\textgreek{m}}}^{N}(\text{\textgreek{y}})n^{\text{\textgreek{m}}}$.
For $s<0$, the conservation of the $J^{T}$ current in the domain
bounded by $\{t=s\}\cap\{t_{-}\ge0\}$ and $\{t_{-}=0\}$ together
with Lemma \ref{lem:ComparisonTT-} imply that $\int_{\{t=s\}\cap\{t_{-}\ge0\}}J_{\text{\textgreek{m}}}^{N}(\text{\textgreek{y}})n^{\text{\textgreek{m}}}\le C\cdot\int_{t=0}J_{\text{\textgreek{m}}}^{N}(\text{\textgreek{y}})n^{\text{\textgreek{m}}}$
in this case as well. Thus, (\ref{eq:AnIntermediateInequality}) yields
\begin{equation}
\int_{t=s}|Th_{t^{*}}\cdot\text{\textgreek{y}}|^{2}\le C\cdot\big(1+dist\{s,[0,t^{*}]\}\big)^{2}\cdot\int_{t=0}J_{\text{\textgreek{m}}}^{N}(\text{\textgreek{y}})n^{\text{\textgreek{m}}}.\label{eq:ThirdBoundTDerivative}
\end{equation}

Substituting (\ref{eq:SecondBoundTDerivative}) and (\ref{eq:ThirdBoundTDerivative})
in \ref{eq:Boundedness1}, we obtain
\begin{equation}
\int_{t=\text{\textgreek{t}}}|T\text{\textgreek{y}}_{k}|^{2}\le C_{q}\cdot\text{\textgreek{w}}_{0}\Big\{\int_{-\infty}^{\infty}\frac{(1+dist\{s,[0,t^{*}]\})^{2}}{1+(\text{\textgreek{w}}_{0}|\text{\textgreek{t}}-s|)^{q}}\, ds\Big\}\int_{t=0}J_{\text{\textgreek{m}}}^{N}(\text{\textgreek{y}})n^{\text{\textgreek{m}}}
\end{equation}
 and hence, since $\text{\textgreek{t}}\in[0,t^{*}]$, if we fix $q=4$
we conclude

\begin{equation}
\int_{t=\text{\textgreek{t}}}|T\text{\textgreek{y}}_{k}|^{2}\le C(\text{\textgreek{w}}_{0})\int_{t=0}J_{\text{\textgreek{m}}}^{N}(\text{\textgreek{y}})n^{\text{\textgreek{m}}}.
\end{equation}
In the same way, we can estimate 
\begin{equation}
\int_{t=\text{\textgreek{t}}}|\nabla_{\text{\textgreek{S}}_{\text{\textgreek{t}}}}\text{\textgreek{y}}_{k}|^{2}\le C(\text{\textgreek{w}}_{0})\int_{t=0}J_{\text{\textgreek{m}}}^{N}(\text{\textgreek{y}})n^{\text{\textgreek{m}}},
\end{equation}
and thus attain (\ref{eq:A Bound for psiK}) regarding $\text{\textgreek{y}}_{k}$.

The statement about $\text{\textgreek{y}}_{\le\text{\textgreek{w}}_{+}}$
follows in exactly the same way by considering $h_{\le\text{\textgreek{w}}_{+}}$
instead of $h_{k}$, and using (\ref{eq:SchwartzBounds1}) instead
of (\ref{eq:SchwartzBounds2}). In this case, in view of the fact
that $\text{\textgreek{w}}_{+}\ge1$, we can also make the constant
appearing in (\ref{eq:BoundednessFrmInitialDataOnK-1}) (which in
the case of $\text{\textgreek{y}}_{k}$ is $\sim\text{\textgreek{w}}_{0}^{-2}$)
not to depend on $\text{\textgreek{w}}_{0}$, but we have chosen to
neglect this fact. 

Finally, the statement for $\text{\textgreek{y}}_{\ge\text{\textgreek{w}}_{+}}=\text{\textgreek{y}}_{t^{*}}-\text{\textgreek{y}}_{\le\text{\textgreek{w}}_{+}}$
follows immediately, in view also of the boundedness assumption \hyperref[Assumption 4]{4}
and the support of the cut-off function $h_{t^{*}}$, which allow
us to handle the $\text{\textgreek{y}}_{t^{*}}$ term in exactly the
same way as we did for $\text{\textgreek{y}}_{k}$.
\end{proof}
It would seem useful to estimate the energy of $\text{\textgreek{y}}_{k}$
(and $\text{\textgreek{y}}_{\le\text{\textgreek{w}}_{+}},\text{\textgreek{y}}_{\ge\text{\textgreek{w}}_{+}}$as
well) only in terms of its own initial data, and not the initial data
of $\text{\textgreek{y}}$. Unfortunately, such a bound is not obtainable.
We can instead establish the following estimate:
\begin{lem}
\label{lem:BoundednessFromInitialDataPsiK}For any $q\in\mathbb{N}$,
there exist suitable constants $C,C_{q}(\text{\textgreek{w}}_{0})$
such that for any $-n\le k\le n$, any $0\le t_{st}\le\frac{1}{2}t^{*}$
and any $t_{st}\le t_{1}\le t_{2}\le t^{*}-t_{st}$ satisfying $t_{2}-t_{1}\le t_{st}$
, the following inequality holds for any $\text{\textgreek{t}}\in[t_{1},t_{2}]$:

\begin{equation}
\int_{t=\text{\textgreek{t}}}J_{\text{\textgreek{m}}}^{N}(\text{\textgreek{y}}_{k})n^{\text{\textgreek{m}}}\le C\cdot\int_{t=t_{1}}J_{\text{\textgreek{m}}}^{N}(\text{\textgreek{y}}_{k})n^{\text{\textgreek{m}}}+C_{q}(\text{\textgreek{w}}_{0})(1+t_{st})^{-q}\int_{t=0}J_{\text{\textgreek{m}}}^{N}(\text{\textgreek{y}})n^{\text{\textgreek{m}}}.\label{eq:BoundednessFrmInitialDataOnK}
\end{equation}
 The same also holds for $\text{\textgreek{y}}_{\le\text{\textgreek{w}}_{+}},\text{\textgreek{y}}_{\ge\text{\textgreek{w}}_{+}}$
in place of $\text{\textgreek{y}}_{k}$.\end{lem}
\begin{proof}
Recall that $\text{\textgreek{y}}_{k}$ satisfies $\square_{g}\text{\textgreek{y}}_{k}=F_{k}$.
Therefore, on $J^{+}(\{t=t_{1}\})\cap\mathcal{D}$ we can uniquely
decompose $\text{\textgreek{y}}_{k}=\text{\textgreek{y}}_{k,hom}+\text{\textgreek{y}}_{k,inhom}$,
where 

\begin{equation}
\begin{cases}
\square_{g}\text{\textgreek{y}}_{k,hom}=0\\
\text{\textgreek{y}}_{k,hom}|_{t=t_{1}}=\text{\textgreek{y}}_{k}|_{t=t_{1}}\\
\partial_{t}\text{\textgreek{y}}_{k,hom}|_{t=t_{1}}=\partial_{t}\text{\textgreek{y}}_{k}|_{t=t_{1}}
\end{cases}
\end{equation}
and
\begin{equation}
\begin{cases}
\square_{g}\text{\textgreek{y}}_{k,inhom}=F_{k}\\
\text{\textgreek{y}}_{k,inhom}|_{t=t_{1}}=0\\
\partial_{t}\text{\textgreek{y}}_{k,inhom}|_{t=t_{1}}=0.
\end{cases}
\end{equation}

The boundedness assumtion \hyperref[Assumption 4]{4} applies to $\text{\textgreek{y}}_{k,hom}$
and we obtain (since the initial data on $t=t_{1}$ that $\text{\textgreek{y}}_{k,hom}$
satisfies are the ones induced by $\text{\textgreek{y}}_{k}$)

\begin{equation}
\int_{t=\text{\textgreek{t}}}J_{\text{\textgreek{m}}}^{N}(\text{\textgreek{y}}_{k,hom})n^{\text{\textgreek{m}}}\le C\cdot\int_{t=t_{1}}J_{\text{\textgreek{m}}}^{N}(\text{\textgreek{y}}_{k})n^{\text{\textgreek{m}}}.\label{eq:HomogeneousTerm}
\end{equation}

For $\text{\textgreek{y}}_{k,inhom}$, we will use Duhamel's principle
as follows: For $s\ge t_{1}$, let $u_{s}:J^{+}(\{t=s\})\cap\mathcal{D}\rightarrow\mathbb{C}$
be the unique solution to the initial value problem $\square_{g}u_{s}=0$,
$u_{s}|_{t=s}=0$, $\partial_{t}u_{s}|_{t=s}=F_{k}|_{t=s}$. Then
the following relation holds 
\begin{equation}
\text{\textgreek{y}}_{k,inhom}(t,\cdot)=\int_{t_{1}}^{t}u_{s}(t,\cdot)\, ds.
\end{equation}
This can be deduced by just noting that (in view of the fact that
the vector field $T$ is Killing), the function $\text{\textgreek{Y}}(t,x)=\text{\textgreek{y}}_{k,inhom}(t,x)-\int_{t_{1}}^{t}u_{s}(t,x)\, ds$
(for any $t\in\mathbb{R}$ and $x\in\text{\textgreek{S}}\cap\mathcal{D}$)
satisfies on $J^{+}(\{t=t_{1}\})\cap\mathcal{D}$: 
\[
\square_{g}\text{\textgreek{Y}}=F_{k}(t,x)-F_{k}(t,x)=0,
\]
 with $\text{\textgreek{Y}}|_{t=t_{1}}=0$, $\partial_{t}\text{\textgreek{Y}}|_{t=t_{1}}=0$.
Thus, $\text{\textgreek{Y}}\equiv0$ on $J^{+}(\{t=t_{1}\})\cap\mathcal{D}$. 

Therefore, since the boundedness assumption \hyperref[Assumption 4]{4}
applies to $u_{s}$, yielding for $\text{\textgreek{t}}\ge s$ 
\begin{equation}
\int_{t=\text{\textgreek{t}}}J_{\text{\textgreek{m}}}^{N}(u_{s})n^{\text{\textgreek{m}}}\le C\cdot\int_{t=s}|F_{k}|^{2},
\end{equation}
 we infer that:

\begin{equation}
\int_{t=\text{\textgreek{t}}}J_{\text{\textgreek{m}}}^{N}(\text{\textgreek{y}}_{k,inhom})n^{\text{\textgreek{m}}}\le C\cdot\int_{t=0}\Big(\int_{t_{1}}^{\text{\textgreek{t}}}|\partial u_{s}(\text{\textgreek{t}},\cdot)|\, ds\Big)^{2}\, dx\le C\cdot(\text{\textgreek{t}}-t_{1})\cdot\int_{R(t_{1},\text{\textgreek{t}})}|F_{k}|^{2}.\label{eq:FirstDuhamel}
\end{equation}
 In view of Lemma \ref{lem:BoundF}, and the fact that $t_{st}\le t_{1}\le\text{\textgreek{t}}\le t_{2}\le t^{*}-t_{st}$
and $\text{\textgreek{t}}-t_{1}\le t_{2}-t_{1}\le t_{st}$, from (\ref{eq:FirstDuhamel})
we conclude that

\begin{equation}
\int_{t=\text{\textgreek{t}}}J_{\text{\textgreek{m}}}^{N}(\text{\textgreek{y}}_{k,inhom})n^{\text{\textgreek{m}}}\le C_{q}(\text{\textgreek{w}}_{0})\cdot(1+t_{st})^{-q}\int_{t=0}J_{\text{\textgreek{m}}}^{N}(\text{\textgreek{y}})n^{\text{\textgreek{m}}}.\label{eq:Inhomogeneous term}
\end{equation}

Since $\text{\textgreek{y}}_{k}=\text{\textgreek{y}}_{k,hom}+\text{\textgreek{y}}_{k,inhom}$,
adding \ref{eq:HomogeneousTerm} and \ref{eq:Inhomogeneous term}
yields \ref{eq:BoundednessFrmInitialDataOnK}. The statement about
$\text{\textgreek{y}}_{\le\text{\textgreek{w}}_{+}},\text{\textgreek{y}}_{\ge\text{\textgreek{w}}_{+}}$
follows in exactly the same way.
\end{proof}
We will also need to use a boundedness statement for the energy of
$\text{\textgreek{y}}_{k}$, $\text{\textgreek{y}}_{\le\text{\textgreek{w}}_{+}}$
and $\text{\textgreek{y}}_{\ge\text{\textgreek{w}}_{+}}$ on spacelike
hypersurfaces more general than $\{t=const\}$. The following lemma,
is a straightforward generalisation of Lemma \ref{lem:BoundednessFromInitialDataPsiK},
and its proof is identical (and will be omitted).
\begin{lem}
\label{lem:GeneralisedBoundednessPsiK}Let $\text{\textgreek{j}}:\text{\textgreek{S}}_{0}\rightarrow\mathbb{R}$
be a non negative function such that $t_{\text{\textgreek{j}}}=t-\text{\textgreek{j}}$
has spacelike level sets. Then for any $q\in\mathbb{N}$, there exist
suitable constants $C,C_{q}(\text{\textgreek{w}}_{0})$ (depending
also on the precise choice of $\text{\textgreek{j}}$) such that for
any $-n\le k\le n$, any $0\le t_{st}\le\frac{1}{2}t^{*}$ and any
$t_{st}\le t_{1}\le t_{2}\le t^{*}-t_{st}$ satisfying $t_{2}-t_{1}\le t_{st}$
, the following inequality holds for any $\text{\textgreek{t}}\in[t_{1},t_{2}]$:

\begin{equation}
\int_{\{t_{\text{\textgreek{j}}}=\text{\textgreek{t}}\}\cap\mathcal{R}(0,t^{*})}J_{\text{\textgreek{m}}}^{N}(\text{\textgreek{y}}_{k})n^{\text{\textgreek{m}}}\le C\cdot\int_{\{t_{\text{\textgreek{j}}}=t_{1}\}\cap\mathcal{R}(0,t^{*})}J_{\text{\textgreek{m}}}^{N}(\text{\textgreek{y}}_{k})n^{\text{\textgreek{m}}}+C_{q}(\text{\textgreek{w}}_{0})(1+t_{st})^{-q}\int_{t=0}J_{\text{\textgreek{m}}}^{N}(\text{\textgreek{y}})n^{\text{\textgreek{m}}}.\label{eq:BoundednessFrmInitialDataOnK-1}
\end{equation}
 The same estimate also holds for $\text{\textgreek{y}}_{\le\text{\textgreek{w}}_{+}},\text{\textgreek{y}}_{\ge\text{\textgreek{w}}_{+}}$
in place of $\text{\textgreek{y}}_{k}$.
\end{lem}
\medskip{}

We will also need to localise in time estimates of the form $\int_{-\infty}^{\infty}|\partial_{t}\text{\textgreek{y}}_{k}|^{2}\, dt\sim\text{\textgreek{w}}_{k}^{2}\int_{-\infty}^{\infty}|\text{\textgreek{y}}_{k}|^{2}\, dt$.
To this end, we have to introduce a few more cut off functions.

Let $\text{\ensuremath{\bar{\text{\textgreek{q}}}}}_{3}:\mathbb{R}\rightarrow[0,1]$
be a smooth function that is identically 1 on $[-\frac{9}{8},\frac{9}{8}]$
and identically 0 outside $(-\frac{5}{4},\frac{5}{4})$. Then for
$1\le k\le n$ we will define the functions 
\begin{equation}
\text{\ensuremath{\bar{\lyxmathsym{\textgreek{z}}}}}_{k}=\bar{\text{\textgreek{q}}}_{3}(\frac{\text{\textgreek{w}}-\frac{\text{\textgreek{w}}_{k-1}+\text{\textgreek{w}}_{k}}{2}}{\frac{1}{2}(\text{\textgreek{w}}_{k}-\text{\textgreek{w}}_{k-1})}).
\end{equation}
We will extend this definition for $-n\le k\le-1$ by setting $\text{\ensuremath{\bar{\lyxmathsym{\textgreek{z}}}}}_{k}(\text{\textgreek{w}})=\bar{\text{\textgreek{z}}}_{-k}(-\text{\textgreek{w}})$. 

Note that the $\bar{\text{\textgreek{z}}}_{k}$'s are smooth functions
with compact support, and hence their inverse Fourier transforms $\bar{h}_{k}$
are Schwartz functions. For $1\le k\le n$, the $\bar{h}_{k}$'s are
of the form 
\begin{equation}
\bar{h}_{k}(t)=\frac{\text{\textgreek{w}}_{k}-\text{\textgreek{w}}_{k-1}}{2}e^{i\frac{\text{\textgreek{w}}_{k}+\text{\textgreek{w}}_{k-1}}{2}t}\check{\bar{\text{\textgreek{q}}}}_{3}(\frac{1}{2}(\text{\textgreek{w}}_{k}-\text{\textgreek{w}}_{k-1})t).
\end{equation}
Hence, due to the bound $\frac{1}{4}\text{\textgreek{w}}_{0}<\text{\textgreek{w}}_{k}-\text{\textgreek{w}}_{k-1}<\frac{1}{2}\text{\textgreek{w}}_{0}$,
the fact that $\check{\bar{\text{\textgreek{q}}}}_{3}$ is a Schwartz
function implies that for any $q\in\mathbb{N}$ 
\begin{equation}
\sup_{t}\big(\text{\textgreek{w}}_{0}^{-1}\{1+|\text{\textgreek{w}}_{0}t|^{q}\}|\bar{h}_{k}(t)|\big)\le C_{q}\mbox{ and }\sup_{t}\big(\text{\textgreek{w}}_{0}^{-2}\{1+|\text{\textgreek{w}}_{0}t|^{q}\}\big|\big\{\partial_{t}\bar{h}_{k}(t)-i\frac{\text{\textgreek{w}}_{k}+\text{\textgreek{w}}_{k-1}}{2}\bar{h}_{k}(t)\big\}\big|\big)\le C_{q}\label{eq:SchwarzBounds1}
\end{equation}
 for some constants depending only on the precise choice of $\bar{\text{\textgreek{q}}}_{3}$.
The same bounds also hold for $-n\le k\le-1$.

In the same way, we can define the function $\bar{\text{\textgreek{z}}}_{0}:\mathbb{R}\rightarrow[0,1]$:
\begin{equation}
\text{\ensuremath{\bar{\lyxmathsym{\textgreek{z}}}}}_{0}=\bar{\text{\textgreek{q}}}_{3}(\frac{|\text{\textgreek{w}}|}{\text{\textgreek{w}}_{0}}).
\end{equation}
Then its inverse Fourier transform $\bar{h}_{0}$ also satisfies 
\begin{equation}
\sup_{t}\big(\text{\textgreek{w}}_{0}^{-1}\{1+|\text{\textgreek{w}}_{0}t|^{q}\}|\bar{h}_{0}(t)|\big)\le C_{q}\mbox{ and }\sup_{t}\big(\text{\textgreek{w}}_{0}^{-2}\{1+|\text{\textgreek{w}}_{0}t|^{q}\}|\partial_{t}\bar{h}_{k}(t)|\big)\le C_{q}.\label{eq:SchwarzBounds2}
\end{equation}

Since $\bar{\text{\textgreek{z}}}_{k}\equiv1$ on the support of $\text{\textgreek{z}}_{k}$,
we have the relation $\text{\textgreek{z}}_{k}\cdot\bar{\text{\textgreek{z}}}_{k}=\text{\textgreek{z}}_{k}$.
This relation implies for $-n\le k\le n$ the following self reproducing
formula for $\text{\textgreek{y}}_{k}$:

\begin{equation}
\text{\textgreek{y}}_{k}(t,\cdot)=\int_{-\infty}^{\infty}\bar{h}_{k}(t-s)\cdot\text{\textgreek{y}}_{k}(s,\cdot)\, ds.\label{eq:Reproducing Formula}
\end{equation}

For $1\le|k|\le n$, we can also establish an estimate for the anti-derivative
of $\bar{h}_{k}$: The functions $\tilde{\text{\textgreek{z}}}_{k}(\text{\textgreek{w}})=\frac{1}{i\text{\textgreek{w}}}\bar{\text{\textgreek{z}}}_{k}(\text{\textgreek{w}})$
are smooth functions of compact support, and in particular they are
identically equal to $(i\text{\textgreek{w}})^{-1}$ on the frequency
support of $\hat{\text{\textgreek{y}}}_{k}$ (i.\,e. the support
of $\text{\textgreek{z}}_{k}$). Hence, we have the identity 
\begin{equation}
\tilde{\text{\textgreek{z}}}_{k}(\text{\textgreek{w}})\cdot i\text{\textgreek{w}}\hat{\text{\textgreek{y}}}_{k}(\text{\textgreek{w}},\cdot)=\hat{\text{\textgreek{y}}}_{k}(\text{\textgreek{w}},\cdot).
\end{equation}
Therefore, by denoting with $\tilde{h}_{k}$ the inverse Fourier transform
of $\tilde{\text{\textgreek{z}}}_{k}$, from (\ref{eq:Reproducing Formula})
we obtain the relation

\begin{equation}
\text{\textgreek{y}}_{k}(t,\cdot)=\int_{-\infty}^{\infty}\tilde{h}_{k}(t-s)\cdot T\text{\textgreek{y}}_{k}(s,\cdot)\, ds.\label{eq:ReproducingFormula2}
\end{equation}

Note also that for any $q\in\mathbb{N}$ we can bound
\begin{align}
\int_{-\infty}^{\infty}|\text{\textgreek{w}}_{0}^{q}\frac{d^{q}}{d\text{\textgreek{w}}^{q}}\tilde{\text{\textgreek{z}}}_{k}(\text{\textgreek{w}})|\, d\text{\textgreek{w}} & \le C_{q}\cdot\sum_{l=0}^{q}\Big\{\int_{-\infty}^{\infty}\text{\textgreek{w}}_{0}^{q}\Big|\frac{d^{l}}{d\text{\textgreek{w}}^{l}}\Big(\frac{1}{\text{\textgreek{w}}}\Big)\cdot\text{\textgreek{w}}_{0}^{l-q}\cdot\bar{\text{\textgreek{q}}}_{3}^{(q-l)}(\frac{\text{\textgreek{w}}-\frac{\text{\textgreek{w}}_{k-1}+\text{\textgreek{w}}_{k}}{2}}{\frac{1}{2}(\text{\textgreek{w}}_{k}-\text{\textgreek{w}}_{k-1})})\Big|\, d\text{\textgreek{w}}\Big\}\le\label{eq:BoundFromReproducingFormula}\\
 & \le C_{q}\cdot\sum_{l=0}^{q}\Big\{\text{\textgreek{w}}_{0}^{l}\int_{-\infty}^{\infty}|\text{\textgreek{w}}|^{-l-1}\cdot\Big|\bar{\text{\textgreek{q}}}_{3}^{(q-l)}(\frac{\text{\textgreek{w}}-\frac{\text{\textgreek{w}}_{k-1}+\text{\textgreek{w}}_{k}}{2}}{\frac{1}{2}(\text{\textgreek{w}}_{k}-\text{\textgreek{w}}_{k-1})})\Big|\, d\text{\textgreek{w}}\Big\}\le\nonumber \\
 & \le C_{q}\cdot\text{\textgreek{w}}_{k-1}^{-1}\cdot\text{\textgreek{w}}_{0}.\nonumber 
\end{align}
 However, from the definition of the inverse Fourier transform, for
any $q\in\mathbb{N}$ we have
\begin{equation}
\sup_{t}\{|\text{\textgreek{w}}_{0}t|^{q}\cdot|\tilde{h}_{k}(t)|\}\le C\cdot\int_{-\infty}^{\infty}\Big|\text{\textgreek{w}}_{0}^{q}\frac{d^{q}}{d\text{\textgreek{w}}^{q}}\tilde{\text{\textgreek{z}}}_{k}(\text{\textgreek{w}})\Big|\, d\text{\textgreek{w}}
\end{equation}
 and hence, (\ref{eq:BoundFromReproducingFormula}) implies:

\begin{equation}
\sup_{t}\{\text{\textgreek{w}}_{0}^{-1}\{1+|\text{\textgreek{w}}_{0}t|^{q}\}\cdot|\tilde{h}_{k}(t)|\}\le C_{q}\cdot\text{\textgreek{w}}_{k-1}^{-1}.\label{eq:ScwarzBoundDerivative}
\end{equation}
 The same statement obviously also holds for $-n\le k\le-1$.

We can now establish the following lemma:
\begin{lem}
\label{lem:DtToOmegaInequalities}For any $1\le|k|\le n$, any $0\le t_{1}\le t_{2}\le t^{*}$
and any $R\ge0$, we can bound

\begin{multline}
c\cdot\text{\textgreek{w}}_{k-1}^{2}\Big\{\int_{\mathcal{R}(t_{1},t_{2})\cap\{r\le R\}}|\text{\textgreek{y}}_{k}|^{2}-C(\text{\textgreek{w}}_{0})\cdot\int_{t=0}J_{\text{\textgreek{m}}}^{N}(\text{\textgreek{y}})n^{\text{\textgreek{m}}}\Big\}\le\int_{\mathcal{R}(t_{1},t_{2})\cap\{r\le R\}}|T\text{\textgreek{y}}_{k}|^{2}\le\\
\le C\cdot\text{\textgreek{w}}_{k}^{2}\int_{\mathcal{R}(t_{1},t_{2})\cap\{r\le R\}}|\text{\textgreek{y}}_{k}|^{2}+\text{\textgreek{w}}_{k}^{2}C(\text{\textgreek{w}}_{0})R^{2}\int_{t=0}J_{\text{\textgreek{m}}}^{N}(\text{\textgreek{y}})n^{\text{\textgreek{m}}},\label{eq:w-estimate}
\end{multline}
and similarly for $k=0$:

\begin{equation}
\int_{\mathcal{R}(t_{1},t_{2})\cap\{r\le R\}}|T\text{\textgreek{y}}_{0}|^{2}\le C\cdot\text{\textgreek{w}}_{0}^{2}\int_{\mathcal{R}(t_{1},t_{2})\cap\{r\le R\}}|\text{\textgreek{y}}_{0}|^{2}+C(\text{\textgreek{w}}_{0})R^{2}\cdot\int_{t=0}J_{\text{\textgreek{m}}}^{N}(\text{\textgreek{y}})n^{\text{\textgreek{m}}}.\label{eq:LowFrequencyEstimate}
\end{equation}
 
\end{lem}
\emph{Remark}: Notice that the constant multiplying the error term
in the right hand side of (\ref{eq:w-estimate}) depends on $R$,
while this is not the case in the left hand side. Notice also that
the constants in front of the $\text{\textgreek{y}}_{k}$ terms in
(\ref{eq:w-estimate}) do not depend on $\text{\textgreek{w}}_{0}$.
\begin{proof}
We will just prove the inequality for $1\le k\le n$, since the cases
$-n\le k\le-1$ and $k=0$ follow in exactly the same way. The proof
relies on manipulating the formulas (\ref{eq:Reproducing Formula})
and (\ref{eq:ReproducingFormula2}), together with the bounds (\ref{eq:SchwarzBounds1}),
(\ref{eq:SchwarzBounds2}) and (\ref{eq:ScwarzBoundDerivative}). 

By differentiating (\ref{eq:Reproducing Formula}) with respect to
$t$, we readily deduce as in \cite{DafRod5} (using (\ref{eq:SchwarzBounds1}))
that

\begin{align}
|\partial_{t}\text{\textgreek{y}}_{k}(t,\cdot)| & \le\int_{-\infty}^{\infty}|\partial_{t}\bar{h}_{k}(t-s)|\cdot|\text{\textgreek{y}}_{k}(s,\cdot)|\, ds\le\label{eq:FirstFormulaFromReproducing}\\
 & \le C_{p}\text{\textgreek{w}}_{0}\big(\frac{\text{\textgreek{w}}_{k}+\text{\textgreek{w}}_{k+1}}{2}+\text{\textgreek{w}}_{0}\big)\cdot\int_{-\infty}^{\infty}(1+\text{\textgreek{w}}_{0}|t-s|)^{-p}\cdot|\text{\textgreek{y}}_{k}(s,\cdot)|\, ds\nonumber \\
 & \le C_{p}\text{\textgreek{w}}_{0}\text{\textgreek{w}}_{k}\int_{-\infty}^{\infty}(1+\text{\textgreek{w}}_{0}|t-s|)^{-p}\cdot|\text{\textgreek{y}}_{k}(s,\cdot)|\, ds.\nonumber 
\end{align}
Applying a H\"older inequality, and provided that $p>1$ so that
$\Big(\int_{-\infty}^{\infty}(1+|\text{\textgreek{w}}_{0}s|)^{-p}\, d(\text{\textgreek{w}}_{0}s)\Big)^{1/2}\le C_{p}<\infty$,
(\ref{eq:FirstFormulaFromReproducing}) yields:

\begin{align}
|\partial_{t}\text{\textgreek{y}}_{k}(t,\cdot)| & \le C_{p}\text{\textgreek{w}}_{0}\text{\textgreek{w}}_{k}\int_{-\infty}^{\infty}(1+\text{\textgreek{w}}_{0}|t-s|)^{-p}\cdot|\text{\textgreek{y}}_{k}(s,\cdot)|\, ds\\
 & \le C_{p}\text{\textgreek{w}}_{k}\text{\textgreek{w}}_{0}^{1/2}\cdot\Big(\int_{-\infty}^{\infty}(1+\text{\textgreek{w}}_{0}|t-s|)^{-p}\cdot|\text{\textgreek{y}}_{k}(s,\cdot)|^{2}\, ds\Big)^{1/2}\nonumber \\
 & \le C_{p}\text{\textgreek{w}}_{k}\text{\textgreek{w}}_{0}^{1/2}\Big(\sum_{l=-\infty}^{\infty}(1+|l|)^{-p}\cdot\int_{t+\text{\textgreek{w}}_{0}^{-1}l}^{t+\text{\textgreek{w}}_{0}^{-1}(l+1)}|\text{\textgreek{y}}_{k}(s,\cdot)|^{2}\, ds\Big)^{1/2}.\nonumber 
\end{align}
Hence

\begin{align}
\int_{t_{1}}^{t_{2}}|\partial_{t}\text{\textgreek{y}}_{k}(t,\cdot)|^{2}\, dt & \le C_{p}\text{\textgreek{w}}_{k}^{2}\text{\textgreek{w}}_{0}\int_{t_{1}}^{t_{2}}\Big(\sum_{l=-\infty}^{\infty}(1+|l|)^{-p}\cdot\Big(\int_{t+\text{\textgreek{w}}_{0}^{-1}l}^{t+\text{\textgreek{w}}_{0}^{-1}(l+1)}|\text{\textgreek{y}}_{k}(s,\cdot)|^{2}\, ds\Big)\Big)\, dt\label{eq:FirstDtBound}\\
 & \le C_{p}\text{\textgreek{w}}_{k}^{2}\text{\textgreek{w}}_{0}\sum_{l=-\infty}^{\infty}(1+|l|)^{-p}\cdot\Big(\int_{t_{1}}^{t_{2}}\int_{t+\text{\textgreek{w}}_{0}^{-1}l}^{t+\text{\textgreek{w}}_{0}^{-1}(l+1)}|\text{\textgreek{y}}_{k}(s,\cdot)|^{2}\, dsdt\Big)\nonumber \\
 & \le C_{p}\text{\textgreek{w}}_{k}^{2}\text{\textgreek{w}}_{0}\sum_{l=-\infty}^{\infty}(1+|l|)^{-p}\cdot\Big(\int_{t_{1}}^{t_{2}}\int_{\text{\textgreek{w}}_{0}^{-1}l}^{\text{\textgreek{w}}_{0}^{-1}(l+1)}|\text{\textgreek{y}}_{k}(s+t,\cdot)|^{2}\, dsdt\Big)\nonumber \\
 & \le C_{p}\text{\textgreek{w}}_{k}^{2}\sum_{l=-\infty}^{\infty}(1+|l|)^{-p}\cdot\Big(\int_{t_{1}+\text{\textgreek{w}}_{0}^{-1}l}^{t_{2}+\text{\textgreek{w}}_{0}^{-1}(l+1)}|\text{\textgreek{y}}_{k}(t,\cdot)|^{2}\, dt\Big)\nonumber \\
 & \le C_{p}\text{\textgreek{w}}_{k}^{2}\Big(\int_{t_{1}}^{t_{2}}|\text{\textgreek{y}}_{k}(t,\cdot)|^{2}\, dt\Big)+C_{p}\text{\textgreek{w}}_{k}^{2}\sum_{l=-\infty}^{\infty}(1+|l|)^{-p}\cdot\Big(\int_{t_{1}+\text{\textgreek{w}}_{0}^{-1}l}^{t_{2}+\text{\textgreek{w}}_{0}^{-1}(l+1)}(1-\text{\textgreek{q}}_{[t_{1},t_{2}]})|\text{\textgreek{y}}_{k}(t,\cdot)|^{2}\, dt\Big),\nonumber 
\end{align}
 where $\text{\textgreek{q}}_{[t_{1},t_{2}]}$ is the characteristic
function of $\{t_{1}\le t\le t_{2}\}$.

Integrating (\ref{eq:FirstDtBound}) over $\{r\le R\}$, we obtain

\begin{align}
\int_{\mathcal{R}(t_{1},t_{2})\cap\{r\le R\}}|T\text{\textgreek{y}}_{k}|^{2}\le & C_{p}\text{\textgreek{w}}_{k}^{2}\int_{\mathcal{R}(t_{1},t_{2})\cap\{r\le R\}}|\text{\textgreek{y}}_{k}|^{2}+\label{eq:AfterIntegration}\\
 & +C_{p}\text{\textgreek{w}}_{k}^{2}\sum_{l=-\infty}^{\infty}(1+|l|)^{-p}\cdot\Big(\int_{\{r\le R\}}\int_{t_{1}+\text{\textgreek{w}}_{0}^{-1}l}^{t_{2}+\text{\textgreek{w}}_{0}^{-1}(l+1)}(1-\text{\textgreek{q}}_{[t_{1},t_{2}]})|\text{\textgreek{y}}_{k}(t,x)|^{2}\, dtdx\Big).\nonumber 
\end{align}

For the second term of the right hand side of (\ref{eq:AfterIntegration}),
we estimate:
\begin{equation}
\begin{split}\sum_{l=-\infty}^{-1}(1+|l|)^{-p}\Big(\int_{\{r\le R\}}\int_{t_{1}+\text{\textgreek{w}}_{0}^{-1}l}^{t_{2}+\text{\textgreek{w}}_{0}^{-1}(l+1)} & (1-\text{\textgreek{q}}_{[t_{1},t_{2}]})|\text{\textgreek{y}}_{k}(t,x)|^{2}\, dtdx\Big)\le\\
\le & \sum_{l=-\infty}^{-1}(1+|l|)^{-p}\Big(\int_{\{r\le R\}}\int_{t_{1}+\text{\textgreek{w}}_{0}^{-1}l}^{t_{1}}|\text{\textgreek{y}}_{k}(t,x)|^{2}\, dtdx\Big).
\end{split}
\label{eq:almost there}
\end{equation}

Notice that for any $q\in\mathbb{N}$ we can bound 
\begin{equation}
|\text{\textgreek{y}}_{k}(t,\cdot)|\le\int_{-\infty}^{\infty}|h_{k}(t-s)|\cdot|\text{\textgreek{y}}_{t^{*}}(s,\cdot)|\, ds\le C_{q}\cdot\text{\textgreek{w}}_{0}\int_{-\infty}^{\infty}(1+\text{\textgreek{w}}_{0}|t-s|)^{-q}\cdot|\text{\textgreek{y}}_{t^{*}}(s,\cdot)|\, ds
\end{equation}
due to (\ref{eq:SchwartzBounds2}), and hence (for $q>1$) 
\begin{equation}
|\text{\textgreek{y}}_{k}(t,\cdot)|^{2}\le C_{q}(\text{\textgreek{w}}_{0})\cdot\int_{-\infty}^{\infty}(1+\text{\textgreek{w}}_{0}|t-s|)^{-q}\cdot|\text{\textgreek{y}}_{t^{*}}(s,\cdot)|^{2}\, ds.
\end{equation}
Therefore, substituting in (\ref{eq:almost there}) we infer:

\begin{equation}
\int_{\{r\le R\}}\int_{t_{1}+\text{\textgreek{w}}_{0}^{-1}l}^{t_{1}}|\text{\textgreek{y}}_{k}(t,x)|^{2}\, dtdx\le C_{q}(\text{\textgreek{w}}_{0})\int_{t_{1}+\text{\textgreek{w}}_{0}^{-1}l}^{t_{1}}\int_{-\infty}^{\infty}(1+\text{\textgreek{w}}_{0}|t-s|)^{-q}\cdot\Big(\int_{\{r\le R\}\cap\{t=s\}}|\text{\textgreek{y}}_{t^{*}}|^{2}\Big)\, dsdt.\label{eq:IntermediateBeforeHardy}
\end{equation}

Using a Hardy inequality (such as (\ref{eq:GeneralHardyBound})),
we can bound: 
\begin{equation}
\int_{\{t=s\}\cap\{r\le R\}}|\text{\textgreek{y}}_{t^{*}}|^{2}=\int_{\{t=s\}\cap\{r\le R\}}|h_{t^{*}}|^{2}|\text{\textgreek{y}}|^{2}\le\int_{\{t=s\}\cap\{t_{-}\ge0\}\cap\{r\le R\}}|\text{\textgreek{y}}|^{2}\le C\cdot R^{2}\int_{\{t=s\}\cap\{t_{-}\ge0\}}J_{\text{\textgreek{m}}}^{N}(\text{\textgreek{y}})n^{\text{\textgreek{m}}}\label{eq:Some Hardy}
\end{equation}
 for any $s\in\mathbb{R}$. Moreover, as we did in the proofs of the
previous lemmas, using the boundedness assumption \hyperref[Assumption 4]{4}
for $s\ge0$, and the conservation of the $J^{T}$ current and Lemma
\ref{lem:ComparisonTT-} for $s<0$, we can bound for any $s\in\mathbb{R}$:
$\int_{\{t=s\}\cap\{t_{-}\ge0\}}J_{\text{\textgreek{m}}}^{N}(\text{\textgreek{y}})n^{\text{\textgreek{m}}}\le\int_{t=0}J_{\text{\textgreek{m}}}^{N}(\text{\textgreek{y}})n^{\text{\textgreek{m}}}$.
Thus, inequality (\ref{eq:Some Hardy}) yields

\begin{equation}
\int_{\{t=s\}\cap\{r\le R\}}|\text{\textgreek{y}}_{t^{*}}|^{2}\le C\cdot R^{2}\int_{t=0}J_{\text{\textgreek{m}}}^{N}(\text{\textgreek{y}})n^{\text{\textgreek{m}}}.
\end{equation}
Hence, returning to (\ref{eq:IntermediateBeforeHardy}), we have:

\begin{align}
\int_{\{r\le R\}}\int_{t_{1}+\text{\textgreek{w}}_{0}^{-1}l}^{t_{1}}|\text{\textgreek{y}}_{k}(t,x)|^{2}\, dtdx & \le C_{q}(\text{\textgreek{w}}_{0})R^{2}\Big\{\int_{t_{1}+\text{\textgreek{w}}_{0}^{-1}l}^{t_{1}}\int_{-\infty}^{\infty}(1+\text{\textgreek{w}}_{0}(t-s))^{-q}\, dsdt\Big\}\cdot\int_{\{t=0\}}J_{\text{\textgreek{m}}}^{N}(\text{\textgreek{y}})n^{\text{\textgreek{m}}}\label{eq:ONemoreBound}\\
 & \le C_{q}(\text{\textgreek{w}}_{0})\cdot|l|\cdot R^{2}\int_{\{t=0\}}J_{\text{\textgreek{m}}}^{N}(\text{\textgreek{y}})n^{\text{\textgreek{m}}}.\nonumber 
\end{align}

Substituting (\ref{eq:ONemoreBound}) in \ref{eq:almost there}, and
fixing $p,q$ large enough, we infer the desired bound

\begin{equation}
\sum_{l=-\infty}^{-1}(1+|l|)^{-p}\Big(\int_{\{r\le R\}}\int_{t_{1}+\text{\textgreek{w}}_{0}^{-1}l}^{t_{2}+\text{\textgreek{w}}_{0}^{-1}(l+1)}(1-\text{\textgreek{q}}_{[t_{1},t_{2}]})|\text{\textgreek{y}}_{k}(t,x)|^{2}\, dtdx\Big)\le C(\text{\textgreek{w}}_{0})R^{2}\int_{\{t=0\}}J_{\text{\textgreek{m}}}^{N}(\text{\textgreek{y}})n^{\text{\textgreek{m}}}.\label{eq:NegativePartSum}
\end{equation}
 In the same way, we can bound

\begin{equation}
\sum_{l=0}^{\infty}(1+|l|)^{-p}\Big(\int_{\{r\le R\}}\int_{t_{1}+\text{\textgreek{w}}_{0}^{-1}l}^{t_{2}+\text{\textgreek{w}}_{0}^{-1}(l+1)}(1-\text{\textgreek{q}}_{[t_{1},t_{2}]})|\text{\textgreek{y}}_{k}(t,x)|^{2}\, dtdx\Big)\le C(\text{\textgreek{w}}_{0})R^{2}\int_{\{t=0\}}J_{\text{\textgreek{m}}}^{N}(\text{\textgreek{y}})n^{\text{\textgreek{m}}}.\label{eq:Positivesum}
\end{equation}
 Returning to (\ref{eq:AfterIntegration}), in view of (\ref{eq:NegativePartSum})
and (\ref{eq:Positivesum}) we obtain

\begin{equation}
\int_{\mathcal{R}(t_{1},t_{2})\cap\{r\le R\}}|T\text{\textgreek{y}}_{k}|^{2}\le C\text{\textgreek{w}}_{k}^{2}\int_{R(t_{1},t_{2})\cap\{r\le R\}}|\text{\textgreek{y}}_{k}|^{2}+C(\text{\textgreek{w}}_{0})\text{\textgreek{w}}_{k}^{2}\cdot R^{2}\int_{\{t=0\}}J_{\text{\textgreek{m}}}^{N}(\text{\textgreek{y}})n^{\text{\textgreek{m}}}.\label{eq:Righthandsidebound}
\end{equation}
 Hence, we have established the right ``half'' of inequality (\ref{eq:w-estimate}).

The left ``half'' is proved in exactly the same way, but with the
use of (\ref{eq:ReproducingFormula2}) instead of (\ref{eq:Reproducing Formula}).
That is, in exactly the same way as before (using now (\ref{eq:ScwarzBoundDerivative})
instead of (\ref{eq:SchwarzBounds1})), we obtain

\begin{align}
|\text{\textgreek{y}}_{k}(t,\cdot)| & \le\int_{-\infty}^{\infty}|\tilde{h}_{k}(t-s)|\cdot|T\text{\textgreek{y}}_{k}(s,\cdot)|\, ds\\
 & \le C_{p}\cdot\text{\textgreek{w}}_{k-1}^{-1}\text{\textgreek{w}}_{0}\cdot\int_{-\infty}^{\infty}(1+\text{\textgreek{w}}_{0}|t-s|)^{-p}\cdot|T\text{\textgreek{y}}_{k}(s,\cdot)|\, ds\nonumber \\
 & \le C_{p}\text{\textgreek{w}}_{k-1}^{-1}\text{\textgreek{w}}_{0}^{1/2}\Big(\int_{-\infty}^{\infty}(1+\text{\textgreek{w}}_{0}|t-s|)^{-p}\cdot|T\text{\textgreek{y}}_{k}(s,\cdot)|^{2}\, ds\Big)^{1/2}\nonumber 
\end{align}
and hence (working exactly as before)

\begin{align}
\int_{t_{1}}^{t_{2}}|\text{\textgreek{y}}_{k}(t,\cdot)|^{2}\, dt & \le C_{p}\text{\textgreek{w}}_{k-1}^{-2}\Big(\int_{t_{1}}^{t_{2}}|T\text{\textgreek{y}}_{k}(t,\cdot)|^{2}\, dt\Big)+\label{eq:ToBeIntegrated}\\
 & \hphantom{\le}+C_{p}\text{\textgreek{w}}_{k-1}^{-2}\sum_{l=-\infty}^{\infty}(1+|l|)^{-p}\cdot\Big(\int_{t_{1}+\text{\textgreek{w}}_{0}^{-1}l}^{t_{2}+\text{\textgreek{w}}_{0}^{-1}(l+1)}(1-\text{\textgreek{q}}_{[t_{1},t_{2}]})|T\text{\textgreek{y}}_{k}(t,\cdot)|^{2}\, dt\Big).\nonumber 
\end{align}
After integrating (\ref{eq:ToBeIntegrated}) over $\{r\le R\}$, we
obtain:

\begin{align}
\int_{\mathcal{R}(t_{1},t_{2})\cap\{r\le R\}}|T\text{\textgreek{y}}_{k}|^{2} & \ge c_{p}\cdot\text{\textgreek{w}}_{k-1}^{2}\int_{\mathcal{R}(t_{1},t_{2})\cap\{r\le R\}}|\text{\textgreek{y}}_{k}|^{2}-\label{eq:Almost left hand side}\\
 & \hphantom{\ge}-C_{p}\sum_{l=-\infty}^{\infty}(1+|l|)^{-p}\cdot\Big(\int_{r\le R}\int_{t_{1}+\text{\textgreek{w}}_{0}^{-1}l}^{t_{2}+\text{\textgreek{w}}_{0}^{-1}(l+1)}(1-\text{\textgreek{q}}_{[t_{1},t_{2}]})|T\text{\textgreek{y}}_{k}(t,x)|^{2}\, dtdx\Big).\nonumber 
\end{align}

For the last term of the right hand side of (\ref{eq:Almost left hand side}),
we work similarly as before: For the part of the infinte sum for $l\le-1$,
we can estimate: 
\begin{equation}
\begin{split}\sum_{l=-\infty}^{-1}(1+|l|)^{-p}\Big(\int_{\{r\le R\}}\int_{t_{1}+\text{\textgreek{w}}_{0}^{-1}l}^{t_{2}+\text{\textgreek{w}}_{0}^{-1}(l+1)} & (1-\text{\textgreek{q}}_{[t_{1},t_{2}]})|T\text{\textgreek{y}}_{k}(t,x)|^{2}\, dtdx\Big)\le\\
\le & \sum_{l=-\infty}^{-1}(1+|l|)^{-p}\Big(\int_{t_{1}+\text{\textgreek{w}}_{0}^{-1}l}^{t_{1}}\{\int_{\{t=\text{\textgreek{t}}\}}|T\text{\textgreek{y}}_{k}|^{2}\}\, d\text{\textgreek{t}}\Big).
\end{split}
\label{eq:AlmostThere2}
\end{equation}

In order to estimate the second term of the right hand side of (\ref{eq:AlmostThere2}),
we proceed as follows: Starting from the bound 
\begin{equation}
|T\text{\textgreek{y}}_{k}(t,\cdot)|\le\int_{-\infty}^{\infty}|h_{k}(t-s)|\cdot|T\text{\textgreek{y}}_{t^{*}}(s,\cdot)|\, ds\le C_{q}(\text{\textgreek{w}}_{0})\int_{-\infty}^{\infty}(1+\text{\textgreek{w}}_{0}|t-s|)^{-q}\cdot|T\text{\textgreek{y}}_{t^{*}}(s,\cdot)|\, ds
\end{equation}
for $q>1$, we can control: 
\begin{equation}
|T\text{\textgreek{y}}_{k}(t,\cdot)|^{2}\le C_{q}(\text{\textgreek{w}}_{0})\int_{-\infty}^{\infty}(1+\text{\textgreek{w}}_{0}|t-s|)^{-q}\cdot|T\text{\textgreek{y}}_{t^{*}}(s,\cdot)|^{2}\, ds.
\end{equation}
Thus, we have

\begin{equation}
\int_{\{t=\text{\textgreek{t}}\}}|T\text{\textgreek{y}}_{k}|^{2}\le C_{q}(\text{\textgreek{w}}_{0})\int_{-\infty}^{\infty}(1+\text{\textgreek{w}}_{0}|\text{\textgreek{t}}-s|)^{-q}\cdot(\int_{\{t=s\}}|T\text{\textgreek{y}}_{t^{*}}|^{2})\, ds.\label{eq:IntermediateBeforeHardy2}
\end{equation}
We can also bound 
\begin{equation}
|T\text{\textgreek{y}}_{t^{*}}|^{2}\lesssim|Th_{t^{*}}|^{2}|\text{\textgreek{y}}|^{2}+|h_{t^{*}}|^{2}\cdot|T\text{\textgreek{y}}|^{2}.
\end{equation}
Since $Th_{t^{*}}$ is supported in the region $\{0\le t_{-}\le1\}\cup\{t^{*}-1\le t_{+}\le t^{*}\}$,
while $h_{t^{*}}$ is supported in $\{t_{-}\ge0\}$, we obtain: 
\begin{equation}
\int_{t=s}|h_{t^{*}}\cdot T\text{\textgreek{y}}|^{2}=\int_{\{t=s\}\cap\{t_{-}\ge0\}}|T\text{\textgreek{y}}|^{2}\le C\cdot\int_{\{t=0\}}J_{\text{\textgreek{m}}}^{N}(\text{\textgreek{y}})n^{\text{\textgreek{m}}}\label{eq:AgainBoundTpsi}
\end{equation}
 for any $s\in\mathbb{R}$. Since $\{t=s\}\cap\{0\le t_{-}\le1\}\cup\{t^{*}-1\le t_{+}\le t^{*}\}\subset\{t=s\}\cap\{r\le R_{1}+1+2s\}$,
we also obtain for any $s\in\mathbb{R}$ through a Hardy inequality:

\begin{align}
\int_{\{t=s\}}|Th_{t^{*}}|^{2}|\text{\textgreek{y}}|^{2} & =\int_{\{t=s\}\cap\{0\le t_{-}\le1\}\cup\{t^{*}-1\le t_{+}\le t^{*}\}}|\text{\textgreek{y}}|^{2}\label{eq:AgainBoundTh}\\
 & \le C\cdot(1+dist\{s,[0,t^{*}]\})^{2}\int_{\{t=s\}\cap\{t_{-}\ge0\}}J_{\text{\textgreek{m}}}^{N}(\text{\textgreek{y}})n^{\text{\textgreek{m}}}\nonumber \\
 & \le C\cdot(1+dist\{s,[0,t^{*}]\})^{2}\int_{t=0}J_{\text{\textgreek{m}}}^{N}(\text{\textgreek{y}})n^{\text{\textgreek{m}}}.\nonumber 
\end{align}
 Hence, we deduce for $q>3$ from \ref{eq:IntermediateBeforeHardy2},
(\ref{eq:AgainBoundTpsi}) and (\ref{eq:AgainBoundTh}):
\begin{align}
\int_{t_{1}+\text{\textgreek{w}}_{0}^{-1}l}^{t_{1}}\{\int_{\{t=\text{\textgreek{t}}\}}|T\text{\textgreek{y}}_{k}|^{2}\}\, d\text{\textgreek{t}} & \le C_{q}(\text{\textgreek{w}}_{0})\Big\{\int_{t_{1}+\text{\textgreek{w}}_{0}^{-1}l}^{t_{1}}\int_{-\infty}^{\infty}(1+\text{\textgreek{w}}_{0}|\text{\textgreek{t}}-s|)^{-q}\cdot(1+dist\{s,[0,t^{*}]\})^{2}\, dsd\text{\textgreek{t}}\Big\}\int_{t=0}J_{\text{\textgreek{m}}}^{N}(\text{\textgreek{y}})n^{\text{\textgreek{m}}}\label{eq:AlmostThere3}\\
 & \le C_{q}(\text{\textgreek{w}}_{0})(1+|l|^{3})\int_{t=0}J_{\text{\textgreek{m}}}^{N}(\text{\textgreek{y}})n^{\text{\textgreek{m}}}.\nonumber 
\end{align}

From \ref{eq:AlmostThere2} and (\ref{eq:AlmostThere3}), we conclude
for $p,q$ large enough

\begin{equation}
\sum_{l=-\infty}^{-1}(1+|l|)^{-p}\Big(\int_{\{r\le R\}}\int_{t_{1}+\text{\textgreek{w}}_{0}^{-1}l}^{t_{2}+\text{\textgreek{w}}_{0}^{-1}(l+1)}(1-\text{\textgreek{q}}_{[t_{1},t_{2}]})|T\text{\textgreek{y}}_{k}(t,x)|^{2}\, dtdx\Big)\le C(\text{\textgreek{w}}_{0})\int_{t=0}J_{\text{\textgreek{m}}}^{N}(\text{\textgreek{y}})n^{\text{\textgreek{m}}}.\label{eq:BoundLowerSum}
\end{equation}
Similarly, we can also bound:

\begin{equation}
\sum_{l=0}^{\infty}(1+|l|)^{-p}\Big(\int_{\{r\le R\}}\int_{t_{1}+\text{\textgreek{w}}_{0}^{-1}l}^{t_{2}+\text{\textgreek{w}}_{0}^{-1}(l+1)}(1-\text{\textgreek{q}}_{[t_{1},t_{2}]})|T\text{\textgreek{y}}_{k}(t,x)|^{2}\, dtdx\Big)\le C(\text{\textgreek{w}}_{0})\int_{t=0}J_{\text{\textgreek{m}}}^{N}(\text{\textgreek{y}})n^{\text{\textgreek{m}}}.\label{eq:BoundUpperSum}
\end{equation}
Hence, we conclude from \ref{eq:Almost left hand side}, (\ref{eq:BoundLowerSum})
and (\ref{eq:BoundUpperSum}):

\begin{equation}
\int_{\mathcal{R}(t_{1},t_{2})\cap\{r\le R\}}|T\text{\textgreek{y}}_{k}|^{2}\ge c_{p}\cdot\text{\textgreek{w}}_{k-1}^{2}\int_{\mathcal{R}(t_{1},t_{2})\cap\{r\le R\}}|\text{\textgreek{y}}_{k}|^{2}-C(\text{\textgreek{w}}_{0})\int_{t=0}J_{\text{\textgreek{m}}}^{N}(\text{\textgreek{y}})n^{\text{\textgreek{m}}},
\end{equation}
 thus completing the proof of the Lemma.
\end{proof}
We can also establish the following variant of the previous lemma
in exactly the same way as before - hence the proof will be omitted:
\begin{lem}
\label{lem:DtToOmegaInequalities2}For any continuous function $\text{\textgreek{q}}:\mathcal{D}\rightarrow[0,+\infty)$
which satisfies the relation $T\text{\textgreek{q}}=0$, for any $1\le|k|\le n$
and any $0\le t_{1}\le t_{2}\le t^{*}$, $R\ge0$ we can bound:

\begin{multline}
c\cdot\text{\textgreek{w}}_{k-1}^{2}\Big\{\int_{\mathcal{R}(t_{1},t_{2})\cap\{r\le R\}}\text{\textgreek{q}}|\text{\textgreek{y}}_{k}|^{2}-C(\text{\textgreek{w}}_{0},\text{\textgreek{q}})\cdot\int_{t=0}J_{\text{\textgreek{m}}}^{N}(\text{\textgreek{y}})n^{\text{\textgreek{m}}}\Big\}\le\int_{\mathcal{R}(t_{1},t_{2})\cap\{r\le R\}}\text{\textgreek{q}}|T\text{\textgreek{y}}_{k}|^{2}\le\\
\le C\cdot\text{\textgreek{w}}_{k}^{2}\int_{\mathcal{R}(t_{1},t_{2})\cap\{r\le R\}}\text{\textgreek{q}}|\text{\textgreek{y}}_{k}|^{2}+\text{\textgreek{w}}_{k}^{2}C(\text{\textgreek{w}}_{0},\text{\textgreek{q}})R^{2}\int_{t=0}J_{\text{\textgreek{m}}}^{N}(\text{\textgreek{y}})n^{\text{\textgreek{m}}},\label{eq:=0003C9-estimate-1}
\end{multline}
and similarly for $k=0$

\begin{equation}
\int_{\mathcal{R}(t_{1},t_{2})\cap\{r\le R\}}\text{\textgreek{q}}|T\text{\textgreek{y}}_{0}|^{2}\le C\cdot\text{\textgreek{w}}_{0}^{2}\int_{\mathcal{R}(t_{1},t_{2})\cap\{r\le R\}}\text{\textgreek{q}}|\text{\textgreek{y}}_{0}|^{2}+C(\text{\textgreek{w}}_{0},\text{\textgreek{q}})R^{2}\cdot\int_{t=0}J_{\text{\textgreek{m}}}^{N}(\text{\textgreek{y}})n^{\text{\textgreek{m}}}.\label{eq:LowFrequencyEstimate-1}
\end{equation}

\end{lem}
Finally, we will need the following bounds for the energy of the high
freqency part $\text{\textgreek{y}}_{\le\text{\textgreek{w}}_{+}}$
in terms of the initial energy of higher derivatives of $\text{\textgreek{y}}$:
\begin{lem}
\label{lem:HighFrequencies}For any $\text{\textgreek{t}}\in\mathbb{R}$
and any $m\in\mathbb{N}$, there exists a positive constant $C_{m}$
such that

\begin{equation}
\int_{\{t=\text{\textgreek{t}}\}\cap\{r\le R_{1}\}}J_{\text{\textgreek{m}}}^{N}(\text{\textgreek{y}}_{\ge\text{\textgreek{w}}_{+}})n^{\text{\textgreek{m}}}\le\frac{C_{m}}{\text{\textgreek{w}}_{+}^{2m}}\sum_{j=0}^{m}\int_{t=0}J_{\text{\textgreek{m}}}^{N}(T^{j}\text{\textgreek{y}})n^{\text{\textgreek{m}}}.\label{eq:HighFrequencyBoundLemma}
\end{equation}
\end{lem}
\begin{proof}
We can assume without loss of generality that $m\ge1$, since the
$m=0$ case is a direct consequence of Lemma (\ref{lem:BoundednessPsiK}).

Recall that we have defined $\text{\textgreek{q}}_{3}:\mathbb{R}\rightarrow[0,1]$
to be a smooth function which is identically $1$ on $[-1,1]$ and
identically $0$ outside of $(-\frac{9}{8},\frac{9}{8})$. We define,
in terms of $\text{\textgreek{q}}_{3}$, the function $\text{\textgreek{q}}_{3}^{c}\doteq1-\text{\textgreek{q}}_{3}$,
which vanishes in $[-1,1]$ and is identically $1$ outside of $(-\frac{9}{8},\frac{9}{8})$.
Then, for any $m\in\mathbb{N},\, m\ge1$, the function $\text{\textgreek{x}}_{m}:\mathbb{R}\rightarrow\mathbb{C}$,
\begin{equation}
\text{\textgreek{x}}_{m}(y)=\frac{1}{(iy)^{m}}\text{\textgreek{q}}_{3}^{c}(y)\label{eq:DefinitionKsim}
\end{equation}
 is smooth, vanishing in $[-1,1]$ and equal to $\frac{1}{(iy)^{m}}$
for $|y|\ge\frac{9}{8}$. 

We will need the inverse Fourier transform of $\text{\textgreek{x}}_{m}$,
namely $\check{\text{\textgreek{x}}}_{m}(\text{\textgreek{r}})=\int_{-\infty}^{\infty}e^{i\text{\textgreek{r}}y}\cdot\text{\textgreek{x}}_{m}(y)\, dy$,
defined in the sense of tempered distributions. Due to the fact that
\begin{equation}
\Big|\int_{1}^{+\infty}\frac{1}{y}e^{i\text{\textgreek{l}}y}\, dy\Big|\le C\cdot\big(|\log(\text{\textgreek{l}})|+1\big)
\end{equation}
 for any $\text{\textgreek{l}}>0$, the $\frac{1}{(iy)^{m}}$ asymptotics
of $\text{\textgreek{x}}_{m}$ imply that $\check{\text{\textgreek{x}}}_{m}$
is actually a measurable function, satisfying for almost all $y\in\mathbb{R}$
the estimate

\begin{equation}
|\check{\text{\textgreek{x}}}_{m}(\text{\textgreek{r}})|\le C\cdot\big(\big|\log|\text{\textgreek{r}}|\big|+1\big).\label{eq:LogBoundHighFrequencies}
\end{equation}
Of course, in the case $m>1$, when $\frac{1}{y^{m}}$ is integrable
away from $0$, the $\big|\log|\text{\textgreek{r}}|\big|$ summand
can be removed from (\ref{eq:LogBoundHighFrequencies}).

Moreover, in view of the fact that $\frac{d}{dy}\text{\textgreek{q}}_{3}^{c}$
is smooth and compactly supported and $\int_{1}^{\infty}\frac{1}{y^{m+1}}\, dy\le C$
for any $m\in\mathbb{N},$$m\ge1$, we deduce that $|\text{\textgreek{r}}\check{\text{\textgreek{x}}}_{m}(\text{\textgreek{r}})|=|\big(\check{\frac{d}{d\text{\textgreek{r}}}\text{\textgreek{x}}_{m}}\big)(\text{\textgreek{r}})|\le C$
and hence $|\check{\text{\textgreek{x}}}_{m}(\text{\textgreek{r}})|\le\frac{C}{|\text{\textgreek{r}}|}$.
Thus, by an induction argument we infer for any $q\in\mathbb{N}$,
$q\ge1$:

\begin{equation}
|\check{\text{\textgreek{x}}}_{m}(\text{\textgreek{r}})|\le C_{q,m}\cdot|\text{\textgreek{r}}|^{-q}.\label{eq:SchwartzTailHighFrequencies}
\end{equation}
Hence, combining \ref{eq:LogBoundHighFrequencies} and \ref{eq:SchwartzTailHighFrequencies}
we obtain for any $q\in\mathbb{N},q\ge1$:

\begin{equation}
|\check{\text{\textgreek{x}}}_{m}(\text{\textgreek{r}})|\le C_{q,m}\cdot\frac{\big|\log|\text{\textgreek{r}}|\big|+1}{1+|\text{\textgreek{r}}|^{q}}.\label{eq:BoundForHighFrequencyCutOff}
\end{equation}

Due to the fact that $\hat{\text{\textgreek{y}}}_{\ge\text{\textgreek{w}}_{+}}(\cdot,x)$
is supported in $\{|\text{\textgreek{w}}|\ge\text{\textgreek{w}}_{+}\}$,
the following identity holds trivially for all $\text{\textgreek{w}}\in\mathbb{R}$:
\begin{equation}
\hat{\text{\textgreek{y}}}_{\ge\text{\textgreek{w}}_{+}}(\text{\textgreek{w}},\cdot)=\text{\textgreek{q}}_{3}^{c}(\frac{2\text{\textgreek{w}}}{\text{\textgreek{w}}_{+}})\cdot\hat{\text{\textgreek{y}}}_{\ge\text{\textgreek{w}}_{+}}(\text{\textgreek{w}},\cdot).
\end{equation}
In view of the definition (\ref{eq:DefinitionKsim}) of $\text{\textgreek{x}}_{m}$,
we infer that for any $m\in\mathbb{N}$, $m\ge1$:

\begin{equation}
\hat{\text{\textgreek{y}}}_{\ge\text{\textgreek{w}}_{+}}(\text{\textgreek{w}},\cdot)=\big(\frac{\text{\textgreek{w}}_{+}}{2}\big)^{-m}\cdot\text{\textgreek{x}}_{m}(\frac{2\text{\textgreek{w}}}{\text{\textgreek{w}}_{+}})\cdot\hat{T^{m}\text{\textgreek{y}}}_{\ge\text{\textgreek{w}}_{+}}(\text{\textgreek{w}},\cdot).
\end{equation}
 Thus, applying the inverse Fourier transform we obtain:

\begin{equation}
\text{\textgreek{y}}_{\ge\text{\textgreek{w}}_{+}}(t,\cdot)=2^{m-1}\text{\textgreek{w}}_{+}^{-m+1}\cdot\int_{-\infty}^{\infty}\text{\ensuremath{\check{\text{\textgreek{x}}}}}_{m}(\frac{1}{2}\text{\textgreek{w}}_{+}(t-s))\cdot T^{m}\text{\textgreek{y}}_{\ge\text{\textgreek{w}}_{+}}(s,\cdot)\, ds.\label{eq:SelfReplicatingFormulaHighFrequencies}
\end{equation}
This formula is also valid for the derivatives of $\text{\textgreek{y}}_{\ge\text{\textgreek{w}}_{+}}$:

\begin{equation}
d\text{\textgreek{y}}_{\ge\text{\textgreek{w}}_{+}}(t,\cdot)=2^{m-1}\text{\textgreek{w}}_{+}^{-m+1}\cdot\int_{-\infty}^{\infty}\text{\ensuremath{\check{\text{\textgreek{x}}}}}_{m}(\frac{1}{2}\text{\textgreek{w}}_{+}(t-s))\cdot d\big(T^{m}\text{\textgreek{y}}_{\ge\text{\textgreek{w}}_{+}}\big)(s,\cdot)\, ds.\label{eq:Expression1Form}
\end{equation}

In view of the fact that $\int_{-\infty}^{\infty}|\check{\text{\textgreek{x}}}_{m}(\frac{1}{2}\text{\textgreek{w}}_{+}t)|\, dt\le C_{m}\cdot\text{\textgreek{w}}_{+}^{-1}$,
we compute from (\ref{eq:Expression1Form}) after contracting with
$T$ that:

\begin{align}
|T\text{\textgreek{y}}_{\ge\text{\textgreek{w}}_{+}}(t,\cdot)|^{2} & =2^{2m-2}\text{\textgreek{w}}_{+}^{-2m+2}\cdot|\int_{-\infty}^{\infty}\text{\ensuremath{\check{\text{\textgreek{x}}}}}_{m}(\frac{1}{2}\text{\textgreek{w}}_{+}(t-s))\cdot T^{m+1}\text{\textgreek{y}}_{\ge\text{\textgreek{w}}_{+}}(s,\cdot)\, ds|^{2}\le\\
 & \le2^{2m-2}\text{\textgreek{w}}_{+}^{-2m+2}\cdot\Big\{\int_{-\infty}^{\infty}|\text{\ensuremath{\check{\text{\textgreek{x}}}}}_{m}(\frac{1}{2}\text{\textgreek{w}}_{+}(t-s))|\, ds\Big\}\cdot\Big\{\int_{-\infty}^{\infty}|\text{\ensuremath{\check{\text{\textgreek{x}}}}}_{m}(\frac{1}{2}\text{\textgreek{w}}_{+}(t-s))|\cdot|T^{m+1}\text{\textgreek{y}}_{\ge\text{\textgreek{w}}_{+}}(s,\cdot)|^{2}\, ds\Big\}\le\nonumber \\
 & \le C_{m}\cdot\text{\textgreek{w}}_{+}^{-2m+1}\cdot\Big\{\int_{-\infty}^{\infty}|\text{\ensuremath{\check{\text{\textgreek{x}}}}}_{m}(\frac{1}{2}\text{\textgreek{w}}_{+}(t-s))|\cdot|T^{m+1}\text{\textgreek{y}}_{\ge\text{\textgreek{w}}_{+}}(s,\cdot)|^{2}\, ds\Big\}\nonumber 
\end{align}
and, after applying (\ref{eq:BoundForHighFrequencyCutOff}), we obtain:

\begin{equation}
|T\text{\textgreek{y}}_{\ge\text{\textgreek{w}}_{+}}(t,\cdot)|^{2}\le C_{m,q}\cdot\text{\textgreek{w}}_{+}^{-2m+1}\cdot\Big\{\int_{-\infty}^{\infty}|\frac{|\log\big(\text{\textgreek{w}}_{+}(t-s)\big)|+1}{1+|\text{\textgreek{w}}_{+}(t-s)|^{q}}|\cdot|T^{m+1}\text{\textgreek{y}}_{\ge\text{\textgreek{w}}_{+}}(s,\cdot)|^{2}\, ds\Big\}.\label{eq:StillLogConvolution}
\end{equation}
Hence, for $0\le\text{\textgreek{t}}\le t^{*}$ we can bound

\begin{equation}
\int_{\{t=\text{\textgreek{t}}\}\cap\{r\le R_{1}\}}|T\text{\textgreek{y}}_{\ge\text{\textgreek{w}}_{+}}|^{2}\le C_{m,q}\cdot\text{\textgreek{w}}_{+}^{-2m+1}\cdot\Big\{\int_{-\infty}^{\infty}\big|\frac{|\log\big(\text{\textgreek{w}}_{+}(\text{\textgreek{t}}-s)\big)|+1}{1+|\text{\textgreek{w}}_{+}(\text{\textgreek{t}}-s)|^{q}}\big|\cdot\big\{\int_{\{t=s\}\cap\{r\le R_{1}\}}|T^{m+1}\text{\textgreek{y}}_{\ge\text{\textgreek{w}}_{+}}|^{2}\big\}\, ds\Big\}.\label{eq:LogConvolutionAlmostEnergy}
\end{equation}

Using (\ref{eq:SchwartzBounds1}) for some $q'>1$ and the fact that
$\text{\textgreek{y}}_{\ge\text{\textgreek{w}}_{+}}(t,\cdot)=\int_{-\infty}^{\infty}h_{\ge\text{\textgreek{w}}_{+}}(t-s)\text{\textgreek{y}}_{t^{*}}(s)\, ds$,
we compute after an application of H\"older's inequality that 
\begin{equation}
|T^{m+1}\text{\textgreek{y}}_{\ge\text{\textgreek{w}}_{+}}(t,\cdot)|^{2}\le C_{m,q'}\cdot\text{\textgreek{w}}_{+}\int_{-\infty}^{\infty}\frac{1}{1+|\text{\textgreek{w}}_{+}(t-s)|^{q'}}\cdot|T^{m+1}\text{\textgreek{y}}_{t^{*}}(s,\cdot)|^{2}\, ds,
\end{equation}
which after an integration over $\{r\le R_{1}\}$ gives us for any
$\text{\textgreek{t}}\in\mathbb{R}$: 
\begin{equation}
\int_{\{t=\text{\textgreek{t}}\}\cap\{r\le R_{1}\}}|T^{m+1}\text{\textgreek{y}}_{\ge\text{\textgreek{w}}_{+}}|^{2}\le C_{m,q'}\cdot\text{\textgreek{w}}_{+}\int_{-\infty}^{\infty}\frac{1}{1+|\text{\textgreek{w}}_{+}(\text{\textgreek{t}}-s)|^{q'}}\cdot\Big\{\int_{\{t=s\}\cap\{r\le R_{1}\}}|T^{m+1}\text{\textgreek{y}}_{t^{*}}|^{2}\Big\}\, ds.\label{eq:OneMoreInequality}
\end{equation}
 Similarly, repeating the same procedure after contracting (\ref{eq:Expression1Form})
with $T$ invariant vector fields tangential to the foliations $\text{\textgreek{S}}_{\text{\textgreek{t}}}$,
we can also bound
\begin{equation}
\int_{\{t=\text{\textgreek{t}}\}\cap\{r\le R_{1}\}}|\nabla_{\text{\textgreek{S}}}(T^{m}\text{\textgreek{y}}_{\ge\text{\textgreek{w}}_{+}})|_{g_{\text{\textgreek{S}}}}^{2}\le C_{m,q'}\cdot\text{\textgreek{w}}_{+}\int_{-\infty}^{\infty}\frac{1}{1+|\text{\textgreek{w}}_{+}(\text{\textgreek{t}}-s)|^{q'}}\cdot\Big\{\int_{\{t=s\}\cap\{r\le R_{1}\}}|\nabla_{\text{\textgreek{S}}}(T^{m}\text{\textgreek{y}}_{t^{*}})|_{g_{\text{\textgreek{S}}}}^{2}\Big\}\, ds.\label{eq:OneMoreInequality-1}
\end{equation}
Recall that $\text{\textgreek{y}}_{t^{*}}=h_{t^{*}}\cdot\text{\textgreek{y}}$,
and in $\{r\le R_{1}\}$ $h_{t^{*}}$ is non zero only for $\{0\le t\le t^{*}\}$.
Thus, (\ref{eq:OneMoreInequality}) and (\ref{eq:OneMoreInequality-1})
for $q'=2$ yield, in view of the Hardy inequality 
\begin{equation}
\int_{\{t=s\}\cap\{r\le R_{1}\}}|\text{\textgreek{y}}|^{2}\le C(R_{1})\cdot\int_{\{t=s\}}J_{\text{\textgreek{m}}}^{N}(\text{\textgreek{y}})n^{\text{\textgreek{m}}}
\end{equation}
 and the boundedness assumption \hyperref[Assumption 4]{4}, that
\begin{equation}
\int_{\{t=\text{\textgreek{t}}\}\cap\{r\le R_{1}\}}J_{\text{\textgreek{m}}}^{N}(T^{m}\text{\textgreek{y}}_{\ge\text{\textgreek{w}}_{+}})n^{\text{\textgreek{m}}}\le C_{m}\sum_{j=0}^{m}\int_{t=0}J_{\text{\textgreek{m}}}^{N}(T^{j}\text{\textgreek{y}})n^{\text{\textgreek{m}}}.\label{eq:IntermediateBoundTDerivatives}
\end{equation}

Using (\ref{eq:IntermediateBoundTDerivatives}) and (\ref{eq:LogConvolutionAlmostEnergy})
for $q=2$ we infer the desired inequality for any $\text{\textgreek{t}}\in\mathbb{R}$:

\begin{equation}
\int_{\{t=\text{\textgreek{t}}\}\cap\{r\le R_{1}\}}J_{\text{\textgreek{m}}}^{N}(\text{\textgreek{y}}_{\ge\text{\textgreek{w}}_{+}})n^{\text{\textgreek{m}}}\le\frac{C_{m}}{\text{\textgreek{w}}_{+}^{2m}}\sum_{j=0}^{m}\int_{t=0}J_{\text{\textgreek{m}}}^{N}(T^{j}\text{\textgreek{y}})n^{\text{\textgreek{m}}}.
\end{equation}

\end{proof}

\section{\label{sub:EstimatesPsiKAsymptoticallyFlat}Estimates for $\text{\textgreek{y}}_{k}$
in the asymptotically flat region}

In this section, we will specialise some of the estimates established
in \cite{Moschidisc} for general asymptotically flat spacetimes to
our setting. As we did in Section (\ref{sec:FreqDecomposition}),
we will assume that we are given a smooth function $\text{\textgreek{y}}:\mathcal{D}\rightarrow\mathbb{C}$
solving $\square_{g}\text{\textgreek{y}}=0$ on $J^{+}(\text{\textgreek{S}})\cap\mathcal{D}$
with compactly supported initial data on $\text{\textgreek{S}}$,
together with a set of parameters $t^{*},\text{\textgreek{w}}_{0},\text{\textgreek{w}}_{+}$,
leading to the construction of the functions $\text{\textgreek{y}}_{t^{*}}$,
$\text{\textgreek{y}}_{\le\text{\textgreek{w}}_{+}}$, $\text{\textgreek{y}}_{\ge\text{\textgreek{w}}_{+}}$
and $\text{\textgreek{y}}_{k}$ (as performed in Section \ref{sec:FreqDecomposition}).
We will derive estimates for the functions $\text{\textgreek{y}}_{k}$,
$\text{\textgreek{y}}_{\le\text{\textgreek{w}}_{+}}$ and $\text{\textgreek{y}}_{\ge\text{\textgreek{w}}_{+}}$
in the asymptotically flat region $\{r\gg1\}$ of $\mathcal{D}$

We will make use of the fact that there exists a function $u$ in
the region $\{r\gg1\}$ (see the Appendix) such that in the $(u,r,\text{\textgreek{sv}})$
coordinate system the metric $g$ has the form:
\begin{align}
g=-4\big(1- & \frac{2M}{r}+O_{3}(r^{-1-a})\big)du^{2}-4\big(1+O_{3}(r^{-1-a})\big)dudr+r^{2}(g_{\mathbb{S}^{d-1}}+O_{3}(r^{-1-a}))+\label{eq:MetricUR-1-2}\\
 & +O_{3}(r^{-a})dud\text{\textgreek{sv}}+O_{3}(r^{-a})drd\text{\textgreek{sv}}.\nonumber 
\end{align}
 We will also introduce the function $v=u+r$, noting that in the
$(u,v,\text{\textgreek{sv}})$ coordinate chart the metric $g$ takes
the form: 
\begin{align}
g=-\Big(4+ & O_{3}(r^{-1-a})\Big)dvdu+r^{2}\cdot\Big(g_{\mathbb{S}^{d-1}}+O_{3}(r^{-1-a})\Big)+O_{3}(r^{-a})dud\text{\textgreek{sv}}+\label{eq:MetricUV}\\
 & +O_{3}(r^{-a})dvd\text{\textgreek{sv}}+4\Big(-\frac{2M}{r}+O_{3}(r^{-1-a})\Big)du^{2}.\nonumber 
\end{align}

Let us remark that most of the results of \cite{Moschidisc} were
stated under the requirement that the metric $g$ in the asymptotic
region is of the form 
\begin{align}
g=-\Big(4+ & O_{m}(r^{-1-a})\Big)dvdu+r^{2}\cdot\Big(g_{\mathbb{S}^{d-1}}+O_{m}(r^{-1-a})\Big)+O_{m}(r^{-a})dud\text{\textgreek{sv}}+\label{eq:MetricUV-1}\\
 & +O_{m}(r^{-a})dvd\text{\textgreek{sv}}+4\Big(-\frac{2M}{r}+O_{m}(r^{-1-a})\Big)du^{2}\nonumber 
\end{align}
for some large enough integer $m$, i.\,e.~it was assumed in \cite{Moschidisc}
that $g$ is smoother on $\mathcal{I}^{+}$ than what is assumed in
the present paper (actually, in \cite{Moschidisc} the metrics considered
where of more general form than (\ref{eq:MetricUV-1}), but this fact
is not relevant for the present paper).However, the results of \cite{Moschidisc}
that we will make use of in this section can be established also for
a metric with the rougher asymptotics (\ref{eq:MetricUR-1-2}) (as
can be readily verified by an inspection of their proof).

\subsection{Some $\partial_{r}$-Morawetz type estimates}

We will establish the following Lemma:
\begin{lem}
\label{lem:MorawetzDrLemma}For any given $0<\text{\textgreek{h}}<a$,
there exists an $R=R(\text{\textgreek{h}})>0$ and $C(\text{\textgreek{h}}),\, C(\text{\textgreek{w}}_{0},\text{\textgreek{h}})>0$
such that for any $-n\le k\le n$ and any smooth cut-off function
$\text{\textgreek{q}}:\mathcal{D}\rightarrow[0,1]$ supported in $\{r\ge R\}$,
we can bound:

\begin{multline}
\int_{\mathcal{R}(0,t^{*})}\text{\textgreek{q}}\cdot\Big(r^{-1-\text{\textgreek{h}}}\Big(|\partial_{t}\text{\textgreek{y}}_{k}|^{2}+|\partial_{r}\text{\textgreek{y}}_{k}|^{2}\Big)+r^{-1}|r^{-1}\partial_{\text{\textgreek{sv}}}\text{\textgreek{y}}_{k}|^{2}+r^{-3-\text{\textgreek{h}}}|\text{\textgreek{y}}_{k}|^{2}\Big)\le\\
\le C(\text{\textgreek{h}})\cdot\int_{\{supp(\partial\text{\textgreek{q}})\}\cap\mathcal{R}(0,t^{*})}|\partial\text{\textgreek{q}}|\cdot J_{\text{\textgreek{m}}}^{N}(\text{\textgreek{y}}_{k})n^{\text{\textgreek{m}}}+C(\text{\textgreek{w}}_{0},\text{\textgreek{h}})\cdot\int_{t=0}J_{\text{\textgreek{m}}}^{N}(\text{\textgreek{y}})n^{\text{\textgreek{m}}}.\label{eq:Morawetz2}
\end{multline}

The same estimate holds for $\text{\textgreek{y}}_{\le\text{\textgreek{w}}_{+}},\text{\textgreek{y}}_{\ge\text{\textgreek{w}}_{+}}$
in place of $\text{\textgreek{y}}_{k}$.\end{lem}
\begin{rem*}
We have used the notation $|\partial h|^{2}=|\partial_{t}h|^{2}+|\partial_{r}h|^{2}+|\frac{1}{r}\partial_{\text{\textgreek{sv}}}h|^{2}$.\end{rem*}
\begin{proof}
From Lemma 4.2 of \cite{Moschidisc}, we can bound for any smooth
function $\text{\textgreek{f}}:\mathcal{D}\rightarrow\mathbb{C}$
with compact support in space (in the $(u,v,\text{\textgreek{sv}})$
coordinate system constructed in the Appendix on the region $\{r\gg1\}$):
\begin{equation}
\begin{split}\int_{\mathcal{R}(0,t^{*})}\text{\textgreek{q}}\cdot\Big(r^{-1-\text{\textgreek{h}}}\Big(|\partial_{u}\text{\textgreek{f}}|^{2}+ & |\partial_{v}\text{\textgreek{f}}|^{2}\Big)+r^{-1}|r^{-1}\partial_{\text{\textgreek{sv}}}\text{\textgreek{f}}|^{2}+r^{-3-\text{\textgreek{h}}}|\text{\textgreek{f}}|^{2}\Big)\le\\
\le & C_{\text{\textgreek{h}}}\int_{\mathcal{R}(0,t^{*})}|\partial\text{\textgreek{q}}|\cdot\big(|\partial\text{\textgreek{f}}|^{2}+r^{-2}|\text{\textgreek{f}}|^{2}\big)+\\
 & +C_{\text{\textgreek{h}}}\int_{\{t=0\}\cap\{r\ge R\}}J_{\text{\textgreek{m}}}^{T}(\text{\textgreek{f}})n^{\text{\textgreek{m}}}+C_{\text{\textgreek{h}}}\int_{\{t=t^{*}\}\cap\{r\ge R\}}J_{\text{\textgreek{m}}}^{T}(\text{\textgreek{f}})n^{\text{\textgreek{m}}}\\
 & +\int_{\mathcal{R}(0,t^{*})}\text{\textgreek{q}}\cdot Re\big\{\big(O_{\text{\textgreek{h}}}(1)(\partial_{v}-\partial_{u})\bar{\text{\textgreek{f}}}+O_{\text{\textgreek{h}}}(r^{-1})\bar{\text{\textgreek{f}}}\big)\cdot\square_{g}\text{\textgreek{f}}\big\}.
\end{split}
\label{eq:MorawetzGeneralCaseRadiative}
\end{equation}

In view of Lemmas (\ref{lem:BoundF}) and (\ref{lem:BoundednessPsiK})
(as well as the properties of the functions $u,v$), inequality (\ref{eq:Morawetz2})
follows readily from (\ref{eq:MorawetzGeneralCaseRadiative}) after
substituting $\text{\textgreek{y}}_{k}$ (or $\text{\textgreek{y}}_{\le\text{\textgreek{w}}_{+}},\text{\textgreek{y}}_{\ge\text{\textgreek{w}}_{+}}$
respectively) in place of $\text{\textgreek{f}}$ (and using a Cauchy--Schwarz
inequality for the last term of the right hand side).%
\footnote{Notice that here we have used the fact that the functions $\text{\textgreek{y}}_{k},\text{\textgreek{y}}_{\le\text{\textgreek{w}}_{+}},\text{\textgreek{y}}_{\ge\text{\textgreek{w}}_{+}}$
have compact support in space, since they are supported in a cylinder
$\{r\lesssim R_{sup}+t^{*}\}$.%
}
\end{proof}
We can also present the previous estimate expressed in a more refined
form in terms of the ``boundary'' terms of the right hand side: 
\begin{lem}
\label{lem:MorawetzRefinedboundary}For any given $0<\text{\textgreek{h}}<a$
and any $R>0$ large enough in terms of $\text{\textgreek{h}}$, there
exist $C(\text{\textgreek{h}}),\, C(\text{\textgreek{w}}_{0},\text{\textgreek{h}})>0$
such that for any $-n\le k\le n$ we can bound:

\begin{multline}
\int_{\{r\ge2R\}\cap\mathcal{R}(0,t^{*})}\Big(r^{-1-\text{\textgreek{h}}}\Big(|\partial_{t}\text{\textgreek{y}}_{k}|^{2}+|\partial_{r}\text{\textgreek{y}}_{k}|^{2}\Big)+r^{-1}|r^{-1}\partial_{\text{\textgreek{sv}}}\text{\textgreek{y}}_{k}|^{2}+r^{-3-\text{\textgreek{h}}}|\text{\textgreek{y}}_{k}|^{2}\Big)\le\\
\le C(\text{\textgreek{h}})\cdot\int_{\{R\le r\le2R\}\cap\mathcal{R}(0,t^{*})}\big(R^{-1}J_{\text{\textgreek{m}}}^{N}(\text{\textgreek{y}}_{k})n^{\text{\textgreek{m}}}+R^{-3}|\text{\textgreek{y}}_{k}|^{2}\big)+C(\text{\textgreek{w}}_{0},\text{\textgreek{h}})\int_{t=0}J_{\text{\textgreek{m}}}^{N}(\text{\textgreek{y}})n^{\text{\textgreek{m}}}.\label{eq:Morawetz2-1}
\end{multline}
 The same estimate also holds for $\text{\textgreek{y}}_{\le\text{\textgreek{w}}_{+}},\text{\textgreek{y}}_{\ge\text{\textgreek{w}}_{+}}$
in place of $\text{\textgreek{y}}_{k}$.\end{lem}
\begin{proof}
Inequality (\ref{eq:Morawetz2-1}) can be proven by fixing a smooth
cut-off $\text{\textgreek{x}}:[0,+\infty)\rightarrow[0,1]$ such that
$\text{\textgreek{x}}\equiv0$ for $r\le1$ and $\text{\textgreek{x}}\equiv1$
for $r\ge2$, and defining $\text{\textgreek{q}}=\text{\textgreek{q}}_{R}:D\rightarrow[0,1]$,
$\text{\textgreek{q}}_{R}=\text{\textgreek{x}}\circ(\frac{r}{R})$
for any $R>0$. We easily calculate that there exists a constant $C>0$
depending on the precise choice of $\text{\textgreek{x}}$, such that
$|\partial\text{\textgreek{q}}_{R}|\le C\cdot R^{-1}$. Moreover,
$supp(\partial\text{\textgreek{q}}_{R})\subseteq\{R\le r\le2R\}$.
The result then follows by using $\text{\textgreek{q}}_{R}$ as a
cut-off function in the statement of the previous lemma (\ref{lem:MorawetzDrLemma}),
for any $R>0$ sufficiently large in terms of $\text{\textgreek{h}}$.
\end{proof}
We will also need a variant of the previous lemma, that provides improved
control of the spacetime integral of the energy of $\text{\textgreek{y}}_{k},\text{\textgreek{y}}_{\le\text{\textgreek{w}}_{+}},\text{\textgreek{y}}_{\ge\text{\textgreek{w}}_{+}}$
over compact subsets. We we will only need to use it in the case of
$\text{\textgreek{y}}_{0}$, but we can state it for all the components
$\text{\textgreek{y}}_{k},\text{\textgreek{y}}_{\le\text{\textgreek{w}}_{+}},\text{\textgreek{y}}_{\ge\text{\textgreek{w}}_{+}}$
of $\text{\textgreek{y}}$:
\begin{lem}
\label{lem:ImprovedMorawetz}For any given $0<\text{\textgreek{h}}<a$,
any $R>0$ sufficiently large in terms of $\text{\textgreek{h}}$,
there exists some $C(\text{\textgreek{h}})>0$ such that for any $R_{c}\ge2R$
and any $-n\le k\le n$ we can bound

\begin{multline}
\int_{\{2R\le r\le R_{c}\}\cap\mathcal{R}(0,t^{*})}\Big(r^{-\text{\textgreek{h}}}\Big(|\partial_{t}\text{\textgreek{y}}_{k}|^{2}+|\partial_{r}\text{\textgreek{y}}_{k}|^{2}+|r^{-1}\partial_{\text{\textgreek{sv}}}\text{\textgreek{y}}_{k}|^{2}\Big)+r^{-2-\text{\textgreek{h}}}|\text{\textgreek{y}}_{k}|^{2}\Big)\le\\
\le C(\text{\textgreek{h}})\cdot\int_{\{R\le r\le2R\}\cap\mathcal{R}(0,t^{*})}\big(J_{\text{\textgreek{m}}}^{N}(\text{\textgreek{y}}_{k})n^{\text{\textgreek{m}}}+R^{-2}|\text{\textgreek{y}}_{k}|^{2}\big)+C(\text{\textgreek{w}}_{0},R_{c},\text{\textgreek{h}})\int_{t=0}J_{\text{\textgreek{m}}}^{N}(\text{\textgreek{y}})n^{\text{\textgreek{m}}}.\label{eq:Morawetz2-1-1}
\end{multline}

The same estimate also holds for $\text{\textgreek{y}}_{\le\text{\textgreek{w}}_{+}},\text{\textgreek{y}}_{\ge\text{\textgreek{w}}_{+}}$
in place of $\text{\textgreek{y}}_{k}$.\end{lem}
\begin{rem*}
Notice that the constant in front of the $\{R\le r\le2R\}$ terms
does not depend on $R_{c}$.\end{rem*}
\begin{proof}
This is an immediate corollary of Lemma 4.4 of \cite{Moschidisc}
for $\text{\textgreek{y}}_{k}$ in place of $\text{\textgreek{f}}$,
which yields
\begin{equation}
\begin{split}\int_{\mathcal{R}(0,t^{*})\cap\{r\le R_{c}\}}\text{\textgreek{q}}\cdot\Big(\big(|\partial_{u}\text{\textgreek{y}}_{k}|^{2}+|\partial_{v} & \text{\textgreek{y}}_{k}|^{2}\big)+|r^{-1}\partial_{\text{\textgreek{sv}}}\text{\textgreek{y}}_{k}|^{2}+r^{-2}|\text{\textgreek{y}}_{k}|^{2}\Big)\le\\
\le & C(\text{\textgreek{h}})\cdot\int_{\mathcal{R}(0,t^{*})}|\partial\text{\textgreek{q}}|\cdot r\cdot\big(|\partial\text{\textgreek{y}}_{k}|^{2}+r^{-2}|\text{\textgreek{y}}_{k}|^{2}\big)+\\
 & +C(\text{\textgreek{h}},R_{c})\cdot\int_{\{t=0\}\cap\{r\ge R\}}|\partial\text{\textgreek{y}}_{k}|^{2}+C(\text{\textgreek{h}},R_{c})\cdot\int_{\{t=t^{*}\}\cap\{r\ge R\}}|\partial\text{\textgreek{y}}_{k}|^{2}+\\
 & +\int_{\mathcal{R}(0,t^{*})}\text{\textgreek{q}}\cdot Re\big\{\big(O_{R_{c},\text{\textgreek{h}}}(1)(\partial_{v}-\partial_{u})\bar{\text{\textgreek{y}}}_{k}+O_{R_{c},\text{\textgreek{h}}}(r^{-1})\bar{\text{\textgreek{y}}}_{k}\big)\cdot\square_{g}\text{\textgreek{y}}_{k}\big\},
\end{split}
\label{eq:MorawetzGeneralCaseImproved}
\end{equation}
combined with Lemmas \ref{lem:BoundF} and \ref{lem:BoundednessPsiK}.
Similarly for $\text{\textgreek{y}}_{\le\text{\textgreek{w}}_{+}},\text{\textgreek{y}}_{\ge\text{\textgreek{w}}_{+}}$.
\end{proof}

\subsection{\label{sub:NewMethodEstimates}Some $r^{p}$-weighted energy estimates}

The following $r^{p}$-weighted energy estimate can be established
as a Corollary of Theorem 5.3 of \cite{Moschidisc} (notice that $(v,\text{\textgreek{sv}})$
defines a regular polar coordinate map in the region $\{r\gg1\}$
of the hypersurfaces $\{t=const\}$):
\begin{lem}
\label{lem:FinalStatementNewMethodPsiK}For any $0<p\le2$, any $0<\text{\textgreek{h}}<a$,
any $0<\text{\textgreek{d}}<1$ and any $R>0$ large enough in terms
of $p,\text{\textgreek{h}},\text{\textgreek{d}}$, the following inequality
holds for any $-n\le k\le n$, any $0\le t_{1}\le t_{2}\le t^{*}$
and any smooth cut-off $\text{\textgreek{q}}_{R}:\mathcal{D}\rightarrow[0,1]$
supported in $\{r\ge R\}$ so that $\text{\textgreek{q}}_{R}\equiv1$
on $\{r\ge2R\}$: (we also set $\text{\textgreek{Y}}_{k}\doteq r^{\frac{d-1}{2}}\cdot\text{\textgreek{y}}_{k}$)
\begin{equation}
\begin{split}\int_{\{t=t_{2}\}}\text{\textgreek{q}}_{R}\cdot\Big(r^{p}|\partial_{v}\text{\textgreek{Y}}_{k}|^{2} & +r^{p}|r^{-1}\partial_{\text{\textgreek{sv}}}\text{\textgreek{Y}}_{k}|^{2}+\big((d-3)r^{p-2}+\min\{r^{p-2},r^{-\text{\textgreek{d}}}\}\big)|\text{\textgreek{Y}}_{k}|^{2}\Big)\, dvd\text{\textgreek{sv}}+\\
+\int_{\mathcal{R}(t_{1},t_{2})}\text{\textgreek{q}}_{R}\cdot\Big(pr^{p-1}\big|\partial_{v} & \text{\textgreek{Y}}_{k}\big|^{2}+\big\{\big((2-p)r^{p-1}+r^{p-1-\text{\textgreek{d}}}\big)\big|r^{-1}\partial_{\text{\textgreek{sv}}}\text{\textgreek{Y}}_{k}\big|^{2}+\\
\hphantom{+\int_{\{t_{1}\le t\le t_{2}\}}\text{\textgreek{q}}_{R}\cdot\Big(}+\big((2- & p)(d-3)r^{p-3}+\min\{r^{p-3},r^{-1-\text{\textgreek{d}}}\}\big)\big|\text{\textgreek{Y}}_{k}\big|^{2}\big\}+r^{-1-\text{\textgreek{h}}}|\partial_{u}\text{\textgreek{Y}}_{k}|^{2}\Big)\, dudvd\text{\textgreek{sv}}\le\\
\le & C(p,\text{\textgreek{h}},\text{\textgreek{d}})\cdot\int_{\mathcal{R}(t_{1},t_{2})}|\partial\text{\textgreek{q}}_{R}|\cdot\big(r^{p}|\partial\text{\textgreek{Y}}_{k}|^{2}+r^{p-2}|\text{\textgreek{Y}}_{k}|^{2}\big)\, dudvd\text{\textgreek{sv}}+\\
 & \hphantom{C(}+C(\text{\textgreek{w}}_{0},p,\text{\textgreek{h}},R,\text{\textgreek{q}}_{R},\text{\textgreek{d}})\cdot\int_{\{t=0\}}(1+r^{p})\cdot J_{\text{\textgreek{m}}}^{N}(\text{\textgreek{y}})n^{\text{\textgreek{m}}}.
\end{split}
\label{eq:FinalNewMethodPsiK}
\end{equation}

The same estimate also holds for $\text{\textgreek{y}}_{\le\text{\textgreek{w}}_{+}},\text{\textgreek{y}}_{\ge\text{\textgreek{w}}_{+}}$
in place of $\text{\textgreek{y}}_{k}$.\end{lem}
\begin{proof}
We will establish the result only for $\text{\textgreek{y}}_{k}$,
since the proof for the cases of $\text{\textgreek{y}}_{\le\text{\textgreek{w}}_{+}},\text{\textgreek{y}}_{\ge\text{\textgreek{w}}_{+}}$
is identical. 

We will set 
\begin{equation}
\text{\textgreek{Y}}_{k}\doteq r^{\frac{d-1}{2}}\text{\textgreek{y}}_{k}.
\end{equation}
 In view of Theorem 5.3 of \cite{Moschidisc} (with $\text{\textgreek{y}}_{k}$
in place of $\text{\textgreek{f}}$), using also of the fact that
in the region $\{r\gg1\}$ in the $(u,v,\text{\textgreek{sv}})$ coordinate
system we have the relation 
\begin{equation}
det(g)=-4r^{2(d-1)}\big(1+O_{3}(r^{-1})\big),
\end{equation}
we can bound: 
\begin{equation}
\begin{split}\int_{\{t=t_{2}\}}\text{\textgreek{q}}_{R}\cdot\Big(r^{p}|\partial_{v}\text{\textgreek{Y}}_{k}|^{2} & +r^{p}|r^{-1}\partial_{\text{\textgreek{sv}}}\text{\textgreek{Y}}_{k}|^{2}+\big((d-3)r^{p-2}+\min\{r^{p-2},r^{-\text{\textgreek{d}}}\}\big)|\text{\textgreek{Y}}_{k}|^{2}\Big)\, dvd\text{\textgreek{sv}}+\\
+\int_{\mathcal{R}(t_{1},t_{2})}\text{\textgreek{q}}_{R}\cdot\Big(pr^{p-1}\big|\partial_{v} & \text{\textgreek{Y}}_{k}\big|^{2}+\big\{\big((2-p)r^{p-1}+r^{p-1-\text{\textgreek{d}}}\big)\big|r^{-1}\partial_{\text{\textgreek{sv}}}\text{\textgreek{Y}}_{k}\big|^{2}+\\
\hphantom{+\int_{\{t_{1}\le t\le t_{2}\}}\text{\textgreek{q}}_{R}\cdot\Big(}+\big((2- & p)(d-3)r^{p-3}+\min\{r^{p-3},r^{-1-\text{\textgreek{d}}}\}\big)\big|\text{\textgreek{Y}}_{k}\big|^{2}\big\}+r^{-1-\text{\textgreek{h}}}|\partial_{u}\text{\textgreek{Y}}_{k}|^{2}\Big)\, dudvd\text{\textgreek{sv}}\lesssim_{p,\text{\textgreek{d}}}\\
\lesssim_{p,\text{\textgreek{h}},\lyxmathsym{\textgreek{d}}} & \int_{\{t=t_{1}\}}\text{\textgreek{q}}_{R}\cdot\Big(r^{p}|\partial_{v}\text{\textgreek{Y}}_{k}|^{2}+r^{p}|r^{-1}\partial_{\text{\textgreek{sv}}}\text{\textgreek{Y}}_{k}|^{2}+\big((d-3)r^{p-2}+\min\{r^{p-2},r^{-\text{\textgreek{d}}}\}\big)|\text{\textgreek{Y}}_{k}|^{2}\Big)\, dvd\text{\textgreek{sv}}+\\
 & +\int_{\{t=t_{1}\}}\text{\textgreek{q}}_{R}J_{\text{\textgreek{m}}}^{T}(\text{\textgreek{y}}_{k})n^{\text{\textgreek{m}}}+\int_{\mathcal{R}(t_{1},t_{2})}|\partial\text{\textgreek{q}}_{R}|\cdot\big(r^{p}|\partial\text{\textgreek{y}}_{k}|^{2}+r^{p-2}\cdot|\text{\textgreek{y}}_{k}|^{2}\big)+\\
 & +\int_{\mathcal{R}(t_{1},t_{2})}\text{\textgreek{q}}_{R}\cdot(r^{p+1}+r^{1+\text{\textgreek{h}}})\cdot|F_{k}|^{2}\,\text{\textgreek{W}}^{2}dudvd\text{\textgreek{sv}}.
\end{split}
\label{eq:newMethodFinalStatement-1}
\end{equation}
Thus, in view also of Lemmas \ref{lem:BoundF} and \ref{lem:BoundednessPsiK},
in order to reach (\ref{eq:FinalNewMethodPsiK}) it suffices to prove
that 
\begin{equation}
\begin{split}\int_{\{t=t_{1}\}}\text{\textgreek{q}}_{R}\cdot\Big(r^{p}|\partial_{v}\text{\ensuremath{\text{\textgreek{Y}}_{k}}}|^{2}+r^{p}|r^{-1}\partial_{\text{\textgreek{sv}}}\text{\textgreek{Y}}_{k}|^{2}+\big((d-3)r^{p-2}+\min\{r^{p-2} & ,r^{-\text{\textgreek{d}}}\}\big)|\text{\textgreek{Y}}_{k}|^{2}\Big)\, dvd\text{\textgreek{sv}}\le\\
\le & C(\text{\textgreek{w}}_{0},R,p,\text{\textgreek{q}}_{R})\cdot\int_{\{t=0\}}(1+r^{p})\cdot J_{\text{\textgreek{m}}}^{N}(\text{\textgreek{y}})n^{\text{\textgreek{m}}}.
\end{split}
\label{eq:DesiredInequality}
\end{equation}

Inequality (\ref{eq:DesiredInequality}) will be established in the
way that Lemma \ref{lem:BoundednessPsiK} was proven, the only difference
being that instead of the boundedness assumption \hyperref[Assumption 4]{4}
for $\text{\textgreek{y}}$, we will mainly use the following estimate
for \textgreek{y} (obtained by applying Theorem 5.3 of \cite{Moschidisc},
after repeating the proof in the region $\{t\le s\}\cap\{t_{-}\ge0\}$
for any $s\in\mathbb{R}$): setting $\text{\textgreek{Y}}=r^{\frac{d-1}{2}}\text{\textgreek{y}}$,
we can bound for any $s\in\mathbb{R}$
\begin{equation}
\begin{split}\int_{\{t=s\}\cap\{t_{-}\ge0\}}\text{\textgreek{q}}_{R}\cdot\Big(r^{p}|\partial_{v}\text{\textgreek{Y}}|^{2} & +r^{p}|r^{-1}\partial_{\text{\textgreek{sv}}}\text{\textgreek{Y}}|^{2}+\big((d-3)r^{p-2}+\min\{r^{p-2},r^{-\text{\textgreek{d}}}\}\big)|\text{\textgreek{Y}}|^{2}\Big)\, dvd\text{\textgreek{sv}}+\\
+\int_{\{t_{-}\ge0\}\cap\{t\le s\}}\text{\textgreek{q}}_{R}\cdot\Big(pr^{p-1}\big|\partial_{v} & \text{\textgreek{Y}}\big|^{2}+\big\{\big((2-p)r^{p-1}+r^{p-1-\text{\textgreek{d}}}\big)\big|r^{-1}\partial_{\text{\textgreek{sv}}}\text{\textgreek{Y}}\big|^{2}+\\
\hphantom{+\int_{\{t_{1}\le t\le t_{2}\}}\text{\textgreek{q}}_{R}\cdot\Big(}+\big((2- & p)(d-3)r^{p-3}+\min\{r^{p-3},r^{-1-\text{\textgreek{d}}}\}\big)\big|\text{\textgreek{Y}}\big|^{2}\big\}+r^{-1-\text{\textgreek{h}}}|\partial_{u}\text{\textgreek{Y}}|^{2}\Big)\, dudvd\text{\textgreek{sv}}\lesssim_{p,\text{\textgreek{d}}}\\
\lesssim_{p,\text{\textgreek{h}},\lyxmathsym{\textgreek{d}}} & \int_{\{t_{-}=0\}\cap\{t\le s\}}\text{\textgreek{q}}_{R}\cdot\Big(r^{p}|\partial_{v}\text{\textgreek{Y}}|^{2}+r^{p}|r^{-1}\partial_{\text{\textgreek{sv}}}\text{\textgreek{Y}}|^{2}+\big((d-3)r^{p-2}+\min\{r^{p-2},r^{-\text{\textgreek{d}}}\}\big)|\text{\textgreek{Y}}|^{2}\Big)\, dvd\text{\textgreek{sv}}+\\
 & +\int_{\{t_{-}=0\}\cap\{t\le s\}}\text{\textgreek{q}}_{R}J_{\text{\textgreek{m}}}^{T}(\text{\textgreek{y}})n^{\text{\textgreek{m}}}+\int_{\{t_{-}\ge0\}\cap\{t\le s\}}|\partial\text{\textgreek{q}}_{R}|\cdot\big(r^{p}|\partial\text{\textgreek{y}}|^{2}+r^{p-2}\cdot|\text{\textgreek{y}}|^{2}\big).
\end{split}
\label{eq:AuxiliaryBound}
\end{equation}

Using the boundedness assumption \hyperref[Assumption 4]{4} for $s\ge0$,
and the conservation of the $J^{T}$ current in the region $\{t_{-}\ge0\}\cap\{t\le s\}$
for $s<0$, as well as Lemma \ref{lem:ComparisonTT-}, we can bound
\begin{equation}
\int_{\{t_{-}=0\}\cap\{t\le s\}}J_{\text{\textgreek{m}}}^{N}(\text{\textgreek{y}})n^{\text{\textgreek{m}}}\le C\cdot\int_{\{t=0\}}J_{\text{\textgreek{m}}}^{N}(\text{\textgreek{y}})n^{\text{\textgreek{m}}}.
\end{equation}
 Using Lemma \ref{lem:ComparisonTT-} and a Hardy inequality 
\begin{equation}
\int_{\{t_{-}=0\}\cap\{t\le s\}}r^{p-2}|\text{\textgreek{y}}|^{2}\le C(p)\cdot\int_{\{t_{-}=0\}\cap\{t\le s\}}r^{p}J_{\text{\textgreek{m}}}^{N}(\text{\textgreek{y}})n^{\text{\textgreek{m}}}
\end{equation}
we can also bound:
\begin{equation}
\begin{split}\int_{\{t_{-}=0\}\cap\{t\le s\}}\text{\textgreek{q}}_{R}\cdot\Big(r^{p}|\partial_{v}\text{\textgreek{Y}}|^{2}+r^{p}|r^{-1}\partial_{\text{\textgreek{sv}}}\text{\textgreek{Y}}|^{2}+\big((d-3)r^{p-2}+\min\{r^{p-2},r^{-\text{\textgreek{d}}}\} & \big)|\text{\textgreek{Y}}|^{2}\Big)\, dvd\text{\textgreek{sv}}\le\\
\le & C(p)\cdot\int_{t=0}(1+r^{p})\cdot J_{\text{\textgreek{m}}}^{N}(\text{\textgreek{y}})n^{\text{\textgreek{m}}}.
\end{split}
\end{equation}
 Therefore, returning to (\ref{eq:AuxiliaryBound}) we obtain:

\begin{equation}
\begin{split}\int_{\{t=s\}\cap\{t_{-}\ge0\}}\text{\textgreek{q}}_{R}\cdot\Big(r^{p}|\partial_{v}\text{\textgreek{Y}}|^{2} & +r^{p}|r^{-1}\partial_{\text{\textgreek{sv}}}\text{\textgreek{Y}}|^{2}+\big((d-3)r^{p-2}+\min\{r^{p-2},r^{-\text{\textgreek{d}}}\}\big)|\text{\textgreek{Y}}|^{2}\Big)\, dvd\text{\textgreek{sv}}+\\
+\int_{\{t_{-}\ge0\}\cap\{t\le s\}}\text{\textgreek{q}}_{R}\cdot\Big(pr^{p-1}\big|\partial_{v} & \text{\textgreek{Y}}\big|^{2}+\big\{\big((2-p)r^{p-1}+r^{p-1-\text{\textgreek{d}}}\big)\big|r^{-1}\partial_{\text{\textgreek{sv}}}\text{\textgreek{Y}}\big|^{2}+\\
\hphantom{+\int_{\{t_{1}\le t\le t_{2}\}}\text{\textgreek{q}}_{R}\cdot\Big(}+\big((2- & p)(d-3)r^{p-3}+\min\{r^{p-3},r^{-1-\text{\textgreek{d}}}\}\big)\big|\text{\textgreek{Y}}\big|^{2}\big\}+r^{-1-\text{\textgreek{h}}}|\partial_{u}\text{\textgreek{Y}}|^{2}\Big)\, dudvd\text{\textgreek{sv}}\lesssim_{p,\text{\textgreek{d}}}\\
\lesssim_{p,\text{\textgreek{h}},\lyxmathsym{\textgreek{d}}} & \int_{t=0}(1+r^{p})\cdot J_{\text{\textgreek{m}}}^{N}(\text{\textgreek{y}})n^{\text{\textgreek{m}}}+\int_{\{t_{-}\ge0\}\cap\{t\le s\}}|\partial\text{\textgreek{q}}_{R}|\cdot\big(r^{p}|\partial\text{\textgreek{y}}|^{2}+r^{p-2}\cdot|\text{\textgreek{y}}|^{2}\big).
\end{split}
\label{eq:BoundPsiNewMethod}
\end{equation}

We recall that 
\begin{equation}
\text{\textgreek{y}}_{k}(t,\cdot)=\int_{-\infty}^{+\infty}h_{k}(t-s)\cdot\text{\textgreek{y}}_{t^{*}}(s,\cdot)\, ds,
\end{equation}
 where $h_{k}$ satisfies (\ref{eq:SchwartzBounds2}). Hence, we can
estimate for $q$ large enough (the precise value of which will be
specified exactly later in the proof ): 

\begin{align}
\int_{\{t=t_{1}\}}\text{\textgreek{q}}_{R}\cdot r^{p}|\partial_{v}\text{\text{\textgreek{Y}}}_{k}|^{2}\, dvd\text{\textgreek{sv}} & =\int_{\{t=t_{1}\}}r^{p}\text{\textgreek{q}}_{R}\cdot|\int_{-\infty}^{\infty}h_{k}(t_{1}-s)\cdot\partial_{v}(r^{\frac{d-1}{2}}\text{\textgreek{y}}_{t^{*}})(s,x)\, ds|^{2}\, dvd\text{\textgreek{sv}}\le\label{eq:Boundedness1-1-1}\\
 & \le\int_{\{t=t_{1}\}}r^{p}\text{\textgreek{q}}_{R}\cdot\big(\int_{-\infty}^{\infty}\frac{C_{q}(\text{\textgreek{w}}_{0})}{1+|t_{1}-s|^{q}}|\partial_{v}(r^{\frac{d-1}{2}}\text{\textgreek{y}}_{t^{*}})(s,x)|\, ds\big)^{2}\, dvd\text{\textgreek{sv}}\le\nonumber \\
 & \le C_{q}(\text{\textgreek{w}}_{0})\int_{\{t=t_{1}\}}r^{p}\text{\textgreek{q}}_{R}\cdot\big(\int_{-\infty}^{\infty}\frac{1}{1+|t_{1}-s|^{q}}|\partial_{v}(r^{\frac{d-1}{2}}\text{\textgreek{y}}_{t^{*}})(s,x)|^{2}\, ds\big)\, dvd\text{\textgreek{sv}}\le\nonumber \\
 & \le C_{q}(\text{\textgreek{w}}_{0})\int_{-\infty}^{\infty}\frac{1}{1+|t_{1}-s|^{q}}\big(\int_{\{t=s\}}r^{p}\text{\textgreek{q}}_{R}\cdot|\partial_{v}(r^{\frac{d-1}{2}}\text{\textgreek{y}}_{t^{*}})|^{2}\, dvd\text{\textgreek{sv}}\big)\, ds,\nonumber 
\end{align}
and similarly 
\begin{equation}
\begin{split}\int_{\{t=t_{1}\}} & \text{\textgreek{q}}_{R}\cdot\big(r^{p}|r^{-1}\partial_{\text{\textgreek{sv}}}\text{\text{\textgreek{Y}}}_{k}|^{2}+\big((d-3)r^{p-2}+\min\{r^{p-2},r^{-\text{\textgreek{d}}}\}\big)|\text{\textgreek{Y}}_{k}|^{2}\big)\, dvd\text{\textgreek{sv}}\le\\
\le & C_{q}(\text{\textgreek{w}}_{0})\int_{-\infty}^{\infty}\frac{1}{1+|t_{1}-s|^{q}}\Big\{\int_{\{t=s\}}\text{\textgreek{q}}_{R}\cdot\Big(r^{p}|r^{-1}\partial_{\text{\textgreek{sv}}}(r^{\frac{d-1}{2}}\text{\textgreek{y}}_{t^{*}})|^{2}+\big((d-3)r^{p-2}+\min\{r^{p-2},r^{-\text{\textgreek{d}}}\}\big)|r^{\frac{d-1}{2}}\text{\textgreek{y}}_{t^{*}}|^{2}\Big)\, dvd\text{\textgreek{sv}}\Big\}\, ds.
\end{split}
\label{eq:ExtraTermsBoundedness}
\end{equation}
 In view of the fact that $\text{\textgreek{y}}_{t^{*}}=h_{t^{*}}\cdot\text{\textgreek{y}}$,
from (\ref{eq:Boundedness1-1-1}) and (\ref{eq:ExtraTermsBoundedness})
we obtain: 
\begin{equation}
\begin{split}\int_{\{t=t_{1}\}} & \text{\textgreek{q}}_{R}\cdot\big(r^{p}|\partial_{v}\text{\textgreek{Y}}_{k}|^{2}+r^{p}|r^{-1}\partial_{\text{\textgreek{sv}}}\text{\text{\textgreek{Y}}}_{k}|^{2}+\big((d-3)r^{p-2}+\min\{r^{p-2},r^{-\text{\textgreek{d}}}\}\big)|\text{\textgreek{Y}}_{k}|^{2}\big)\, dvd\text{\textgreek{sv}}\le\\
\le & C_{q}(\text{\textgreek{w}}_{0})\int_{-\infty}^{\infty}\frac{1}{1+|t_{1}-s|^{q}}\Big(\int_{\{t=s\}}\text{\textgreek{q}}_{R}\cdot h_{t^{*}}^{2}\big(r^{p}|\partial_{v}\text{\textgreek{Y}}|^{2}+r^{p}|r^{-1}\partial_{\text{\textgreek{sv}}}\text{\textgreek{Y}}|^{2}+\big((d-3)r^{p-2}+\min\{r^{p-2},r^{-\text{\textgreek{d}}}\}\big)|\text{\textgreek{Y}}|^{2}\big)\, dvd\text{\textgreek{sv}}\Big)\, ds+\\
 & +C_{q}(\text{\textgreek{w}}_{0})\int_{-\infty}^{\infty}\frac{1}{1+|t_{1}-s|^{q}}\Big(\int_{\{t=s\}}\text{\textgreek{q}}_{R}\cdot|\partial h_{t^{*}}|^{2}r^{p}|\text{\textgreek{Y}}|^{2}\, dvd\text{\textgreek{sv}}\Big)\, ds.
\end{split}
\label{eq:ExtraTermsBoundedness-1}
\end{equation}

Since $h_{t^{*}}$ is supported only in $\{t_{-}\ge0\}$, due to (\ref{eq:BoundPsiNewMethod}),
we can estimate for any $s\in\mathbb{R}$:

\begin{align}
\int_{\{t=s\}}r^{p}\text{\textgreek{q}}_{R}\cdot h_{t^{*}}^{2} & \big(r^{p}|\partial_{v}\text{\textgreek{Y}}|^{2}+r^{p}|r^{-1}\partial_{\text{\textgreek{sv}}}\text{\textgreek{Y}}|^{2}+\big((d-3)r^{p-2}+\min\{r^{p-2},r^{-\text{\textgreek{d}}}\}\big)|\text{\textgreek{Y}}|^{2}\big)\, dvd\text{\textgreek{sv}}\le\label{eq:IntermediateBound}\\
\le & C(p)\int_{t=0}(1+r^{p})J_{\text{\textgreek{m}}}^{N}(\text{\textgreek{y}})n^{\text{\textgreek{m}}}+C(p)\cdot\int_{\{supp(\partial\text{\textgreek{q}}_{R})\}\cap\{0-\frac{1}{2}\text{\textgreek{q}}_{1}\cdot(r-R_{1})\le t\le s\}}|\partial\text{\textgreek{q}}_{R}|\cdot\big(r^{p}J_{\text{\textgreek{m}}}^{N}(\text{\textgreek{y}})n^{\text{\textgreek{m}}}+r^{p-2}|\text{\textgreek{y}}|^{2}\big).\nonumber 
\end{align}
In view of the inclusion $supp(\partial\text{\textgreek{q}}_{R})\subseteq\{R\le r\le2R\}$,
we can also bound 
\begin{equation}
\int_{\{supp(\partial\text{\textgreek{q}}_{R})\}\cap\{0-\frac{1}{2}\text{\textgreek{q}}_{1}\cdot(r-R_{1})\le t\le s\}}|\partial\text{\textgreek{q}}_{R}|\cdot\big(r^{p}J_{\text{\textgreek{m}}}^{N}(\text{\textgreek{y}})n^{\text{\textgreek{m}}}+r^{p-2}|\text{\textgreek{y}}|^{2}\big)\le C(\text{\textgreek{q}}_{R})\cdot(R+max\{s,0\})\cdot\int_{t=0}(1+r^{p})\cdot J_{\text{\textgreek{m}}}^{N}(\text{\textgreek{y}})n^{\text{\textgreek{m}}},
\end{equation}
 and hence, from (\ref{eq:IntermediateBound}) we obtain: 
\begin{align}
\int_{\{t=s\}}r^{p}\text{\textgreek{q}}_{R}\cdot h_{t^{*}}^{2}\big(r^{p}|\partial_{v}\text{\textgreek{Y}}|^{2}+r^{p}|r^{-1}\partial_{\text{\textgreek{sv}}}\text{\textgreek{Y}}|^{2}+\big((d-3)r^{p-2}+\min\{ & r^{p-2},r^{-\text{\textgreek{d}}}\}\big)|\text{\textgreek{Y}}|^{2}\big)\, dvd\text{\textgreek{sv}}\le\label{eq:BoundOnePsiCutOff}\\
\le & C(R,p,\text{\textgreek{q}}_{R})\cdot(1+|s|)\cdot\int_{t=0}(1+r^{p})J_{\text{\textgreek{m}}}^{N}(\text{\textgreek{y}})n^{\text{\textgreek{m}}}.\nonumber 
\end{align}

Since $supp(\partial h_{t^{*}})\subseteq\{0\le t_{-}\le1\}\cup\{t^{*}-1\le t_{+}\le t^{*}\}$,
the following inclusion 
\[
\{t=s\}\cap supp(\partial h_{t^{*}})\subseteq\{r\le R_{1}+C\cdot|s|\}
\]
holds for some $C\gg1$. We can therefore bound through a Hardy inequality

\begin{align}
\int_{\{t=s\}}r^{p}\text{\textgreek{q}}_{R}\cdot|\partial h_{t^{*}}|^{2}\cdot|\text{\textgreek{Y}}|^{2}\, dvd\text{\textgreek{sv}} & \le C\cdot\int_{\{t=s\}\cap supp(\partial h_{t^{*}})\cap\{r\ge R\}}r^{p+2}\cdot J_{\text{\textgreek{m}}}^{N}(\text{\textgreek{y}})n^{\text{\textgreek{m}}}\le\label{eq:BoundTwoOnPsiCutOff}\\
 & \le C\cdot(1+|s|)^{p+2}\cdot\int_{\{t=s\}\cap supp(\partial h_{t^{*}})\cap\{r\ge R\}}J_{\text{\textgreek{m}}}^{N}(\text{\textgreek{y}})n^{\text{\textgreek{m}}}\le\nonumber \\
 & \le C\cdot(1+|s|)^{p+2}\cdot\int_{\{t=0\}}J_{\text{\textgreek{m}}}^{N}(\text{\textgreek{y}})n^{\text{\textgreek{m}}}.\nonumber 
\end{align}

Returning to (\ref{eq:ExtraTermsBoundedness}) and using (\ref{eq:BoundOnePsiCutOff})
and (\ref{eq:BoundTwoOnPsiCutOff}), we can finally estimate (fixing
$q=6>(p+2)+2$)
\begin{equation}
\begin{split}\int_{\{t=t_{1}\}}\text{\textgreek{q}}_{R}\cdot\big(r^{p}|r^{-1}\partial_{\text{\textgreek{sv}}}\text{\text{\textgreek{Y}}}_{k}|^{2}+ & \big((d-3)r^{p-2}+\min\{r^{p-2},r^{-\text{\textgreek{d}}}\}\big)|\text{\textgreek{Y}}_{k}|^{2}\big)\, dvd\text{\textgreek{sv}}\le\\
\le & C_{q}(\text{\textgreek{w}}_{0},R,p,\text{\textgreek{q}}_{R})\cdot\Big\{\int_{-\infty}^{\infty}\frac{1+|s|^{p+2}}{1+|s|^{q}}\, ds\Big\}\cdot\int_{t=0}(1+r^{p})J_{\text{\textgreek{m}}}^{N}(\text{\textgreek{y}})n^{\text{\textgreek{m}}}\le\\
\le & C(\text{\textgreek{w}}_{0},R,p,\text{\textgreek{q}}_{R})\cdot\int_{t=0}(1+r^{p})J_{\text{\textgreek{m}}}^{N}(\text{\textgreek{y}})n^{\text{\textgreek{m}}}.
\end{split}
\label{eq:FinalBoundForInitialBoundaryTerm}
\end{equation}

\noindent Thus, in view of (\ref{eq:newMethodFinalStatement-1}) and
Lemmas \ref{lem:BoundF} and \ref{lem:BoundednessPsiK}, from (\ref{eq:FinalBoundForInitialBoundaryTerm})
we infer the required estimate (\ref{eq:FinalNewMethodPsiK}): 
\begin{equation}
\begin{split}\int_{\{t=t_{2}\}}\text{\textgreek{q}}_{R}\cdot\Big(r^{p}|\partial_{v}\text{\textgreek{Y}}_{k}|^{2} & +r^{p}|r^{-1}\partial_{\text{\textgreek{sv}}}\text{\textgreek{Y}}_{k}|^{2}+\big((d-3)r^{p-2}+\min\{r^{p-2},r^{-\text{\textgreek{d}}}\}\big)|\text{\textgreek{Y}}_{k}|^{2}\Big)\, dvd\text{\textgreek{sv}}+\\
+\int_{\mathcal{R}(t_{1},t_{2})}\text{\textgreek{q}}_{R}\cdot\Big(pr^{p-1}\big|\partial_{v} & \text{\textgreek{Y}}_{k}\big|^{2}+\big\{\big((2-p)r^{p-1}+r^{p-1-\text{\textgreek{d}}}\big)\big|r^{-1}\partial_{\text{\textgreek{sv}}}\text{\textgreek{Y}}_{k}\big|^{2}+\\
\hphantom{+\int_{\{t_{1}\le t\le t_{2}\}}\text{\textgreek{q}}_{R}\cdot\Big(}+\big((2- & p)(d-3)r^{p-3}+\min\{r^{p-3},r^{-1-\text{\textgreek{d}}}\}\big)\big|\text{\textgreek{Y}}_{k}\big|^{2}\big\}+r^{-1-\text{\textgreek{h}}}|\partial_{u}\text{\textgreek{Y}}_{k}|^{2}\Big)\, dudvd\text{\textgreek{sv}}\le\\
\le & C(p,\text{\textgreek{h}},\text{\textgreek{d}})\cdot\int_{\mathcal{R}(t_{1},t_{2})}|\partial\text{\textgreek{q}}_{R}|\cdot\big(r^{p}|\partial\text{\textgreek{Y}}_{k}|^{2}+r^{p-2}|\text{\textgreek{Y}}_{k}|^{2}\big)\, dudvd\text{\textgreek{sv}}+\\
 & \hphantom{C(}+C(\text{\textgreek{w}}_{0},p,\text{\textgreek{h}},R,\text{\textgreek{q}}_{R},\text{\textgreek{d}})\cdot\int_{\{t=0\}}(1+r^{p})\cdot J_{\text{\textgreek{m}}}^{N}(\text{\textgreek{y}})n^{\text{\textgreek{m}}}.
\end{split}
\label{eq:newMethodFinalStatement-1-1}
\end{equation}

The same estimate for $\text{\textgreek{y}}_{\le\text{\textgreek{w}}_{+}},\text{\textgreek{y}}_{\ge\text{\textgreek{w}}_{+}}$
in place of $\text{\textgreek{y}}_{k}$ follows in exactly the same
way. Hence, the proof of the lemma is complete
\end{proof}
Finally, let us note that we will need to use Lemma \ref{lem:FinalStatementNewMethodPsiK}
in order to estimate the energy of $\text{\textgreek{y}}_{k}$ on
the leaves of a hyperboloidal foliation restricted on $\mathcal{R}(0,t^{*})$.
Namely, we will consider the energy of $\text{\textgreek{y}}_{k}$
on $\tilde{S}_{\text{\textgreek{t}}}\cap\mathcal{R}(0,t^{*})$, where
the hyperboloids $\tilde{S}_{\text{\textgreek{t}}}$ are defined as
follows:
\begin{defn*}
For a fixed $\text{\textgreek{h}}'<1+2a$, we define function $\tilde{u}_{\text{\textgreek{h}}'}:\{r\gg1\}\rightarrow\mathbb{R}$
as 
\begin{equation}
\tilde{u}_{\text{\textgreek{h}}'}\doteq u-\frac{1}{1+r^{\text{\textgreek{h}}'}}.\label{eq:Hyperboloids}
\end{equation}
 We also define for any $\text{\textgreek{t}}\in\mathbb{R}$ the hypersurface
$\tilde{S}_{\text{\textgreek{t}}}\subset\{r\gg1\}$ as 
\begin{equation}
\tilde{S}_{\text{\textgreek{t}}}\doteq\{\tilde{u}_{\text{\textgreek{h}}'}=\text{\textgreek{t}}\}.
\end{equation}

\end{defn*}
\noindent Notice that the level sets of $\tilde{S}_{\text{\textgreek{t}}}$
are indeed spacelike hypersurfaces for $r$ large enough depending
on the chosen \textgreek{h}'. This follows from the computation:
\begin{align*}
g^{\text{\textgreek{m}\textgreek{n}}}\partial_{\text{\textgreek{m}}}\tilde{u}_{\text{\textgreek{h}}'}\cdot\partial_{\text{\textgreek{n}}}\tilde{u}_{\text{\textgreek{h}}'} & =g^{\text{\textgreek{m}\textgreek{n}}}\partial_{\text{\textgreek{m}}}u\cdot\partial_{\text{\textgreek{n}}}u+2g^{\text{\textgreek{m}\textgreek{n}}}\partial_{\text{\textgreek{m}}}u\cdot\partial_{\text{\textgreek{n}}}(\frac{-1}{1+r^{\text{\textgreek{h}}'}})+g^{\text{\textgreek{m}\textgreek{n}}}\partial_{\text{\textgreek{m}}}(\frac{-1}{1+r^{\text{\textgreek{h}}'}})\cdot\partial_{\text{\textgreek{n}}}(\frac{-1}{1+r^{\text{\textgreek{h}}'}})=\\
 & =-2\text{\textgreek{h}}'r^{-1-\text{\textgreek{h}}'}+O_{\text{\textgreek{h}}'}(r^{-2-a}+r^{-2-\text{\textgreek{h}'}})<0.
\end{align*}
 for $r$ large enough in terms of $\text{\textgreek{h}}'$. Moreover,
$|\tilde{u}_{\text{\textgreek{h}}'}-u|\le1$, and hence $\tilde{S}_{\text{\textgreek{t}}}$
are spacelike hypersurfaces terminating at future null infinity, according
to the definition (\ref{eq:HyperboloidTerminatingAtNullInfinity}). 

In view of (\ref{eq:Hyperboloids}), for any smooth function $\text{\textgreek{f}}:\{r\gg1\}\rightarrow\mathbb{C}$
the following estimate is true:

\begin{equation}
J_{\text{\textgreek{m}}}^{N}(\text{\textgreek{f}})\tilde{n}^{\text{\textgreek{m}}}\sim\frac{1}{r^{1+\text{\textgreek{h}}'}}|\partial_{u}\text{\textgreek{f}}|^{2}+|\partial_{v}\text{\textgreek{f}}|^{2}+|\frac{1}{r}\partial_{\text{\textgreek{sv}}}\text{\textgreek{f}}|^{2},\label{eq:EnergyEstimateOnTheHyperboloids}
\end{equation}
where $\tilde{n}$ denotes the future directed unit normal to $\{\tilde{\text{S}}_{\text{\textgreek{t}}}\}_{\text{\textgreek{t}}\in\mathbb{R}}$

We can now infer the following Lemma:
\begin{lem}
\label{lem:IntegratedBoundEnergyHyperboloids}Assume that $R\gg1$
is large enough in terms of the fixed value of $\text{\textgreek{h}}'$.
Then for any $\text{\textgreek{t}}\in\mathbb{R}$ and any $-n\le k\le n$,
the following inequality is true:

\begin{multline}
\int_{0}^{\infty}\big\{\int_{\tilde{S}_{\text{\textgreek{t}}}\cap\mathcal{R}(0,t^{*})\cap\{r\ge R+1\}}J_{\text{\textgreek{m}}}^{N}(\text{\textgreek{y}}_{k})\tilde{n}^{\text{\textgreek{m}}}\big\}\, d\text{\textgreek{t}}\le\\
\le C(R)\int_{\{R\le r\le R+1\}\cap\mathcal{R}(0,t^{*})}\big(J_{\text{\textgreek{m}}}^{N}(\text{\textgreek{y}}_{k})n^{\text{\textgreek{m}}}+|\text{\textgreek{y}}_{k}|^{2}\big)+C(\text{\textgreek{w}}_{0},R)\int_{t=0}(1+r)J_{\text{\textgreek{m}}}^{N}(\text{\textgreek{y}})n^{\text{\textgreek{m}}}.\label{eq:IntegratedDecayForTheEnergyOfTheHyperboloids}
\end{multline}

The same estimate also holds for $\text{\textgreek{y}}_{\le\text{\textgreek{w}}_{+}},\text{\textgreek{y}}_{\ge\text{\textgreek{w}}_{+}}$
in place of $\text{\textgreek{y}}_{k}$.

Moreover, for any $t_{1}\le t_{2}$ the following estimate for $\text{\textgreek{y}}$
holds:

\begin{equation}
\int_{t_{1}}^{t_{2}}\Big\{\int_{\tilde{S}_{\text{\textgreek{t}}}\cap\{r\ge R\}}J_{\text{\textgreek{m}}}^{N}(\text{\textgreek{y}})\tilde{n}^{\text{\textgreek{m}}}\Big\}\, d\text{\textgreek{t}}\le C(R)\cdot\int_{\{R\le r\le R+1\}\cap\mathcal{R}(t_{1},t_{2})}\big(J_{\text{\textgreek{m}}}^{N}(\text{\textgreek{y}})n^{\text{\textgreek{m}}}+|\text{\textgreek{y}}|^{2}\big)+C\cdot\int_{\{t=t_{1}\}}(1+r)J_{\text{\textgreek{m}}}^{N}(\text{\textgreek{y}})\tilde{n}^{\text{\textgreek{m}}}.\label{eq:IntegratedDecayForTheEnergyOfTheHyperboloidsPsi}
\end{equation}
\end{lem}
\begin{proof}
In view of (\ref{eq:EnergyEstimateOnTheHyperboloids}), fixing some
smooth cut-off function $\text{\textgreek{q}}_{R}:\mathcal{D}\rightarrow[0,1]$
such that $\text{\textgreek{q}}_{R}\equiv0$ on $\{r\le R\}$ and
$\text{\textgreek{q}}_{R}\equiv1$ on $\{r\ge R+1\}$, choosing some
$0<\text{\textgreek{h}}<\min\{a,\text{\textgreek{h}}'\}$ and $0<\text{\textgreek{d}}<1$
and setting $p=1$ in Lemma \ref{lem:FinalStatementNewMethodPsiK},
we can readily bound for any $-n\le k\le n$:

\begin{multline}
\int_{0}^{\infty}\Big\{\int_{\tilde{\text{S}}_{\text{\textgreek{t}}}\cap\mathcal{R}(0,t^{*})\cap\{r\ge R+1\}}J_{\text{\textgreek{m}}}^{N}(\text{\textgreek{y}}_{k})\tilde{n}^{\text{\textgreek{m}}}\Big\}\, d\text{\textgreek{t}}\le\\
\le C(R)\int_{\{R\le r\le R+1\}\cap\mathcal{R}(0,t^{*})}\big(J_{\text{\textgreek{m}}}^{N}(\text{\textgreek{y}}_{k})n^{\text{\textgreek{m}}}+|\text{\textgreek{y}}_{k}|^{2}\big)+C(\text{\textgreek{w}}_{0},R)\int_{t=0}(1+r)J_{\text{\textgreek{m}}}^{N}(\text{\textgreek{y}})n^{\text{\textgreek{m}}}.\label{eq:IntegratedDecayForTheEnergyOfTheHyperboloids-1}
\end{multline}
The same is also true for $\text{\textgreek{y}}_{\le\text{\textgreek{w}}_{+}}$,
$\text{\textgreek{y}}_{\ge\text{\textgreek{w}}_{+}}$ in place of
$\text{\textgreek{y}}_{k}$.

The estimate (\ref{eq:IntegratedDecayForTheEnergyOfTheHyperboloidsPsi})
follows by applying Theorem 5.1 of \cite{Moschidisc} for $\text{\textgreek{y}}$,
adapting the proof in the region spanned by $\tilde{S}_{t_{2}}$ and
$\text{\textgreek{S}}_{t_{1}}$ (instead of the region spanned by
$\tilde{S}_{t_{2}}$ and $\tilde{S}_{t_{1}}$, as is the case in the
original theorem).
\end{proof}

\section{\label{sub:LowFrequencies}Integrated local energy decay for $\text{\textgreek{y}}_{0}$}

As we explained in our sketch of the proof of Theorem \ref{thm:Theorem}
in Section \ref{sub:SketchOfProof}, the treatment of the very low
frequency component $\text{\textgreek{y}}_{0}$ of $\text{\textgreek{y}}$
differs substantially from the treatment of the rest of $\text{\textgreek{y}}_{k}$,
$1\le|k|\le n$. In this section, we will prove an integrated local
energy decay statement for $\text{\textgreek{y}}_{0}$, as described
in the proposition that follows. It is in this section that the parameter
$\text{\textgreek{w}}_{0}$ will be fixed. Moreover, in the proof
of Proposition \ref{prop:LowFrequencies} it is specified how small
the value of $\text{\textgreek{e}}=sup\{g(T,T)$ should be in the
statement of Theorem \ref{thm:Theorem}.
\begin{prop}
\label{prop:LowFrequencies}Suppose that $\text{\textgreek{e}}=sup\{g(T,T)\}$
(defined in Assumption \hyperref[Assumption 3]{3} regarding the smallness
of the ergoregion; recall that, in view of our remark in Section \ref{sub:Remark},
we have assumed that $\mathcal{H}^{+}\neq\emptyset$) is small enough
in terms of the geometry of the region $\{r>\frac{r_{0}}{2}\}$ (i.\,e.~the
geometry of a subset of $(\mathcal{D},g)$ not containing the ergoregion).
Then for any $R>0$ , there exists a $\text{\textgreek{d}}=\text{\textgreek{d}}(R)>0$,
such that if $\text{\textgreek{w}}_{0}<\text{\textgreek{d}}$, we
can bound:

\begin{equation}
\int_{\{r\le R\}\cap\mathcal{R}(0,t^{*})}\big(J_{\text{\textgreek{m}}}^{N}(\text{\textgreek{y}}_{0})n^{\text{\textgreek{m}}}+|\text{\textgreek{y}}_{0}|^{2})\le C(R,\text{\textgreek{w}}_{0})\cdot\int_{t=0}J_{\text{\textgreek{m}}}^{N}(\text{\textgreek{y}})n^{\text{\textgreek{m}}}\label{eq:LowFrequencies}
\end{equation}
\end{prop}
\begin{proof}
It suffices to establish (\ref{eq:LowFrequencies}) for $Re(\text{\textgreek{y}}_{0})$
and $Im(\text{\textgreek{y}}_{0})$ in place of $\text{\textgreek{y}}_{0}$,
since then by adding the two inequalities one arrives immediately
at the statement for $\text{\textgreek{y}}_{0}$. Hence, we can assume
without loss of generality that $\text{\textgreek{y}}_{0}$ is real
valued. 

Without loss of generality, we can also assume that $R$ is large
enough in terms of the geometry of $(\mathcal{D},g)$, since we can
always choose an even larger $R$ than the one given in the statement. 

We will need to fix a $C^{1}$ and piecewise $C^{2}$ function $h:[0,+\infty)\rightarrow[0,1]$
with the following properties for some $l>2$:

\begin{enumerate}

\item For $\{r\le1\}$: $h\equiv1$

\item For $\{1<r<2\}$: $h^{\prime\prime}<0$ and $h'<0$

\item For $\{2\le r\le l-c\}$: $h(r)=\frac{3}{2r}$ for some $0<c\ll1$

\item For $\{l-c<r<l+c\}$: $h^{\prime\prime}<0$, $h'<0$

\item For $\{l+c<r<2l\}$: $h(r)=\frac{3}{2l^{3}}(r-2l)^{2}$

\item For $\{r\ge2l\}$: $h\equiv0$

\end{enumerate}

\noindent The existence of such a function is established as follows:
First of all, the existence of a function satisfying the three first
assumptions on $[0,l-c]$ and moreover satisfying $h(r)=\frac{3}{2r}$
on the whole of $\{2\le r\le l\}$ is obvious, since the only obstruction
posed by the conditions $h^{\prime\prime}<0$ on $(1,2)$ and $h'(1)=0,h'(2)=-\frac{3}{8}$
is 
\[
h'(2)\cdot(2-1)<h(2)-h(1).
\]
 And this condition is indeed verified, since $-\frac{3}{8}<-\frac{1}{4}$.
Extending $h$ on the whole of $[0,2l]$ by the expression $h(r)=\frac{3}{2l^{3}}(r-2l)^{2}$
on $(l,2l]$, we infer that 
\[
\lim_{r\rightarrow l^{+}}h(r)=\frac{3}{2l}=\lim_{r\rightarrow l^{-}}h(r)
\]
 and 
\[
\lim_{r\rightarrow l^{+}}h'(r)=-\frac{3}{l^{2}}<-\frac{3}{2l^{2}}=\lim_{r\rightarrow l^{-}}h'(r).
\]
Thus, we conclude that $h$ can be mollified around $r=l$ so that
it is $C^{1}$ and concave around this point, and moreover $h'<0$
in the region of the mollification. Finally, the construction of $h$
is completed by extending $h$ to be identically $0$ on $[2l,+\infty)$
(notice that $h$ remains $C^{1}$ under this extension, since $h(2l)=h'(2l)=0$)

We can now define the function $h_{R}:\mathcal{D}\rightarrow[0,1]$
through the composition $h_{R}\doteq h\circ(\frac{r}{R})$. Using
this function, we can then define the current

\begin{equation}
J_{\text{\textgreek{m}}}=e\cdot T_{\text{\textgreek{m}\textgreek{n}}}(\text{\textgreek{y}}_{0})N^{\text{\textgreek{n}}}+h_{R}\text{\textgreek{y}}_{0}\partial_{\text{\textgreek{m}}}\text{\textgreek{y}}_{0}-\frac{1}{2}\partial_{\text{\textgreek{m}}}h_{R}\cdot\text{\textgreek{y}}_{0}^{2}
\end{equation}
 for a small parameter $e>0$ to be specified later.%
\footnote{Let us remark that in the case where $\mathcal{H}^{+}=\emptyset$,
the proof would follow by using the current $J_{\text{\textgreek{m}}}=h_{R}\text{\textgreek{y}}_{0}\partial_{\text{\textgreek{m}}}\text{\textgreek{y}}_{0}-\frac{1}{2}\partial_{\text{\textgreek{m}}}h_{R}\cdot\text{\textgreek{y}}_{0}^{2}$.%
} Since $\text{\textgreek{y}}_{0}$ satisfies the equation $\square_{g}\text{\textgreek{y}}_{0}=F_{0}$,
we compute (using the abstract index notation):

\begin{equation}
\nabla^{\text{\textgreek{m}}}J_{\text{\textgreek{m}}}=e\cdot K^{N}(\text{\textgreek{y}}_{0})+e\cdot F_{0}\cdot N\text{\ensuremath{\text{\textgreek{y}}_{0}}}+h_{R}\cdot\partial^{\text{\textgreek{m}}}\text{\textgreek{y}}_{0}\cdot\partial_{\text{\textgreek{m}}}\text{\textgreek{y}}_{0}-\frac{1}{2}(\square_{g}h_{R})\text{\textgreek{y}}_{0}^{2}+h_{R}\cdot\text{\textgreek{y}}_{0}\cdot F_{0}.\label{eq:Divergence}
\end{equation}

We will integrate (\ref{eq:Divergence}) over $\mathcal{R}(0,t^{*})$.
Since $\text{\textgreek{y}}_{0}$ is supported in some cylinder of
the form $\{r\lesssim R_{sup}+t^{*}\}$ (see the remark in Section
(\ref{sub:Remark})), we need not worry about boundary terms occuring
near spacelike infinty when applying the divergence theorem in $\mathcal{R}(0,t^{*})$.
After applying the divergence theorem on $\mathcal{R}(0,t^{*})$,
we obtain:

\begin{multline}
\int_{\mathcal{H}^{+}\cap\mathcal{R}(0,t^{*})}J_{\text{\textgreek{m}}}n_{\mathcal{H}^{+}}^{\text{\textgreek{m}}}+\int_{\{t=t^{*}\}}J_{\text{\textgreek{m}}}n^{\text{\textgreek{m}}}+e\int_{\mathcal{R}(0,t^{*})}K^{N}(\text{\textgreek{y}}_{0})+\int_{\mathcal{R}(0,t^{*})}h_{R}\partial^{\text{\textgreek{m}}}\text{\textgreek{y}}_{0}\cdot\partial_{\text{\textgreek{m}}}\text{\textgreek{y}}_{0}-\frac{1}{2}\int_{\mathcal{R}(0,t^{*})}(\square_{g}h_{R})\text{\textgreek{y}}_{0}^{2}=\\
=\int_{\{t=0\}}J_{\text{\textgreek{m}}}n^{\text{\textgreek{m}}}-e\int_{\mathcal{R}(0,t^{*})}F_{0}\cdot N\text{\textgreek{y}}_{0}-\int_{\mathcal{R}(0,t^{*})}h_{R}\cdot\text{\textgreek{y}}_{0}\cdot F_{0},\label{eq:IdentityNearHorizon}
\end{multline}
where $n_{\mathcal{H}^{+}}^{\text{\textgreek{m}}}$ denotes a future
directed $T$-invariant generator of $\mathcal{H}^{+}\backslash(\mathcal{H}^{+}\cap\mathcal{H}^{-})$
which coincides with the horizon Killing field $V$ on $\mathcal{H}\cap J^{+}(\text{\textgreek{S}})$.
For the volume form involved in this integration, see Section \ref{sec:Notational-conventions}
for our conventions regarding integration over null hypersurfaces.

Before proceeding to the extraction of an inequality from identity
(\ref{eq:IdentityNearHorizon}), we should note the following: Having
identified $\mathcal{D}\backslash\mathcal{H}^{-}$ with $\mathbb{R}\times(\text{\textgreek{S}}\cap\mathcal{D})$
through the flow of $T$ (see Section \ref{sub:CoordinateCharts}),
the inverse $g^{-1}$ of the metric $g$ (that is to say, $g^{-1}$
is the induced metric on the cotangent bundle $T^{*}\mathcal{M}$)
splits as

\begin{equation}
g^{-1}=a\cdot T\otimes T+T\otimes b+b\otimes T+\mathcal{C},\label{eq:InverseMetric}
\end{equation}
 where, due to our assumption \hyperref[Assumption 3]{3} on the smallness
of the ergoregion, the following conditions are satisfied on $\{t\ge0\}\subset\mathcal{D}$:

\begin{itemize}

\item $a$ is a real bounded function.

\item $b$ is a vector field tangential to the hypersurfaces $\text{\textgreek{S}}_{\text{\textgreek{t}}}$
and uniformly bounded with respect to the induced Riemannian metric
$g_{\text{\textgreek{S}}}$ on these hypersurfaces.

\item $\mathcal{C}$ is a symmetric $(0,2)$-tensor on the hypersurfaces
$\text{\textgreek{S}}_{\text{\textgreek{t}}}$, uniformly bounded
when measured with $g_{\text{\textgreek{S}}}$, satisfying the lower
bound $\mathcal{C}\ge-C\text{\textgreek{e}}\cdot g_{\text{\textgreek{S}}}^{-1}$
in the region $\{r\le\frac{3r_{0}}{4}\}$and $\mathcal{C}\ge c\cdot g_{\text{\textgreek{S}}}^{-1}$
in the region $\{r\ge\frac{3r_{0}}{4}\}$ for some $C,c>0$ (here
we have used the notation that $A\ge B$ for two symmetric $(0,2)$-tensors
if $A(\text{\textgreek{w}},\text{\textgreek{w}})\ge B(\text{\textgreek{w}},\text{\textgreek{w}})$
for all 1-forms $\text{\textgreek{w}}$).

\end{itemize}

Since $K^{N}(\cdot)\gtrsim J_{\text{\textgreek{m}}}^{N}(\cdot)n^{\text{\textgreek{m}}}$
on $\{r\le r_{0}\}$ (in view of Assumption \hyperref[Assumption 2]{2}),
we can bound from below on $\text{\textgreek{S}}_{t}\cap\{r\le r_{0}\}$
for any $t\ge0$ and any smooth function $v$ on $\{t\ge0\}$ provided
$e$ is chosen (independently of $t$) so that $e\gtrsim\text{\textgreek{e}}$:
\begin{equation}
e\cdot K^{N}(v)+\mathcal{C}(\partial v,\partial v)\ge0.\label{eq:FirstBoundEpsilon}
\end{equation}
Moreover, since $K^{N}(\cdot)\lesssim J_{\text{\textgreek{m}}}^{N}(\cdot)n^{\text{\textgreek{m}}}$
on $\{r_{0}\le r\le2r_{0}\}$ and $K^{N}$ vanishes anywhere else,
we can bound on $\text{\textgreek{S}}_{t}\cap\{r\ge r_{0}\}$: 
\begin{equation}
e\cdot K^{N}(v)+\mathcal{C}(\partial v,\partial v)+(Tv)^{2}\ge0\label{eq:SecondBoundEpsilon}
\end{equation}
 provided $e$ is chosen sufficiently small depending only on the
geometric aspects of our spacetime for $\{r\ge r_{0}\}$. This assumption
is compatible with the assumption $\text{\textgreek{e}}\lesssim e$
that we imposed for (\ref{eq:FirstBoundEpsilon}) to hold, provided
that $\text{\textgreek{e}}$ is small enough in terms of the geometry
of the spacetime region $\{r\ge r_{0}\}$. 

All in all, if \textgreek{e} is small enough in terms of the geometry
of $\{r\ge r_{0}\}$, we infer that we can choose a suitable $e$
such that for any smooth function $v$ on $\{t\ge0\}$ we can bound
(from (\ref{eq:FirstBoundEpsilon}) and (\ref{eq:SecondBoundEpsilon})):
\[
e\cdot K^{N}(v)+\mathcal{C}(\partial v,\partial v)+(Tv)^{2}\ge c(e)\cdot J_{\text{\textgreek{m}}}^{N}(v)n^{\text{\textgreek{m}}}
\]
 everywhere on $\{t\ge0\}$. This inequality implies, due to (\ref{eq:InverseMetric}),
that for some constant $C>0$ depending on the bounds for $a,b$ in
(\ref{eq:InverseMetric}) we can estimate from below (using the abstract
index notation): 
\begin{equation}
e\cdot K^{N}(v)+g^{\text{\textgreek{m}\textgreek{n}}}\partial_{\text{\textgreek{m}}}v\cdot\partial_{\text{\textgreek{n}}}v+C\cdot(Tv)^{2}\ge c(e)\cdot J_{\text{\textgreek{m}}}^{N}(v)n^{\text{\textgreek{m}}}\label{eq:First bound near horizon}
\end{equation}
everywhere on $\{t\ge0\}$. Therefore, from now on we can assume that
that the parameter $e$ has been fixed.

Since $h_{R}\equiv1$ near $\mathcal{H}^{+}$ and $h_{R}\ge0$ everywhere
(and $K^{N}$ vanishes for $r\ge2r_{0}$), we can bound due to (\ref{eq:First bound near horizon})
(denoting with $\nabla_{\text{\textgreek{S}}}\text{\textgreek{y}}_{0}$
the gradient of $\text{\textgreek{y}}_{0}$ on the $(\text{\textgreek{S}}_{t},g_{\text{\textgreek{S}}})$
hypersurfaces):

\begin{equation}
eK^{N}(\text{\textgreek{y}}_{0})+h_{R}\partial^{\text{\textgreek{m}}}\text{\textgreek{y}}_{0}\cdot\partial_{\text{\textgreek{m}}}\text{\textgreek{y}}_{0}\ge c\cdot h_{R}|\nabla_{\text{\textgreek{S}}}\text{\textgreek{y}}_{0}|_{g_{\text{\textgreek{S}}}}^{2}-C\cdot h_{R}|T\text{\textgreek{y}}_{0}|^{2}.\label{eq:AnotherIntermediateBound}
\end{equation}
Returning to the identity (\ref{eq:IdentityNearHorizon}), in view
of (\ref{eq:AnotherIntermediateBound}) we can bound:

\begin{multline}
\int_{H^{+}\cap\mathcal{R}(0,t^{*})}J_{\text{\textgreek{m}}}n_{H^{+}}^{\text{\textgreek{m}}}+\int_{\mathcal{R}(0,t^{*})}h_{R}\cdot|\nabla_{\text{\textgreek{S}}}\text{\textgreek{y}}_{0}|_{g_{\text{\textgreek{S}}}}^{2}-C\cdot\int_{\mathcal{R}(0,t^{*})}h_{R}\cdot|T\text{\textgreek{y}}_{0}|^{2}-\frac{1}{2}\int_{\mathcal{R}(0,t^{*})}(\square_{g}h_{R})\text{\textgreek{y}}_{0}^{2}\le\\
\le-e\int_{\mathcal{R}(0,t^{*})}F_{0}\cdot N\text{\textgreek{y}}_{0}-\int_{\mathcal{R}(0,t^{*})}h_{R}\cdot\text{\textgreek{y}}_{0}\cdot F_{0}+\int_{t=0}J_{\text{\textgreek{m}}}n^{\text{\textgreek{m}}}-\int_{t=t^{*}}J_{\text{\textgreek{m}}}n^{\text{\textgreek{m}}}.\label{eq:IdentityNearHorizon2}
\end{multline}

Along $\mathcal{H}^{+}$, the current $J_{\text{\textgreek{m}}}$
takes the form 
\begin{equation}
J_{\text{\textgreek{m}}}=e\cdot T_{\text{\textgreek{m}\textgreek{n}}}(\text{\ensuremath{\text{\textgreek{y}}_{0}}})N^{\text{\textgreek{n}}}+\text{\textgreek{y}}_{0}\cdot\partial_{\text{\textgreek{m}}}\text{\textgreek{y}}_{0}=e\cdot J_{\text{\textgreek{m}}}^{N}(\text{\textgreek{y}}_{0})+\frac{1}{2}\partial_{\text{\textgreek{m}}}(\text{\textgreek{y}}_{0}^{2}).
\end{equation}
Therefore, we compute

\begin{align}
-\int_{\mathcal{H}^{+}\cap\mathcal{R}(0,t^{*})}J_{\text{\textgreek{m}}}n_{\mathcal{H}^{+}}^{\text{\textgreek{m}}} & =-e\int_{\mathcal{H}^{+}\cap\{0\le t\le t^{*}\}}J_{\text{\textgreek{m}}}^{N}(\text{\textgreek{y}}_{0})n_{\mathcal{H}^{+}}^{\text{\textgreek{m}}}-\frac{1}{2}\int_{\mathcal{H}^{+}\cap\{0\le t\le t^{*}\}}n_{\mathcal{H}^{+}}(\text{\textgreek{y}}_{0}^{2})\le\label{eq:FirstHorizonEstimate}\\
 & \le-\frac{1}{2}\int_{\mathcal{H}^{+}\cap\{0\le t\le t^{*}\}}n_{\mathcal{H}^{+}}(\text{\textgreek{y}}_{0}^{2})=\nonumber \\
 & =\frac{1}{2}\big(\int_{\mathcal{H}^{+}\cap\{t=0\}}\text{\textgreek{y}}_{0}^{2}-\int_{\mathcal{H}^{+}\cap\{t=t^{*}\}}\text{\textgreek{y}}_{0}^{2}\big),\nonumber 
\end{align}
the last equality following from the definition of the volume element
on $\mathcal{H}^{+}$ with respect to $n_{\mathcal{H}^{+}}$, combined
with the fact that $n_{\mathcal{H}^{+}}$ was assumed to coincide
with the horizon Killing field $V$ on $\mathcal{H}^{+}\cap J^{+}(\text{\textgreek{S}})$.
Due to a trace theorem, a Hardy inequality and Lemma \ref{lem:BoundednessPsiK},
the following estimate holds for each $\text{\textgreek{t}}\in[0,t^{*}]$:
\begin{equation}
\int_{\{t=\text{\textgreek{t}}\}\cap\mathcal{H}^{+}}\text{\textgreek{y}}_{0}^{2}\le C\cdot\int_{\{t=\text{\textgreek{t}}\}\cap\{r\le1\}}(\text{\textgreek{y}}_{0}^{2}+|\nabla_{\text{\textgreek{S}}}\text{\textgreek{y}}_{0}|_{g_{\text{\textgreek{S}}}}^{2})\le C\cdot\int_{t=\text{\textgreek{t}}}J_{\text{\textgreek{m}}}^{N}(\text{\textgreek{y}}_{0})n^{\text{\textgreek{m}}}\le C(\text{\textgreek{w}}_{0})\cdot\int_{t=0}J_{\text{\textgreek{m}}}^{N}(\text{\textgreek{y}})n^{\text{\textgreek{m}}}.
\end{equation}
 Hence, we can bound in view of (\ref{eq:FirstHorizonEstimate}):

\begin{equation}
-\int_{\mathcal{H}^{+}\cap\mathcal{R}(0,t^{*})}J_{\text{\textgreek{m}}}n_{\mathcal{H}^{+}}^{\text{\textgreek{m}}}\le C(\text{\textgreek{w}}_{0})\cdot\int_{t=0}J_{\text{\textgreek{m}}}^{N}(\text{\textgreek{y}})n^{\text{\textgreek{m}}}.\label{eq:HorizonBoundaryTerm}
\end{equation}

We will now estimate the boundary terms at $\{t=0,t^{*}\}$ of (\ref{eq:IdentityNearHorizon2}).
Due to Lemma \ref{lem:BoundednessPsiK} (using also a Hardy inequalty
for the 0-th order terms, since $h_{R}$ and $\partial h_{R}$ decay
faster than $C(R)\cdot r^{-2}$), we can bound: 

\begin{equation}
\int_{\{t=const\}}J_{\text{\textgreek{m}}}(\text{\ensuremath{\text{\textgreek{y}}_{0}}})n^{\text{\textgreek{m}}}\le C(\text{\textgreek{w}}_{0},R)\int_{\{t=0\}}J_{\text{\textgreek{m}}}^{N}(\text{\textgreek{y}})n^{\text{\textgreek{m}}}.\label{eq:BoundaryTerms}
\end{equation}

By substituting estimates (\ref{eq:HorizonBoundaryTerm}) and (\ref{eq:BoundaryTerms})
into (\ref{eq:IdentityNearHorizon2}), we infer:

\begin{multline}
\int_{\mathcal{R}(0,t^{*})}h_{R}|\nabla_{\text{\textgreek{S}}}\text{\textgreek{y}}_{0}|_{g_{\text{\textgreek{S}}}}^{2}-\frac{1}{2}\int_{\mathcal{R}(0,t^{*})}(\square_{g}h_{R})\cdot\text{\textgreek{y}}_{0}^{2}\lesssim\int_{\mathcal{R}(0,t^{*})}h_{R}\cdot|T\text{\textgreek{y}}_{0}|^{2}+\\
-e\int_{\mathcal{R}(0,t^{*})}F_{0}\cdot N\text{\textgreek{y}}_{0}-\int_{\mathcal{R}(0,t^{*})}h_{R}\cdot\text{\textgreek{y}}_{0}\cdot F_{0}+C(\text{\textgreek{w}}_{0},R)\int_{t=0}J_{\text{\textgreek{m}}}^{N}(\text{\textgreek{y}})n^{\text{\textgreek{m}}}.\label{eq:Almost reached Igor's argument}
\end{multline}

We will now proceed to estimate the term $\int_{\mathcal{R}(0,t^{*})}(\square_{g}h_{R})\cdot\text{\textgreek{y}}_{0}^{2}$
in (\ref{eq:Almost reached Igor's argument}). Due to the properties
of $h$, we have $\square_{g}h_{R}\equiv0$ for $r\le R$. We can
assume without loss of generality that $R\ge R_{1}$. Then, in view
of the expression (\ref{eq:metric}) of $g$ in the $(t,r,\text{\textgreek{sv}})$
coordinate chart for $r\ge R_{1}$, we compute: 
\begin{equation}
\square_{g}h_{R}=\frac{1}{\sqrt{-det(g)}}\partial_{r}\big(g^{rr}\cdot\sqrt{-det(g)}\partial_{r}h_{R}\big)=(1+O(r^{-1}))\Big(\partial_{r}^{2}h_{R}+\frac{d-1}{r}\partial_{r}h_{R}\Big)+O(r^{-1-a})\cdot\partial_{r}h_{R}.\label{eq:SquareOfH}
\end{equation}
 Thus, we can estimate in $\{R<r<2R\}$ (in view again of the properties
of the constructed function $h$): 
\[
\square_{g}h_{R}<-c\cdot R^{-2},
\]
 provided $R\gg1$. For $2R<r<lR-c$ we calculate

\begin{equation}
\square_{g}h_{R}=\frac{3}{2}\Big\{(1+O(r^{-1}))\Big(2Rr^{-3}-\frac{d-1}{r}Rr^{-2}\Big)+O(r^{-1-a})\cdot R\cdot r^{-2}\Big\}\le C\cdot\frac{R}{r^{3+a}}.\label{eq:SquareOfHAway}
\end{equation}
Similarly, for $\{lR-c<r<lR+c\}$ we can readily bound 
\begin{equation}
\square_{g}h_{R}\le0,
\end{equation}
while for $\{lR+c<r<2lR\}$ we can estimate (provided again that $R\gg1$)
\begin{equation}
\square_{g}h_{R}\le\frac{C}{l^{3}R^{2}}.
\end{equation}
Since $r\sim l\cdot R$ in this interval, we can estimate for any
given (fixed) $0<\text{\textgreek{b}}<1$: 
\begin{equation}
\square_{g}h_{R}\le C\cdot\big(\frac{(lR)^{\text{\textgreek{b}}}}{l}\big)r^{-2-\text{\textgreek{b}}}.\label{eq:BoundBoxFarAwayRegion}
\end{equation}

It is at this point that we will determine the value of $m$ in terms
of $R$: We set $l=R^{\frac{\text{\textgreek{b}}}{\text{1-\textgreek{b}}}}$,
and with this choice we have $\frac{(lR)^{\text{\textgreek{b}}}}{l}=1$.
With this choice, therefore, for $lR+c<r<2lR$ we can estimate from
(\ref{eq:BoundBoxFarAwayRegion}): 
\begin{equation}
\square_{g}h_{R}\le C\cdot r^{-2-\text{\textgreek{b}}}.
\end{equation}
Of course, for $r>2l$ we have 
\begin{equation}
\square_{g}h_{R}\equiv0.
\end{equation}

Returning to (\ref{eq:Almost reached Igor's argument}) and using
the above estimates for $\square h_{R}$, we can bound 
\begin{equation}
\begin{split}\int_{\mathcal{R}(0,t^{*})}h_{R}|\nabla_{\text{\textgreek{S}}}\text{\ensuremath{\text{\textgreek{y}}_{0}}}|_{g_{\text{\textgreek{S}}}}^{2}+ & c\cdot\int_{\mathcal{R}(0,t^{*})\cap\{R\le r\le2R\}}\frac{\text{\textgreek{y}}_{0}^{2}}{R^{2}}\le\\
\le & C\cdot\int_{\mathcal{R}(0,t^{*})}h_{R}\cdot|T\text{\textgreek{y}}_{0}|^{2}+C\cdot\int_{\mathcal{R}(0,t^{*})\cap\{2R\le r\le lR\}}\frac{R}{r^{3+a}}\text{\textgreek{y}}_{0}^{2}+C\int_{\mathcal{R}(0,t^{*})\cap\{lR\le r\le2lR\}}\frac{\text{\textgreek{y}}_{0}^{2}}{r^{2+\text{\textgreek{b}}}}+\\
 & +C\Big|\int_{\mathcal{R}(0,t^{*})}F_{0}\cdot N\text{\textgreek{y}}_{0}\Big|+C\cdot\Big|\int_{\mathcal{R}(0,t^{*})}h_{R}\cdot\text{\textgreek{y}}_{0}\cdot F_{0}\Big|+C(\text{\textgreek{w}}_{0},R)\int_{t=0}J_{\text{\textgreek{m}}}^{N}(\text{\textgreek{y}})n^{\text{\textgreek{m}}},
\end{split}
\label{eq:Almost reached Igor's argument-1-1-1}
\end{equation}
 or, after adding a multiple of $\int_{\mathcal{R}(0,t^{*})}h_{R}\cdot|T\text{\textgreek{y}}_{0}|^{2}$
on both sides: 
\begin{equation}
\begin{split}\int_{\mathcal{R}(0,t^{*})}h_{R}J_{\text{\textgreek{m}}}^{N}(\text{\textgreek{y}}_{0})n^{\text{\textgreek{m}}}+ & c\cdot\int_{\mathcal{R}(0,t^{*})\cap\{R\le r\le2R\}}\frac{\text{\textgreek{y}}_{0}^{2}}{R^{2}}\le\\
\le & C\cdot\int_{\mathcal{R}(0,t^{*})}h_{R}\cdot|T\text{\textgreek{y}}_{0}|^{2}+C\cdot\int_{\mathcal{R}(0,t^{*})\cap\{2R\le r\le lR\}}\frac{R}{r^{3+a}}\text{\textgreek{y}}_{0}^{2}+C\int_{\mathcal{R}(0,t^{*})\cap\{lR\le r\le2lR\}}\frac{\text{\textgreek{y}}_{0}^{2}}{r^{2+\text{\textgreek{b}}}}+\\
 & +C\Big|\int_{\mathcal{R}(0,t^{*})}F_{0}\cdot N\text{\textgreek{y}}_{0}\Big|+C\Big|\int_{\mathcal{R}(0,t^{*})}h_{R}\cdot\text{\textgreek{y}}_{0}\cdot F_{0}\Big|+C(\text{\textgreek{w}}_{0},R)\int_{t=0}J_{\text{\textgreek{m}}}^{N}(\text{\textgreek{y}})n^{\text{\textgreek{m}}}.
\end{split}
\label{eq:Almost reached Igor's argument-1}
\end{equation}

We will dispense with the third term of the right hand side of (\ref{eq:Almost reached Igor's argument-1})
using Lemma \ref{lem:ImprovedMorawetz}. According to \ref{lem:ImprovedMorawetz}
for $\text{\textgreek{h}}=\min\{\frac{1}{2}\text{\textgreek{b}},a\}$
and $R_{c}=2lR=2R^{\frac{1}{1-\text{\textgreek{b}}}}$, we can bound
provided $R\gg1$:

\begin{equation}
\int_{\mathcal{R}(0,t^{*})\cap\{lR\le r\le2lR\}}\frac{\text{\textgreek{y}}_{0}^{2}}{r^{2+\frac{1}{2}\text{\textgreek{b}}}}\le C\int_{\mathcal{R}(0,t^{*})\cap\{R\le r\le2R\}}\Big(J_{\text{\textgreek{m}}}^{N}(\text{\textgreek{y}}_{0})n^{\text{\textgreek{m}}}+R^{-2}|\text{\textgreek{y}}_{0}|^{2}\Big)+C(\text{\textgreek{w}}_{0},R)\int_{t=0}J_{\text{\textgreek{m}}}^{N}(\text{\textgreek{y}})n^{\text{\textgreek{m}}}.\label{eq:BoundByNewMethod-1}
\end{equation}
Substituting (\ref{eq:BoundByNewMethod-1}) in (\ref{eq:Almost reached Igor's argument-1})
we have:
\begin{equation}
\begin{split}\int_{\mathcal{R}(0,t^{*})}h_{R} & J_{\text{\textgreek{m}}}^{N}(\text{\textgreek{y}}_{0})n^{\text{\textgreek{m}}}+c\cdot\int_{\mathcal{R}(0,t^{*})\cap\{R\le r\le2R\}}\frac{\text{\textgreek{y}}_{0}^{2}}{R^{2}}\le\\
\le & C\cdot\int_{\mathcal{R}(0,t^{*})}h_{R}\cdot|T\text{\textgreek{y}}_{0}|^{2}+C\cdot\int_{\mathcal{R}(0,t^{*})\cap\{2R\le r\le lR\}}\frac{R}{r^{3+a}}\text{\textgreek{y}}_{0}^{2}+CR^{-\frac{\text{\textgreek{b}}}{2}}\int_{\{R\le r\le2R\}\cap\mathcal{R}(0,t^{*})}\Big(J_{\text{\textgreek{m}}}^{N}(\text{\textgreek{y}}_{0})n^{\text{\textgreek{m}}}+R^{-2}\text{\textgreek{y}}_{0}^{2}\Big)+\\
 & +C\Big|\int_{\mathcal{R}(0,t^{*})}F_{0}\cdot N\text{\textgreek{y}}_{0}\Big|+C\cdot\Big|\int_{\mathcal{R}(0,t^{*})}h_{R}\cdot\text{\textgreek{y}}_{0}\cdot F_{0}\Big|+C(\text{\textgreek{w}}_{0},R)\int_{t=0}J_{\text{\textgreek{m}}}^{N}(\text{\textgreek{y}})n^{\text{\textgreek{m}}}.
\end{split}
\label{eq:BeforeTheTterms-3-1}
\end{equation}
Thus, since $h_{R}\ge\frac{1}{2}$ on $\{r\le2R\}$, if $R$ is large
enough, the third term of the right hand side of (\ref{eq:BeforeTheTterms})
can be absorbed by the left hand side, yielding
\begin{equation}
\begin{split}\int_{\mathcal{R}(0,t^{*})}h_{R}J_{\text{\textgreek{m}}}^{N}(\text{\textgreek{y}}_{0})n^{\text{\textgreek{m}}}+ & c\cdot\int_{\mathcal{R}(0,t^{*})\cap\{R\le r\le2R\}}\frac{\text{\textgreek{y}}_{0}^{2}}{R^{2}}\le\\
\le & C\cdot\int_{\mathcal{R}(0,t^{*})}h_{R}\cdot|T\text{\textgreek{y}}_{0}|^{2}+C\cdot\int_{\mathcal{R}(0,t^{*})\cap\{2R\le r\le lR\}}\frac{R}{r^{3+a}}\text{\textgreek{y}}_{0}^{2}+\\
 & +C\Big|\int_{\mathcal{R}(0,t^{*})}F_{0}\cdot N\text{\textgreek{y}}_{0}\Big|+C\cdot\Big|\int_{\mathcal{R}(0,t^{*})}h_{R}\cdot\text{\textgreek{y}}_{0}\cdot F_{0}\Big|+C(\text{\textgreek{w}}_{0},R)\int_{t=0}J_{\text{\textgreek{m}}}^{N}(\text{\textgreek{y}})n^{\text{\textgreek{m}}}.
\end{split}
\label{eq:BeforeTheSecondTermOfTheRHS}
\end{equation}

We will now absorb the second term of the righ hand side of (\ref{eq:BeforeTheSecondTermOfTheRHS})
by the left hand side through the use of Lemma \ref{lem:MorawetzRefinedboundary}
if $R\gg1$. According to Lemma \ref{lem:MorawetzRefinedboundary}
for any $0<\text{\textgreek{h}}<\text{\textgreek{a}}$, (if $R\gg1$):

\begin{equation}
\int_{\mathcal{R}(0,t^{*})\cap\{r\ge2R\}}\frac{1}{r^{3+\text{\textgreek{h}}}}\text{\textgreek{y}}_{0}^{2}\le C(\text{\textgreek{h}})\cdot\int_{\mathcal{R}(0,t^{*})\cap\{R\le r\le2R\}}\big\{\frac{1}{R}J_{\text{\textgreek{m}}}^{N}(\text{\textgreek{y}}_{0})n^{\text{\textgreek{m}}}+\frac{1}{R^{3}}|\text{\textgreek{y}}_{0}|^{2}\big\}+C(\text{\textgreek{h}},\text{\textgreek{w}}_{0})\int_{t=0}J_{\text{\textgreek{m}}}^{N}(\text{\textgreek{y}})n^{\text{\textgreek{m}}},
\end{equation}
 and hence 
\begin{align}
\int_{\mathcal{R}(0,t^{*})\cap\{2R\le r\le lR\}}\frac{R}{r^{3+a}}\text{\textgreek{y}}_{0}^{2} & \le C\cdot\int_{\mathcal{R}(0,t^{*})\cap\{r\ge2R\}}\frac{R^{1-a+\text{\textgreek{h}}}}{r^{3+\text{\textgreek{h}}}}\text{\textgreek{y}}_{0}^{2}\label{eq:FromSimpleMorawetz}\\
 & \le C\cdot\int_{\mathcal{R}(0,t^{*})\cap\{R\le r\le2R\}}\big\{ R^{-a+\text{\textgreek{h}}}J_{\text{\textgreek{m}}}^{N}(\text{\textgreek{y}}_{0})n^{\text{\textgreek{m}}}+R^{-2-a+\text{\textgreek{h}}}|\text{\textgreek{y}}_{0}|^{2}\big\}+C(\text{\textgreek{w}}_{0},R)\int_{t=0}J_{\text{\textgreek{m}}}^{N}(\text{\textgreek{y}})n^{\text{\textgreek{m}}}.\nonumber 
\end{align}
Thus, from (\ref{eq:BeforeTheSecondTermOfTheRHS}) and (\ref{eq:FromSimpleMorawetz})
for $\text{\textgreek{h}}=\frac{a}{2}$ we deduce that: 
\begin{equation}
\begin{split}\int_{\mathcal{R}(0,t^{*})}h_{R}J_{\text{\textgreek{m}}}^{N}(\text{\textgreek{y}}_{0})n^{\text{\textgreek{m}}}+ & c\cdot\int_{\mathcal{R}(0,t^{*})\cap\{R\le r\le2R\}}R^{-2}\text{\textgreek{y}}_{0}^{2}\le\\
\le & C\cdot\int_{\mathcal{R}(0,t^{*})}h_{R}\cdot|T\text{\textgreek{y}}_{0}|^{2}+C\cdot\int_{\mathcal{R}(0,t^{*})\cap\{R\le r\le2R\}}\big(R^{-\frac{a}{2}}J_{\text{\textgreek{m}}}^{N}(\text{\textgreek{y}}_{0})n^{\text{\textgreek{m}}}+R^{-2-\frac{a}{2}}\text{\textgreek{y}}_{0}^{2}\big)+\\
 & +C\Big|\int_{\mathcal{R}(0,t^{*})}F_{0}\cdot N\text{\textgreek{y}}_{0}\Big|+C\cdot\Big|\int_{\mathcal{R}(0,t^{*})}h_{R}\cdot\text{\textgreek{y}}_{0}\cdot F_{0}\Big|+C(\text{\textgreek{w}}_{0},R)\int_{t=0}J_{\text{\textgreek{m}}}^{N}(\text{\textgreek{y}})n^{\text{\textgreek{m}}},
\end{split}
\label{eq:Almost reached Igor's argument-1-1}
\end{equation}
which, after absorbing the second term of the right hand side into
the left hand side, yields:

\begin{multline}
\int_{\mathcal{R}(0,t^{*})}h_{R}\cdot J_{\text{\textgreek{m}}}^{N}(\text{\textgreek{y}}_{0})n^{\text{\textgreek{m}}}+c\cdot\int_{\mathcal{R}(0,t^{*})\cap\{R\le r\le2R\}}R^{-2}\text{\textgreek{y}}_{0}^{2}\le C\cdot\int_{\mathcal{R}(0,t^{*})}h_{R}\cdot|T\text{\textgreek{y}}_{0}|^{2}+\\
+C\Big|\int_{\mathcal{R}(0,t^{*})}F_{0}\cdot N\text{\textgreek{y}}_{0}\Big|+C\cdot\Big|\int_{\mathcal{R}(0,t^{*})}h_{R}\cdot\text{\textgreek{y}}_{0}\cdot F_{0}\Big|+C(\text{\textgreek{w}}_{0},R)\int_{t=0}J_{\text{\textgreek{m}}}^{N}(\text{\textgreek{y}})n^{\text{\textgreek{m}}}.\label{eq:BeforeTheTterms}
\end{multline}

Fixing a smooth $\text{\textgreek{q}}:[0,+\infty)\rightarrow[0,1]$
such that $\text{\textgreek{q}}\equiv1$ on $r\le1$ and $\text{\textgreek{q}}\equiv0$
on $r\ge2$, by applying a Poincare inequality for the compactly supported
function $(\text{\textgreek{q}}\circ(\frac{r}{R}))\cdot\text{\textgreek{y}}_{0}$
we obtain
\begin{align}
\int_{\mathcal{R}(0,t^{*})\cap\{r\le R\}}\text{\textgreek{y}}_{0}^{2} & \le C\cdot R^{2}\int_{\mathcal{R}(0,t^{*})\cap\{r\le2R\}}|\nabla_{\text{\textgreek{S}}}\big\{\text{\textgreek{y}}_{0}\cdot(\text{\textgreek{q}}\circ(\frac{r}{R}))\big\}|_{g_{\text{\textgreek{S}}}}^{2}\le\\
 & \le C(R)\Big(\int_{\mathcal{R}(0,t^{*})\cap\{r\le2R\}}|\nabla_{\text{\textgreek{S}}}\text{\textgreek{y}}_{0}|_{g_{\text{\textgreek{S}}}}^{2}+\int_{\mathcal{R}(0,t^{*})\cap\{R\le r\le2R\}}|\text{\textgreek{y}}_{0}|^{2}\Big).\nonumber 
\end{align}
 Therefore, (\ref{eq:BeforeTheTterms}) can be upgraded to:

\begin{multline}
\int_{\mathcal{R}(0,t^{*})}h_{R}\cdot J_{\text{\textgreek{m}}}^{N}(\text{\textgreek{y}}_{0})n^{\text{\textgreek{m}}}+\int_{\mathcal{R}(0,t^{*})\cap\{r\le2R\}}\text{\textgreek{y}}_{0}^{2}\le C(R)\cdot\int_{\mathcal{R}(0,t^{*})}h_{R}\cdot|T\text{\textgreek{y}}_{0}|^{2}+\\
+C(R)\Big|\int_{\mathcal{R}(0,t^{*})}F_{0}\cdot N\text{\textgreek{y}}_{0}\Big|+C(R)\cdot\Big|\int_{\mathcal{R}(0,t^{*})}h_{R}\cdot\text{\textgreek{y}}_{0}\cdot F_{0}\Big|+C(\text{\textgreek{w}}_{0},R)\int_{t=0}J_{\text{\textgreek{m}}}^{N}(\text{\textgreek{y}})n^{\text{\textgreek{m}}}.\label{eq:OneMoreTerm}
\end{multline}

We will now absorb the first term of the right hand side of (\ref{eq:OneMoreTerm})
by the left hand side under the condition that $\text{\textgreek{w}}_{0}$
is small in terms of $R$. Recall that, due to Lemma \ref{lem:DtToOmegaInequalities2},
we can bound%
\footnote{note that $supp(h_{R})\subseteq\{r\le2lR\}=\{r\le2R^{\frac{1}{1-\text{\textgreek{b}}}}\}$%
}: 
\begin{equation}
\int_{\mathcal{R}(0,t^{*})}h_{R}\cdot|T\text{\textgreek{y}}_{0}|^{2}\le C\cdot\text{\textgreek{w}}_{0}^{2}\int_{\mathcal{R}(0,t^{*})\cap\{r\le2lR\}}\text{\textgreek{y}}_{0}^{2}+C(\text{\textgreek{w}}_{0},R)\int_{t=0}J_{\text{\textgreek{m}}}^{N}(\text{\textgreek{y}})n^{\text{\textgreek{m}}}.\label{eq:BeforeAbsorbingThetTerm}
\end{equation}
 Since, due to Lemma \ref{lem:MorawetzRefinedboundary}, we can bound
\begin{equation}
\int_{\mathcal{R}(0,t^{*})\cap\{2R\le r\le2lR\}}\text{\textgreek{y}}_{0}^{2}\le C(R)\cdot\int_{\mathcal{R}(0,t^{*})\cap\{R\le r\le2R\}}(J_{\text{\textgreek{m}}}^{N}(\text{\textgreek{y}}_{0})n^{\text{\textgreek{m}}}+\text{\textgreek{y}}_{0}^{2})+C(\text{\textgreek{w}}_{0},R)\int_{t=0}J_{\text{\textgreek{m}}}^{N}(\text{\textgreek{y}})n^{\text{\textgreek{m}}},
\end{equation}
the inequality (\ref{eq:BeforeAbsorbingThetTerm}) yields

\begin{equation}
\int_{\mathcal{R}(0,t^{*})}h_{R}\cdot|\partial_{t}\text{\textgreek{y}}_{0}|^{2}\le C(R)\cdot\text{\textgreek{w}}_{0}^{2}\int_{\mathcal{R}(0,t^{*})\cap\{r\le2R\}}(J_{\text{\textgreek{m}}}^{N}(\text{\textgreek{y}}_{0})n^{\text{\textgreek{m}}}+\text{\textgreek{y}}_{0}^{2})+C(\text{\textgreek{w}}_{0},R)\int_{t=0}J_{\text{\textgreek{m}}}^{N}(\text{\textgreek{y}})n^{\text{\textgreek{m}}}.\label{eq:OneMoreBoundFromMorawetz}
\end{equation}

Substituting (\ref{eq:OneMoreTerm-1}) in (\ref{eq:OneMoreTerm}),
we obtain:

\begin{multline}
\int_{\mathcal{R}(0,t^{*})}h_{R}\cdot J_{\text{\textgreek{m}}}^{N}(\text{\textgreek{y}}_{0})n^{\text{\textgreek{m}}}+\int_{\mathcal{R}(0,t^{*})\cap\{r\le2R\}}\text{\textgreek{y}}_{0}^{2}\le C(R)\cdot\text{\textgreek{w}}_{0}^{2}\int_{\mathcal{R}(0,t^{*})\cap\{r\le2R\}}(J_{\text{\textgreek{m}}}^{N}(\text{\textgreek{y}}_{0})n^{\text{\textgreek{m}}}+\text{\textgreek{y}}_{0}^{2})+\\
+C(R)\Big|\int_{\mathcal{R}(0,t^{*})}F_{0}\cdot N\text{\textgreek{y}}_{0}\Big|+C(R)\cdot\Big|\int_{\mathcal{R}(0,t^{*})}h_{R}\cdot\text{\textgreek{y}}_{0}\cdot F_{0}\Big|+C(\text{\textgreek{w}}_{0},R)\int_{t=0}J_{\text{\textgreek{m}}}^{N}(\text{\textgreek{y}})n^{\text{\textgreek{m}}}.\label{eq:OneMoreTerm-1}
\end{multline}
Hence, if $\text{\textgreek{d}}$ in the statement of the current
Proposition (and thus $\text{\textgreek{w}}_{0}$) is small enough
in terms of the given $R$, the first term of the right hand side
of (\ref{eq:OneMoreTerm-1}) can be absorbed into the left hand side
(since $h_{R}\ge c>0$ on $\{r\le2R\}$). Thus, we deduce:

\begin{multline}
\int_{\mathcal{R}(0,t^{*})}h_{R}\cdot J_{\text{\textgreek{m}}}^{N}(\text{\textgreek{y}}_{0})n^{\text{\textgreek{m}}}+\int_{\mathcal{R}(0,t^{*})\cap\{r\le2R\}}\text{\textgreek{y}}_{0}^{2}\le\\
\le C(R)\Big|\int_{\mathcal{R}(0,t^{*})}F_{0}\cdot N\text{\textgreek{y}}_{0}\Big|+C(R)\cdot\Big|\int_{\mathcal{R}(0,t^{*})}h_{R}\cdot\text{\textgreek{y}}_{0}\cdot F_{0}\Big|+C(\text{\textgreek{w}}_{0},R)\int_{t=0}J_{\text{\textgreek{m}}}^{N}(\text{\textgreek{y}})n^{\text{\textgreek{m}}}.\label{eq:OnlyTheErroeTermsLeft}
\end{multline}

Finally, as we noted before, we can bound due to Lemma \ref{lem:MorawetzRefinedboundary}:

\begin{equation}
\int_{\mathcal{R}(0,t^{*})\cap\{r\ge2R\}}r^{-2}|N\text{\textgreek{y}}_{0}|^{2}+r^{-4}\text{\textgreek{y}}_{0}^{2}\le C\cdot\int_{\mathcal{R}(0,t^{*})\cap\{R\le r\le2R\}}(J_{\text{\textgreek{m}}}^{N}(\text{\textgreek{y}}_{0})n^{\text{\textgreek{m}}}+\text{\textgreek{y}}_{0}^{2})+C(\text{\textgreek{w}}_{0},R)\int_{t=0}J_{\text{\textgreek{m}}}^{N}(\text{\textgreek{y}})n^{\text{\textgreek{m}}}
\end{equation}
 and thus we can estimate

\begin{equation}
\int_{\mathcal{R}(0,t^{*})}(1+r)^{-2}|N\text{\textgreek{y}}_{0}|^{2}+(1+r)^{-4}\text{\textgreek{y}}_{0}^{2}\le C\cdot\int_{\mathcal{R}(0,t^{*})\cap\{r\le2R\}}(J_{\text{\textgreek{m}}}^{N}(\text{\textgreek{y}}_{0})n^{\text{\textgreek{m}}}+\text{\textgreek{y}}_{0}^{2})+C(\text{\textgreek{w}}_{0},R)\int_{t=0}J_{\text{\textgreek{m}}}^{N}(\text{\textgreek{y}})n^{\text{\textgreek{m}}}.
\end{equation}
Therefore, by a Cauchy-Schwarz inequality and an application of Lemma
\ref{lem:BoundF}, we can bound

\begin{multline}
C(R)\Big|\int_{\mathcal{R}(0,t^{*})}F_{0}\cdot N\text{\textgreek{y}}_{0}\Big|+C(R)\cdot\Big|\int_{\mathcal{R}(0,t^{*})}h_{R}\cdot\text{\textgreek{y}}_{0}\cdot F_{0}\Big|\le\\
\le\Big\{ C\cdot\int_{\mathcal{R}(0,t^{*})\cap\{r\le2R\}}(J_{\text{\textgreek{m}}}^{N}(\text{\textgreek{y}}_{0})n^{\text{\textgreek{m}}}+\text{\textgreek{y}}_{0}^{2})+C(\text{\textgreek{w}}_{0},R)\int_{t=0}J_{\text{\textgreek{m}}}^{N}(\text{\textgreek{y}})n^{\text{\textgreek{m}}}\Big\}^{1/2}\cdot\Big\{ C(\text{\textgreek{w}}_{0},R)\int_{t=0}J_{\text{\textgreek{m}}}^{N}(\text{\textgreek{y}})n^{\text{\textgreek{m}}}\Big\}^{1/2}.\label{eq:AfterCauchy}
\end{multline}
Substituting in (\ref{eq:OnlyTheErroeTermsLeft}) and absorbing the
$J_{\text{\textgreek{m}}}^{N}(\text{\textgreek{y}}_{0})n^{\text{\textgreek{m}}}+\text{\textgreek{y}}_{0}^{2}$
term of the first factor in (\ref{eq:AfterCauchy}) into the left
hand side of (\ref{eq:OnlyTheErroeTermsLeft}), we conclude:

\begin{equation}
\int_{\mathcal{R}(0,t^{*})}h_{R}J_{\text{\textgreek{m}}}^{N}(\text{\textgreek{y}}_{0})n^{\text{\textgreek{m}}}+\int_{\mathcal{R}(0,t^{*})\cap\{r\le2R\}}\text{\textgreek{y}}_{0}^{2}\le C(\text{\textgreek{w}}_{0},R)\int_{t=0}J_{\text{\textgreek{m}}}^{N}(\text{\textgreek{y}})n^{\text{\textgreek{m}}}.
\end{equation}
Thus, the proof of the Proposition is complete.
\end{proof}

\section{\label{sec:ILEDlow}Integrated local energy decay for $\text{\textgreek{y}}_{\le\text{\textgreek{w}}_{+}}$}

In this section, we will establish an integrated local energy decay
estimate for the low frequency part $\text{\textgreek{y}}_{\le\text{\textgreek{w}}_{+}}$
of $\text{\textgreek{y}}$. This will be accomplished by using of
a suitable Carleman type inequality and an ODE lemma appearing in
\cite{Rodnianski2011}. Of course we only have to obtain an integrated
local energy decay for $\text{\textgreek{y}}_{k}$ with $1\le|k|\le n$,
since the related estimate for the very low frequency part $\text{\textgreek{y}}_{0}$
was obtained in Section \ref{sub:LowFrequencies}.

\subsection{\label{sub:CarlemannEstimates}A Carleman-type estimate for $\text{\textgreek{y}}_{k}$,
$1\le|k|\le n$}

We will need to establish a Carleman type inequality for the functions
$\text{\textgreek{y}}_{k}$, $1\le|k|\le n$, that will be necessary
for proving the full integrated local energy decay statement for $\text{\textgreek{y}}_{\le\text{\textgreek{w}}_{+}}$
in the next section. The main ideas contained in the proof of the
following lemma have already been presented and used in \cite{Rodnianski2011}.
\begin{lem}
\label{lem:Carlemann}There exist functions $w,w':\mathcal{D}\rightarrow[0,1]$,
satisfying $T(w)=T(w')=0$ and which for $r\ge R_{1}$ are both equal
and strictly increasing functions of the variable $r$, such that
for any $R>0$ , any $0<\text{\textgreek{w}}_{0}\ll1$, any $\text{\textgreek{w}}_{+}>1$
and any $1\le|k|\le n$ the following bound holds (setting $s=C(\text{\textgreek{w}}_{0},R)\cdot\text{\textgreek{w}}_{k}$
for some large constant $C(\text{\textgreek{w}}_{0},R)$ depending
only on $\text{\textgreek{w}}_{0}$ and $R$):

\begin{align}
\int_{\mathcal{R}(0,t^{*})\cap\{r\le R\}}\Big(e^{2sw}+e^{2sw'}\Big)\cdot\Big(J_{\text{\textgreek{m}}}^{N}(\text{\textgreek{y}}_{k})n^{\text{\textgreek{m}}}+|\text{\textgreek{y}}_{k}|^{2}\Big)\le & C(\text{\textgreek{w}}_{0},R)\cdot\int_{\mathcal{R}(0,t^{*})\cap\{R\le r\le R+1\})}(1+s^{2})e^{2sw}\cdot\Big(J_{\text{\textgreek{m}}}^{N}(\text{\textgreek{y}}_{k})n^{\text{\textgreek{m}}}+|\text{\textgreek{y}}_{k}|^{2}\Big)+\label{eq:CarlemannTypeInequality}\\
 & +e^{C(R)\cdot s}\cdot C(\text{\textgreek{w}}_{0},R)\int_{t=0}J_{\text{\textgreek{m}}}^{N}(\text{\textgreek{y}})n^{\text{\textgreek{m}}}.\nonumber 
\end{align}
\end{lem}
\begin{proof}
We will assume without loss of generality that $R$ is large enough
in terms of the geometry of $(\mathcal{D},g)$, since we can always
pick a larger $R$ if necessary. Recall that we have assumed without
loss of generality that $\mathcal{H}^{+}\neq\emptyset$ (see Section
\ref{sub:Remark}). The proof of (\ref{eq:CarlemannTypeInequality})
in the case where $\mathcal{H}^{+}=\emptyset$ is almost identical,
and we will explain in the footnotes the differences of the proof
in the two cases.

For any smooth function $w:\mathcal{D}\rightarrow\mathbb{R}$ and
some real number $s\in\mathbb{R}$ to be fixed more precisely later
on (it will eventually agree with the parameter $s$ in the statement
of the lemma), we define, in accordance with \cite{Rodnianski2011},
the following current for any $1\le|k|\le n$ using the abstract index
notation:

\begin{equation}
J_{\text{\textgreek{m}}}=\partial^{\text{\textgreek{n}}}(e^{2sw})T_{\text{\textgreek{m}\textgreek{n}}}(\text{\textgreek{y}}_{k})+\frac{1}{2}\square_{g}e^{2sw}\cdot Re\{\bar{\text{\textgreek{y}}}_{k}\cdot\partial_{\text{\textgreek{m}}}\text{\textgreek{y}}_{k}\}+\frac{1}{2}s\Big\{(\partial_{\text{\textgreek{m}}}w)(\square_{g}e^{2sw})-(\partial_{\text{\textgreek{m}}}\square_{g}w)e^{2sw}\Big\}|\text{\textgreek{y}}_{k}|^{2}.\label{eq:CarlemannCurrent}
\end{equation}
The divergence of (\ref{eq:CarlemannCurrent}) takes the form:

\begin{align}
\nabla^{\text{\textgreek{m}}}J_{\text{\textgreek{m}}}= & \big\{ s^{2}(\square_{g}w)^{2}+4s^{3}\partial^{\text{\textgreek{m}}}w\cdot\partial_{\text{\textgreek{m}}}w\cdot(\square_{g}w)+4s^{2}\nabla_{\text{\textgreek{m}\textgreek{n}}}^{2}w\cdot\partial^{\text{\textgreek{m}}}w\cdot\partial^{\text{\textgreek{n}}}w+4s^{4}(\partial^{\text{\textgreek{m}}}w\cdot\partial_{\text{\textgreek{m}}}w)^{2}-\frac{1}{2}s\square_{g}^{2}w\big\} e^{2sw}|\text{\textgreek{y}}_{k}|^{2}+\label{eq:CarlemannCurrentDivergence1}\\
 & +\big\{4s^{2}\partial_{\text{\textgreek{m}}}w\cdot\square_{g}w+4s^{2}\nabla_{\text{\textgreek{m}\textgreek{n}}}^{2}w\cdot\partial^{\text{\textgreek{n}}}w+8s^{3}\partial^{\text{\textgreek{n}}}w\cdot\partial_{\text{\textgreek{n}}}w\cdot\partial_{\text{\textgreek{m}}}w\big\} e^{2sw}Re\{\bar{\text{\textgreek{y}}}_{k}\cdot\partial^{\text{\textgreek{m}}}\text{\textgreek{y}}_{k}\}+\nonumber \\
 & +\big\{2s\nabla_{\text{\textgreek{m}\textgreek{n}}}^{2}w+4s^{2}\partial_{\text{\textgreek{m}}}w\cdot\partial_{\text{\textgreek{n}}}w\big\} e^{2sw}\partial^{\text{\textgreek{m}}}\bar{\text{\textgreek{y}}}_{k}\partial^{\text{\textgreek{n}}}\text{\textgreek{y}}_{k}+\nonumber \\
 & +\frac{1}{2}(\square_{g}e^{2sw})Re\{\bar{\text{\textgreek{y}}}_{k}\cdot F_{k}\}+\partial^{\text{\textgreek{m}}}e^{2sw}\cdot Re\{\partial_{\text{\textgreek{m}}}\bar{\text{\textgreek{y}}}_{k}\cdot F_{k}\}.\nonumber 
\end{align}
Proceeding as in \cite{Rodnianski2011}, after completion of the square,
(\ref{eq:CarlemannCurrentDivergence1}) becomes:

\begin{align}
\nabla^{\text{\textgreek{m}}}J_{\text{\textgreek{m}}}= & e^{2sw}|2s\partial^{\text{\textgreek{m}}}w\cdot\partial_{\text{\textgreek{m}}}\text{\textgreek{y}}_{k}+(2s^{2}\partial^{\text{\textgreek{m}}}w\cdot\partial_{\text{\textgreek{m}}}w+s\cdot\square_{g}w)\cdot\text{\textgreek{y}}_{k}|^{2}+2s\nabla_{\text{\textgreek{m}\textgreek{n}}}^{2}w\cdot\partial^{\text{\textgreek{m}}}(e^{sw}\bar{\text{\textgreek{y}}}_{k})\cdot\partial^{\text{\textgreek{n}}}(e^{sw}\text{\textgreek{y}}_{k})+\label{eq:CarlemannCurrentDivergence}\\
 & +2s^{3}\nabla_{\text{\textgreek{m}\textgreek{n}}}^{2}w\cdot\partial^{\text{\textgreek{m}}}w\cdot\partial^{\text{\textgreek{n}}}w\cdot e^{2sw}|\text{\textgreek{y}}_{k}|^{2}-\frac{1}{2}s(\square_{g}^{2}w)\cdot e^{2sw}|\text{\textgreek{y}}_{k}|^{2}+\nonumber \\
 & +\frac{1}{2}(\square_{g}e^{2sw})Re\{\bar{\text{\textgreek{y}}}_{k}\cdot F_{k}\}+\partial^{\text{\textgreek{m}}}e^{2sw}\cdot Re\{\partial_{\text{\textgreek{m}}}\bar{\text{\textgreek{y}}}_{k}\cdot F_{k}\}.\nonumber 
\end{align}

Let $\text{\textgreek{q}}(x):\text{\textgreek{S}}\rightarrow[0,1]$
be a smooth cut-off function to be defined more precisely later on,
with compact support away from $\mathcal{H}^{+}\cap\text{\textgreek{S}}$.
We extend it to the whole of $\mathcal{D}\backslash\mathcal{H}^{-}$
by the condition $T\text{\textgreek{q}}=0$. Then the divergence identity
yields:

\begin{equation}
\int_{\mathcal{R}(0,t^{*})}\text{\textgreek{q}}\cdot\nabla^{\text{\textgreek{m}}}J_{\text{\textgreek{m}}}=\int_{\{t=t^{*}\}}\text{\textgreek{q}}\cdot J_{\text{\textgreek{m}}}n^{\text{\textgreek{m}}}-\int_{\{t=0\}}\text{\textgreek{q}}\cdot J_{\text{\textgreek{m}}}n^{\text{\textgreek{m}}}-\int_{\mathcal{R}(0,t^{*})}\partial^{\text{\textgreek{m}}}\text{\textgreek{q}}\cdot J_{\text{\textgreek{m}}}.\label{eq:Green'sIdentity}
\end{equation}

Substituting (\ref{eq:CarlemannCurrentDivergence}) in (\ref{eq:Green'sIdentity})
and setting (for convenience) 
\begin{equation}
U_{k}\doteq e^{sw}(2s\partial^{\text{\textgreek{m}}}w\cdot\partial_{\text{\textgreek{m}}}\text{\textgreek{y}}_{k}+(2s^{2}\partial^{\text{\textgreek{m}}}w\cdot\partial_{\text{\textgreek{m}}}w+s\cdot\square_{g}w)\cdot\text{\textgreek{y}}_{k}),
\end{equation}
we obtain the following Carleman identity (in accordance with the
Riemannian setting in \cite{Rodnianski2011}): 
\begin{equation}
\begin{split}\int_{\mathcal{R}(0,t^{*})}\text{\textgreek{q}}\cdot|U_{k}|^{2} & +2s\int_{\mathcal{R}(0,t^{*})}\text{\textgreek{q}}\nabla_{\text{\textgreek{m}\textgreek{n}}}^{2}w\cdot\partial^{\text{\textgreek{m}}}(e^{sw}\bar{\text{\textgreek{y}}}_{k})\cdot\partial^{\text{\textgreek{n}}}(e^{sw}\text{\textgreek{y}}_{k})+2s^{3}\int_{\mathcal{R}(0,t^{*})}\text{\textgreek{q}}\nabla_{\text{\textgreek{m}\textgreek{n}}}^{2}w\cdot\partial^{\text{\textgreek{m}}}w\cdot\partial^{\text{\textgreek{n}}}w\cdot e^{2sw}|\text{\textgreek{y}}_{k}|^{2}=\\
= & \frac{1}{2}s\int_{\mathcal{R}(0,t^{*})}\text{\textgreek{q}}(\square_{g}^{2}w)\cdot e^{2sw}|\text{\textgreek{y}}_{k}|^{2}+\int_{\mathcal{R}(0,t^{*})}\text{\textgreek{q}}Re\{e^{sw}U_{k}\cdot\bar{F}_{k}\}+\int_{\{t=t^{*}\}}\text{\textgreek{q}}\cdot J_{\text{\textgreek{m}}}n^{\text{\textgreek{m}}}-\int_{\{t=0\}}\text{\textgreek{q}}\cdot J_{\text{\textgreek{m}}}n^{\text{\textgreek{m}}}-\int_{\mathcal{R}(0,t^{*})}\partial^{\text{\textgreek{m}}}\text{\textgreek{q}}\cdot J_{\text{\textgreek{m}}}.
\end{split}
\label{eq:CarlemannIdentity}
\end{equation}

Using a Cauchy--Schwarz inequality 
\[
\int_{\mathcal{R}(0,t^{*})}\text{\textgreek{q}}Re\{e^{sw}U_{k}\cdot\bar{F}_{k}\}\le\Big(\int_{\mathcal{R}(0,t^{*})}\text{\textgreek{q}}e^{2sw}|F_{k}|^{2}\Big)^{1/2}\Big(\int_{\mathcal{R}(0,t^{*})}\text{\textgreek{q}}|U_{k}|^{2}\Big)^{1/2},
\]
after absorbing the term $\int_{\mathcal{R}(0,t^{*})}\text{\textgreek{q}}|U_{k}|^{2}$
into the corresponding term on the left hand side of (\ref{eq:CarlemannIdentity})
and then dropping once and for all the positive $\int_{\mathcal{R}(0,t^{*})}\text{\textgreek{q}}|U_{k}|^{2}$
term from the left hand side, we obtain the following inequality:

\begin{multline}
2s\int_{\mathcal{R}(0,t^{*})}\text{\textgreek{q}}Hess_{\text{\textgreek{m}\textgreek{n}}}(w)\cdot\partial^{\text{\textgreek{m}}}(e^{sw}\bar{\text{\textgreek{y}}}_{k})\cdot\partial^{\text{\textgreek{n}}}(e^{sw}\text{\textgreek{y}}_{k})+2s^{3}\int_{\mathcal{R}(0,t^{*})}\text{\textgreek{q}}\nabla_{\text{\textgreek{m}\textgreek{n}}}^{2}w\cdot\partial^{\text{\textgreek{m}}}w\cdot\partial^{\text{\textgreek{n}}}w\cdot e^{2sw}|\text{\textgreek{y}}_{k}|^{2}\le\\
\le\frac{1}{2}s\int_{\mathcal{R}(0,t^{*})}\text{\textgreek{q}}(\square_{g}^{2}w)\cdot e^{2sw}|\text{\textgreek{y}}_{k}|^{2}+C\cdot\int_{\mathcal{R}(0,t^{*})}\text{\textgreek{q}}e^{2sw}|F_{k}|^{2}+\int_{\{t=t^{*}\}}\text{\textgreek{q}}\cdot J_{\text{\textgreek{m}}}n^{\text{\textgreek{m}}}-\int_{\{t=0\}}\text{\textgreek{q}}\cdot J_{\text{\textgreek{m}}}n^{\text{\textgreek{m}}}-\int_{\mathcal{R}(0,t^{*})}\partial^{\text{\textgreek{m}}}\text{\textgreek{q}}\cdot J_{\text{\textgreek{m}}}.\label{eq:PreliminaryCarlemannInequality}
\end{multline}

We will now proceed to define more precisely the function $w$. In
order to be able to control the boundary terms in the above inequality,
we will impose the condition $Tw=0$, and thus $w$ will be uniquely
determined by its restriction on $\text{\textgreek{S}}\cap\mathcal{D}=\{t=0\}$. 

If it were possible to construct a geodesically convex function $w$
(that is, one with positive definite Hessian), inequality (\ref{eq:PreliminaryCarlemannInequality})
would readily yield the desired local integrated decay statement for
$\text{\textgreek{y}}_{k}$. However, in general the construction
of such a function $w$ is not possible for the class of manifolds
under consideration, as, for example, even the existence of a single
geodesic of $\mathcal{D}$ contained in a subset of the form $\{r\le R\}$
would exclude the existence of such a $w$. See also the relevant
comments in \cite{Rodnianski2011}. For this reason, we will construct
the function $w$ so that it is geodesically convex along the gradient
direction (thus guaranteeing the positivity of the second term of
the left hand side of (\ref{eq:PreliminaryCarlemannInequality})),
at least away from the ergoregion, and away from a discrete set of
points which will be the necessary local extrema of $w$, which are
necessitated by the lack of serious restrictions on the topology of
$\text{\textgreek{S}}$. This approach was originally adopted in the
Carleman-type estimates established in \cite{Rodnianski2011}. 

According to \cite{Rodnianski2011} (see also \cite{Burq1998}), we
can construct a Morse function $\text{\textgreek{f}}:\text{\textgreek{S}}\rightarrow\mathbb{R}$,
such that \textgreek{f} has no local minima in $\text{\textgreek{S}}$,
and such that moreover $\text{\textgreek{f}}\equiv r$ for $\big(\{r\le2r_{0}\}\cup\{r\ge R_{1}\}\big)\cap\text{\textgreek{S}}$.
For instance, $\text{\textgreek{f}}$ can be obtained by first solving
the Dirichlet problem 
\begin{equation}
\begin{cases}
\text{\textgreek{D}}\bar{\text{\textgreek{f}}}=-1\, on\,\{2r_{0}<r<R_{1}\}\cap\text{\textgreek{S}}\\
\bar{\text{\textgreek{f}}}=2r_{0}\, on\, r=2r_{0}\\
\bar{\text{\textgreek{f}}}=R_{1}\, on\, r=R_{1},
\end{cases}
\end{equation}
 where $\text{\textgreek{D}}$ is the Laplacian of the induced Riemannian
metric on $\text{\textgreek{S}}$. Such a function $\text{\ensuremath{\bar{\text{\textgreek{f}}}}}$
is smooth on $\{2r_{0}\le r\le R_{1}\}\cap\text{\textgreek{S}}$,
and has no local minima in $\{2r_{0}<r<R_{1}\}\cap\text{\textgreek{S}}$,
since if such a local minimum existed then we should have $\text{\textgreek{D}}\bar{\text{\textgreek{f}}}\ge0$
there. The set of Morse functions is dense in the $C^{2}$ topology,
and thus we can approximate $\bar{\text{\textgreek{f}}}$ in $C^{2}$
by a Morse function $\text{\textgreek{f}}$ on $\{2r_{0}\le r\le R_{1}\}\cap\text{\textgreek{S}}$.
In view of the maximum principle and Hopf's lemma, we must have 
\[
2r_{0}<\bar{\text{\textgreek{f}}}<R_{1}\text{ on }\{2r_{0}<r<R_{1}\}\cap\text{\textgreek{S}}
\]
 and 
\[
\partial_{r}\bar{\text{\textgreek{f}}}|_{\{r=2r_{0}\}\cap\text{\textgreek{S}}},\partial_{r}\bar{\text{\textgreek{f}}}|_{\{r=R_{1}\}\cap\text{\textgreek{S}}}>0.
\]
Note that $dr\neq0$ in $\{r\le3r_{0}\}$, and hence the expression
$\partial_{r}\bar{\text{\textgreek{f}}}|_{\{r=2r_{0}\}\cap\text{\textgreek{S}}}>0$
makes sense in any local chart around the horizon where $r$ is a
coordinate function; similarly for the expression $\partial_{r}\bar{\text{\textgreek{f}}}|_{\{r=R_{1}\}\cap\text{\textgreek{S}}}>0$.
Hence, if $\text{\textgreek{f}}$ was chosen close enough to $\bar{\text{\textgreek{f}}}$,
$\text{\textgreek{f}}$ can be mollified around $\big(\{r=2r_{0}\}\cup\{r=R_{1}\}\big)\cap\text{\textgreek{S}}$,
so that it can be smoothly extended as a Morse function equal to $r$
on $\big(\{r\le2r_{0}\}\cup\{r\ge R_{1}\}\big)\cap\text{\textgreek{S}}$.
In this way, we obtain a Morse function $\text{\textgreek{f}}$ with
no local minima in $\text{\textgreek{S}}$ and such that $\text{\textgreek{f}}\equiv r$
for $\big(\{r\le2r_{0}\}\cup\{r\ge R_{1}\}\big)\cap\text{\textgreek{S}}$.%
\footnote{In the case where $\mathcal{H}^{+}=\emptyset$, a similar construction
yields a Morse function $\text{\textgreek{f}}$ such that $\text{\textgreek{f}}\equiv r$
on $\{r\ge R_{1}\}\cap\text{\textgreek{S}}$.%
}

Since $\text{\textgreek{f}}$ is a Morse function, it only has a finite
number of critical points $\{x_{j}\}_{j=1}^{n}$, all lying in the
region$\{2r_{0}\le r\le R_{1}\}\cap\text{\textgreek{S}}$, and there
exist positive numbers $\{\bar{\text{\textgreek{e}}}_{j}\}_{j=1}^{n}$
such that the balls $\{B(x_{j},\bar{\text{\textgreek{e}}}_{j})\}_{j=1}^{n}$
are disjoint (the balls considered with respect to the induced Riemannian
metric $g_{\text{\textgreek{S}}}$ on \textgreek{S}). 

We will also need to construct a second Morse function $\text{\textgreek{f}}'$
with the same properties as $\text{\textgreek{f}}$, but with a different
set of points of local extrema. Due to the construction of $\text{\textgreek{f}}$
none of the $x_{j}$'s are points of local minima, and thus there
exist points $y_{j}\in B(x_{j},\frac{1}{2}\bar{\text{\textgreek{e}}}_{j})$
such that $\text{\textgreek{f}}(y_{j})>\text{\textgreek{f}}(x_{j})$.
We can find a diffeomorphism $\text{\textgreek{z}}:\text{\textgreek{S}}\rightarrow\text{\textgreek{S}}$
which is equal to the identity outside $\cup_{j=1}^{n}B(x_{j},\frac{1}{2}\bar{\text{\textgreek{e}}}_{j})$,
and which interchanges $x_{j}$ with $y_{j}$. Considering, hence,
the function 
\begin{equation}
\text{\textgreek{f}}'=\text{\textgreek{f}}\circ\text{\textgreek{z}},\label{eq:SecondMorseFunction}
\end{equation}
we infer that $\text{\textgreek{f}}\equiv\text{\textgreek{f}}'$ on
$\text{\textgreek{S}}\backslash\cup_{j=1}^{n}B(x_{j},\frac{1}{2}\bar{\text{\textgreek{e}}}_{j})$,
and the only critical points of $\text{\textgreek{f}}'$ are $y_{j}$.
Moreover there exist positive numbers $\{\text{\textgreek{e}}_{j}\}_{j=1}^{n}$
such that $\text{\textgreek{f}}<\text{\textgreek{f}}'$ on $\cup_{j=1}^{n}B(x_{j},\text{\textgreek{e}}_{j})$
and $\text{\textgreek{f}}'<\text{\textgreek{f}}$ on $\cup_{j=1}^{n}B(y_{j},\text{\textgreek{e}}_{j})$,
with $\cup_{j=1}^{n}B(x_{j},\text{\textgreek{e}}_{j})$ and $\cup_{j=1}^{n}B(y_{j},\text{\textgreek{e}}_{j})$
being disjoint. 

Extending $\text{\textgreek{f}}$ to the whole of $\mathcal{D}\backslash\mathcal{H}^{-}$
by the condition $T\text{\textgreek{f}}=0$, and setting $w=e^{\text{\textgreek{l}}\text{\textgreek{f}}}$,
for some $\text{\textgreek{l}}>0$ to be determined later, we compute

\begin{equation}
\nabla_{\text{\textgreek{m}\textgreek{n}}}^{2}w\cdot\partial^{\text{\textgreek{m}}}w\partial^{\text{\textgreek{n}}}w=\Big\{\text{\textgreek{l}}^{2}(\partial^{\text{\textgreek{m}}}\text{\textgreek{f}}\cdot\partial_{\text{\textgreek{m}}}\text{\textgreek{f}})^{2}+\text{\textgreek{l}}Hess_{\text{\textgreek{m}\textgreek{n}}}(\text{\textgreek{f}})\partial^{\text{\textgreek{m}}}\text{\textgreek{f}}\cdot\partial^{\text{\textgreek{n}}}\text{\textgreek{f}}\Big\} e^{3\text{\textgreek{l}}\text{\textgreek{f}}}.\label{eq:HessW}
\end{equation}
 Moreover, for any vector field $X\in\text{\textgreek{G}}(T\text{\textgreek{S}})$,
we calculate

\begin{equation}
\nabla_{\text{\textgreek{m}\textgreek{n}}}^{2}w\cdot X^{\text{\textgreek{m}}}X^{\text{\textgreek{n}}}=\Big\{\text{\textgreek{l}}^{2}(X\text{\textgreek{f}})^{2}+\text{\textgreek{l}}Hess_{\text{\textgreek{m}\textgreek{n}}}(\text{\textgreek{f}})X^{\text{\textgreek{m}}}X^{\text{\textgreek{n}}}\Big\} e^{\text{\textgreek{l}}\text{\textgreek{f}}}.\label{eq:HessW,X}
\end{equation}

Notice that $\partial^{\text{\textgreek{m}}}\text{\textgreek{f}}\cdot\partial_{\text{\textgreek{m}}}\text{\textgreek{f}}\neq0$
on the spacetime region $\{r>\frac{1}{2}r_{0}\}\backslash\big(\mathbb{R}\times\cup_{j=1}^{n}\{x_{j}\}\big)$.%
\footnote{Recall that we have identified $\mathcal{D}\backslash\mathcal{H}^{-}$
with $\mathbb{R}\times(\text{\textgreek{S}}\cap\mathcal{D})$.%
} This is due to the fact that outside the ergoregion, the level sets
of $\text{\textgreek{f}}$ will always contain the timelike tangent
vector $T$, and hence their normal will never be a null vector. Thus,
$\partial^{\text{\textgreek{m}}}\text{\textgreek{f}}\cdot\partial_{\text{\textgreek{m}}}\text{\textgreek{f}}\neq0$
whenever $\partial_{\text{\textgreek{m}}}\text{\textgreek{f}}$ is
non zero in $\{r>\frac{1}{2}r_{0}\}$. Hence, for any $R\gg1$, there
exists a large enough $\text{\textgreek{l}}=\text{\textgreek{l}}(R)$
such that in the region $\{\frac{3r_{0}}{4}\le r\le R+1\}\backslash\big(\mathbb{R}\times\cup_{j=1}^{n}B(x_{j},\frac{1}{4}\text{\textgreek{e}}_{j})\big)$:

\begin{equation}
\nabla_{\text{\textgreek{m}\textgreek{n}}}^{2}w\cdot\partial^{\text{\textgreek{m}}}w\partial^{\text{\textgreek{n}}}w\ge1>0.\label{eq:PositiveGradientDirection}
\end{equation}

After fixing an auxiliary Riemannian metric $h$ on $\mathcal{D}$
which is $T$ invariant and equal to the Euclidean metric inherited
by the chart we have fixed near the asymptotically flat end (i.\,e.
the chart $(t,r,\text{\textgreek{sv}})$ chart where $g$ takes the
form (\ref{eq:metric})), we can also bound in the region $\{\frac{3r_{0}}{4}\le r\le R+1\}\backslash\big(\mathbb{R}\times\cup_{j=1}^{n}B(x_{j},\frac{1}{4}\text{\textgreek{e}}_{j})\big)$
due to (\ref{eq:HessW,X}) for any fixed small $c_{0}>0$ (having
chosen $\text{\textgreek{l}}=\text{\textgreek{l}}(R,c_{0})$ large
enough):

\begin{equation}
\nabla_{\text{\textgreek{m}\textgreek{n}}}^{2}w\cdot X^{\text{\textgreek{m}}}X^{\text{\textgreek{n}}}\ge-c_{0}\cdot\big(\nabla_{\text{\textgreek{m}\textgreek{n}}}^{2}w\frac{\partial^{\text{\textgreek{m}}}w\cdot\partial^{\text{\textgreek{n}}}w}{|\partial w|_{h}^{2}}\big)|X|_{h}^{2}.\label{eq:NonGradientDirection}
\end{equation}
 That is, the negative part of the Hessian of $w$ can be controlled
by the value of the Hessian in the gradient direction.%
\footnote{In the case where $\mathcal{H}^{+}=\emptyset$, in view of the fact
that in this case $g(T,T)<0$ everywhere on $\mathcal{D}$ according
to Assumption \hyperref[Assumption 3]{3}, we can similarly arrange
so that (\ref{eq:PositiveGradientDirection}) and (\ref{eq:NonGradientDirection})
hold on the whole of $\{r\le R+1\}\backslash\big(\mathbb{R}\times\cup_{j=1}^{n}B(x_{j},\frac{1}{4}\text{\textgreek{e}}_{j})\big)$.%
} 

We will denote for simplicity the Hessian in the gradient direction
as 
\begin{equation}
\nabla_{nn}^{2}w\doteq\nabla_{\text{\textgreek{m}\textgreek{n}}}^{2}w\frac{\partial^{\text{\textgreek{m}}}w\cdot\partial^{\text{\textgreek{n}}}w}{|\partial w|_{h}^{2}}.
\end{equation}
Then, (\ref{eq:NonGradientDirection}) and a Cauchy--Schwarz inequality
imply that:

\begin{equation}
\nabla_{\text{\textgreek{m}\textgreek{n}}}^{2}w\cdot\partial^{\text{\textgreek{m}}}(e^{sw}\bar{\text{\textgreek{y}}}_{k})\cdot\partial^{\text{\textgreek{n}}}(e^{sw}\text{\textgreek{y}}_{k})\ge-2c_{0}\nabla_{\text{\textgreek{m}\textgreek{n}}}^{2}we^{2sw}\big(s^{2}|\partial w|_{h}^{2}|\text{\textgreek{y}}_{k}|^{2}+|\partial\text{\textgreek{y}}_{k}|_{h}^{2}\big).\label{eq:DerivativeTerm}
\end{equation}

We will now fix the cut-off function $\text{\textgreek{q}}$ appearing
in (\ref{eq:PreliminaryCarlemannInequality}). Fixing a smooth function
$\text{\textgreek{q}}=\text{\textgreek{q}}_{R}:\text{\textgreek{S}}\rightarrow[0,1]$
such that $\text{\textgreek{q}}_{R}\equiv0$ on $\{r\le\frac{3r_{0}}{4}\}\cup\{r\ge R+1\}\cup_{j=1}^{n}B(x_{j},\frac{1}{4}\text{\textgreek{e}}_{j})$
and $\text{\textgreek{q}}_{R}\equiv1$ on $\{\frac{7r_{0}}{8}\le r\le R\}\backslash\cup_{j=1}^{n}B(x_{j},\frac{1}{2}\text{\textgreek{e}}_{j})$,
extending it on $\mathcal{D}\backslash\mathcal{H}^{-}$ by the condition
$T\text{\textgreek{q}}_{R}=0$, we deduce from (\ref{eq:PreliminaryCarlemannInequality}),
in view also of (\ref{eq:PositiveGradientDirection}), (\ref{eq:NonGradientDirection})
and (\ref{eq:DerivativeTerm}): 
\begin{equation}
\begin{split}(1-2c_{0})\cdot s^{3}\int_{\mathcal{R}(0,t^{*})}\text{\textgreek{q}}_{R}\nabla_{nn}^{2}w & \cdot|\partial w|_{h}^{2}e^{2sw}|\text{\textgreek{y}}_{k}|^{2}\le\\
\le & C\cdot s\int_{\mathcal{R}(0,t^{*})}\text{\textgreek{q}}_{R}(\square_{g}^{2}w)\cdot e^{2sw}|\text{\textgreek{y}}_{k}|^{2}+2c_{0}s\cdot\int_{\mathcal{R}(0,t^{*})}\text{\textgreek{q}}_{R}\nabla_{nn}^{2}w\cdot e^{2sw}|\partial\text{\textgreek{y}}_{k}|_{h}^{2}+\\
 & +\int_{\mathcal{R}(0,t^{*})}\text{\textgreek{q}}_{R}e^{2sw}|F_{k}|^{2}+\int_{\{t=t^{*}\}}\text{\textgreek{q}}_{R}\cdot J_{\text{\textgreek{m}}}n^{\text{\textgreek{m}}}-\int_{\{t=0\}}\text{\textgreek{q}}_{R}\cdot J_{\text{\textgreek{m}}}n^{\text{\textgreek{m}}}-\int_{\mathcal{R}(0,t^{*})}\partial^{\text{\textgreek{m}}}\text{\textgreek{q}}_{R}\cdot J_{\text{\textgreek{m}}}.
\end{split}
\label{eq:PreliminaryCarlemannInequality-1}
\end{equation}

In view of Lemmas \ref{lem:BoundF} and \ref{lem:BoundednessPsiK}
(using also some Hardy-type inequality to control the 0-th order terms
on the slices $t=0,t^{*}$), we obtain from \ref{eq:PreliminaryCarlemannInequality-1}
after dividing with $(1-2c_{0})$ (provided $c_{0}$ was chosen so
that $c_{0}\le\frac{1}{4}$):
\begin{equation}
\begin{split}s^{3}\int_{\mathcal{R}(0,t^{*})}\text{\textgreek{q}}_{R}\nabla_{nn}^{2}w & \cdot|\partial w|_{h}^{2}e^{2sw}|\text{\textgreek{y}}_{k}|^{2}\le\\
\le & C\cdot s\int_{\mathcal{R}(0,t^{*})}\text{\textgreek{q}}_{R}(\square_{g}^{2}w)\cdot e^{2sw}|\text{\textgreek{y}}_{k}|^{2}+\frac{2c_{0}}{1-2c_{0}}\cdot s\cdot\int_{\mathcal{R}(0,t^{*})}\text{\textgreek{q}}_{R}\nabla_{nn}^{2}w\cdot e^{2sw}|\partial\text{\textgreek{y}}_{k}|_{h}^{2}\Big\}+\\
 & +C(R)\cdot\int_{\mathcal{R}(0,t^{*})\cap\{supp(\partial\text{\textgreek{q}})\}}(1+s^{2})e^{2sw}\big\{|\partial\text{\textgreek{y}}_{k}|_{h}^{2}+|\text{\textgreek{y}}_{k}|^{2}\big\}+e^{C(R)\cdot s}\cdot C(\text{\textgreek{w}}_{0},R)\cdot\int_{t=0}J_{\text{\textgreek{m}}}^{N}(\text{\textgreek{y}})n^{\text{\textgreek{m}}}.
\end{split}
\label{eq:AlmostCarlemannInequality}
\end{equation}

We will now proceed to absorb the first two terms of the righ hand
side of (\ref{eq:AlmostCarlemannInequality}) into the left hand side
after choosing $s$ sufficirently large. We will use the auxiliary
current 
\begin{equation}
J_{\text{\textgreek{m}}}^{aux}=\nabla_{nn}^{2}w\cdot e^{2sw}\cdot Re\{\bar{\text{\textgreek{y}}}_{k}\cdot\partial_{\text{\textgreek{m}}}\text{\textgreek{y}}_{k}\}-\frac{1}{2}\partial_{\text{\textgreek{m}}}\Big\{\nabla_{nn}^{2}w\cdot e^{2sw}\Big\}|\text{\textgreek{y}}_{k}|^{2},
\end{equation}
for which the divergence theorem yields 
\begin{multline}
\int_{\mathcal{R}(0,t^{*})}\text{\textgreek{q}}_{R}\cdot\Big(\nabla_{nn}^{2}w\cdot e^{2sw}\cdot\partial^{\text{\textgreek{m}}}\bar{\text{\textgreek{y}}}_{k}\cdot\partial_{\text{\textgreek{m}}}\text{\textgreek{y}}_{k}-\frac{1}{2}\square_{g}\big\{\nabla_{nn}^{2}w\cdot e^{2sw}\big\}\cdot|\text{\textgreek{y}}_{k}|^{2}\Big)=\\
=-\int_{\mathcal{R}(0,t^{*})}\nabla_{nn}^{2}w\cdot e^{2sw}\cdot Re\{\bar{\text{\textgreek{y}}}_{k}\cdot F_{k}\}-\int_{t=t^{*}}\text{\textgreek{q}}_{R}\cdot J_{\text{\textgreek{m}}}^{aux}n^{\text{\textgreek{m}}}+\int_{t=0}\text{\textgreek{q}}_{R}\cdot J_{\text{\textgreek{m}}}^{aux}n^{\text{\textgreek{m}}}-\int_{\mathcal{R}(0,t^{*})}\partial^{\text{\textgreek{m}}}\text{\textgreek{q}}_{R}\cdot J_{\text{\textgreek{m}}}^{aux}.\label{eq:BeforeLagrangeIdentity}
\end{multline}
Therefore, from (\ref{eq:BeforeLagrangeIdentity}) we obtain the following
Langrangean identity:

\begin{multline}
\int_{\mathcal{R}(0,t^{*})}\text{\textgreek{q}}_{R}\cdot\nabla_{nn}^{2}w\cdot e^{2sw}\cdot\partial^{\text{\textgreek{m}}}\bar{\text{\textgreek{y}}}_{k}\cdot\partial_{\text{\textgreek{m}}}\text{\textgreek{y}}_{k}=\frac{1}{2}\int_{\mathcal{R}(0,t^{*})}\text{\textgreek{q}}_{R}\cdot\square_{g}\big\{\nabla_{nn}^{2}w\cdot e^{2sw}\big\}\cdot|\text{\textgreek{y}}_{k}|^{2}-\\
-\int_{\mathcal{R}(0,t^{*})}\nabla_{nn}^{2}w\cdot e^{2sw}\cdot Re\{\bar{\text{\textgreek{y}}}_{k}\cdot F_{k}\}-\int_{t=t^{*}}\text{\textgreek{q}}_{R}\cdot J_{\text{\textgreek{m}}}^{aux}n^{\text{\textgreek{m}}}+\int_{t=0}\text{\textgreek{q}}_{R}\cdot J_{\text{\textgreek{m}}}^{aux}n^{\text{\textgreek{m}}}-\int_{\mathcal{R}(0,t^{*})}\partial^{\text{\textgreek{m}}}\text{\textgreek{q}}_{R}\cdot J_{\text{\textgreek{m}}}^{aux}.\label{eq:LagrangeIdentity}
\end{multline}

In order to estimate the second term of the right hand side of (\ref{eq:AlmostCarlemannInequality}),
we will use the trivial bound 
\begin{equation}
|\partial\text{\textgreek{y}}_{k}|_{h}^{2}\le C_{1}\cdot\partial^{\text{\textgreek{m}}}\bar{\text{\textgreek{y}}}_{k}\cdot\partial_{\text{\textgreek{m}}}\text{\textgreek{y}}_{k}+C_{2}|T\text{\textgreek{y}}_{k}|^{2}.
\end{equation}
After using Lemmas \ref{lem:BoundF} and \ref{lem:BoundednessPsiK}
to deal with the $F_{k}$-terms and the boundary terms in (\ref{eq:LagrangeIdentity})
(as we did in the derivation of (\ref{eq:AlmostCarlemannInequality})),
from (\ref{eq:AlmostCarlemannInequality}), (\ref{eq:LagrangeIdentity})
and Lemma \ref{lem:DtToOmegaInequalities2} (which allows us to bound
$|T\text{\textgreek{y}}_{k}|^{2}$ bulk terms by $\text{\textgreek{w}}_{k}^{2}|\text{\textgreek{y}}_{k}|^{2}$)
terms we obtain: 
\begin{align}
s^{3}\int_{\mathcal{R}(0,t^{*})}\text{\textgreek{q}}_{R}\cdot\nabla_{nn}^{2}w\cdot|\partial w|_{h}^{2}e^{2sw}|\text{\textgreek{y}}_{k}|^{2}\le & C\cdot s\int_{\mathcal{R}(0,t^{*})}\text{\textgreek{q}}_{R}\cdot(\square_{g}^{2}w)\cdot e^{2sw}|\text{\textgreek{y}}_{k}|^{2}+C(\text{\textgreek{w}}_{0})\cdot s\cdot\text{\textgreek{w}}_{k}^{2}\int_{\mathcal{R}(0,t^{*})}\text{\textgreek{q}}_{R}\nabla_{nn}^{2}w\cdot e^{2sw}|\text{\textgreek{y}}_{k}|^{2}+\label{eq:AlmostCarlemannInequality-1}\\
 & +C\cdot\frac{2c_{0}}{1-2c_{0}}\cdot s\cdot\int_{\mathcal{R}(0,t^{*})}\text{\textgreek{q}}_{R}\cdot\square_{g}\big\{\nabla_{nn}^{2}w\cdot e^{2sw}\big\}\cdot|\text{\textgreek{y}}_{k}|^{2}+\nonumber \\
 & +C(\text{\textgreek{w}}_{0},R)\int_{\mathcal{R}(0,t^{*})\cap\{supp(\partial\text{\textgreek{q}}_{R})\}}(1+s^{2})e^{2sw}\big\{|\partial\text{\textgreek{y}}_{k}|_{h}^{2}+|\text{\textgreek{y}}_{k}|^{2}\big\}+\nonumber \\
 & +e^{C(R)\cdot s}\cdot C(\text{\textgreek{w}}_{0},R)\int_{t=0}J_{\text{\textgreek{m}}}^{N}(\text{\textgreek{y}})n^{\text{\textgreek{m}}}.\nonumber 
\end{align}

In view of the fact that in the support of $\text{\textgreek{q}}_{R}$
we can bound $\nabla_{nn}^{2}w\cdot|\partial w|_{h}^{2}\ge1$ (due
to (\ref{eq:PositiveGradientDirection})), the third term of the right
hand side of (\ref{eq:AlmostCarlemannInequality-1}) can be bounded
as follows: 
\begin{equation}
C\cdot\frac{2c_{0}}{1-2c_{0}}\cdot s\cdot\int_{\mathcal{R}(0,t^{*})}\text{\textgreek{q}}_{R}\cdot\square_{g}\big\{\nabla_{nn}^{2}w\cdot e^{2sw}\big\}\cdot|\text{\textgreek{y}}_{k}|^{2}\le C(R)\cdot\frac{2c_{0}}{1-2c_{0}}\cdot(s+s^{3})\cdot\int_{\mathcal{R}(0,t^{*})}\text{\textgreek{q}}_{R}\cdot\nabla_{nn}^{2}w\cdot|\partial w|_{h}^{2}e^{2sw}|\text{\textgreek{y}}_{k}|^{2}.
\end{equation}
Thus, if $c_{0}$ is small enough in terms of $R$ (which due to (\ref{eq:NonGradientDirection})
corresponds to $\text{\textgreek{l}}$ having been chosen large enough
in terms of $R$), the third term of the right hand of (\ref{eq:AlmostCarlemannInequality-1})
can be absorbed into the left hand side. Moreover, if we set $s=C\cdot\text{\textgreek{w}}_{k}$,
where $C=C(R,\text{\textgreek{w}}_{0})>0$ is large enough in terms
of $R$, $\text{\textgreek{w}}_{0}$, the left hand side of (\ref{eq:AlmostCarlemannInequality-1})
can also absorb the first two terms of the right hand side, yielding
the following Carleman type inequality:

\begin{equation}
s^{3}\int_{\mathcal{R}(0,t^{*})}\text{\textgreek{q}}_{R}e^{2sw}|\text{\textgreek{y}}_{k}|^{2}\le C(\text{\textgreek{w}}_{0},R)\cdot\int_{\mathcal{R}(0,t^{*})\cap\{supp(\partial\text{\textgreek{q}})\}}(1+s^{2})e^{2sw}\big\{|\partial\text{\textgreek{y}}_{k}|_{h}^{2}+|\text{\textgreek{y}}_{k}|^{2}\big\}+e^{C(R)\cdot s}\cdot C(\text{\textgreek{w}}_{0},R)\int_{t=0}J_{\text{\textgreek{m}}}^{N}(\text{\textgreek{y}})n^{\text{\textgreek{m}}}.\label{eq:CarlemannInequality1}
\end{equation}
In view of the Lagrangean identity (\ref{eq:LagrangeIdentity}), we
can also upgrade inequality \ref{eq:CarlemannInequality1} to:

\begin{multline}
s^{3}\int_{\mathcal{R}(0,t^{*})}\text{\textgreek{q}}_{R}e^{2sw}|\text{\textgreek{y}}_{k}|^{2}+c(\text{\textgreek{w}}_{0},R)\cdot s\int_{\mathcal{R}(0,t^{*})}\text{\textgreek{q}}_{R}e^{2sw}|\partial\text{\textgreek{y}}_{k}|_{h}^{2}\le\\
\le C(\text{\textgreek{w}}_{0},R)\int_{\mathcal{R}(0,t^{*})\cap\{supp(\partial\text{\textgreek{q}}_{R})\}}(1+s^{2})e^{2sw}\big\{|\partial\text{\textgreek{y}}_{k}|_{h}^{2}+|\text{\textgreek{y}}_{k}|^{2}\big\}+e^{C(R)\cdot s}\cdot C(\text{\textgreek{w}}_{0},R)\int_{t=0}J_{\text{\textgreek{m}}}^{N}(\text{\textgreek{y}})n^{\text{\textgreek{m}}}\label{eq:CarlemannInequality}
\end{multline}
for some constant $c(\text{\textgreek{w}}_{0},R)$ small in terms
of $\text{\textgreek{w}}_{0},R$.

The support of $\partial\text{\textgreek{q}}_{R}$ in $\mathcal{D}\backslash\mathcal{H}^{-}$
breaks into three pieces: The part contained in $\mathbb{R}\times(\cup_{j=1}^{n}B(x_{j},\text{\textgreek{e}}_{j}))$,
the part contained in $\{\frac{3r_{0}}{4}\le r\le\frac{7r_{0}}{8}\}$
and the part contained in $\{R\le r\le R+1\}$. Thus, the first term
of the right hand side of (\ref{eq:CarlemannInequality}) naturally
splits into three summands (corresponding to the aforementioned partition
of $supp(\partial\text{\textgreek{q}}_{R})$):
\begin{align}
\int_{\mathcal{R}(0,t^{*})\cap\{supp(\partial\text{\textgreek{q}}_{R})\}}(1+s^{2})e^{2sw}\big\{|\partial\text{\textgreek{y}}_{k}|_{h}^{2}+|\text{\textgreek{y}}_{k}|^{2}\big\}= & \int_{\mathcal{R}(0,t^{*})\cap\{supp(\partial\text{\textgreek{q}}_{R})\}\cap\big(\mathbb{R}\times(\cup_{j=1}^{n}B(x_{j},\text{\textgreek{e}}_{j}))\big)}(1+s^{2})e^{2sw}\big\{|\partial\text{\textgreek{y}}_{k}|_{h}^{2}+|\text{\textgreek{y}}_{k}|^{2}\big\}+\label{eq:PartitionOfNearBoundaryTerms}\\
 & +\int_{\mathcal{R}(0,t^{*})\cap\{supp(\partial\text{\textgreek{q}}_{R})\}\cap\{\frac{3r_{0}}{4}\le r\le\frac{7r_{0}}{8}\}}(1+s^{2})e^{2sw}\big\{|\partial\text{\textgreek{y}}_{k}|_{h}^{2}+|\text{\textgreek{y}}_{k}|^{2}\big\}+\nonumber \\
 & +\int_{\mathcal{R}(0,t^{*})\cap\{supp(\partial\text{\textgreek{q}}_{R})\}\cap\{R\le r\le R+1\}}(1+s^{2})e^{2sw}\big\{|\partial\text{\textgreek{y}}_{k}|_{h}^{2}+|\text{\textgreek{y}}_{k}|^{2}\big\}.\nonumber 
\end{align}
 Therefore, in order to reach a statement close to (\ref{eq:CarlemannTypeInequality}),
we have to dispense with the first two summands of the right hand
side of (\ref{eq:PartitionOfNearBoundaryTerms}).%
\footnote{Notice that the same steps apply in the case where $\mathcal{H}^{+}=\emptyset$,
the only difference being that in that case the second term of the
right hand side of (\ref{eq:PartitionOfNearBoundaryTerms}) can be
dropped.%
} We will accomplish this task in two steps.

\smallskip{}

\noindent \emph{1.} In order to deal with the first term of the right
hand side of (\ref{eq:PartitionOfNearBoundaryTerms}), we will make
use of the second Morse function function $\text{\textgreek{f}}'$
defined as (\ref{eq:SecondMorseFunction}).

By defining $w'=e^{\text{\textgreek{l}}\text{\textgreek{f}}'}$ and
$\text{\textgreek{q}}{}_{R}^{\prime}=\text{\textgreek{q}}_{R}\circ\text{\textgreek{z}}$,
and repeating the same procedure as before (leading to (\ref{eq:CarlemannInequality}))
with $w'$ and $\text{\textgreek{q}}_{R}^{\prime}$ in place of $w$
and $\text{\textgreek{q}}_{R}$, we obtain the inequality

\begin{multline}
s^{3}\int_{\mathcal{R}(0,t^{*})}\text{\textgreek{q}}_{R}^{\prime}e^{2sw'}|\text{\textgreek{y}}_{k}|^{2}+c(\text{\textgreek{w}}_{0},R)\int_{\mathcal{R}(0,t^{*})}\text{\textgreek{q}}_{R}^{\prime}e^{2sw'}|\partial\text{\textgreek{y}}_{k}|_{h}^{2}\le\\
\le C(\text{\textgreek{w}}_{0},R)\int_{\mathcal{R}(0,t^{*})\cap\{supp(\partial\text{\textgreek{q}}_{R}^{\prime})\}}(1+s^{2})e^{2sw'}\big\{|\partial\text{\textgreek{y}}_{k}|_{h}^{2}+|\text{\textgreek{y}}_{k}|^{2}\big\}+e^{C(R)\cdot s}\cdot C(\text{\textgreek{w}}_{0},R)\int_{t=0}J_{\text{\textgreek{m}}}^{N}(\text{\textgreek{y}})n^{\text{\textgreek{m}}}.\label{eq:CarlemannInequality2}
\end{multline}

Recall that $\text{\textgreek{f}}<\text{\textgreek{f}}'$ on $\cup_{j=1}^{n}B(x_{j},\text{\textgreek{e}}_{j})$
and $\text{\textgreek{f}}'<\text{\textgreek{f}}$ on $\cup_{j=1}^{n}B(y_{j},\text{\textgreek{e}}_{j})$,
with $\cup_{j=1}^{n}B(x_{j},\text{\textgreek{e}}_{j})$ and $\cup_{j=1}^{n}B(y_{j},\text{\textgreek{e}}_{j})$
being disjoint. This condition guarantees that if the constant $C(\text{\textgreek{w}}_{0},R)$
in the definition $s=C(\text{\textgreek{w}}_{0},R)\cdot\text{\textgreek{w}}_{k}$
is chosen large enough, then the term 
\[
\int_{\mathcal{R}(0,t^{*})\cap\{supp(\partial\text{\textgreek{q}}_{R})\}\cap\big(\mathbb{R}\times(\cup_{j=1}^{n}B(x_{j},\text{\textgreek{e}}_{j}))\big)}(1+s^{2})e^{2sw}\big\{|\partial\text{\textgreek{y}}_{k}|_{h}^{2}+|\text{\textgreek{y}}_{k}|^{2}\big\}
\]
 of (\ref{eq:CarlemannInequality}) can be controlled by the left
hand side of (\ref{eq:CarlemannInequality2}), while the term 
\[
\int_{\mathcal{R}(0,t^{*})\cap\{supp(\partial\text{\textgreek{q}}_{R}^{\prime})\}\cap\big(\mathbb{R}\times(\cup_{j=1}^{n}B(y_{j},\text{\textgreek{e}}_{j}))\big)}(1+s^{2})e^{2sw'}\big\{|\partial\text{\textgreek{y}}_{k}|_{h}^{2}+|\text{\textgreek{y}}_{k}|^{2}\big\}
\]
 of (\ref{eq:CarlemannInequality2}) can be controlled by the left
hand side of \ref{eq:CarlemannInequality}. Therefore, recalling that
$w=w'$ on the complement of $\mathbb{R}\times\cup_{j=1}^{n}B(x_{j},\frac{1}{2}\bar{\text{\textgreek{e}}}_{j})$,
and redefining $\text{\textgreek{q}}_{R}$ to be equal to the maximum
of the previous two cut-offs $\text{\textgreek{q}}_{R},\text{\textgreek{q}}'_{R}$
\footnote{note that in this way, we have $\text{\textgreek{q}}_{R}\equiv1$
on the whole of $\{\frac{7r_{0}}{8}\le r\le R\}$%
}, we obtain after adding (\ref{eq:CarlemannInequality}) and (\ref{eq:CarlemannInequality2}):

\begin{multline}
\int_{\mathcal{R}(0,t^{*})}\text{\textgreek{q}}_{R}\cdot\Big(e^{2sw}+e^{2sw'}\Big)|\text{\textgreek{y}}_{k}|^{2}+\int_{\mathcal{R}(0,t^{*})}\text{\textgreek{q}}_{R}\cdot\Big(e^{2sw}+e^{2sw'}\Big)|\partial\text{\textgreek{y}}_{k}|_{h}^{2}\le\\
\le C(\text{\textgreek{w}}_{0},R)\int_{\mathcal{R}(0,t^{*})\cap(\{\frac{3r_{0}}{4}\le r\le\frac{7r_{0}}{8}\}\cup\{R\le r\le R+1\})}(1+s^{2})e^{2sw}\big\{|\partial\text{\textgreek{y}}_{k}|_{h}^{2}+|\text{\textgreek{y}}_{k}|^{2}\big\}+e^{C(R)\cdot s}\cdot C(\text{\textgreek{w}}_{0},R)\int_{t=0}J_{\text{\textgreek{m}}}^{N}(\text{\textgreek{y}})n^{\text{\textgreek{m}}}.\label{eq:NonDegenerateCarlemannInequality}
\end{multline}

\smallskip{}

\noindent \emph{2.} In order to deal with the first term of the right
hand side of (\ref{eq:NonDegenerateCarlemannInequality}), we will
use the red shift current $K^{N}$, which is positive definite near
the horizon (see Assumption \hyperref[Assumption 2]{2}). This will
also help us extend our control of the inegral of $|\text{\textgreek{y}}_{k}|^{2}+|\partial\text{\textgreek{y}}_{k}|_{h}^{2}$
on the whole of $\{r\le R\}$.%
\footnote{Notice that this step is unnecessary in the case where $\mathcal{H}^{+}=\emptyset$,
as in that case the second term of the right hand side of (\ref{eq:PartitionOfNearBoundaryTerms})
can be dropped.%
} 

Setting $\tilde{\text{\textgreek{q}}}_{R}=max(1_{\{0\le r\le r_{0}\}},\text{\textgreek{q}}_{R})$
(note that $\tilde{\text{\textgreek{q}}}_{R}\equiv1$ on $\{r\le R\}$),
we compute using the divergence theorem:
\begin{align}
\int_{\mathcal{R}(0,t^{*})}\tilde{\text{\textgreek{q}}}_{R}K^{N}(\text{\textgreek{y}}_{k})= & -\int_{t=t^{*}}\tilde{\text{\textgreek{q}}}_{R}J_{\text{\textgreek{m}}}^{N}(\text{\textgreek{y}}_{k})n^{\text{\textgreek{m}}}-\int_{\mathcal{H}^{+}\cap\mathcal{R}(0,t^{*})}\tilde{\text{\textgreek{q}}}_{R}J_{\text{\textgreek{m}}}^{N}(\text{\textgreek{y}}_{k})n^{\text{\textgreek{m}}}+\int_{t=0}\tilde{\text{\textgreek{q}}}_{R}J_{\text{\textgreek{m}}}^{N}(\text{\textgreek{y}}_{k})n^{\text{\textgreek{m}}}-\label{eq:RedShiftIdentityILED}\\
 & -\int_{\mathcal{R}(0,t^{*})}\partial^{\text{\textgreek{m}}}\tilde{\text{\textgreek{q}}}_{R}\cdot J_{\text{\textgreek{m}}}^{N}(\text{\textgreek{y}}_{k})-\int_{\mathcal{R}(0,t^{*})}\tilde{\text{\textgreek{q}}}_{R}Re\{N\bar{\text{\textgreek{y}}}_{k}\cdot F_{k}\}.\nonumber 
\end{align}
Hence, in view of (\ref{eq:PositiveKN}), Lemmas \ref{lem:BoundF}
and \ref{lem:BoundednessPsiK}, and a Cauchy-Schwarz inequality for
the term $\int_{\mathcal{R}(0,t^{*})}\tilde{\text{\textgreek{q}}}_{R}Re\{N\bar{\text{\textgreek{y}}}_{k}\cdot F_{k}\}$,
we obtain from (\ref{eq:RedShiftIdentityILED}) for any small $\text{\textgreek{d}}>0$:

\begin{multline}
\int_{R(0,t^{*})\cap\{r\le\frac{7}{8}r_{0}\}}|\partial\text{\textgreek{y}}_{k}|_{h}^{2}\le C\cdot\int_{\mathcal{R}(0,t^{*})\cap\{r_{0}\le r\le2r_{0}\}}|\partial\text{\textgreek{y}}_{k}|_{h}^{2}+\text{\textgreek{d}}\cdot\int_{\mathcal{R}(0,t^{*})}\text{\textgreek{q}}_{R}|\partial\text{\textgreek{y}}_{k}|_{h}^{2}+\\
+C(\text{\textgreek{w}}_{0})\text{\textgreek{d}}^{-1}\int_{t=0}J_{\text{\textgreek{m}}}^{N}(\text{\textgreek{y}}_{k})n^{\text{\textgreek{m}}}+C(R)\cdot\int_{\mathcal{R}(0,t^{*})\cap\{R\le r\le R+1\}}\big\{|\partial\text{\textgreek{y}}_{k}|_{h}^{2}+|\text{\textgreek{y}}_{k}|^{2}\}.\label{eq:NearTheHorizonDerivativeEstimate}
\end{multline}

Hence, after adding (\ref{eq:NearTheHorizonDerivativeEstimate}) multiplied
by $\sup_{\{r\le7r_{0}/8\}}e^{2sw}$ to (\ref{eq:NonDegenerateCarlemannInequality})
for some $\text{\textgreek{d}}$ small enough in terms of $\text{\textgreek{w}}_{0},R$,
we infer:
\begin{equation}
\begin{split}\int_{\mathcal{R}(0,t^{*})\cap\{\frac{7r_{0}}{8}\le r\le R\}}\Big(e^{2sw}+e^{2sw'}\Big) & |\text{\textgreek{y}}_{k}|^{2}+\int_{\mathcal{R}(0,t^{*})\cap\{r\le R\}}\Big(e^{2sw}+e^{2sw'}\Big)|\partial\text{\textgreek{y}}_{k}|_{h}^{2}\le\\
\le & C(\text{\textgreek{w}}_{0},R)\cdot\big(\sup_{\{r\le7r_{0}/8\}}e^{2sw}\big)\int_{\mathcal{R}(0,t^{*})\cap\{r_{0}\le r\le2r_{0}\}}|\partial\text{\textgreek{y}}_{k}|_{h}^{2}+\\
 & +C(\text{\textgreek{w}}_{0},R)\int_{\mathcal{R}(0,t^{*})\cap\{R\le r\le R+1\})}(1+s^{2})e^{2sw}\big\{|\partial\text{\textgreek{y}}_{k}|_{h}^{2}+|\text{\textgreek{y}}_{k}|^{2}\big\}+e^{C(R)\cdot s}\cdot C(\text{\textgreek{w}}_{0},R)\int_{t=0}J_{\text{\textgreek{m}}}^{N}(\text{\textgreek{y}})n^{\text{\textgreek{m}}}.
\end{split}
\label{eq:AlmostThereCarlemann}
\end{equation}
 Since, due to the construction of $\text{\textgreek{f}}$, we have
\begin{equation}
\sup_{\{r\le7r_{0}/8\}}w<\inf_{\{r_{0}\le r\le2r_{0}\}}w,
\end{equation}
 for $s=C(\text{\textgreek{w}}_{0},R)\cdot\text{\textgreek{w}}_{k}$
large enough the first term of the right hand side of (\ref{eq:AlmostThereCarlemann})
can be absorbed into the second term of the left hand side, yielding:

\begin{multline}
\int_{\mathcal{R}(0,t^{*})\cap\{\frac{7r_{0}}{8}\le r\le R\}}\Big(e^{2sw}+e^{2sw'}\Big)\cdot|\text{\textgreek{y}}_{k}|^{2}+\int_{\mathcal{R}(0,t^{*})\cap\{r\le R\}}\Big(e^{2sw}+e^{2sw'}\Big)\cdot|\partial\text{\textgreek{y}}_{k}|_{h}^{2}\le\\
\le C(\text{\textgreek{w}}_{0},R)\int_{\mathcal{R}(0,t^{*})\cap\{R\le r\le R+1\})}(1+s^{2})e^{2sw}\big\{|\partial\text{\textgreek{y}}_{k}|_{h}^{2}+|\text{\textgreek{y}}_{k}|^{2}\big\}+e^{C(R)\cdot s}\cdot C(\text{\textgreek{w}}_{0},R)\int_{t=0}J_{\text{\textgreek{m}}}^{N}(\text{\textgreek{y}})n^{\text{\textgreek{m}}}\label{eq:AlmostThereCarlemann-1}
\end{multline}

Estimating through a Hardy-type inequality

\begin{align}
\int_{\mathcal{R}(0,t^{*})\cap\{r\le\frac{7}{8}r_{0}\}}|\text{\textgreek{y}}_{k}|^{2} & =\int_{\mathcal{R}(0,t^{*})\cap\{r\le\frac{7}{8}r_{0}\}}|\text{\ensuremath{\tilde{\lyxmathsym{\textgreek{q}}}}}_{R}\text{\textgreek{y}}_{k}|^{2}\le\label{eq:HardyForthe0OrderTermsNearTheHorizon}\\
 & \le C(R)\int_{\mathcal{R}(0,t^{*})}\tilde{\text{\textgreek{q}}}_{R}|\partial\text{\textgreek{y}}_{k}|_{h}^{2}+C(R)\int_{\mathcal{R}(0,t^{*})}|\partial\tilde{\text{\textgreek{q}}}_{R}|_{h}^{2}|\text{\textgreek{y}}_{k}|^{2}\le\nonumber \\
 & \le C(R)\int_{\mathcal{R}(0,t^{*})\cap\{r\le R\}}|\partial\text{\textgreek{y}}_{k}|_{h}^{2}+C(R)\int_{\mathcal{R}(0,t^{*})\cap\{R\le r\le R+1\}}\big\{|\partial\text{\textgreek{y}}_{k}|_{h}^{2}+|\text{\textgreek{y}}_{k}|^{2}\big\},\nonumber 
\end{align}
after adding (\ref{eq:HardyForthe0OrderTermsNearTheHorizon}) and
(\ref{eq:AlmostThereCarlemann-1}) we obtain the desired inequality:

\begin{align}
\int_{\mathcal{R}(0,t^{*})\cap\{r\le R\}}\Big(e^{2sw}+e^{2sw'}\Big)\cdot\Big(|\text{\textgreek{y}}_{k}|^{2}+|\partial\text{\textgreek{y}}_{k}|_{h}^{2}\Big)\le & C(\text{\textgreek{w}}_{0},R)\int_{\mathcal{R}(0,t^{*})\cap\{R\le r\le R+1\})}(1+s^{2})e^{2sw}\big\{|\partial\text{\textgreek{y}}_{k}|_{h}^{2}+|\text{\textgreek{y}}_{k}|^{2}\big\}+\label{eq:FinalCarlemann}\\
 & +e^{C(R)\cdot s}\cdot C(\text{\textgreek{w}}_{0},R)\int_{t=0}J_{\text{\textgreek{m}}}^{N}(\text{\textgreek{y}})n^{\text{\textgreek{m}}}.\nonumber 
\end{align}

\end{proof}

\subsection{\label{sub:ILED}Integrated local energy decay for $\text{\textgreek{y}}_{\le\text{\textgreek{w}}_{+}}$}

In this section, we will remove the first term of the right hand side
of (\ref{eq:CarlemannTypeInequality}) from Lemma \ref{lem:Carlemann},
in order to obtain a genuine integrated local energy decay statement
for $\text{\textgreek{y}}_{k}$, $1\le|k|\le n$. 

The main argument involved in the proof of Proposition \ref{prop:ILEDPsiK}
originates in \cite{Rodnianski2011}. In particular, we will use an
ODE Lemma stated and proven in Sections 9 and 10 of \cite{Rodnianski2011}.
We will mainly work in the region $\{r\gg1\}$, where in the $(t,r,\text{\textgreek{sv}})$
coordinate chart, according to Assumption \hyperref[Assumption 1]{1}
(see also Section \ref{sub:Remark} on our simplifying assumption
that this region has a single connected component), the metric takes
the form

\begin{equation}
g=-\big(1-\frac{2M}{r}+O_{4}(r^{-1-a})\big)dt^{2}+\big(1+\frac{2M}{r}+O_{4}(r^{-1-a})\big)dr{}^{2}+r^{2}(g_{\mathbb{S}^{d-1}}+O_{4}(r^{-1-a}))+O_{4}(r^{-a})dtd\text{\textgreek{sv}}.\label{eq:Metric1}
\end{equation}

We will first establish the following Proposition:
\begin{prop}
\label{prop:ILEDPsiK}For any $R>0$ and $0<\text{\textgreek{w}}_{0}\ll1$,
there exists a positive constant $C=C(R,\text{\textgreek{w}}_{0})$
such that for any smooth solution $\text{\textgreek{y}}$ to the wave
equation on $J^{+}(\text{\textgreek{S}})\cap D$ with compactly supported
initial data on $\{t=0\}$, any $\text{\textgreek{w}}_{+}>1$ and
any $1\le|k|\le n$ we can bound:

\begin{equation}
\int_{\{r\le R\}\cap\mathcal{R}(0,t^{*})}\big\{ J_{\text{\textgreek{m}}}^{N}(\text{\textgreek{y}}_{k})n^{\text{\textgreek{m}}}+|\text{\textgreek{y}}_{k}|^{2}\big\}\le C(R,\text{\textgreek{w}}_{0})\cdot e^{C(R,\text{\textgreek{w}}_{0})\cdot\text{\textgreek{w}}_{+}}\int_{t=0}J_{\text{\textgreek{m}}}^{N}(\text{\textgreek{y}})n^{\text{\textgreek{m}}}.\label{eq:ILED}
\end{equation}

\end{prop}
\medskip{}

\noindent \emph{Proof.} We will assume without loss of generality
that $1\le k\le n$, since the case $-n\le k\le-1$ follows in exactly
the same way. We will also assume without loss of generality that
$\text{\textgreek{y}}_{k}$ is real valued (since otherwise we can
first establish (\ref{eq:ILED}) first for $Re(\text{\textgreek{y}}_{k})$
and $Im(\text{\textgreek{y}}_{k})$ and then add the resulting inequalities). 

The proof of the estimate (\ref{eq:ILED}) will follow from a number
of auxilliary lemmas that aim to provide us with control over the
boundary term near $r\sim R$ in the right hand side of (\ref{eq:CarlemannTypeInequality}).
To this end, we introduce a large constant $C_{1}=C_{1}(R,\text{\textgreek{w}}_{0})$,
the magnitude of which will be defined more precisely later in the
proof. We will examine the behavior of $\int_{\{r=\text{\textgreek{r}}\}\cap\mathcal{R}(0,t^{*})}|\text{\textgreek{y}}_{k}|^{2}$
in the region $\text{\textgreek{r}}\ge C_{1}$. 

We will make the following normalization, so that our notations are
in agreement with \cite{Rodnianski2011}: We will set (provided, of
course, $\text{\textgreek{y}}_{k}$ is not identically $0$ on $\mathcal{R}(0,t^{*})$,
in which case (\ref{eq:ILED}) would follow immediately) 
\begin{equation}
\tilde{\text{\textgreek{y}}}_{k}=\big(\int_{\mathcal{R}(0,t^{*})}(1+r)^{-2}|\text{\textgreek{y}}_{k}|^{2}\big)^{-1/2}\text{\textgreek{y}}_{k}\label{eq:RenormalisedPsi}
\end{equation}
 and 
\begin{equation}
\tilde{F}_{k}=\big(\int_{\mathcal{R}(0,t^{*})}(1+r)^{-2}|\text{\textgreek{y}}_{k}|^{2}\big)^{-1/2}F_{k}.\label{eq:RenormalisedF}
\end{equation}
Notice that $\tilde{\text{\textgreek{y}}}_{k}$ solves $\square_{g}\tilde{\text{\textgreek{y}}}_{k}=\tilde{F}_{k}$. 

We will also set, in order to stay close to the notations of \cite{Rodnianski2011}:
\begin{equation}
\mathfrak{D}\doteq\big(\int_{\mathcal{R}(0,t^{*})}r^{-2}|\text{\textgreek{y}}_{k}|^{2}\big)^{-1}\int_{t=0}J_{\text{\textgreek{m}}}^{N}(\text{\textgreek{y}})n^{\text{\textgreek{m}}}.\label{eq:Delta}
\end{equation}
Notice that all the previous quantities (\ref{eq:RenormalisedPsi}),
(\ref{eq:RenormalisedF}) and (\ref{eq:Delta}) are finite, since,
in view of the fact that the initial data for $\text{\textgreek{y}}$
on $\{t=0\}$ were assumed to be supported in a set of the form $\{r\le R_{sup}\}$,
$\text{\textgreek{y}}_{t^{*}}$ and $\text{\textgreek{y}}_{k},\text{\textgreek{y}}_{\le\text{\textgreek{w}}_{+}},\text{\textgreek{y}}_{\ge\text{\textgreek{w}}_{+}}$
are all supported in the cylinder $\{r\lesssim R_{sup}+t^{*}\}$.%
\footnote{Of course, as we remarked in Section \ref{sub:Frequency-cut-off},
no constant in what follows will be allowed to depend on $R_{sup},t^{*}$.%
}

Under the above normalizations, we compute that 
\[
\mathcal{\int}_{R(0,t^{*})}(1+r)^{-2}|\tilde{\text{\textgreek{y}}}_{k}|^{2}=1,
\]
and Lemma \ref{lem:BoundF} implies that for any $q\in\mathbb{N}$:
\[
\int_{\mathcal{R}(0,t^{*})}r^{q}\tilde{F}_{k}\le C_{q}(\text{\textgreek{w}}_{0})\cdot\mathfrak{D}.
\]
Moreover, in case $\mathfrak{D}\ge1$, the desired integrated local
energy decay statement (\ref{eq:ILED}) for $\text{\textgreek{y}}_{k}$
would readily follow: In this case, $\mathfrak{D}\ge1$ immediately
yields 
\begin{equation}
\int_{\mathcal{R}(0,t^{*})}(1+r)^{-2}|\text{\textgreek{y}}_{k}|^{2}\le\int_{t=0}J_{\text{\textgreek{m}}}^{N}(\text{\textgreek{y}})n^{\text{\textgreek{m}}}.
\end{equation}
Using the Lagrangean inequality 
\begin{equation}
\int_{\mathcal{R}(0,t^{*})}\text{\textgreek{q}}\cdot J_{\text{\textgreek{m}}}^{N}(\text{\textgreek{y}}_{k})n^{\text{\textgreek{m}}}\le C(\text{\textgreek{w}}_{0},\text{\textgreek{q}})\cdot(1+\text{\textgreek{w}}_{k}^{2})\cdot\int_{\mathcal{R}(0,t^{*})\cap supp(\text{\textgreek{q}})}|\text{\textgreek{y}}_{k}|^{2}+C(\text{\textgreek{w}}_{0},\text{\textgreek{q}})\cdot\int_{t=0}J_{\text{\textgreek{m}}}^{N}(\text{\textgreek{y}})n^{\text{\textgreek{m}}},
\end{equation}
 which holds for any compactly supported cut-off $\text{\textgreek{q}}$
and is easily obtained by using the current $J_{\text{\textgreek{m}}}=\text{\textgreek{q}}\cdot Re\big\{\text{\textgreek{y}}_{k}\cdot\partial_{\text{\textgreek{m}}}\bar{\text{\textgreek{y}}}_{k}\big\}-(\partial_{\text{\textgreek{m}}}\text{\textgreek{q}})\cdot|\text{\textgreek{y}}_{k}|^{2}$
(similarly with the extraction of (\ref{eq:LagrangeIdentity})), as
well as Lemmas \ref{lem:BoundF}, \ref{lem:BoundednessPsiK} and \ref{lem:DtToOmegaInequalities2}
(the latter enabling us to bound any $|T\text{\textgreek{y}}_{k}|^{2}$
bulk terms by $\text{\textgreek{w}}_{k}^{2}|\text{\textgreek{y}}_{k}|^{2}$
terms), one easily infers (\ref{eq:ILED}). Therefore, we will assume
without loss of generality that 
\begin{equation}
\mathfrak{D}<1.\label{eq:AssumptionDelta}
\end{equation}

We will also set 
\begin{equation}
\tilde{\text{\textgreek{Y}}}_{k}=D\cdot\text{\ensuremath{\tilde{\lyxmathsym{\textgreek{y}}}}}_{k},
\end{equation}
where 
\begin{equation}
D\doteq\big(\sqrt{-det(g)}\cdot g^{rr}\big)^{\frac{1}{2}}=r^{\frac{d-1}{2}}\big(1+O_{4}(r^{-1})\big).
\end{equation}
The factor $D$ was chosen so that $\tilde{\text{\textgreek{Y}}}_{k}$
satisfies in the region $\{r\ge C_{1}\}$ an equation of the form 

\begin{align}
D\cdot\tilde{F}_{k}=-\big(1+ & a_{tt}\big)\partial_{t}^{2}\tilde{\text{\textgreek{Y}}}_{k}+\big(1+a_{rr}\big)\partial_{r}^{2}\tilde{\text{\textgreek{Y}}}_{k}+r^{-2}\cdot\big(\text{\textgreek{D}}_{g_{\mathbb{S}^{d-1}}+O_{4}(r^{-1-a})}-\frac{(d-1)(d-3)}{4}\big)\tilde{\text{\textgreek{Y}}}_{k}+\label{eq:Equation1}\\
 & +a_{t\text{\textgreek{sv}}}\cdot\partial_{t}\partial_{\text{\textgreek{sv}}}\tilde{\text{\textgreek{Y}}}_{k}+a_{\text{\textgreek{sv}}}\cdot\partial_{\text{\textgreek{sv}}}\tilde{\text{\textgreek{Y}}}_{k}+a_{t}\cdot\partial_{t}\tilde{\text{\textgreek{Y}}}_{k}+a\cdot\tilde{\text{\textgreek{Y}}}_{k},\nonumber 
\end{align}
where
\begin{itemize}
\item $a_{tt}=\frac{2M}{r}+O_{4}(r^{-1-a})$, 
\item $a_{rr}=-\frac{2M}{r}+O_{4}(r^{-1-a})$, 
\item $a_{t\text{\textgreek{sv}}}=O_{4}(r^{-2-a})$,
\item $a_{\text{\textgreek{sv}}}=O_{3}(r^{-3-a})$, 
\item $a_{t}=O_{3}(r^{-1-a})$ and 
\item $a=O_{2}(r^{-2-a})$. 
\end{itemize}
Note that the $\partial_{r}$ derivatives appear only in the $\partial_{r}^{2}\tilde{\text{\textgreek{Y}}}_{k}$
term in this expression (owing to the normal form of the metric (\ref{eq:metric})).

Let us define for $r\ge C_{1}$ the following \emph{spherical energies}
(see \cite{Rodnianski2011}):
\begin{itemize}
\item $\mathscr{M}(\text{\textgreek{r}})\doteq\int_{\{r=\text{\textgreek{r}}\}\cap\mathcal{R}(0,t^{*})}|\tilde{\text{\textgreek{Y}}}_{k}|^{2}\, dg_{\mathbb{S}^{d-1}}dt.$ 
\item $\mathscr{F}(\text{\textgreek{r}})\doteq\int_{\{r=\text{\textgreek{r}}\}\cap\mathcal{R}(0,t^{*})}\tilde{\text{\textgreek{Y}}}_{k}\cdot\partial_{r}\tilde{\text{\textgreek{Y}}}_{k}\, dg_{\mathbb{S}^{d-1}}dt.$
\item $\mathscr{R}(\text{\textgreek{r}})\doteq\int_{\{r=\text{\textgreek{r}}\}\cap\mathcal{R}(0,t^{*})}|\partial_{r}\tilde{\text{\textgreek{Y}}}_{k}|^{2}\, dg_{\mathbb{S}^{d-1}}dt.$
\item $\mathscr{A}(\text{\textgreek{r}})\doteq\int_{\{r=\text{\textgreek{r}}\}\cap\mathcal{R}(0,t^{*})}|\frac{1}{r}\partial_{\text{\textgreek{sv}}}\tilde{\text{\textgreek{Y}}}_{k}|^{2}\, dg_{\mathbb{S}^{d-1}}dt.$
\end{itemize}
Notice that the measure of integration used is not the geometric one,
but rather the ``coordinate'' one.

We will also need the following Pohozaev-type flux $\mathcal{\mathscr{P}}$,
which is defined as 
\begin{multline}
\mathcal{\mathscr{P}}(\text{\textgreek{r}})=\int_{\{r=\text{\textgreek{r}}\}\cap\mathcal{R}(0,t^{*}\}}\Big\{|\partial_{r}\tilde{\text{\textgreek{Y}}}_{k}|^{2}+(1+a_{tt})(1+a_{rr})^{-1}|\partial_{t}\tilde{\text{\textgreek{Y}}}_{k}|^{2}-(1+a_{rr})^{-1}r^{-2}\big(|\partial_{\text{\textgreek{sv}}}\tilde{\text{\textgreek{Y}}}_{k}|^{2}+\frac{(d-1)(d-3)}{4}|\tilde{\text{\textgreek{Y}}}_{k}|^{2}\big)+\\
-(1+a_{rr})^{-1}a_{t\text{\textgreek{sv}}}\partial_{t}\tilde{\text{\textgreek{Y}}}_{k}\partial_{\text{\textgreek{sv}}}\tilde{\text{\textgreek{Y}}}_{k}+\big((1+a_{rr})^{-1}a_{\text{\textgreek{sv}}}-r^{-2}\partial_{\text{\textgreek{sv}}}(1+a_{rr})^{-1}\big)\cdot\tilde{\text{\textgreek{Y}}}_{k}\partial_{\text{\textgreek{sv}}}\tilde{\text{\textgreek{Y}}}_{k}+a_{t}\tilde{\text{\textgreek{Y}}}_{k}\partial_{t}\tilde{\text{\textgreek{Y}}}_{k}+a\cdot|\tilde{\text{\textgreek{Y}}}_{k}|^{2}\Big\}\, dg_{\mathbb{S}^{d-1}}dt.\label{eq:PohazevFlux}
\end{multline}
The above integrand can be obtained from (\ref{eq:Equation1}) by
applying multiplying the right hand side of equation (\ref{eq:Equation1})
with $(1+a_{rr})^{-1}\tilde{\text{\textgreek{Y}}}_{k}$ and then formally
integrating by parts in all second order terms schematically as: $a\cdot\partial^{2}\tilde{\text{\textgreek{Y}}}_{k}\cdot\tilde{\text{\textgreek{Y}}}_{k}\rightarrow-\partial\tilde{\text{\textgreek{Y}}}_{k}\cdot\partial(a\cdot\tilde{\text{\textgreek{Y}}}_{k})$,
after first multiplying the the $\partial_{r}^{2}$ term with $-1$. 

Notice that if $C_{1}\gg1$, then 
\begin{equation}
\mathcal{\mathscr{P}}+2\mathcal{\mathscr{A}}+r^{-2-a}\cdot\mathcal{\mathscr{M}}\ge c\cdot\int_{\{r=\text{\textgreek{r}}\}\cap\mathcal{R}(0,t^{*})}(|\partial_{t}\tilde{\text{\textgreek{Y}}}_{k}|^{2}+|\partial_{r}\tilde{\text{\textgreek{Y}}}_{k}|^{2}+|r^{-1}\partial_{\text{\textgreek{sv}}}\tilde{\text{\textgreek{Y}}}_{k}|^{2})\, dg_{\mathbb{S}^{d-1}}dt.
\end{equation}

We will denote with $G$ any function $G:[0,+\infty)\rightarrow\mathbb{R}$
defined in terms of $\tilde{\text{\textgreek{Y}}}_{k}$, such that
there exist $r_{1}\ge0\mbox{ and }C(\text{\textgreek{w}}_{0},R)>0$
for which we can bound 
\begin{equation}
\int_{r_{1}}^{\infty}|G(r)|\, dr\le C(\text{\textgreek{w}}_{0},R)\cdot\mathfrak{D},
\end{equation}
with the constant $C(\text{\textgreek{w}}_{0},R)$ depending only
on $\text{\textgreek{w}}_{0},R$ (and thus independent of $\text{\textgreek{y}},\text{\textgreek{y}}_{k},t^{*}$
etc.) 

It readily follows from the definitions that 
\begin{equation}
\mathcal{\mathscr{M}},\mathcal{\mathscr{A}},\mathcal{\mathscr{R}}\ge0.
\end{equation}
Furthermore, using a Cauchy--Schwarz inequality, we can bound:
\begin{equation}
|\mathcal{\mathscr{F}}|\le(\mathscr{M})^{1/2}\cdot(\mathcal{\mathscr{R}})^{1/2}.\label{eq:Cauchy-Schwartz}
\end{equation}
 Finally, since $\text{\textgreek{y}}_{k}$ is supported in the cylinder
$\{r\lesssim R_{sup}+t^{*}\}$, we have:
\begin{equation}
\lim_{r\rightarrow\infty}\frac{1}{r}\int_{r}^{2r}\big\{\mathcal{\mathscr{M}}(\text{\textgreek{r}})+\mathcal{\mathscr{R}}(\text{\textgreek{r}})+\mathcal{\mathscr{A}}(\text{\textgreek{r}})+|\mathcal{\mathscr{P}}(\text{\textgreek{r}})|\big\}\, d\text{\textgreek{r}}=0.\label{eq:InitialDataInfinity}
\end{equation}

The main result of this section (i.\,e.~(\ref{eq:ILED})) is a consequence
of the fact that the spherical energies obey a particular system of
equations of motion (see also \cite{Rodnianski2011}):
\begin{lem}
The spherical energies satisfy the following system of ODE's:

\begin{equation}
\begin{cases}
\frac{d}{dr}\mathcal{\mathscr{M}}=2\mathcal{\mathscr{F}}\\
\frac{d}{dr}\mathcal{\mathscr{F}}=2\mathscr{R}-\mathcal{\mathscr{P}}+G+O(r^{-3})\cdot\mathcal{\mathscr{M}}\\
\frac{d}{dr}\mathcal{\mathscr{P}}=\frac{2}{r}\mathcal{\mathscr{A}}+O(r^{-3-a})\cdot\mathcal{\mathscr{M}}+O(r^{-1-a})\{\mathcal{\mathscr{P}}+2\mathcal{\mathscr{A}}\}+G.
\end{cases}\label{eq:SystemODE's}
\end{equation}
 The constants in the Big $O$ notation are allowed to depend on $\text{\textgreek{w}}_{0},R$.\end{lem}
\begin{proof}
The first equation of (\ref{eq:SystemODE's}) follows readily due
to the particular choice of measure on $\{r=const\}\cap\mathcal{R}(0,t^{*})$
slices, which makes it invariant under the flow of $\partial_{r}$.
Namely, 
\begin{equation}
\frac{d}{d\text{\textgreek{r}}}\Big(\int_{\{r=\text{\textgreek{r}}\}\cap\mathcal{R}(0,t^{*})}|\tilde{\text{\textgreek{Y}}}_{k}|^{2}\, dg_{\mathbb{S}^{d-1}}dt\Big)=2\int_{\{r=\text{\textgreek{r}}\}\cap\mathcal{R}(0,t^{*})}\partial_{r}\tilde{\text{\textgreek{Y}}}_{k}\cdot\tilde{\text{\textgreek{Y}}}_{k}\, dg_{\mathbb{S}^{d-1}}dt.
\end{equation}

In order to extract the second equation of (\ref{eq:SystemODE's}),
we must use the fact that $\tilde{\text{\textgreek{Y}}}_{k}$ satisfies
equation (\ref{eq:Equation1}). In particular, we compute 
\begin{equation}
\frac{d}{d\text{\textgreek{r}}}\Big(\int_{\{r=\text{\textgreek{r}}\}\cap\mathcal{R}(0,t^{*})}\partial_{r}\tilde{\text{\textgreek{Y}}}_{k}\cdot\tilde{\text{\textgreek{Y}}}_{k}\, dg_{\mathbb{S}^{d-1}}dt\Big)=\int_{\{r=\text{\textgreek{r}}\}\cap\mathcal{R}(0,t^{*})}\big(|\partial_{r}\tilde{\text{\textgreek{Y}}}_{k}|^{2}+\tilde{\text{\textgreek{Y}}}_{k}\cdot\partial_{r}^{2}\tilde{\text{\textgreek{Y}}}_{k}\big)\, dg_{\mathbb{S}^{d-1}}dt.\label{eq:FirstDrDerivative}
\end{equation}
Since we can use equation (\ref{eq:Equation1}) to replace $\partial_{r}^{2}\tilde{\text{\textgreek{Y}}}_{k}$
with 
\begin{equation}
\begin{split}(1+a_{rr})^{-1}\Big\{ & D\cdot\tilde{F}_{k}+(1+a_{tt})\partial_{t}^{2}\tilde{\text{\textgreek{Y}}}_{k}-r^{-2}\big(\text{\textgreek{D}}_{g_{\mathbb{S}^{d-1}}+O_{4}(r^{-1-a})}\tilde{\text{\textgreek{Y}}}_{k}-\frac{(d-1)(d-3)}{4}\tilde{\text{\textgreek{Y}}}_{k}\big)-\\
 & -a_{t\text{\textgreek{sv}}}\partial_{t}\partial_{\text{\textgreek{sv}}}\tilde{\text{\textgreek{Y}}}_{k}-a_{\text{\textgreek{sv}}}\partial_{\text{\textgreek{sv}}}\tilde{\text{\textgreek{Y}}}_{k}-a_{t}\partial_{t}\tilde{\text{\textgreek{Y}}}_{k}-a\cdot\tilde{\text{\textgreek{Y}}}_{k}\Big\},
\end{split}
\label{eq:ExpressionDrSquarePsiK}
\end{equation}
 we calculate from (\ref{eq:FirstDrDerivative}) that: 
\begin{equation}
\begin{split}\frac{d}{d\text{\textgreek{r}}}\Big(\int_{\{r=\text{\textgreek{r}}\}\cap\mathcal{R}(0,t^{*})}\partial_{r}\tilde{\text{\textgreek{Y}}}_{k}\cdot\tilde{\text{\textgreek{Y}}}_{k}\, dg_{\mathbb{S}^{d-1}} & dt\Big)=\int_{\{r=\text{\textgreek{r}}\}\cap\mathcal{R}(0,t^{*})}\Big(|\partial_{r}\tilde{\text{\textgreek{Y}}}_{k}|^{2}+(1+a_{rr})^{-1}\tilde{\text{\textgreek{Y}}}_{k}\cdot\Big\{ D\cdot\tilde{F}_{k}+(1+a_{tt})\partial_{t}^{2}\tilde{\text{\textgreek{Y}}}_{k}-\\
 & -r^{-2}\big(\text{\textgreek{D}}_{g_{\mathbb{S}^{d-1}}+O_{4}(r^{-1-a})}\tilde{\text{\textgreek{Y}}}_{k}-\frac{(d-1)(d-3)}{4}\tilde{\text{\textgreek{Y}}}_{k}\big)-a_{t\text{\textgreek{sv}}}\partial_{t}\partial_{\text{\textgreek{sv}}}\tilde{\text{\textgreek{Y}}}_{k}-a_{\text{\textgreek{sv}}}\partial_{\text{\textgreek{sv}}}\tilde{\text{\textgreek{Y}}}_{k}-\\
 & -a_{t}\partial_{t}\tilde{\text{\textgreek{Y}}}_{k}-a\cdot\tilde{\text{\textgreek{Y}}}_{k}\Big\}\Big)\, dg_{\mathbb{S}^{d-1}}dt.
\end{split}
\label{eq:BoundBeforePohozaev}
\end{equation}
 Hence, after integrating by parts in $t$ and $\lyxmathsym{\textgreek{sv}}$
in order to we eliminate the second order derivatives, we conclude
in view of (\ref{eq:PohazevFlux}) Lemma \ref{lem:BoundednessPsiK}%
\footnote{which shows with the boundary terms that result from integration by
parts in $t$ sum up to a $G$-type function of $r$%
} that:

\begin{equation}
\frac{d}{d\text{\textgreek{r}}}\Big(\int_{\{r=\text{\textgreek{r}}\}\cap\mathcal{R}(0,t^{*})}\partial_{r}\tilde{\text{\textgreek{Y}}}_{k}\cdot\tilde{\text{\textgreek{Y}}}_{k}\, dg_{\mathbb{S}^{d-1}}dt\Big)=2\mathcal{\mathscr{R}}-\mathcal{\mathscr{P}}+G+\int_{\{r=\text{\textgreek{r}}\}\cap\mathcal{R}(0,t^{*})}D\cdot|\tilde{F}_{k}|\cdot|\tilde{\text{\textgreek{Y}}}_{k}|\, dg_{\mathbb{S}^{d-1}}dt.\label{eq:BeforeSecondEquation}
\end{equation}
The second equation of (\ref{eq:SystemODE's}) now follows from (\ref{eq:BeforeSecondEquation}),
after applying a Cauchy--Schwarz inequality for the last term of the
right hand side and using Lemma \ref{lem:BoundF} to show that $\int_{\{r=\text{\textgreek{r}}\}\cap R(0,t^{*})}r^{3}D^{2}|\tilde{F}_{k}|^{2}\, dg_{\mathbb{S}^{d-1}}dt$
is a $G$-type function in $\text{\textgreek{r}}$. 

For the extraction of the third equation of (\ref{eq:SystemODE's}),
we work in a similar way. We calculate from the expression (\ref{eq:PohazevFlux})
for $\mathcal{\mathscr{P}}(\text{\textgreek{r}})$ that

\begin{align}
\frac{d}{d\text{\textgreek{r}}}\mathcal{\mathscr{P}}(\text{\textgreek{r}})= & \int_{\{r=\text{\textgreek{r}}\}\cap\mathcal{R}(0,t^{*}\}}2r^{-3}(1+O(r^{-1}))\big(|\partial_{\text{\textgreek{sv}}}\tilde{\text{\textgreek{Y}}}_{k}|^{2}+\frac{(d-1)(d-3)}{4}|\tilde{\text{\textgreek{Y}}}_{k}|^{2}\big)\, dg_{\mathbb{S}^{2}}dt+\label{eq:PohazevFlux-1}\\
 & +\int_{\{r=\text{\textgreek{r}}\}\cap\mathcal{R}(0,t^{*}\}}\Big\{2\partial_{r}^{2}\tilde{\text{\textgreek{Y}}}_{k}\cdot\partial_{r}\tilde{\text{\textgreek{Y}}}_{k}+2(1+a_{tt})(1+a_{rr})^{-1}\partial_{r}\partial_{t}\tilde{\text{\textgreek{Y}}}_{k}\cdot\partial_{t}\tilde{\text{\textgreek{Y}}}_{k}-\nonumber \\
 & \hphantom{+\int_{\{r=\text{\textgreek{r}}\}\cap\mathcal{R}(0,t^{*}\}}\Big\{}-2(1+a_{rr})^{-1}r^{-2}\big(g_{\mathbb{S}^{d-1}}(\nabla^{\mathbb{S}^{d-1}}\partial_{r}\tilde{\text{\textgreek{Y}}}_{k},\nabla^{\mathbb{S}^{d-1}}\tilde{\text{\textgreek{Y}}}_{k})+\frac{(d-1)(d-3)}{4}\partial_{r}\tilde{\text{\textgreek{Y}}}_{k}\cdot\tilde{\text{\textgreek{Y}}}_{k}\big)-\nonumber \\
 & \hphantom{+\int_{\{r=\text{\textgreek{r}}\}\cap\mathcal{R}(0,t^{*}\}}\Big\{}-(1+a_{rr})^{-1}a_{t\text{\textgreek{sv}}}\partial_{r}\partial_{t}\tilde{\text{\textgreek{Y}}}_{k}\cdot\partial_{\text{\textgreek{sv}}}\tilde{\text{\textgreek{Y}}}_{k}-(1+a_{rr})^{-1}a_{t\text{\textgreek{sv}}}\partial_{t}\tilde{\text{\textgreek{Y}}}_{k}\cdot\partial_{r}\partial_{\text{\textgreek{sv}}}\tilde{\text{\textgreek{Y}}}_{k}+\\
 & \hphantom{+\int_{\{r=\text{\textgreek{r}}\}\cap\mathcal{R}(0,t^{*}\}}\Big\{}+\big((1+a_{rr})^{-1}a_{\text{\textgreek{sv}}}-r^{-2}\partial_{\text{\textgreek{sv}}}(1+a_{rr})^{-1}\big)\tilde{\text{\textgreek{Y}}}_{k}\cdot\partial_{r}\partial_{\text{\textgreek{sv}}}\tilde{\text{\textgreek{Y}}}_{k}+a_{t}\tilde{\text{\textgreek{Y}}}_{k}\cdot\partial_{r}\partial_{t}\tilde{\text{\textgreek{Y}}}_{k}+\nonumber \\
 & \hphantom{+\int_{\{r=\text{\textgreek{r}}\}\cap\mathcal{R}(0,t^{*}\}}\Big\{}+O(r^{-3-a})\partial_{r}\tilde{\text{\textgreek{Y}}}_{k}\cdot\partial_{\text{\textgreek{sv}}}\tilde{\text{\textgreek{Y}}}_{k}+O(r^{-1-a})\partial_{r}\tilde{\text{\textgreek{Y}}}_{k}\cdot\partial_{t}\tilde{\text{\textgreek{Y}}}_{k}+O(r^{-2-a})\partial_{r}\tilde{\text{\textgreek{Y}}}_{k}\cdot\tilde{\text{\textgreek{Y}}}_{k}+\nonumber \\
 & \hphantom{+\int_{\{r=\text{\textgreek{r}}\}\cap\mathcal{R}(0,t^{*}\}}\Big\{}+O(r^{-2-a})|\partial_{t}\tilde{\text{\textgreek{Y}}}_{k}|^{2}+O(r^{-3-a})\partial_{t}\tilde{\text{\textgreek{Y}}}_{k}\cdot\partial_{\text{\textgreek{sv}}}\tilde{\text{\textgreek{Y}}}_{k}+O(r^{-4-a})\tilde{\text{\textgreek{Y}}}_{k}\cdot\partial_{\text{\textgreek{sv}}}\tilde{\text{\textgreek{Y}}}_{k}+\nonumber \\
 & \hphantom{+\int_{\{r=\text{\textgreek{r}}\}\cap\mathcal{R}(0,t^{*}\}}\Big\{}+O(r^{-2-a})\tilde{\text{\textgreek{Y}}}_{k}\cdot\partial_{t}\tilde{\text{\textgreek{Y}}}_{k}+O(r^{-3-a})|\tilde{\text{\textgreek{Y}}}_{k}|^{2}\Big\}\, dg_{\mathbb{S}^{d-1}}dt.\nonumber 
\end{align}
Using equation (\ref{eq:Equation1}) to replace $\partial_{r}^{2}\tilde{\text{\textgreek{Y}}}_{k}$
with (\ref{eq:ExpressionDrSquarePsiK}) in (\ref{eq:PohazevFlux-1})
and then integrating by parts in $t$ and $\text{\textgreek{sv}}$
in the highest order terms, then all the second order derivatives
of $\tilde{\text{\textgreek{Y}}}_{k}$ and $\tilde{\text{\textgreek{Y}}}_{k}$
in (\ref{eq:PohazevFlux-1}) cancel out. Thus, the third equation
of (\ref{eq:SystemODE's}) readily follows, after, of course, treating
the boundary terms resulting from the integration by parts, as well
as the $\tilde{F}_{k}$ terms, using Lemmas \ref{lem:BoundednessPsiK}
and \ref{lem:BoundF}.\end{proof}
\begin{lem}
The Pohozaev flux $\mathcal{\mathscr{P}}$ satisfies in $\{r\ge C_{1}\}$
for $C_{1}$ suitably large:

\begin{equation}
\mathcal{\mathscr{P}}\lesssim\mathfrak{D}+O(r^{-2-a})\cdot\mathcal{\mathscr{M}}.\label{eq:PohazevBound}
\end{equation}
 \end{lem}
\begin{proof}
If we set $\mathcal{\mathscr{P}}^{*}\doteq\mathcal{\mathscr{P}}-C_{0}r^{-2-a}\mathcal{\mathscr{M}}$,
for a positive constant $C_{0}$ to be determined, we calculate from
(\ref{eq:SystemODE's}):

\begin{align}
\frac{d}{dr}\mathscr{P}^{*}= & \frac{2}{r}\mathscr{A}+O(r^{-3-a})\mathcal{\mathscr{M}}+O(r^{-1-a})\{\mathcal{\mathscr{P}}^{*}+C_{0}r^{-2-a}\mathcal{\mathscr{M}}+2\mathcal{\mathscr{A}}\}+G+\label{eq:AlmostPohazevBound}\\
 & +(2+a)C_{0}r^{-3-a}\mathcal{\mathscr{M}}-2C_{0}r^{-2-a}\mathcal{\mathscr{F}}+C_{0}O(r^{-4-2a})\mathcal{\mathscr{M}}.\nonumber 
\end{align}
 By bounding 
\begin{equation}
|\mathcal{\mathscr{F}}|\le r^{-1}\mathcal{\mathscr{M}}+r\mathcal{\mathscr{R}}\le r^{-1}(1+C_{0}\cdot O(r^{-a}))\mathcal{\mathscr{M}}+Cr(1+O(r^{-a}))\mathcal{\mathscr{A}}+Cr\mathscr{P}^{*}
\end{equation}
and substituting in (\ref{eq:AlmostPohazevBound}), we obtain

\begin{align}
\frac{d}{dr}\mathcal{\mathscr{P}}^{*}\ge & \frac{2}{r}\{1+O(r^{-a})-C_{0}(r^{-a}+O(r^{-2a}))\}\mathcal{\mathscr{A}}+\Big(C_{0}r^{-3-a}(1+O(r^{-1}))+O(r^{-3-a})+C_{0}\cdot O(r^{-3-2a})\Big)\mathcal{\mathscr{M}}+\label{eq:AlmostPohazevBound2}\\
 & +C_{0}\cdot O(r^{-1-a})\mathcal{\mathscr{P}}^{*}+G.\nonumber 
\end{align}

Hence, if $C_{0}$ is chosen sufficiently large, and $C_{1}$ (and
thus $r$) is large enough, both in terms of the geometry of $(\mathcal{D},g)$,
the coefficients in front of the $\mathcal{\mathscr{M}}$ and $\mathcal{\mathscr{A}}$
terms in (\ref{eq:AlmostPohazevBound2}) will be positive, resulting
in the inequality

\begin{equation}
\frac{d}{dr}\mathscr{P}^{*}\ge O(r^{-1-a})\mathcal{\mathscr{P}}^{*}+G,
\end{equation}
 from which (\ref{eq:PohazevBound}) follows by an application of
Gronwall's inequality (in view of the initial conditions (\ref{eq:InitialDataInfinity})
at infinity, as well as the definition of $\mathscr{P}^{*}$).
\end{proof}
\medskip{}

In view of Lemma \ref{lem:DtToOmegaInequalities} (notice that we
can allow $R\rightarrow\infty$ in the statement of that lemma, since
at this point we only care about the left hand side of (\ref{eq:w-estimate})
which is independent of $R$), we can bound from below 
\begin{equation}
\int_{\{r=\text{\textgreek{r}}\}\cap\mathcal{R}(0,t^{*})}|\partial_{t}\tilde{\text{\textgreek{Y}}}_{k}|^{2}\, dg_{\mathbb{S}^{2}}dt\ge c\cdot\text{\textgreek{w}}_{k-1}^{2}\mathscr{M}+G.
\end{equation}
Hence, from (\ref{eq:PohazevFlux}) we can bound from below: 
\begin{equation}
\mathcal{\mathscr{A}}\ge(1+O(r^{-1}))(\mathcal{\mathscr{R}}-\mathcal{\mathscr{P}})+c\cdot\text{\textgreek{w}}_{k-1}^{2}\mathcal{\mathscr{M}}+O(r^{-2-a})\mathcal{\mathscr{M}}.\label{eq:BoundFromFrequency}
\end{equation}

Moreover, in view of (\ref{eq:Cauchy-Schwartz}), we can eliminate
the quantity $\mathscr{R}$ from (\ref{eq:SystemODE's}) through the
estimate 
\begin{equation}
\mathcal{\mathscr{R}}\ge\frac{\mathcal{\mathscr{F}}^{2}}{\mathcal{\mathscr{M}}}.\label{eq:BoundFromCS}
\end{equation}
Hence, substituting (\ref{eq:BoundFromFrequency}) and (\ref{eq:BoundFromCS})
in (\ref{eq:SystemODE's}), we obtain the following system of differential
inequalities:

\begin{equation}
\begin{cases}
\frac{d}{dr}\mathcal{\mathscr{M}}=2\mathcal{\mathscr{F}}\\
\frac{d}{dr}\mathscr{F}\ge2\frac{\mathcal{\mathscr{F}}^{2}}{\mathcal{\mathscr{M}}}-\mathcal{\mathscr{P}}+G+O(r^{-3})\mathcal{\mathscr{M}}\\
\frac{d}{dr}\mathcal{\mathscr{P}}\ge\frac{2}{r}(1+O(r^{-1}))\cdot\Big(\frac{\mathcal{\mathscr{F}}^{2}}{\mathcal{\mathscr{M}}}-\mathcal{\mathscr{P}}+c\cdot\text{\textgreek{w}}_{k-1}^{2}\mathcal{\mathscr{M}}\Big)+O(r^{-3-a})\mathcal{\mathscr{M}}+O(r^{-1-a})\mathcal{\mathscr{P}}+G.
\end{cases}\label{eq:SystemODEinequalities}
\end{equation}
Note that in the third equation of (\ref{eq:SystemODEinequalities}),
the last term of the right hand side can also be replaced for free
by $O(r^{-1}+\text{\textgreek{w}}_{k-1})\cdot G$, since we have $\text{\textgreek{w}}_{0}\le\text{\textgreek{w}}_{k-1}$
and we allow the constants in the big $O$ notation to depend on $\text{\textgreek{w}}_{0}$.
Therefore, we see that we can immediately apply on the above system
of differential inequalities the ODE lemma from Section 10 in \cite{Rodnianski2011},
which can be stated as follows:
\begin{lem}
\label{thm:ODELemma} (Rodnianski--Tao) Let $\mathcal{\mathscr{M}},\mathcal{\mathscr{F}},\mathcal{\mathscr{P}}:[C_{1},+\infty)\rightarrow(0,+\infty)$
satisfy (\ref{eq:InitialDataInfinity}), (\ref{eq:PohazevBound})
and (\ref{eq:SystemODEinequalities}). Then, provided $C_{1}$ is
large enough (independently of $\text{\textgreek{w}}_{k-1}$), for
any positive constant $C_{2}\gg C_{1}$ (again with magnitude independent
of $\text{\textgreek{w}}_{k-1}$) there exists another positive constant
$C(C_{1},C_{2})$ such that one of the following scenarios holds:

1. Boundedness Scenario: There exists an $R_{0}$ with $C_{1}<R_{0}<C(C_{1},C_{2})$
such that
\begin{equation}
\mathcal{\mathscr{M}}(r)\le C(C_{1},C_{2})\big(\text{\textgreek{w}}_{k-1}^{-2C(C_{1},C_{2})}+1)\mathfrak{D}\label{eq:Boundedness}
\end{equation}
 for all $r\in[\frac{1}{2}R_{0},4R_{0}]$

2. Exponential growth from infinity Scenario: For all $C_{1}\le r\le10C_{1}$
we can bound 
\begin{equation}
\frac{d}{dr}\mathcal{\mathscr{M}}(r)\le-C_{2}(1+\text{\textgreek{w}}_{k-1})\mathcal{\mathscr{M}}(r).\label{eq:ExponentialGrowth}
\end{equation}

\end{lem}
Using Lemma \ref{thm:ODELemma} we will suitably absorb the first
term of the right hand side of (\ref{eq:CarlemannTypeInequality}),
thus obtaining the desired integrated local energy decay statement
(\ref{eq:ILED}). 

We first fix $C_{1}=R$, assuming of course without loss of generality
that $R$ was large enough in terms of the geometry of $(\mathcal{D},g)$.
We then choose $C_{2}=C_{2}(C_{1},\text{\textgreek{w}}_{0})$ large
enough in terms of $C_{1}$ and $\text{\textgreek{w}}_{0}$ so that
\begin{equation}
C_{2}\gg_{\text{\textgreek{w}}_{0},R}C(\text{\textgreek{w}}_{0},R)\cdot\sup_{C_{1}\le r\le10C_{1}}w,
\end{equation}
where $C(\text{\textgreek{w}}_{0},R)$ is the constant appearing in
the $s$ parameter in inequality (\ref{eq:CarlemannTypeInequality}).
In this way, we can also bound (since $s=C\cdot\text{\textgreek{w}}_{k}$
in (\ref{eq:CarlemannTypeInequality}), and $\text{\textgreek{w}}_{k}\le\text{\textgreek{w}}_{k-1}+1$):

\begin{equation}
C_{2}\cdot(1+\text{\textgreek{w}}_{k-1})-s\cdot\sup_{C_{1}\le r\le10C_{1}}w\gg_{\text{\textgreek{w}}_{0},R}\text{\textgreek{w}}_{k-1}\label{eq:exponentBound}
\end{equation}

In view of Lemma \ref{thm:ODELemma}, there are only two possible
scenarios for $\mathcal{\mathscr{M}}$:
\begin{casenv}
\item Assume that the boundedness scenario in Lemma \ref{thm:ODELemma}
holds, that is (\ref{eq:Boundedness}) is true. From (\ref{eq:CarlemannTypeInequality})
we obtain (replacing $R$ in (\ref{eq:CarlemannTypeInequality}) by
the value $R_{0}$ which is provided by the boundedness scenario):
\begin{equation}
\begin{split}\int_{\mathcal{R}(0,t^{*})\cap\{r\le R_{0}\}}\Big(e^{2sw}+e^{2sw'}\Big)\cdot\Big(J_{\text{\textgreek{m}}}^{N}(\text{\textgreek{y}}_{k})n^{\text{\textgreek{m}}} & +|\text{\textgreek{y}}_{k}|^{2}\Big)\le\\
\le & C(\text{\textgreek{w}}_{0},R_{0})\int_{\mathcal{R}(0,t^{*})\cap\{R_{0}\le r\le2R_{0}\})}(1+s^{2})e^{2sw}\big\{ J_{\text{\textgreek{m}}}^{N}(\text{\textgreek{y}}_{k})n^{\text{\textgreek{m}}}+|\text{\textgreek{y}}_{k}|^{2}\big\}+\\
 & +e^{C(R_{0})\cdot s}\cdot C(\text{\textgreek{w}}_{0},R_{0})\int_{t=0}J_{\text{\textgreek{m}}}^{N}(\text{\textgreek{y}})n^{\text{\textgreek{m}}}.
\end{split}
\label{eq:BoundednessCarlemann}
\end{equation}
Furthermore, applying (as we did earlier) the divergence identity
for the current $J_{\text{\textgreek{m}}}=\text{\textgreek{q}}\cdot\text{\textgreek{y}}_{k}\cdot\partial_{\text{\textgreek{m}}}\text{\textgreek{y}}_{k}-(\partial_{\text{\textgreek{m}}}\text{\textgreek{q}})\cdot|\text{\textgreek{y}}_{k}|^{2}$,
combined with Lemmas \ref{lem:BoundF} and \ref{lem:BoundednessPsiK},
as well the bounded frequency estimate \ref{lem:DtToOmegaInequalities2}
(serving to bound $|\partial_{t}\text{\textgreek{y}}_{k}|^{2}$ bulk
terms by $\text{\textgreek{w}}_{k}^{2}|\text{\textgreek{y}}_{k}|^{2}$
terms), we obtain the following Lagrangean inequality: 
\begin{equation}
\int_{\mathcal{R}(0,t^{*})}\text{\textgreek{q}}\cdot J_{\text{\textgreek{m}}}^{N}(\text{\textgreek{y}}_{k})n^{\text{\textgreek{m}}}\le C(\text{\textgreek{w}}_{0},\text{\textgreek{q}})\cdot(1+\text{\textgreek{w}}_{k}^{2})\cdot\int_{\mathcal{R}(0,t^{*})\cap supp(\text{\textgreek{q}})}|\text{\textgreek{y}}_{k}|^{2}+C(\text{\textgreek{w}}_{0},\text{\textgreek{q}})\cdot\int_{t=0}J_{\text{\textgreek{m}}}^{N}(\text{\textgreek{y}})n^{\text{\textgreek{m}}}.\label{eq:OneMoreLagrangeanIdentity}
\end{equation}
Therefore, using (\ref{eq:OneMoreLagrangeanIdentity}) for a cut off
$\text{\textgreek{q}}$ supported in $\{\frac{1}{2}R_{0}\le r\le4R_{0}\}$
and being equal to $1$ in $\{R_{0}\le r\le2R_{0}\}$, combined with
(\ref{eq:Boundedness}), we can readily bound: 
\begin{equation}
\int_{\mathcal{R}(0,t^{*})\cap\{R_{0}\le r\le2R_{0}\})}(1+s^{2})e^{2sw}\big\{ J_{\text{\textgreek{m}}}^{N}(\text{\textgreek{y}}_{k})n^{\text{\textgreek{m}}}+|\text{\textgreek{y}}_{k}|^{2}\big\}\le C(R,\text{\textgreek{w}}_{0})e^{C(R,\text{\textgreek{w}}_{0})\text{\textgreek{w}}_{+}}\int_{t=0}J_{\text{\textgreek{m}}}^{N}(\text{\textgreek{y}})n^{\text{\textgreek{m}}}.\label{eq:BoundBoundaryTerm}
\end{equation}
Hence, since $R=C_{1}<R_{0}<C_{2}(R,\text{\textgreek{w}}_{0})$, in
this case we can estimate from (\ref{eq:BoundednessCarlemann}) and
(\ref{eq:BoundBoundaryTerm}):
\begin{equation}
\int_{\{r\le R\}\cap\mathcal{R}(0,t^{*})}\big\{ J_{\text{\textgreek{m}}}^{N}(\text{\textgreek{y}}_{k})n^{\text{\textgreek{m}}}+|\text{\textgreek{y}}_{k}|^{2}\big\}\le C(R,\text{\textgreek{w}}_{0})e^{C(R,\text{\textgreek{w}}_{0})\text{\textgreek{w}}_{+}}\int_{t=0}J_{\text{\textgreek{m}}}^{N}(\text{\textgreek{y}})n^{\text{\textgreek{m}}}.
\end{equation}

\item Assume that the exponential growth scenario in Lemma \ref{thm:ODELemma}
holds, that is (\ref{eq:ExponentialGrowth}) is true. Then a simple
application of Gronwall's inequality on (\ref{eq:ExponentialGrowth})
yields:
\begin{equation}
\int_{\{5C_{1}\le r\le10C_{1}\}}\mathscr{M}(r)\, dr\le C_{3}\cdot e^{-C_{2}(1+\text{\textgreek{w}}_{k-1})}\int_{\{2C_{1}\le r\le4C_{1}\}}\mathscr{M}(r)\, dr.\label{eq:ExponentialBeforeCarlemann}
\end{equation}
 Applying (\ref{eq:CarlemannTypeInequality}) (for $6C_{1}$ in place
of $R$), we obtain 
\begin{equation}
\begin{split}\int_{\mathcal{R}(0,t^{*})\cap\{r\le6C_{1}\}}\Big(e^{2sw}+e^{2sw'}\Big)\cdot\Big(J_{\text{\textgreek{m}}}^{N}( & \text{\textgreek{y}}_{k})n^{\text{\textgreek{m}}}+|\text{\textgreek{y}}_{k}|^{2}\Big)\le\\
\le & C(\text{\textgreek{w}}_{0},R)\int_{\mathcal{R}(0,t^{*})\cap\{6C_{1}\le r\le7C_{1}\})}(1+s^{2})e^{2sw}\Big(J_{\text{\textgreek{m}}}^{N}(\text{\textgreek{y}}_{k})n^{\text{\textgreek{m}}}+|\text{\textgreek{y}}_{k}|^{2}\Big)+\\
 & +e^{C(C_{1})\cdot s}\cdot C(\text{\textgreek{w}}_{0},R)\int_{t=0}J_{\text{\textgreek{m}}}^{N}(\text{\textgreek{y}})n^{\text{\textgreek{m}}}.
\end{split}
\label{eq:ExponantialCarlemann}
\end{equation}
 In view of (\ref{eq:OneMoreLagrangeanIdentity}) applied for a cut
off $\text{\textgreek{q}}$ supported in $5C_{1}\le r\le8C_{1}$ and
being equal to 1 in $6C_{1}\le r\le7C_{1}$, we can also bound 
\begin{equation}
\begin{split}\int_{\mathcal{R}(0,t^{*})\cap\{6C_{1}\le r\le7C_{1}\})}(1+s^{2})e^{2sw}\Big(J_{\text{\textgreek{m}}}^{N}( & \text{\textgreek{y}}_{k})n^{\text{\textgreek{m}}}+|\text{\textgreek{y}}_{k}|^{2}\Big)\le\\
\le & C(R,\text{\textgreek{w}}_{0})(1+\text{\textgreek{w}}_{k}^{2})e^{s\cdot\sup_{C_{1}\le r\le10C_{1}}w}\int_{\{5C_{1}\le r\le8C_{1}\}}\mathscr{M}(r)\, dr+\\
 & +C(R,\text{\textgreek{w}}_{0})e^{C(R,\text{\textgreek{w}}_{0})\text{\textgreek{w}}_{k}}\int_{t=0}J_{\text{\textgreek{m}}}^{N}(\text{\textgreek{y}})n^{\text{\textgreek{m}}}.
\end{split}
\label{eq:ExponentialDecayBound}
\end{equation}
 From (\ref{eq:ExponentialDecayBound}) and (\ref{eq:ExponentialBeforeCarlemann})we
infer:
\begin{equation}
\begin{split}\int_{\mathcal{R}(0,t^{*})\cap\{6C_{1}\le r\le7C_{1}\})}(1+s^{2})e^{2sw}\Big(J_{\text{\textgreek{m}}}^{N}( & \text{\textgreek{y}}_{k})n^{\text{\textgreek{m}}}+|\text{\textgreek{y}}_{k}|^{2}\Big)\le\\
\le & C(R,\text{\textgreek{w}}_{0})(1+\text{\textgreek{w}}_{k}^{2})e^{s\cdot\sup_{C_{1}\le r\le10C_{1}}w-C_{2}\text{\textgreek{w}}_{k-1}}\int_{\{2C_{1}\le r\le4C_{1}\}}\mathscr{M}(r)\, dr+\\
 & +C(R,\text{\textgreek{w}}_{0})e^{C(R,\text{\textgreek{w}}_{0})\text{\textgreek{w}}_{k}}\int_{t=0}J_{\text{\textgreek{m}}}^{N}(\text{\textgreek{y}})n^{\text{\textgreek{m}}}.
\end{split}
\label{eq:ExponentialDecayBound-1}
\end{equation}
 Therefore, since in view of (\ref{eq:exponentBound}) the first term
of the right hand side of (\ref{eq:ExponentialDecayBound-1}) can
be absorbed into the left hand side of (\ref{eq:ExponantialCarlemann}),
we conclude from (\ref{eq:ExponantialCarlemann}): 
\begin{equation}
\int_{\{r\le R\}\cap\mathcal{R}(0,t^{*})}\big\{ J_{\text{\textgreek{m}}}^{N}(\text{\textgreek{y}}_{k})n^{\text{\textgreek{m}}}+|\text{\textgreek{y}}_{k}|^{2}\big\}\le C(R,\text{\textgreek{w}}_{0})e^{C(R,\text{\textgreek{w}}_{0})\text{\textgreek{w}}_{+}}\int_{t=0}J_{\text{\textgreek{m}}}^{N}(\text{\textgreek{y}})n^{\text{\textgreek{m}}}.
\end{equation}

\end{casenv}
Thus, the proof of Proposition \ref{prop:ILEDPsiK} is complete.\qed

\medskip{}

As a corollary of Propositions \ref{prop:ILEDPsiK} and \ref{prop:LowFrequencies},
we will obtain an integrated local energy decay statement for the
$\text{\textgreek{y}}_{\le\text{\textgreek{w}}_{+}}$ component of
$\text{\textgreek{y}}$. This will be the final result of this section.
\begin{cor}
\label{cor:BasicILED}For any $R>0$ there exists a $\text{\textgreek{d}}=\text{\textgreek{d}}(R)>0$,
such that for any smooth solution $\text{\textgreek{y}}$ to the wave
equation on $J^{+}(\text{\textgreek{S}})\cap\mathcal{D}$ with compactly
supported initial data on $\text{\textgreek{S}}$, any $0<\text{\textgreek{w}}_{0}<\text{\textgreek{d}}$
and any $\text{\textgreek{w}}_{+}>1$ we can bound:\end{cor}
\begin{quote}
\begin{equation}
\int_{\{r\le R\}\cap\mathcal{R}(0,t^{*})}\big\{ J_{\text{\textgreek{m}}}^{N}(\text{\textgreek{y}}_{\le\text{\textgreek{w}}_{+}})n^{\text{\textgreek{m}}}+|\text{\textgreek{y}}_{\le\text{\textgreek{w}}_{+}}|^{2}\big\}\le C(R,\text{\textgreek{w}}_{0})\cdot e^{C(R,\text{\textgreek{w}}_{0})\text{\textgreek{w}}_{+}}\int_{t=0}J_{\text{\textgreek{m}}}^{N}(\text{\textgreek{y}})n^{\text{\textgreek{m}}}.\label{eq:IntegratedLocalEnergydecayForTheLowFrequencies}
\end{equation}
\end{quote}
\begin{proof}
Proposition \ref{prop:LowFrequencies} for the zero frequency part
$\text{\textgreek{y}}_{0}$ of $\text{\textgreek{y}}$ provides an
integrated local energy decay statement for $\text{\textgreek{y}}_{0}$:
\[
\int_{\{r\le R\}\cap\mathcal{R}(0,t^{*})}\big(J_{\text{\textgreek{m}}}^{N}(\text{\textgreek{y}}_{0})n^{\text{\textgreek{m}}}+|\text{\textgreek{y}}_{0}|^{2})\le C(R,\text{\textgreek{w}}_{0})\cdot\int_{t=0}J_{\text{\textgreek{m}}}^{N}(\text{\textgreek{y}})n^{\text{\textgreek{m}}}.
\]
Adding (\ref{eq:LowFrequencies}) and (\ref{eq:ILED}) for all $1\le|k|\le n$,
and recalling that $n\sim\frac{\text{\textgreek{w}}_{+}}{\text{\textgreek{w}}_{0}}$,
we obtain:

\[
\sum_{k=-n}^{n}\Big\{\int_{\{r\le R\}\cap\mathcal{R}(0,t^{*})}\Big(J_{\text{\textgreek{m}}}^{N}(\text{\textgreek{y}}_{k})n^{\text{\textgreek{m}}}+|\text{\textgreek{y}}_{k}|^{2}\Big)\Big\}\le C(R,\text{\textgreek{w}}_{0})\cdot\text{\textgreek{w}}_{+}\cdot e^{C(R,\text{\textgreek{w}}_{0})\cdot\text{\textgreek{w}}_{+}}\int_{t=0}J_{\text{\textgreek{m}}}^{N}(\text{\textgreek{y}})n^{\text{\textgreek{m}}}.
\]
The result now follows readily from the fact that $\text{\textgreek{y}}_{\le\text{\textgreek{w}}_{+}}=\sum_{k=-n}^{n}\text{\textgreek{y}}_{k}$:
\begin{align}
\int_{\{r\le R\}\cap\mathcal{R}(0,t^{*})}\Big(J_{\text{\textgreek{m}}}^{N}(\text{\textgreek{y}}_{\le\text{\textgreek{w}}_{+}})n^{\text{\textgreek{m}}}+|\text{\textgreek{y}}_{\le\text{\textgreek{w}}_{+}}|^{2}\Big) & \le\int_{\{r\le R\}\cap\mathcal{R}(0,t^{*})}2n\cdot\sum_{k=-n}^{n}\Big(J_{\text{\textgreek{m}}}^{N}(\text{\textgreek{y}}_{k})n^{\text{\textgreek{m}}}+|\text{\textgreek{y}}_{k}|^{2}\Big)\le\\
 & \le C(R,\text{\textgreek{w}}_{0})\cdot\text{\textgreek{w}}_{+}^{2}\cdot e^{C(R,\text{\textgreek{w}}_{0})\cdot\text{\textgreek{w}}_{+}}\int_{t=0}J_{\text{\textgreek{m}}}^{N}(\text{\textgreek{y}})n^{\text{\textgreek{m}}}\le\nonumber \\
 & \le C(R,\text{\textgreek{w}}_{0})\cdot e^{C(R)\cdot\text{\textgreek{w}}_{+}}\int_{t=0}J_{\text{\textgreek{m}}}^{N}(\text{\textgreek{y}})n^{\text{\textgreek{m}}}.\nonumber 
\end{align}

\end{proof}

\section{\label{sec:Polynomial-decay}Polynomial decay of the local energy
of $\text{\textgreek{y}}_{\le\text{\textgreek{w}}_{+}}$}

In this Section, we will establish a polynomial decay rate for the
local energy of $\text{\textgreek{y}}_{\le\text{\textgreek{w}}_{+}}$,
using the integrated local energy decay estimate of Corollary \ref{cor:BasicILED},
together with the generalisation of the $r^{p}$-weighted energy method
of Dafermos and Rodnianski established in \cite{Moschidisc}. 

We will thus examine the energy of $\text{\textgreek{y}}_{\le\text{\textgreek{w}}_{+}}$
on a foliation of spacelike hypersurfaces terminating at null infinity.%
\footnote{Since the energy of $\text{\textgreek{y}}_{\le\text{\textgreek{w}}_{+}}$
on the $\text{\textgreek{S}}_{t}$ hypersurfaces can not be expected
to decay in $t$%
} In the $r\gg1$ region, such a foliation has been constructed in
terms of the level sets of the function $\tilde{u}_{\text{\textgreek{h}}'}$,
see (\ref{eq:Hyperboloids}). We will extend $\tilde{u}_{\text{\textgreek{h}}'}$
to a function defined on the whole of $\mathcal{D}\backslash\mathcal{H}^{-}$,
with everywhere spacelike level sets, such that it coincides with
the time function $t$ in the region $r\lesssim1$: 
\begin{defn*}
For any fixed $\text{\textgreek{h}}'\in(0,1+2a)$ and any constant
$R_{3}\gg1$ large in terms of the geometry of $(\mathcal{D},g)$,
we will define the function $\bar{t}:D\backslash\mathcal{H}^{-}\rightarrow\mathbb{R}$,
\begin{equation}
\bar{t}=t+\text{\textgreek{q}}_{R_{3}}(r)\cdot(\tilde{u}_{\text{\textgreek{h}'}}-t),\label{eq:EverywhereDefinedHyperboloidalFoliation}
\end{equation}
where $u_{\text{\textgreek{h}}'}$ is defined as (\ref{eq:Hyperboloids})
and $\text{\textgreek{q}}_{R_{3}}:\mathbb{R}\rightarrow[0,1]$ is
a smooth cut-off function equal to $0$ for $r\le R_{3}$ and equal
to $1$ for $r\ge R_{3}+1$. We will also denote with $S_{\text{\textgreek{t}}}$
the hyperboloids $\{\bar{t}=\text{\textgreek{t}}\}$, which for $\{r\ge R_{3}\}$
coincide with $\tilde{S}_{\text{\textgreek{t}}}$. 
\end{defn*}
We will establish the following proposition:
\begin{prop}
\label{prop:PolynomialDecayLowFrequencies}There exists a positive
constant $C$, depending only on the geometry of $(\mathcal{D},g)$,
such that for any smooth solution $\text{\textgreek{y}}$ to the wave
equation on $J^{+}(\text{\textgreek{S}})\cap D$ with compactly supported
initial data on $\text{\textgreek{S}}$, any $0\le\text{\textgreek{d}}_{0}\le1$,
any $\text{\textgreek{w}}_{+}>1$ and any $\text{\textgreek{t}}\in[\frac{1}{4}t^{*},\frac{3}{4}t^{*}]$
we can bound:

\begin{equation}
\int_{S_{\text{\textgreek{t}}}\cap\mathcal{R}(0,t^{*})}J_{\text{\textgreek{m}}}^{N}(\text{\textgreek{y}}_{\le\text{\textgreek{w}}_{+}})n_{S_{\text{\textgreek{t}}}}^{\text{\textgreek{m}}}\le\frac{C}{\big(\text{\textgreek{t}}-\frac{1}{4}t^{*}\big)^{\text{\textgreek{d}}_{0}}}\cdot\Big\{ e^{C\text{\textgreek{w}}_{+}}\int_{t=0}J_{\text{\textgreek{m}}}^{N}(\text{\textgreek{y}})n^{\text{\textgreek{m}}}+\int_{t=0}(1+r)^{\text{\textgreek{d}}_{0}}J_{\text{\textgreek{m}}}^{N}(\text{\textgreek{y}})n^{\text{\textgreek{m}}}\Big\}.\label{eq:PolynomialDecayLowFrequencies}
\end{equation}
\end{prop}
\begin{proof}
Let us fix $R=R_{3}$, assuming, of course, without loss of generality
that $R_{3}$ has been chosen sufficiently large in terms of the geometry
of $(\mathcal{D},g)$. We will also fix a value $\text{\textgreek{w}}_{0}=\text{\textgreek{w}}_{0}(R)$,
so that according to Corollary \ref{cor:BasicILED} we can estimate:

\begin{equation}
\int_{\{r\le R\}\cap\mathcal{R}(0,t^{*})}\big\{ J_{\text{\textgreek{m}}}^{N}(\text{\textgreek{y}}_{\le\text{\textgreek{w}}_{+}})n^{\text{\textgreek{m}}}+|\text{\textgreek{y}}_{\le\text{\textgreek{w}}_{+}}|^{2}\big\}\le C\cdot e^{C\text{\textgreek{w}}_{+}}\int_{t=0}J_{\text{\textgreek{m}}}^{N}(\text{\textgreek{y}})n^{\text{\textgreek{m}}}.\label{eq:ILEDlowfrequencies}
\end{equation}

In view of (\ref{eq:IntegratedDecayForTheEnergyOfTheHyperboloids-1}),
we can bound:

\begin{multline}
\int_{0}^{\infty}\big\{\int_{S_{t}\cap\mathcal{R}(0,t^{*})\cap\{r\ge R\}}J_{\text{\textgreek{m}}}^{N}(\text{\textgreek{y}}_{\le\text{\textgreek{w}}_{+}})n_{S_{t}}^{\text{\textgreek{m}}}\big\}\, dt\le\\
\le C\int_{\{R-1\le r\le R\}\cap\mathcal{R}(0,t^{*})}\big(J_{\text{\textgreek{m}}}^{N}(\text{\textgreek{y}}_{\le\text{\textgreek{w}}_{+}})n^{\text{\textgreek{m}}}+|\text{\textgreek{y}}_{\le\text{\textgreek{w}}_{+}}|^{2}\big)+C\int_{t=0}(1+r)J_{\text{\textgreek{m}}}^{N}(\text{\textgreek{y}})n^{\text{\textgreek{m}}},\label{eq:NearInfinityDecay}
\end{multline}
 where $n_{S_{t}}$ is the future directed unit normal of the foliation
$S_{t}$.

From (\ref{eq:NearInfinityDecay}) and (\ref{eq:ILEDlowfrequencies})
we thus conclude

\begin{equation}
\int_{0}^{\infty}\big\{\int_{S_{\text{\textgreek{t}}}\cap\mathcal{R}(0,t^{*})}J_{\text{\textgreek{m}}}^{N}(\text{\textgreek{y}}_{\le\text{\textgreek{w}}_{+}})n_{S_{\text{\textgreek{t}}}}^{\text{\textgreek{m}}}\big\}\, d\text{\textgreek{t}}\le C\cdot\Big\{ e^{C\text{\textgreek{w}}_{+}}\int_{t=0}J_{\text{\textgreek{m}}}^{N}(\text{\textgreek{y}})n^{\text{\textgreek{m}}}+\int_{t=0}rJ_{\text{\textgreek{m}}}^{N}(\text{\textgreek{y}})n^{\text{\textgreek{m}}}\Big\}.\label{eq:ILEDLowFrequencies2}
\end{equation}
In order to simplify the expressions in the next few lines, we will
set: 
\begin{equation}
f(\text{\textgreek{t}})\doteq\int_{S_{\text{\textgreek{t}}}\cap\mathcal{R}(0,t^{*})}J_{\text{\textgreek{m}}}^{N}(\text{\textgreek{y}}_{\le\text{\textgreek{w}}_{+}})n_{S_{\text{\textgreek{t}}}}^{\text{\textgreek{m}}}.
\end{equation}
Then (\ref{eq:ILEDLowFrequencies2}) reads

\begin{equation}
\int_{0}^{\infty}f(\text{\textgreek{t}})\, d\text{\textgreek{t}}\le C\cdot\Big\{ e^{C\text{\textgreek{w}}_{+}}\int_{t=0}J_{\text{\textgreek{m}}}^{N}(\text{\textgreek{y}})n^{\text{\textgreek{m}}}+\int_{t=0}rJ_{\text{\textgreek{m}}}^{N}(\text{\textgreek{y}})n^{\text{\textgreek{m}}}\Big\}.\label{eq:ILEDLowFrequencies2-1}
\end{equation}

Due to Lemma \ref{lem:GeneralisedBoundednessPsiK} for $q=2$, for
any $\text{\textgreek{t}}_{1}\le\text{\textgreek{t}}_{2}$ in the
interval $[\frac{1}{4}t^{*},\frac{3}{4}t^{*}]$, setting $t_{st}=\frac{1}{4}t^{*}$
in (\ref{eq:BoundednessFrmInitialDataOnK-1}) (notice that this implies
that $t_{st}\sim t_{1}\sim t_{2}-t_{1}$), the following boundedness
statement holds for $f$:

\begin{equation}
f(\text{\textgreek{t}}_{2})\le C\cdot f(\text{\textgreek{t}}_{1})+C\text{\textgreek{t}}_{2}^{2}\int_{t=0}J_{\text{\textgreek{m}}}^{N}(\text{\textgreek{y}})n^{\text{\textgreek{m}}}.\label{eq:BoundednessForf}
\end{equation}
Hence, we infer from (\ref{eq:ILEDLowFrequencies2-1}) and (\ref{eq:BoundednessForf})
for any $\text{\textgreek{t}}\in[\frac{1}{4}t^{*},\frac{3}{4}t^{*}]$:
\begin{align}
C\cdot\Big\{ e^{C\text{\textgreek{w}}_{+}}\int_{t=0}J_{\text{\textgreek{m}}}^{N}(\text{\textgreek{y}})n^{\text{\textgreek{m}}}+\int_{t=0}rJ_{\text{\textgreek{m}}}^{N}(\text{\textgreek{y}})n^{\text{\textgreek{m}}}\Big\} & \ge\int_{0}^{\infty}f(s)\, ds\ge\int_{\frac{1}{4}t^{*}}^{\text{\textgreek{t}}}f(s)\, ds\ge\label{eq:FinalBoundDecay}\\
 & \ge\int_{\frac{1}{4}t^{*}}^{\text{\textgreek{t}}}\big(\frac{f(\text{\textgreek{t}})}{C}-\frac{1}{\text{\textgreek{t}}^{2}}\int_{t=0}J_{\text{\textgreek{m}}}^{N}(\text{\textgreek{y}})n^{\text{\textgreek{m}}}\big)\, ds\ge\nonumber \\
 & \ge c\cdot(\text{\textgreek{t}}-\frac{t^{*}}{4})f(\text{\textgreek{t}})-C\cdot\int_{t=0}J_{\text{\textgreek{m}}}^{N}(\text{\textgreek{y}})n^{\text{\textgreek{m}}}.\nonumber 
\end{align}
From (\ref{eq:FinalBoundDecay}) it readily follows that forany $\text{\textgreek{t}}\in[\frac{1}{4}t^{*},\frac{3}{4}t^{*}]$
we can bound:

\begin{equation}
\int_{S_{\text{\textgreek{t}}}\cap\mathcal{R}(0,t^{*})}J_{\text{\textgreek{m}}}^{N}(\text{\textgreek{y}}_{\le\text{\textgreek{w}}_{+}})n_{S_{\text{\textgreek{t}}}}^{\text{\textgreek{m}}}\le\frac{C}{\text{\textgreek{t}}-\frac{1}{4}t^{*}}\cdot\Big\{ e^{C\text{\textgreek{w}}_{+}}\int_{t=0}J_{\text{\textgreek{m}}}^{N}(\text{\textgreek{y}})n^{\text{\textgreek{m}}}+\int_{t=0}rJ_{\text{\textgreek{m}}}^{N}(\text{\textgreek{y}})n^{\text{\textgreek{m}}}\Big\}.\label{eq:PolynomialDecayLowFrequencies-1}
\end{equation}

Interpolating (using Lemma \ref{lem:InterpolationLemma} of the Appendix)
 between (\ref{eq:PolynomialDecayLowFrequencies-1}) and the boundedness
estimate (following from Lemmas \ref{lem:BoundednessPsiK} and \ref{lem:GeneralisedBoundednessPsiK})
\begin{equation}
\int_{S_{\text{\textgreek{t}}}\cap\mathcal{R}(0,t^{*})}J_{\text{\textgreek{m}}}^{N}(\text{\textgreek{y}}_{\le\text{\textgreek{w}}_{+}})n_{S_{\text{\textgreek{t}}}}^{\text{\textgreek{m}}}\le C\int_{t=0}J_{\text{\textgreek{m}}}^{N}(\text{\textgreek{y}})n^{\text{\textgreek{m}}},\label{eq:PolynomialDecayLowFrequencies-1-1}
\end{equation}
 we readily obtain the desired estimate: 
\begin{equation}
\int_{S_{\text{\textgreek{t}}}\cap\mathcal{R}(0,t^{*})}J_{\text{\textgreek{m}}}^{N}(\text{\textgreek{y}}_{\le\text{\textgreek{w}}_{+}})n_{S_{\text{\textgreek{t}}}}^{\text{\textgreek{m}}}\le\frac{C}{\big(\text{\textgreek{t}}-\frac{1}{4}t^{*}\big)^{\text{\textgreek{d}}_{0}}}\cdot\Big\{ e^{C\text{\textgreek{w}}_{+}}\int_{t=0}J_{\text{\textgreek{m}}}^{N}(\text{\textgreek{y}})n^{\text{\textgreek{m}}}+\int_{t=0}r^{\text{\textgreek{d}}_{0}}J_{\text{\textgreek{m}}}^{N}(\text{\textgreek{y}})n^{\text{\textgreek{m}}}\Big\}.\label{eq:PolynomialDecayLowFrequencies-1-2}
\end{equation}

\end{proof}

\section{\label{sec:The-interpolation-argument}The frequency interpolation
argument}

In this section, we will fix $t^{*}$, $\text{\textgreek{w}}_{0}$
and $\text{\textgreek{w}}_{+}$, and we will complete the proof of
Theorem \ref{thm:Theorem}. 

Let $\text{\textgreek{y}}:\mathcal{D}\rightarrow\mathbb{C}$ be a
smooth function as in the statement of Theorem \ref{thm:Theorem},
with compactly supported initial data on $\text{\textgreek{S}}\cap\mathcal{D}$.
For any given $\text{\textgreek{t}}>0$, let $t^{*}=2\text{\textgreek{t}}$.
Of course, we can safely asume without loss of generality that $\text{\textgreek{t}}\gg1$.
Let us introduce a parameter $\text{\textgreek{w}}_{+}>1$, and define
$\text{\textgreek{y}}_{t^{*}},$$\text{\textgreek{y}}_{\le\text{\textgreek{w}}_{+}}$and
$\text{\textgreek{y}}_{\ge\text{\textgreek{w}}_{+}}$ as in Section
\ref{sec:FreqDecomposition}. 

Due to the fact that $\text{\textgreek{y}}_{t^{*}}=\text{\textgreek{y}}_{\le\text{\textgreek{w}}_{+}}+\text{\textgreek{y}}_{\ge\text{\textgreek{w}}_{+}}$,
we can bound: 
\begin{align}
\int_{\{t=\text{\textgreek{t}}\}\cap\{r\le R_{1}\}}J_{\text{\textgreek{m}}}^{N}(\text{\textgreek{y}}) & =\int_{\{t=\text{\textgreek{t}}\}\cap\{r\le R_{1}\}}J_{\text{\textgreek{m}}}^{N}(\text{\textgreek{y}}_{t^{*}})\lesssim\label{eq:AlmostThereForTheend}\\
 & \lesssim\int_{\{t=\text{\textgreek{t}}\}\cap\{r\le R_{1}\}}J_{\text{\textgreek{m}}}^{N}(\text{\textgreek{y}}_{\le\text{\textgreek{w}}_{+}})+\int_{\{t=\text{\textgreek{t}}\}\cap\{r\le R_{1}\}}J_{\text{\textgreek{m}}}^{N}(\text{\textgreek{y}}_{\ge\text{\textgreek{w}}_{+}}).\nonumber 
\end{align}
In view of Lemma \ref{lem:HighFrequencies}, we can estimate for the
high frequency part:

\begin{equation}
\int_{\{t=\text{\textgreek{t}}\}\cap\{r\le R_{1}\}}J_{\text{\textgreek{m}}}^{N}(\text{\textgreek{y}}_{\ge\text{\textgreek{w}}_{+}})n^{\text{\textgreek{m}}}\le\frac{C_{m}}{\text{\textgreek{w}}_{+}^{2m}}\sum_{j=0}^{m}\int_{t=0}J_{\text{\textgreek{m}}}^{N}(T^{j}\text{\textgreek{y}})n^{\text{\textgreek{m}}}.\label{eq:BoundHighFrequency}
\end{equation}
In view of Proposition \ref{prop:PolynomialDecayLowFrequencies},
we can estimate for the low frequency part:

\begin{equation}
\int_{\{t=\text{\textgreek{t}}\}\cap\{r\le R_{1}\}}J_{\text{\textgreek{m}}}^{N}(\text{\textgreek{y}}_{\le\text{\textgreek{w}}_{+}})\le C\frac{1}{\text{\textgreek{t}}^{\text{\textgreek{d}}_{0}}}\cdot\Big\{ e^{C\text{\textgreek{w}}_{+}}\int_{t=0}J_{\text{\textgreek{m}}}^{N}(\text{\textgreek{y}})n^{\text{\textgreek{m}}}+\int_{t=0}r^{\text{\textgreek{d}}_{0}}J_{\text{\textgreek{m}}}^{N}(\text{\textgreek{y}})n^{\text{\textgreek{m}}}\Big\}.\label{eq:BoundLowFrequenct}
\end{equation}
 From (\ref{eq:AlmostThereForTheend}), (\ref{eq:BoundLowFrequenct})
and (\ref{eq:BoundHighFrequency}), we thus deduce: 
\begin{align}
\int_{\{t=\text{\textgreek{t}}\}\cap\{r\le R_{1}\}}J_{\text{\textgreek{m}}}^{N}(\text{\textgreek{y}})n^{\text{\textgreek{m}}}\le & C\frac{1}{\text{\textgreek{t}}^{\text{\textgreek{d}}_{0}}}\cdot\Big\{ e^{C\text{\textgreek{w}}_{+}}\int_{t=0}J_{\text{\textgreek{m}}}^{N}(\text{\textgreek{y}})n^{\text{\textgreek{m}}}+\int_{t=0}r^{\text{\textgreek{d}}_{0}}J_{\text{\textgreek{m}}}^{N}(\text{\textgreek{y}})n^{\text{\textgreek{m}}}\Big\}+\label{eq:BeforeLastInterpolation}\\
 & +\frac{C_{m}}{\text{\textgreek{w}}_{+}^{2m}}\sum_{j=0}^{m}\int_{t=0}J_{\text{\textgreek{m}}}^{N}(T^{j}\text{\textgreek{y}})n^{\text{\textgreek{m}}}.\nonumber 
\end{align}

It is at this point that we will choose a suitable value for $\text{\textgreek{w}}_{+}$.
Choosing $\text{\textgreek{w}}_{+}=\frac{log(2+\text{\textgreek{t}})}{2C}$
in (\ref{eq:BeforeLastInterpolation}), we conclude:

\begin{equation}
\int_{\{t=\text{\textgreek{t}}\}\cap\{r\le R_{1}\}}J_{\text{\textgreek{m}}}^{N}(\text{\textgreek{y}})n^{\text{\textgreek{m}}}\le\frac{C_{m}}{\{log(2+\text{\textgreek{t}})\}^{2m}}\Big(\sum_{j=0}^{m}\int_{t=0}J_{\text{\textgreek{m}}}^{N}(T^{j}\text{\textgreek{y}})n^{\text{\textgreek{m}}}\Big)+\frac{C}{\text{\textgreek{t}}^{\text{\textgreek{d}}_{0}}}\int_{t=0}(1+r)^{\text{\textgreek{d}}_{0}}J_{\text{\textgreek{m}}}^{N}(\text{\textgreek{y}})n^{\text{\textgreek{m}}}
\end{equation}
 and hence Theorem \ref{eq:BasicTheorem} has been established. \qed

\section{\label{sec:Proof-of-corollary}Proof of Corollary \ref{cor:Corollary}}

In view of the boundedness assumption \hyperref[Assumption 4]{4}
and the conservation of the $J^{T}$ flux, it suffices to prove Corollary
\ref{cor:Corollary} for $S_{\text{\textgreek{t}}}$ being the level
sets $\{\bar{t}=\text{\textgreek{t}}\}$ of the function $\bar{t}$
defined in Section \ref{sec:Polynomial-decay}. 

Fix an $R_{1}\gg1$ in the statement of Theorem \ref{thm:Theorem}.
According to Theorem \ref{eq:BasicTheorem}, for any $m\in\mathbb{N}$
and any $l>0$ (that will be fixed later), there exists a constant
$C_{m}(l)>0$ (depending also on $R_{1}$) such that:

\begin{equation}
\int_{\{t=\text{\textgreek{t}}\}\cap\{r\le l\cdot R_{1}\}}J_{\text{\textgreek{m}}}^{N}(\text{\textgreek{y}})n^{\text{\textgreek{m}}}\le\frac{C_{m}(l)}{\{log(2+\text{\textgreek{t}})\}^{2m}}\Big(\sum_{j=0}^{m}\int_{t=0}J_{\text{\textgreek{m}}}^{N}(T^{j}\text{\textgreek{y}})n^{\text{\textgreek{m}}}\Big)+\frac{C_{m}(l)}{\text{\textgreek{t}}^{\text{\textgreek{d}}_{0}}}\int_{\{t=0\}}(1+r)^{\text{\textgreek{d}}_{0}}\cdot J_{\text{\textgreek{m}}}^{N}(\text{\textgreek{y}})n^{\text{\textgreek{m}}}.\label{eq:ResultFromTheorem}
\end{equation}
Applying (\ref{eq:IntegratedDecayForTheEnergyOfTheHyperboloidsPsi}),
we obtain for any $0\le t'\le t^{\prime\prime}$: %
\footnote{Recall that $S_{t}$ and $\tilde{S}_{t}$ coincide for $\{r\ge R_{1}\}$
if $R_{1}\gg1$%
}

\begin{equation}
\int_{t'}^{t^{\prime\prime}}\big\{\int_{S_{\text{s}}\cap\{r\ge R_{1}\}}J_{\text{\textgreek{m}}}^{N}(\text{\textgreek{y}})n_{S}^{\text{\textgreek{m}}}\big\}\, ds\le C\cdot\int_{\mathcal{R}(t',t^{\prime\prime})\cap\{R_{1}\le r\le R_{1}+1\}}\big(J_{\text{\textgreek{m}}}^{N}(\text{\textgreek{y}})n^{\text{\textgreek{m}}}+|\text{\textgreek{y}}|^{2}\big)+C\int_{\{t=0\}}(1+r)^{\text{\textgreek{d}}_{0}}J_{\text{\textgreek{m}}}^{N}(\text{\textgreek{y}})n^{\text{\textgreek{m}}}.\label{eq:ResultFromNewMethod}
\end{equation}

Using a Hardy inequality on $S_{\text{\textgreek{t}}}$ (a corollary
of (\ref{eq:GeneralHardyBound})), we can bound: 
\begin{equation}
\int_{S_{t}\cap\{R_{1}\le r\le R_{1}+1\}}|\text{\textgreek{y}}|^{2}\le C\cdot\int_{S_{t}\cap\{R_{1}\le r\}}r^{-2-\frac{1}{2}}|\text{\textgreek{y}}|^{2}\le C\cdot\int_{S_{t}\cap\{R_{1}\le r\}}r^{-\frac{1}{2}}J_{\text{\textgreek{m}}}^{N}(\text{\textgreek{y}})n_{S}^{\text{\textgreek{m}}}.\label{eq:HardyOnSt}
\end{equation}
Thus, for any $\text{\textgreek{e}}>0$, if $l=l(\text{\textgreek{e}})$
is chosen sufficiently large (independently of $t$), there exists
a $C=C(\text{\textgreek{e}})>0$ such that

\begin{equation}
\int_{S_{t}\cap\{R_{1}\le r\le R_{1}+1\}}|\text{\textgreek{y}}|^{2}\le\text{\textgreek{e}}\cdot\int_{S_{\text{t}}\cap\{r\ge R_{1}\}}J_{\text{\textgreek{m}}}^{N}(\text{\textgreek{y}})n_{S}^{\text{\textgreek{m}}}+C(\text{\textgreek{e}})\int_{S_{t}\cap\{r\le lR_{1}\}}J_{\text{\textgreek{m}}}^{N}(\text{\textgreek{y}})n_{S}^{\text{\textgreek{m}}}.\label{eq:EpsilonBound}
\end{equation}

Returning to (\ref{eq:ResultFromNewMethod}), in view of (\ref{eq:EpsilonBound})
we have for any $\text{\textgreek{e}}>0$: 
\begin{align}
\int_{t'}^{t^{\prime\prime}}\big\{\int_{S_{\text{s}}\cap\{r\ge R_{1}\}}J_{\text{\textgreek{m}}}^{N}(\text{\textgreek{y}})n_{S}^{\text{\textgreek{m}}}\big\}\, ds\le & C\text{\textgreek{e}}\cdot\int_{t'}^{t^{\prime\prime}}\big\{\int_{S_{\text{s}}\cap\{r\ge R_{1}\}}J_{\text{\textgreek{m}}}^{N}(\text{\textgreek{y}})n_{S}^{\text{\textgreek{m}}}\big\}\, ds+\\
 & +C(\text{\textgreek{e}})\int_{t'}^{t^{\prime\prime}}\big\{\int_{S_{t}\cap\{r\le l\cdot R_{1}\}}J_{\text{\textgreek{m}}}^{N}(\text{\textgreek{y}})n_{S}^{\text{\textgreek{m}}}\big\}\, ds+C\int_{\{t=0\}}(1+r)^{\text{\textgreek{d}}_{0}}J_{\text{\textgreek{m}}}^{N}(\text{\textgreek{y}})n^{\text{\textgreek{m}}},\nonumber 
\end{align}
or, after absorbing the first term of the right hand side into the
left hand side:

\begin{equation}
\int_{t'}^{t^{\prime\prime}}\big\{\int_{S_{s}\cap\{r\ge R_{1}\}}J_{\text{\textgreek{m}}}^{N}(\text{\textgreek{y}})n_{S}^{\text{\textgreek{m}}}\big\}\, ds\le C(\text{\textgreek{e}})\cdot\int_{t'}^{t^{\prime\prime}}\big\{\int_{S_{t}\cap\{r\le l\cdot R_{1}\}}J_{\text{\textgreek{m}}}^{N}(\text{\textgreek{y}})n_{S}^{\text{\textgreek{m}}}\big\}\, ds+C\int_{\{t=0\}}(1+r)^{\text{\textgreek{d}}_{0}}J_{\text{\textgreek{m}}}^{N}(\text{\textgreek{y}})n^{\text{\textgreek{m}}}.\label{eq:LastInequality}
\end{equation}
 From now on, $\text{\textgreek{e}}$ and hence $l$ will be considered
fixed.

To make the notations a little simpler, let us for a moment denote
\begin{equation}
f(t)\doteq\int_{S_{\text{t}}\cap\{r\ge R_{1}\}}J_{\text{\textgreek{m}}}^{N}(\text{\textgreek{y}})n_{S}^{\text{\textgreek{m}}}
\end{equation}
 for $t\ge0$. We will also need the initial data quantities 
\begin{equation}
A\doteq\sum_{j=0}^{m}\int_{t=0}J_{\text{\textgreek{m}}}^{N}(T^{j}\text{\textgreek{y}})n^{\text{\textgreek{m}}}
\end{equation}
 and 
\[
B\doteq\int_{\{t=0\}}(1+r)^{\text{\textgreek{d}}_{0}}J_{\text{\textgreek{m}}}^{N}(\text{\textgreek{y}})n^{\text{\textgreek{m}}}.
\]

In view of (\ref{eq:ResultFromTheorem}), inequality (\ref{eq:LastInequality})
becomes (assuming without loss of generality that $t'\ge1$)

\begin{equation}
\int_{t'}^{t^{\prime\prime}}f(s)\, ds\le C_{m}\cdot A\cdot\int_{t'}^{t^{\prime\prime}}\frac{1}{\{log(s+2)\}^{2m}}\, ds+C_{m}\cdot B\int_{t'}^{t^{\prime\prime}}\frac{1}{s}\, ds\le C_{m}\cdot(A+B)\cdot\int_{t'}^{t^{\prime\prime}}\frac{1}{\{log(s+2)\}^{2m}}\, ds.\label{eq:IntegralEstimate}
\end{equation}
 Inequality (\ref{eq:IntegralEstimate}), combined with an application
of the pigeonhole principle, implies that there exists an infinite
dyadic sequence $0<t_{1}<...<t_{k}<...$ of positive numbers with
$t_{k+1}\in[2t_{k},4t_{k}]$,%
\footnote{thus $t_{j}\longrightarrow+\infty$ as $j\rightarrow+\infty$%
} such that $\forall j\in\mathbb{N}:$ 
\begin{equation}
f(t_{j})\le2C_{m}\cdot(A+B)\cdot\frac{1}{\{log(t_{j}+2)\}^{2m}}.
\end{equation}
 It is important that in this case, $C_{m}$ is the same constant
as the one appearing in (\ref{eq:IntegralEstimate}). 

The existence of such a dyadic sequence follows by contradiction:
If such a sequence did not exist, that would mean that for any $t_{1}>0$,
there should exist a $t_{c}=t_{c}(t_{1})\ge t_{1}$, such that for
all $t\in[2t_{c},4t{}_{c}]$: 
\begin{equation}
f(t)>2C_{m}\cdot(A+B)\cdot\frac{1}{\{log(t+2)\}^{2m}}.\label{eq:BoundForContradiction}
\end{equation}
But this would lead to a contradiction, since in this case due to
(\ref{eq:IntegralEstimate}) and (\ref{eq:BoundForContradiction})
we could bound: 
\begin{equation}
1\ge\limsup_{t_{1}\rightarrow+\infty}\frac{\int_{2t_{c}}^{4t_{c}}f(s)\, ds}{C_{m}\cdot(A+B)\cdot\int_{2t_{c}}^{4t_{c}}\frac{1}{\{log(s+2)\}^{2m}}\, ds}>\limsup_{t_{1}\rightarrow+\infty}\frac{2C_{m}\cdot(A+B)\cdot\int_{2t_{c}}^{4t_{c}}\frac{1}{\{log(s+2)\}^{2m}}\, ds}{C_{m}\cdot(A+B)\cdot\int_{2t_{c}}^{4t_{c}}\frac{1}{\{log(s+2)\}^{2m}}\, ds}=2.
\end{equation}

We also note that due to the boundedness assumption \hyperref[Assumption 4]{4},%
\footnote{Notice the remark following assumption \hyperref[Assumption 4]{4}%
} there exists a constant $C>0$ (independent, of course, of $\text{\textgreek{y}}$)
such that for any $t$ in any interval $[t_{j},t_{j+1}]$ defined
by the previous dyadic sequence, we can bound 
\begin{equation}
f(t)\le C\cdot f(t_{j}).
\end{equation}
 Thus, since the sequence $\{t_{j}\}_{j\in\mathbb{N}}$ satisfied
$t_{j+1}\in[2t_{j},4t_{j}]$, we conclude that for any $t\ge0$:

\begin{equation}
f(t)\le2C_{m}\cdot(A+B)\cdot\frac{1}{\{log(t+2)\}^{2m}},
\end{equation}
 and thus Corollary \ref{cor:Corollary} has been established.\qed

\section{\label{sec:Proof-of-Corollary2}Proof of Corollaries \ref{cor:CorollaryQualititative}
and \ref{cor:CorollaryScattering}}

In this Section, we will establish Corollaries \ref{cor:CorollaryQualititative}
and \ref{cor:CorollaryScattering} using Corollary \ref{cor:Corollary}.

\subsection{Proof of Corollary \ref{cor:CorollaryQualititative}}

In order to show that $\lim_{\text{\textgreek{t}}\rightarrow+\infty}\int_{S_{\text{\textgreek{t}}}}J_{\text{\textgreek{m}}}^{N}(\text{\textgreek{y}})n_{S}^{\text{\textgreek{m}}}=0$,
it suffices to show that for any given $\text{\textgreek{e}}>0$,
there exists some $\text{\textgreek{t}}_{0}>0$ such that for any
$\text{\textgreek{t}}\ge\text{\textgreek{t}}_{0}$: $\int_{S_{\text{\textgreek{t}}}}J_{\text{\textgreek{m}}}^{N}(\text{\textgreek{y}})n_{S}^{\text{\textgreek{m}}}<\text{\textgreek{e}}$. 

Since $\int_{t=0}J_{\text{\textgreek{m}}}^{N}(\text{\textgreek{y}})n^{\text{\textgreek{m}}}<\infty$,
using the standard density arguments we can find smooth and compactly
supported initial data $(\tilde{\text{\textgreek{y}}},\partial_{t}\tilde{\text{\textgreek{y}}})$
on $t=0$, with the property that $\int_{t=0}J_{\text{\textgreek{m}}}^{N}(\text{\textgreek{y}}-\tilde{\text{\textgreek{y}}})n^{\text{\textgreek{m}}}<0$. 

Let $\tilde{\text{\textgreek{y}}}$ be the unique solution of $\square_{g}\tilde{\text{\textgreek{y}}}=0$
on $\mathcal{D}$ with initial data $(\tilde{\text{\textgreek{y}}},\partial_{t}\tilde{\text{\textgreek{y}}})|_{t=0}$.
Then, due to the linearity of the wave operator, $\text{\textgreek{y}}-\tilde{\text{\textgreek{y}}}$
will also solve $\square_{g}(\text{\textgreek{y}}-\tilde{\text{\textgreek{y}}})=0$.
Due to the boundedness assumption \hyperref[Assumption 4]{4} and
the conservation of the $J^{T}$-flux, we can then bound 
\begin{align}
\int_{S_{\text{\textgreek{t}}}}J_{\text{\textgreek{m}}}^{N}(\text{\textgreek{y}}-\tilde{\text{\textgreek{y}}})n_{S}^{\text{\textgreek{m}}} & <C\cdot\int_{t=0}J_{\text{\textgreek{m}}}^{N}(\text{\textgreek{y}}-\tilde{\text{\textgreek{y}}})n^{\text{\textgreek{m}}}<\label{eq:DifferneceEstimate}\\
 & <C\cdot\text{\textgreek{e}}.\nonumber 
\end{align}

Moreover, since $(\tilde{\text{\textgreek{y}}},\partial_{t}\tilde{\text{\textgreek{y}}})|_{t=0}$
is smooth and compactly supported and thus has finite $\int_{t=0}r\cdot J_{\text{\textgreek{m}}}^{N}(\tilde{\text{\textgreek{y}}})n^{\text{\textgreek{m}}}$
and $\int_{t=0}J_{\text{\textgreek{m}}}^{N}(T\tilde{\text{\textgreek{y}}})n^{\text{\textgreek{m}}}$
norms, Corollary \ref{cor:Corollary} applies to show that 
\begin{equation}
\int_{S_{\text{\textgreek{t}}}}J_{\text{\textgreek{m}}}^{N}(\tilde{\text{\textgreek{y}}})n_{S}^{\text{\textgreek{m}}}\le\frac{C(\tilde{\text{\textgreek{y}}})}{\{\log(2+\text{\textgreek{t}})\}^{2}}.
\end{equation}
 Thus, there exists some $\text{\textgreek{t}}_{0}$ depending on
$\tilde{\text{\textgreek{y}}}$ itself such that for $\text{\textgreek{t}}\ge\text{\textgreek{t}}_{0}$:
\begin{equation}
\int_{S_{\text{\textgreek{t}}}}J_{\text{\textgreek{m}}}^{N}(\tilde{\text{\textgreek{y}}})n_{S}^{\text{\textgreek{m}}}<\text{\textgreek{e}}.\label{eq:DecayForApproximation}
\end{equation}

Combining (\ref{eq:DifferneceEstimate}) and (\ref{eq:DecayForApproximation}),
and using a triangle inequality, we thus obtain the required estimate:
\begin{equation}
\int_{S_{\text{\textgreek{t}}}}J_{\text{\textgreek{m}}}^{N}(\text{\textgreek{y}})n_{S}^{\text{\textgreek{m}}}<C\cdot\text{\textgreek{e}}.
\end{equation}

The claim that $\lim_{\text{\textgreek{t}}\rightarrow+\infty}\int_{S_{\text{\textgreek{t}}}}J_{\text{\textgreek{m}}}^{T}(\text{\textgreek{y}})n_{S}^{\text{\textgreek{m}}}=0$
if $\int_{t=0}J_{\text{\textgreek{m}}}^{T}(\text{\textgreek{y}})n^{\text{\textgreek{m}}}<\infty$
in the case the $J^{T}$-energy is coercive follows in exactly the
same way, using the conservation of the $J^{T}$-current in place
of the boundedness of the $J^{N}$-energy. \qed

\subsection{Proof of Corollary \ref{cor:CorollaryScattering}}

Assume that $(\mathcal{D},g)$ satisfies the assumptions of Corollary
\ref{cor:CorollaryScattering}. We will also assume without loss of
generality that $\mathcal{I}_{as}$ has only one connected component
(see also Section \ref{sub:Remark}).

We will first show the first part of Corollary \ref{cor:CorollaryScattering},
that is to say, the statement that for any solution $\text{\textgreek{y}}$
to the wave equation $\square_{g}\text{\textgreek{y}}=0$ on $\mathcal{D}$
with 
\begin{equation}
\int_{t=0}J_{\text{\textgreek{m}}}^{T}(\text{\textgreek{y}})n^{\text{\textgreek{m}}}<\infty,\label{eq:FinitenessEnergy}
\end{equation}
 the following equality holds: 
\begin{equation}
\int_{\mathcal{H}^{+}\cap\{t\ge0\}}J_{\text{\textgreek{m}}}^{T}(\text{\textgreek{y}})n_{\mathcal{H}^{+}}^{\text{\textgreek{m}}}+\int_{\mathcal{I}^{+}}J_{\text{\textgreek{m}}}^{T}(\text{\textgreek{y}})n_{\mathcal{I}^{+}}^{\text{\textgreek{m}}}=\int_{t=0}J_{\text{\textgreek{m}}}^{T}(\text{\textgreek{y}})n^{\text{\textgreek{m}}}.\label{eq:EnergyIdentityInfinity}
\end{equation}

Let $\{S_{\text{\textgreek{t}}}\}_{\text{\textgreek{t}}\in\mathbb{R}}$
be the level sets of the function $\bar{t}$ constructed in Section
\ref{sec:Polynomial-decay}. We can assume without loss of generality
(after setting $\bar{t}\rightarrow\bar{t}+c$) that $S_{0}\subset J^{+}(\{t=0\})$.
Moreover, in view of the definition of the future radiation field
on $\mathcal{I}^{+}$ for solutions to (\ref{eq:WaveEquation}) with
merely finite initial energy norm (see the remarks above the statement
of Corollary \ref{eq:CorollaryScattering}), we will assume without
loss of generality that $\text{\textgreek{y}}$ is smooth and induces
compactly supported initial data on $\text{\textgreek{S}}\cap\mathcal{D}$. 

The conservation of the $J_{\text{\textgreek{m}}}^{T}(\text{\textgreek{y}})$
current in the region bounded by $\{t=0\},S_{\text{\textgreek{t}}},\mathcal{H}^{+}$
and a null hyperurface $\mathcal{I}_{n,\text{\textgreek{t}}}^{+}$
of $\mathcal{M}$ of the form 
\begin{equation}
\mathcal{I}_{n,\text{\textgreek{t}}}^{+}\doteq\partial J^{-}(S_{\text{\textgreek{t}}}\cap\{r\le R_{n}\})\backslash\mathcal{H}^{+}
\end{equation}
for a sequence $R_{n}\rightarrow+\infty$ yields: 
\begin{align}
\int_{\{t=0\}\cap J^{-}(S_{\text{\textgreek{t}}}\cap\{r\le R_{n}\})}J_{\text{\textgreek{m}}}^{T}(\text{\textgreek{y}})= & \int_{\mathcal{H}^{+}\cap J^{-}(S_{\text{\textgreek{t}}})\cap\{t\ge0\}}J_{\text{\textgreek{m}}}^{T}(\text{\textgreek{y}})n_{\mathcal{H}^{+}}^{\text{\textgreek{m}}}+\int_{S_{\text{\textgreek{t}}}\cap\{r\le R_{n}\}}J_{\text{\textgreek{m}}}^{T}(\text{\textgreek{y}})n_{S}^{\text{\textgreek{m}}}+\label{eq:TIdentity-1}\\
 & +\int_{\mathcal{I}_{n,\text{\textgreek{t}}}^{+}\cap J^{-}(S_{\text{\textgreek{t}}})}J_{\text{\textgreek{m}}}^{T}(\text{\textgreek{y}})n_{\mathcal{I}_{n,\text{\textgreek{t}}}^{+}}^{\text{\textgreek{m}}}.\nonumber 
\end{align}
The identity (\ref{eq:TIdentity-1}) implies that 
\begin{equation}
\sup_{n}\int_{\mathcal{I}_{n,\text{\textgreek{t}}}^{+}\cap J^{-}(S_{\text{\textgreek{t}}})}J_{\text{\textgreek{m}}}^{T}(\text{\textgreek{y}})n_{\mathcal{I}_{n,\text{\textgreek{t}}}^{+}}^{\text{\textgreek{m}}}\le\int_{t=0}J_{\text{\textgreek{m}}}^{T}(\text{\textgreek{y}})n^{\text{\textgreek{m}}}<+\infty.\label{eq:ForDominatedConvergence}
\end{equation}
Moreover, since $\text{\textgreek{y}}$ was assumed to induce compactly
supported initial data on $\{t=0\}$, there exist $R_{0}\gg1$ and
$u_{0}\in\mathbb{R}$ such that (in view of the domain of dependence
property for the wave equation)
\begin{equation}
supp(\text{\textgreek{y}})\cap\{r\ge R_{0}\}\subset\{u\ge u_{0}\}\label{eq:UniformSmallnessEnergyFlux}
\end{equation}
(where the function $u$ is constructed in the Appendix). Therefore,
in view of (\ref{eq:ForDominatedConvergence}), (\ref{eq:UniformSmallnessEnergyFlux})
and the fact that $r^{\frac{d-1}{2}}\text{\textgreek{y}}(u,r,\text{\textgreek{sv}})\rightarrow\text{\textgreek{Y}}_{\mathcal{I}^{+}}(u,\text{\textgreek{sv}})$
in the $H_{loc}^{1}(\mathbb{R}\times\mathbb{S}^{d-1},du^{2}+g_{\mathbb{S}^{d-1}})$
norm (see the remarks above the statement of Corollary \ref{cor:CorollaryScattering}),
the dominated convergence theorem implies that 
\begin{equation}
\int_{\mathcal{I}_{n}^{+}}J_{\text{\textgreek{m}}}^{T}(\text{\textgreek{y}})n_{\mathcal{I}_{n}^{+}}^{\text{\textgreek{m}}}\rightarrow\int_{\mathcal{I}^{+}\cap J^{-}(S_{\text{\textgreek{t}}})}J_{\text{\textgreek{m}}}^{T}(\text{\textgreek{y}})n_{\mathcal{I}^{+}}^{\text{\textgreek{m}}},
\end{equation}
and thus (\ref{eq:TIdentity-1}) readily implies for any $\text{\textgreek{t}}\ge0$
(in the limit $n\rightarrow+\infty$):
\begin{align}
\int_{t=0}J_{\text{\textgreek{m}}}^{T}(\text{\textgreek{y}})= & \int_{\mathcal{H}^{+}\cap J^{-}(S_{\text{\textgreek{t}}})\cap\{t\ge0\}}J_{\text{\textgreek{m}}}^{T}(\text{\textgreek{y}})n_{\mathcal{H}^{+}}^{\text{\textgreek{m}}}+\int_{S_{\text{\textgreek{t}}}}J_{\text{\textgreek{m}}}^{T}(\text{\textgreek{y}})n_{S}^{\text{\textgreek{m}}}+\label{eq:TIdentity}\\
 & +\int_{\mathcal{I}^{+}\cap J^{-}(S_{\text{\textgreek{t}}})}J_{\text{\textgreek{m}}}^{T}(\text{\textgreek{y}})n_{\mathcal{I}^{+}}^{\text{\textgreek{m}}},\nonumber 
\end{align}
 where $\int_{\mathcal{I}^{+}\cap J^{-}(S_{\text{\textgreek{t}}})}J_{\text{\textgreek{m}}}^{T}(\text{\textgreek{y}})n_{\mathcal{I}^{+}}^{\text{\textgreek{m}}}$
is defined as (\ref{eq:EnergyI+}). 

According to Corollary \ref{cor:CorollaryQualititative}, since $\int_{t=0}J_{\text{\textgreek{m}}}^{T}(\text{\textgreek{y}})n^{\text{\textgreek{m}}}<\infty$
we also have: 
\begin{equation}
\lim_{\text{\textgreek{t}}\rightarrow\infty}\int_{S_{\text{\textgreek{t}}}}J_{\text{\textgreek{m}}}^{T}(\text{\textgreek{y}})n_{S}^{\text{\textgreek{m}}}=0.
\end{equation}
 Thus, in the limit $\text{\textgreek{t}}\rightarrow+\infty$, the
relation (\ref{eq:TIdentity}) yields the desired identity (\ref{eq:EnergyIdentityInfinity}).

In order to complete the proof of Corollary \ref{cor:CorollaryScattering},
it remains to show that for any pair $(\text{\textgreek{y}}|_{\mathcal{H}^{+}},\text{\textgreek{Y}}_{\mathcal{I}^{+}})$
with finite $J^{T}$-energy, there exists an $H_{loc}^{1}$ solution
$\text{\textgreek{y}}$ to the wave equation $\square_{g}\text{\textgreek{y}}=0$
on $\mathcal{D}$ with $\int_{t=0}J_{\text{\textgreek{m}}}^{T}(\text{\textgreek{y}})n^{\text{\textgreek{m}}}<\infty$,
which gives rise to the scattering data pair $(\text{\textgreek{y}}|_{\mathcal{H}^{+}},\text{\textgreek{Y}}_{\mathcal{I}^{+}})$.
To this end, let us define Hilbert space $\mathscr{E}_{0}$ of finite
energy initial data on $\{t=0\}$ consisting of the completion of
the space $C_{0}^{\infty}(\{t=0\})\times C_{0}^{\infty}(\{t=0\})$
with the $J^{T}-$norm: 
\begin{equation}
||(\text{\textgreek{y}}_{0},\text{\textgreek{y}}_{1})||_{J^{T}}\doteq\int_{t=0}J_{\text{\textgreek{m}}}^{T}(\text{\textgreek{y}})n^{\text{\textgreek{m}}},
\end{equation}
 where $\text{\textgreek{y}}$ is the unique function defined on $\{t\ge0\}\cap\mathcal{D}$
such that $\text{\textgreek{y}}|_{t=0}=\text{\textgreek{y}}_{0}$
and $T\text{\textgreek{y}}|_{t=0}=\text{\textgreek{y}}_{1}$. Let
us also define the Hilbert space $\mathscr{E}_{sc}$ of finite energy
scattering data as the completion of the space $C_{0}^{\infty}(\mathcal{H}^{+}\cap\{t\ge0\})\times C_{0}^{\infty}(\mathcal{I}^{+})$
(where $\mathcal{I}^{+}\simeq\mathbb{R}\times\mathbb{S}^{d-1}$) with
the norm 
\begin{equation}
||(\text{\textgreek{y}}|_{\mathcal{H}^{+}\cap\{t\ge0\}},\text{\textgreek{Y}}_{\mathcal{I}^{+}})||_{\mathscr{E}_{sc}}\doteq\int_{\mathcal{H}^{+}\cap\{t\ge0\}}J_{\text{\textgreek{m}}}^{T}(\text{\textgreek{y}})n_{\mathcal{H}^{+}}^{\text{\textgreek{m}}}+\int_{\mathcal{I}^{+}}J_{\text{\textgreek{m}}}^{T}(\text{\textgreek{y}})n_{\mathcal{I}^{+}}^{\text{\textgreek{m}}}.\label{eq:ScatteringNorm}
\end{equation}
We will also define the scattering map $Scatt:\mathscr{E}_{0}\rightarrow\mathscr{E}_{sc}$,
which maps any pair of initial data $(\text{\textgreek{y}}_{0},\text{\textgreek{y}}_{1})$
to the scattering data $(\text{\textgreek{y}}|_{\mathcal{H}^{+}\cap\{t\ge0\}},\text{\textgreek{Y}}_{\mathcal{I}^{+}})$,
induced by the solution $\text{\textgreek{y}}$ to $\square_{g}\text{\textgreek{y}}=0$
with initial data $\text{\textgreek{y}}|_{t=0}=\text{\textgreek{y}}_{0}$
and $T\text{\textgreek{y}}|_{t=0}=\text{\textgreek{y}}_{1}$, according
to the remark above the statement of Corollary \ref{cor:CorollaryScattering}. 

With these notations, in order to complete the proof of Corollary
\ref{cor:CorollaryScattering}, we have to show that the scattering
map $Scatt$ is onto. In view of (\ref{eq:EnergyIdentityInfinity}),
$Scatt$ is an isometric embedding. Thus, it suffices to show that
its image is dense in $\mathcal{E}_{sc}$. In particular, we will
show that for any scattering data pair $(\text{\textgreek{y}}|_{\mathcal{H}^{+}\cap\{t\ge0\}},\text{\textgreek{Y}}_{\mathcal{I}^{+}})$
which are smooth and compactly supported and any $\text{\textgreek{e}}>0$,
there exists a solution $\tilde{\text{\textgreek{y}}}$ to the wave
equation $\square\tilde{\text{\textgreek{y}}}=0$ on $\{t\ge0\}\cap\mathcal{D}$
with $\int_{t=0}J_{\text{\textgreek{m}}}^{T}(\tilde{\text{\textgreek{y}}})n^{\text{\textgreek{m}}}<+\infty$,
such that 
\begin{equation}
||(\tilde{\text{\textgreek{y}}}|_{\mathcal{H}^{+}\cap\{t\ge0\}},\tilde{\text{\textgreek{Y}}}_{\mathcal{I}^{+}})-(\text{\textgreek{y}}|_{\mathcal{H}^{+}\cap\{t\ge0\}},\text{\textgreek{Y}}_{\mathcal{I}^{+}})||_{\mathscr{E}_{sc}}<\text{\textgreek{e}},\label{eq:SmallnessScattering}
\end{equation}
where $\tilde{\text{\textgreek{Y}}}_{\mathcal{I}^{+}}$ is the future
radiation field of $\tilde{\text{\textgreek{y}}}$.

We will assume without loss of generality that the scattering data
pair $(\text{\textgreek{y}}|_{\mathcal{H}^{+}\cap\{t\ge0\}},\text{\textgreek{Y}}_{\mathcal{I}^{+}})$
is real valued. Since $(\text{\textgreek{y}}|_{\mathcal{H}^{+}\cap\{t\ge0\}},\text{\textgreek{Y}}_{\mathcal{I}^{+}})$
was assumed to be compactly supported, there exists some $t_{f}>0$
such that $\text{\textgreek{y}}|_{\mathcal{H}^{+}\cap\{t\ge t_{f}\}}\equiv0$
and $\text{\textgreek{Y}}_{\mathcal{I}^{+}}|_{\{u\ge t_{f}\}}\equiv0$.
Moreover, there exists some $t_{in}<0$ such that $\text{\textgreek{Y}}_{\mathcal{I}^{+}}|_{\{u\le t_{in}\}}\equiv0$. 

Let us also introduce the $C^{1}$ optical function $\bar{v}:\{r\ge R_{far}\}\cap\{\bar{t}\le t_{f}\}\rightarrow\mathbb{R}$
such that the level sets of $\bar{v}$ are the hypersurfaces 
\begin{equation}
\mathcal{I}_{R}^{+}\doteq\partial(J^{-}(S_{t_{f}}\cap\{r\le R\})\backslash\mathcal{H}^{+}
\end{equation}
 and satifying $T\bar{v}=1$. 

In the region $\{\bar{t}\le t_{f}\}\cap\{\bar{v}\ge v_{0}\}\subset\mathcal{D}$
for some $v_{0}\gg1$, we will introduce the flat metric associated
to the $(u,v,\text{\textgreek{sv}})$ coordinate chart there: 
\begin{equation}
g_{fl}=-4dudv+r^{2}g_{\mathbb{S}^{d-1}}.\label{eq:Minkowski metric}
\end{equation}
 In this region, we can construct a solution $\text{\textgreek{y}}_{fl}$
to the \underline{flat} wave equation by solving the following scattering
problem:
\begin{equation}
\begin{cases}
\square_{g_{fl}}\text{\textgreek{y}}_{fl}=0 & \mbox{on }\{\bar{t}\le t_{f}\}\cap\{\bar{v}\ge v_{0}\}\\
(\text{\textgreek{y}}_{fl},T\text{\textgreek{y}}_{fl})|_{S_{t_{f}}\cap\{\bar{v}\ge v_{0}\}}=0\\
\lim_{v\rightarrow+\infty}r^{\frac{d-1}{2}}\text{\textgreek{y}}_{fl}(u,v,\text{\textgreek{sv}})=\text{\textgreek{Y}}_{\mathcal{I}^{+}}(u,\text{\textgreek{sv}}).
\end{cases}\label{eq:FlatScatteringProblem}
\end{equation}
Notice that the scattering problem (\ref{eq:FlatScatteringProblem})
is a scattering problem on Minkowski spacetime, and the existence
and uniqueness of a solution $\text{\textgreek{y}}_{fl}$ to (\ref{eq:FlatScatteringProblem})
follows readily by using the confromal compactification method and
solving (backwards) a regular mixed characteristic-initial value problem
for the conformal wave equation (in the setup of \cite{Christodoulou1986}).
Alternatively, one can use the $r^{p}$-weighted energy method of
\cite{DafRod7}, see Section 9.6 of \cite{Dafermos2014}. Following
either way of proof, one obtains a unique smooth solution $\text{\textgreek{y}}_{fl}$
to (\ref{eq:FlatScatteringProblem}) satisfying the following qualititative
bound 
\begin{equation}
\sum_{j=0}^{k}\sum_{j_{1}+j_{2}+j_{3}=j}\int_{\{t\ge0\}\cap\{\bar{t}\le t_{f}\}\cap\{\bar{v}\ge v_{0}\}}r^{-1-\text{\textgreek{d}}}\cdot|r^{j_{1}}\partial_{v}^{j_{1}}\partial_{\text{\textgreek{sv}}}^{j_{2}}\partial_{u}^{j_{3}}(r^{\frac{d-1}{2}}\text{\textgreek{y}}_{fl})|^{2}\, dudvd\text{\textgreek{sv}}<+\infty\label{eq:qualitativebound}
\end{equation}
 for any $k\in\mathbb{N}$ and any $\text{\textgreek{d}}>0$. 

In view of (\ref{eq:qualitativebound}), for any $C_{0}>0$ (to be
fixed later), there exists a $v_{\text{\textgreek{e}}}\gg v_{0}$
(depending on $\text{\textgreek{e}},\text{\textgreek{d}},C_{0}$ and
the function $\text{\textgreek{y}}_{fl}$ itself) such that 
\begin{equation}
\sum_{j=0}^{3}\sum_{j_{1}+j_{2}+j_{3}=j}\int_{\{t\ge0\}\cap\{\bar{t}\le t_{f}\}\cap\{\bar{v}\ge v_{\text{\textgreek{e}}}\}}r^{-1-\text{\textgreek{d}}}\cdot|r^{j_{1}}\partial_{v}^{j_{1}}\partial_{\text{\textgreek{sv}}}^{j_{2}}\partial_{u}^{j_{3}}(r^{\frac{d-1}{2}}\text{\textgreek{y}}_{fl})|^{2}\, dudvd\text{\textgreek{sv}}<C_{0}^{-1}\text{\textgreek{e}}.\label{eq:qualitativeepsilonbulkbound}
\end{equation}
Let $R_{\text{\textgreek{e}}}>1$ be fixed so that $\{\bar{t}=t_{f}\}\cap\{\bar{v}=v_{\text{\textgreek{e}}}\}\subset\{R_{\text{\textgreek{e}}}-1\le r\le R_{\text{\textgreek{e}}}\}$.
Then, there exists a smooth spacelike hypersurface $\bar{S}_{\text{\textgreek{e}}}$
of $(\mathcal{D},g)$ (depending again on $\text{\textgreek{e}},\text{\textgreek{d}},C_{0}$
and $\text{\textgreek{y}}_{fl}$) such that $\bar{S}_{\text{\textgreek{e}}}\equiv\{t=R_{\text{\textgreek{e}}}+2v_{\text{\textgreek{e}}}\}$
in the region $\{r\ge R_{\text{\textgreek{e}}}+1\}$ and $\bar{S}_{\text{\textgreek{e}}}\equiv S_{t_{f}}$
in the region $\{r\le R_{\text{\textgreek{e}}}\}$. Notice that the
domain of dependence of $\bar{S}_{\text{\textgreek{e}}}\cup(\mathcal{H}^{+}\cap\{t\ge0\})$
contains the whole of $\{t\ge0\}$.

We will construct the solution $\tilde{\text{\textgreek{y}}}$ to
$\square_{g}\tilde{\text{\textgreek{y}}}=0$ on $\{t\ge0\}$ with
initial-characteristic data on $\bar{S}_{\text{\textgreek{e}}}\cup(\mathcal{H}^{+}\cap\{t\ge0\}$
as follows:
\begin{equation}
\begin{cases}
\square_{g}\tilde{\text{\textgreek{y}}}=0 & \mbox{on }J^{+}(\text{\textgreek{S}})\cap\mathcal{D}\\
(\tilde{\text{\textgreek{y}}},T\tilde{\text{\textgreek{y}}})|_{\bar{S}_{\text{\textgreek{e}}}\cap\{r\le R_{\text{\textgreek{e}}}-1\}}=(0,0)\\
(\tilde{\text{\textgreek{y}}},T\tilde{\text{\textgreek{y}}})|_{\bar{S}_{\text{\textgreek{e}}}\cap\{r\ge R_{\text{\textgreek{e}}}-1\}}=(\text{\textgreek{y}}_{fl},T\text{\textgreek{y}}_{fl})|_{\bar{S}_{\text{\textgreek{e}}}\cap\{r\ge R_{\text{\textgreek{e}}}-1\}}\\
\tilde{\text{\textgreek{y}}}|_{\mathcal{H}^{+}\cap\{t\ge0\}}=\text{\textgreek{y}}|_{\mathcal{H}^{+}\cap\{t\ge0\}}.
\end{cases}\label{eq:PerturbedWave}
\end{equation}
Notice that the initial data for $\tilde{\text{\textgreek{y}}}$ on
$\bar{S}_{\text{\textgreek{e}}}$ are smooth, in view of the fact
that $(\text{\textgreek{y}}_{fl},T\text{\textgreek{y}}_{fl})|_{S_{t_{f}}\cap\{\bar{v}\ge v_{0}\}}=0$.
Hence, $\tilde{\text{\textgreek{y}}}$ is a smooth function on $\{t\ge0\}$. 

We will now show that the radiation field of $\tilde{\text{\textgreek{y}}}$
on $\mathcal{I}^{+}$ exists in the $H_{loc}^{1}$ norm and is actually
$\text{\textgreek{e}}$-close to $\text{\textgreek{Y}}_{\mathcal{I}^{+}}$
in the $\int_{\mathcal{I}^{+}}J_{\text{\textgreek{m}}}^{T}(\cdot)n_{\mathcal{I}^{+}}^{\text{\textgreek{m}}}$
norm. Notice that for any smooth function $\text{\textgreek{f}}$
on $\{\bar{t}\le t_{f}\}\cap\{\bar{v}\ge v_{0}\}$ we can estimate
\begin{align}
r^{\frac{d-1}{2}}\square_{g}\text{\textgreek{f}}=r^{\frac{d-1}{2}} & \square_{fl}\text{\textgreek{f}}+O_{3}(r^{-2-a})\partial_{u}^{2}(r^{\frac{d-1}{2}}\text{\textgreek{f}})+O_{3}(r^{-1})\partial_{v}^{2}(r^{\frac{d-1}{2}}\text{\textgreek{f}})+O_{3}(r^{-2-a})\partial_{u}\partial_{\text{\textgreek{sv}}}(r^{\frac{d-1}{2}}\text{\textgreek{f}})+O_{3}(r^{-2})\partial_{v}\partial_{\text{\textgreek{sv}}}(r^{\frac{d-1}{2}}\text{\textgreek{f}})+\label{eq:ErrorTermsMorawetz}\\
 & +O_{2}(r^{-3-a})\partial_{\text{\textgreek{sv}}}\partial_{\text{\textgreek{sv}}}(r^{\frac{d-1}{2}}\text{\textgreek{f}})+O_{2}(r^{-2-a})\partial_{u}(r^{\frac{d-1}{2}}\text{\textgreek{f}})+O_{2}(r^{-1-a})\partial_{v}(r^{\frac{d-1}{2}}\text{\textgreek{f}})+O_{2}(r^{-2-a})\partial_{\text{\textgreek{sv}}}(r^{\frac{d-1}{2}}\text{\textgreek{f}})+O_{1}(r^{-3})(r^{\frac{d-1}{2}}\text{\textgreek{f}}).\nonumber 
\end{align}
In view of (\ref{eq:FlatScatteringProblem}), (\ref{eq:PerturbedWave})
and (\ref{eq:ErrorTermsMorawetz}) for $\text{\textgreek{y}}_{fl}$
in place of $\text{\textgreek{f}}$, the difference $\tilde{\text{\textgreek{y}}}-\text{\textgreek{y}}_{fl}$
solves the following initial value problem on $\{\bar{t}\le t_{f}\}\cap J^{+}(\bar{S}_{\text{\textgreek{e}}})$:
\begin{equation}
\begin{cases}
\square_{g}(\tilde{\text{\textgreek{y}}}-\text{\textgreek{y}}_{fl})=F_{fl} & \mbox{ on }\{\bar{t}\le t_{f}\}\cap J^{+}(\bar{S}_{\text{\textgreek{e}}})\\
\big(\tilde{\text{\textgreek{y}}}-\text{\textgreek{y}}_{fl},T(\tilde{\text{\textgreek{y}}}-\text{\textgreek{y}}_{fl})\big)\big|_{\{\bar{t}\le t_{f}\}\cap\bar{S}_{\text{\textgreek{e}}}}=0,
\end{cases}\label{eq:InitialvalueProblemDifference}
\end{equation}
where, in view of (\ref{eq:ErrorTermsMorawetz}) and (\ref{eq:qualitativeepsilonbulkbound})
we can estimate (for some $\text{\textgreek{d}}_{0}>0$ small in terms
of $\text{\textgreek{d}},a$) 
\begin{equation}
\sum_{j=0}^{1}\sum_{j_{1}+j_{2}+j_{3}=j}\int_{\{\bar{t}\le t_{f}\}\cap J^{+}(\bar{S}_{\text{\textgreek{e}}})}r^{1+\text{\textgreek{d}}_{0}}|r^{j_{1}}\partial_{v}^{j_{1}}\partial_{\text{\textgreek{sv}}}^{j_{2}}\partial_{u}^{j_{3}}(r^{\frac{d-1}{2}}F_{fl})|^{2}\, dudvd\text{\textgreek{sv}}\le C\cdot C_{0}^{-1}\text{\textgreek{e}}.\label{eq:BoundSourceDifference}
\end{equation}
 Repeating the proof of Theorem 7.1 of \cite{Moschidisc} for $\tilde{\text{\textgreek{y}}}-\text{\textgreek{y}}_{fl}$
in the region $\{\bar{t}\le t_{f}\}\cap J^{+}(\bar{S}_{\text{\textgreek{e}}})$,
(\ref{eq:InitialvalueProblemDifference}) and (\ref{eq:BoundSourceDifference})
imply that for any increasing sequence $\{v_{n}\}_{n\in\mathbb{N}}$
with $v_{n}\rightarrow+\infty$, the sequence $r^{\frac{d-1}{2}}(\tilde{\text{\textgreek{y}}}-\text{\textgreek{y}}_{fl})(u,v_{n},\text{\textgreek{sv}})$
converges in the $H_{loc}^{1}$ norm. Since by construction of the
function $\text{\textgreek{y}}_{fl}$ the sequence $r^{\frac{d-1}{2}}\text{\textgreek{y}}_{fl}(u,v_{n},\text{\textgreek{sv}})$
converges in the $H_{loc}^{1}$ topology, we also deduce that the
sequence $r^{\frac{d-1}{2}}\tilde{\text{\textgreek{y}}}(u,v_{n},\text{\textgreek{sv}})$
converges in the $H_{loc}^{1}$ norm, and we set 
\begin{equation}
\tilde{\text{\textgreek{Y}}}_{\mathcal{I}^{+}}(u,\text{\textgreek{sv}})\doteq\lim_{v\rightarrow+\infty}\big(r^{\frac{d-1}{2}}\tilde{\text{\textgreek{y}}}(u,v,\text{\textgreek{sv}})\big).\label{eq:RadiationFieldTilde}
\end{equation}
 Furthermore, by applying the energy identity for the $T$ vector
field on the region $\{\bar{t}\le t_{f}\}\cap J^{+}(\bar{S}_{\text{\textgreek{e}}})$,
we infer from (\ref{eq:InitialvalueProblemDifference}) and (\ref{eq:BoundSourceDifference})
that: 
\begin{equation}
\int_{\mathcal{I}^{+}\cap\{\bar{t}\le t_{f}\}}J_{\text{\textgreek{m}}}^{T}(\tilde{\text{\textgreek{y}}}-\text{\textgreek{y}}_{fl})n_{\mathcal{I}^{+}}^{\text{\textgreek{m}}}\le C\cdot C_{0}^{-1}\text{\textgreek{e}}.\label{eq:EnergyDifferenceNotOnWhole}
\end{equation}

Since the function $\text{\textgreek{y}}_{fl}$ solves 
\begin{equation}
\square_{g}\text{\textgreek{y}}_{fl}=-F_{fl}
\end{equation}
and satisfies (\ref{eq:FlatScatteringProblem}), by applying the energy
identity for the $T$ vector field on the region $\{\bar{t}\le t_{f}\}\cap J^{+}(\bar{S}_{\text{\textgreek{e}}})$
in the backward evolution (notice that $\bar{S}_{\text{\textgreek{e}}}\cap\{\bar{t}\le t_{f}\}$
lies in the domain of dependence of $\{\bar{t}=t_{f}\}\cup(\mathcal{I}^{+}\cap\{u\le t_{f}\}$),
we can bound (in view of (\ref{eq:BoundSourceDifference})): 
\begin{equation}
\int_{\bar{S}_{\text{\textgreek{e}}}\cap\{\bar{t}\le t_{f}\}}J_{\text{\textgreek{m}}}^{T}(\text{\textgreek{y}}_{fl})n^{\text{\textgreek{m}}}\le\int_{\mathcal{I}^{+}}J_{\text{\textgreek{m}}}^{T}(\text{\textgreek{y}})n_{\mathcal{I}^{+}}^{\text{\textgreek{m}}}+C\cdot C_{0}^{-1}\text{\textgreek{e}}.\label{eq:BoundSpacelikeEnergyFlatPsi}
\end{equation}
Thus, since 
\begin{equation}
(\tilde{\text{\textgreek{y}}},T\tilde{\text{\textgreek{y}}})|_{\bar{S}_{\text{\textgreek{e}}}\cap\{r\ge R_{\text{\textgreek{e}}}-1\}}=(\text{\textgreek{y}}_{fl},T\text{\textgreek{y}}_{fl})|_{\bar{S}_{\text{\textgreek{e}}}\cap\{r\ge R_{\text{\textgreek{e}}}-1\}}
\end{equation}
and 
\begin{equation}
(\tilde{\text{\textgreek{y}}},T\tilde{\text{\textgreek{y}}})|_{\bar{S}_{\text{\textgreek{e}}}\cap\{r\le R_{\text{\textgreek{e}}}-1\}}=(0,0),
\end{equation}
(\ref{eq:BoundSpacelikeEnergyFlatPsi}) immediately yields: 
\begin{equation}
\int_{\bar{S}_{\text{\textgreek{e}}}}J_{\text{\textgreek{m}}}^{T}(\tilde{\text{\textgreek{y}}})n_{\bar{S}_{\text{\textgreek{e}}}}^{\text{\textgreek{m}}}\le\int_{\mathcal{I}^{+}}J_{\text{\textgreek{m}}}^{T}(\text{\textgreek{y}})n^{\text{\textgreek{m}}}+C\cdot C_{0}^{-1}\text{\textgreek{e}},\label{eq:BoundSpacelikeEnergyFlatPsi-1}
\end{equation}
where $n_{\bar{S}_{\text{\textgreek{e}}}}$ is the future directed
unit normal of the hypersurface $\bar{S}_{\text{\textgreek{e}}}$.
Therefore, in view of the fact that $\tilde{\text{\textgreek{y}}}|_{\mathcal{H}^{+}\cap\{t\ge0\}}=\text{\textgreek{y}}|_{\mathcal{H}^{+}\cap\{t\ge0\}}$,
by applying once more the energy identity for the $T$ vector field
in the region boundd by $\{t=0\}$, $\mathcal{H}^{+}$ and $\bar{S}_{\text{\textgreek{e}}}$,
we obtain (using (\ref{eq:BoundSpacelikeEnergyFlatPsi-1})): 
\begin{equation}
\int_{t=0}J_{\text{\textgreek{m}}}^{T}(\tilde{\text{\textgreek{y}}})n^{\text{\textgreek{m}}}\le\int_{\mathcal{H}^{+}\cap\{t\ge0\}}J_{\text{\textgreek{m}}}^{T}(\text{\textgreek{y}})n_{\mathcal{H}^{+}}^{\text{\textgreek{m}}}+\int_{\mathcal{I}^{+}}J_{\text{\textgreek{m}}}^{T}(\text{\textgreek{y}})n_{\mathcal{I}^{+}}^{\text{\textgreek{m}}}+C\cdot C_{0}^{-1}\text{\textgreek{e}}.\label{eq:InitialEnergyPerturbedwave}
\end{equation}

Since $\int_{t=0}J_{\text{\textgreek{m}}}^{T}(\tilde{\text{\textgreek{y}}})n^{\text{\textgreek{m}}}<+\infty$
(in view of (\ref{eq:InitialEnergyPerturbedwave})), according to
the definition above the statement of Corollary \ref{cor:CorollaryScattering}
the radiation field $\tilde{\text{\textgreek{Y}}}_{\mathcal{I}^{+}}$
of $\tilde{\text{\textgreek{y}}}$ on $\mathcal{I}^{+}$ exists in
the Hilbert space defined by the $\int_{\mathcal{I}^{+}}J_{\text{\textgreek{m}}}^{T}(\cdot)n_{\mathcal{I}^{+}}^{\text{\textgreek{m}}}$
norm, and its restriction on $\mathcal{I}^{+}\cap\{u\le t_{f}\}$
coincides with the limit (\ref{eq:RadiationFieldTilde}). In view
of (\ref{eq:BoundHorizon}), which also applies for $\tilde{\text{\textgreek{y}}}$,
we can bound: 
\begin{equation}
\int_{\mathcal{H}^{+}\cap\{t\ge0\}}J_{\text{\textgreek{m}}}^{T}(\tilde{\text{\textgreek{y}}})n_{\mathcal{H}^{+}}^{\text{\textgreek{m}}}+\int_{\mathcal{I}^{+}}J_{\text{\textgreek{m}}}^{T}(\tilde{\text{\textgreek{y}}})n_{\mathcal{I}^{+}}^{\text{\textgreek{m}}}\le\int_{t=0}J_{\text{\textgreek{m}}}^{T}(\tilde{\text{\textgreek{y}}})n^{\text{\textgreek{m}}}.\label{eq:BoundHorizon-1}
\end{equation}
Thus, from (\ref{eq:InitialEnergyPerturbedwave}) and (\ref{eq:BoundHorizon-1})
(using also the fact that $\tilde{\text{\textgreek{y}}}|_{\mathcal{H}^{+}\cap\{t\ge0\}}=\text{\textgreek{y}}|_{\mathcal{H}^{+}\cap\{t\ge0\}}$)
we infer: 
\begin{equation}
\int_{\mathcal{I}^{+}}J_{\text{\textgreek{m}}}^{T}(\tilde{\text{\textgreek{y}}})n_{\mathcal{I}^{+}}^{\text{\textgreek{m}}}\le\int_{\mathcal{I}^{+}}J_{\text{\textgreek{m}}}^{T}(\text{\textgreek{y}})n_{\mathcal{I}^{+}}^{\text{\textgreek{m}}}+C\cdot C_{0}^{-1}\text{\textgreek{e}}.\label{eq:EnergyBoundRadiationField}
\end{equation}
From (\ref{eq:EnergyDifferenceNotOnWhole}) and the fact that $\text{\textgreek{Y}}_{\mathcal{I}^{+}}$
was assumed to be supported on $\{u\le t_{f}\}$, we obtain: 
\begin{equation}
\int_{\mathcal{I}^{+}\cap\{u\le t_{f}\}}J_{\text{\textgreek{m}}}^{T}(\tilde{\text{\textgreek{y}}})n_{\mathcal{I}^{+}}^{\text{\textgreek{m}}}\ge\int_{\mathcal{I}^{+}}J_{\text{\textgreek{m}}}^{T}(\text{\textgreek{y}})n_{\mathcal{I}^{+}}^{\text{\textgreek{m}}}-C\cdot C_{0}^{-1}\text{\textgreek{e}}.\label{eq:energyBoundRestricted}
\end{equation}
Therefore, from (\ref{eq:EnergyBoundRadiationField}) and (\ref{eq:energyBoundRestricted})
(as well as the fact that $\text{\textgreek{Y}}_{\mathcal{I}^{+}}\equiv0$
on $\{u\ge t_{f}\}$ we deduce: 
\begin{equation}
\int_{\mathcal{I}^{+}\cap\{u\ge t_{f}\}}J_{\text{\textgreek{m}}}^{T}(\tilde{\text{\textgreek{y}}}-\text{\textgreek{y}})n_{\mathcal{I}^{+}}^{\text{\textgreek{m}}}=\int_{\mathcal{I}^{+}\cap\{u\ge t_{f}\}}J_{\text{\textgreek{m}}}^{T}(\tilde{\text{\textgreek{y}}})n_{\mathcal{I}^{+}}^{\text{\textgreek{m}}}\le2C\cdot C_{0}^{-1}\text{\textgreek{e}},\label{eq:EpsilonBoundEnergyDifference}
\end{equation}
 which, after adding (\ref{eq:EnergyDifferenceNotOnWhole}) and assuming
$C_{0}$ was chosen large enough in terms of the geometry of $(\mathcal{D},g)$,
yields: 
\begin{equation}
\int_{\mathcal{I}^{+}}J_{\text{\textgreek{m}}}^{T}(\tilde{\text{\textgreek{y}}}-\text{\textgreek{y}})n_{\mathcal{I}^{+}}^{\text{\textgreek{m}}}\le\text{\textgreek{e}}.\label{eq:EpsilonBoundEnergyDifference-1}
\end{equation}

Combining (\ref{eq:EpsilonBoundEnergyDifference-1}) with the fact
that $\tilde{\text{\textgreek{y}}}|_{\mathcal{H}^{+}\cap\{t\ge0\}}=\text{\textgreek{y}}|_{\mathcal{H}^{+}\cap\{t\ge0\}}$,
we immediately infer the desired estimate (\ref{eq:SmallnessScattering}).
Thus, the proof of Corllary \ref{cor:CorollaryScattering} is complete.
\qed

\section{\label{sub:Sharpness}Sharpness of the logarithmic decay statement}

In this section, we will construct a spacetime $(\mathcal{M},g)$
satisfying all four assumptions \hyperref[Assumtion 1]{1}, \hyperref[Assumtion 2]{2},
\hyperref[Assumtion 3]{3} and \hyperref[Assumtion 4]{4}, for which
the logarithmic decay rate of Theorem \ref{thm:Theorem} is optimal.
This example is a modification of a counterexample provided by Rodnianski
and Tao in \cite{Rodnianski2011}, in the context of their proof of
effective limiting absorption principles for asymptotically conic
Riemannian manifolds, which in our setting correspond to ultrastatic
spacetimes (namely static spacetimes with $g(T,T)\equiv-1$).

We should notice, however, that a more realistic class of spacetimes
exhibiting the optimality of Theorem \ref{thm:Theorem} has been investigated
by Keir in the recent \cite{Keir2014}. This class includes the family
of spherically symmetric ultracompact neutron stars. The absence of
a horizon and an ergoregion, as well as the consrvation of a positive
definite energy makes these spacetimes valid backgrounds on which
Theorem \ref{thm:Theorem} holds. 

The construction of a simple spacetime $(\mathcal{M},g)$ exhibiting
the optimality of Theorem \ref{thm:Theorem} proceeds as follows:
We consider $\mathbb{R}^{3}$ equipped with a smooth Riemannian metric
$g_{0}$ with the following behaviour: In the usual polar coordinate
system $(r,\text{\textgreek{j}},\text{\textgreek{f}})$, we assume
that $g_{0}$ takes the folowing form:

\begin{equation}
g_{0}=dr^{2}+f(r)\cdot(d\text{\textgreek{j}}^{2}+\sin^{2}\text{\textgreek{j}}\cdot d\text{\textgreek{f}}^{2}),\label{eq:Counterexample1RiemannianMetric}
\end{equation}
where $f:[0,+\infty)\rightarrow[0,+\infty)$ is a smooth function
such that
\begin{itemize}
\item $f(r)=\sin^{2}r$ for $0\le r\le\frac{7\text{\textgreek{p}}}{8},$
\item $f(r)=r^{2}$ for $r\ge\text{\textgreek{p}}$ and
\item $f(r)>0$ for $\frac{7\text{\textgreek{p}}}{8}\le r\le\text{\textgreek{p}}.$
\end{itemize}
That is to say, $g_{0}$ coincides with the usual Euclidean metric
for $r\ge\text{\textgreek{p}}$, while the region $r\le\frac{7\text{\textgreek{p}}}{8}$
equipped with $g_{0}$ is isometric to $\mathbb{S}^{3}$ (with the
usual metric) with an open ball around one of the poles removed. 

Note that $\{r=\frac{\text{\textgreek{p}}}{2}\}$ corresponds to an
equator of $\mathbb{S}^{3}$, and hence this is a surface ruled by
stably trapped geodesics. It is the presence of this stable trapping
that will eventually lead to the desired slow logarithmic decay rate
for waves, in accordance also with the results on ultracompact neutron
stars of \cite{Keir2014}

We set $\mathcal{M}\doteq\mathbb{R}\times\mathbb{R}^{3}$, and we
construct the Lorentzian metric $g$ on $\mathcal{M}$ so that in
the coordinate system $(t,r,\text{\textgreek{j}},\text{\textgreek{f}})$,%
\footnote{$t$ denoting the projection on the first factor of $\mathbb{R}\times\mathbb{R}^{3}$,
while $(r,\text{\textgreek{j}},\text{\textgreek{f}})$ are as before
the usual polar coordinates on $\mathbb{R}^{3}$%
} $g$ takes the form 
\begin{equation}
g=-dt^{2}+dr^{2}+f(r)\cdot(d\text{\textgreek{j}}^{2}+\sin^{2}\text{\textgreek{j}}\cdot d\text{\textgreek{f}}^{2}).\label{eq:Counterexample1LorentzianMetric}
\end{equation}
Hence, for any coordinate chart $(x^{1},x^{2},x^{3})$ on $\mathbb{R}^{3}$
we have 
\begin{equation}
g=-dt^{2}+(g_{0})_{\text{\textgreek{m}\textgreek{n}}}dx^{\text{\textgreek{m}}}dx^{\text{\textgreek{n}}}.
\end{equation}

We readily verify that $(\mathcal{M},g)$ satisfies assumptions \hyperref[Assumtion 1]{1},
\hyperref[Assumtion 2]{2}, \hyperref[Assumtion 3]{3} and \hyperref[Assumtion 4]{4}:
\begin{itemize}
\item $(\mathcal{M},g)$ is a stationary spacetime (with Killing vector
field $T\doteq\partial_{t}$), with Cauchy hypersurface $\text{\textgreek{S}}_{0}\doteq\{t=0\}$,
which is asymptotically flat with one end (since outside $\{r\le\text{\textgreek{p}}\}$
$g$ is the usual Minkowski metric, and the $\text{\textgreek{S}}_{0}$
has Euclidean induced metric and vanishing second fundamental form).
Hence, Assumption \hyperref[Assumption 1]{1} is satisfied.
\item There is no event horizon in $(\mathcal{M},g)$, since any point of
$\mathcal{M}$ communicates with the flat region $\{r\ge\text{\textgreek{p}}\}$
with both future directed and past directed timelike curves: For any
$p\in\mathcal{M}$ and any smooth curve $\text{\textgreek{g}}:[0,1]\rightarrow\mathbb{R}^{3}$
such that $\text{\textgreek{g}}(0)=\tilde{p}$ ($\tilde{p}$ being
the projection of $p$ on the second factor of $\mathcal{M}=\mathbb{R}\times\mathbb{R}^{3}$)
and $\text{\textgreek{g}}(1)\in\{r\ge\text{\textgreek{p}}\}$,%
\footnote{such a curve always exist, it can be e.g. a straight line in $\mathbb{R}^{3}$%
} we construct the smooth curve $\text{\textgreek{G}}:[0,1]\rightarrow\mathcal{M}$,
$\text{\textgreek{G}}(\text{\textgreek{l}})\doteq(t(p)+C\cdot\text{\textgreek{l}},\text{\textgreek{g}}(\text{\textgreek{l}}))$,
for some $C\gg1$. We see that $\text{\textgreek{G}}(0)=p$, $\text{\textgreek{G}}(1)\in\{r\ge\text{\textgreek{p}}\}$
and, if $C\gg1$ is sufficiently large, due to \ref{eq:Counterexample1LorentzianMetric},
$\text{\textgreek{G}}$ is timelike future oriented. Similarly, if
we had instead chosen $C\ll-1$, $\text{\textgreek{G}}$ would have
been timelike past directed. Thus, assumption Assumption \hyperref[Assumption 2]{2}
is satisfied.
\item There is no ergoregion, since $g(T,T)\equiv-1$, and thus assumption
\hyperref[Assumption 3]{3} holds trivially.
\item The boundedness of the energy of solutions to the wave equation holds
trivially, since T is globally timelike and $\int_{t=\text{\textgreek{t}}}J_{\text{\textgreek{m}}}^{T}(\text{\textgreek{y}})n^{\text{\textgreek{m}}}$
is constant (and hence bounded) for any smooth solution $\text{\textgreek{y}}$
to the wave equation $\square_{g}\text{\textgreek{y}}=0$ on $(\mathcal{M},g)$.
So Assumption \hyperref[Assumption 4]{4} is also true.
\end{itemize}
The fact that Theorem \ref{thm:Theorem} is optimal for the spacetime
$(\mathcal{M},g)$ is a consequence of the following proposition:
\begin{prop}
For any function $h:[0,+\infty)\rightarrow[0,+\infty)$ for which
$h(t)\cdot\{\log^{2}(2+t)\}\rightarrow0$ as $t\rightarrow+\infty$,
there exists a $t_{0}>0$ so that for any $\text{\textgreek{t}}>t_{0}$,
one can construct a smooth solution $\text{\textgreek{y}}$ (depending
on $\text{\textgreek{t}}$) to the wave equation on $(\mathcal{M},g)$
with finite $T$-energy on the $\{t=const\}$ hypersurfaces such that
\[
\int_{\{t=\text{\textgreek{t}}\}\cap\{r\le\frac{3\text{\textgreek{p}}}{4}\}}J_{\text{\textgreek{m}}}^{T}(\text{\textgreek{y}})n^{\text{\textgreek{m}}}>h(\text{\textgreek{t}})\cdot\Big(\int_{t=0}(1+r)\cdot J_{\text{\textgreek{m}}}^{T}(\text{\textgreek{y}})n^{\text{\textgreek{m}}}+\int_{t=0}J_{\text{\textgreek{m}}}^{T}(T\text{\textgreek{y}})n^{\text{\textgreek{m}}}\Big).
\]
\end{prop}
\begin{proof}
This is accomplished through the quasimode construction performed
by Rodnianski and Tao in \cite{Rodnianski2011} in the following way:

We first note that on $\mathbb{S}^{3}$, equipped with the usual coordinate
system $(\text{\textgreek{j}}_{1},\text{\textgreek{j}}_{2},\text{\textgreek{j}}_{3})$
in which the metric takes the form $g_{\mathbb{S}^{3}}=d\text{\textgreek{j}}_{1}^{2}+sin^{2}\text{\textgreek{j}}_{1}\cdot(d\text{\textgreek{j}}_{2}^{2}+\sin^{2}\text{\textgreek{j}}_{2}\cdot d\text{\textgreek{j}}_{3}^{2})$,
the Laplacian takes the form 

\begin{equation}
\text{\textgreek{D}}_{\mathbb{S}^{3}}=\partial_{\text{\textgreek{j}}_{1}}^{2}+2c\cot\text{\textgreek{j}}_{1}\cdot\partial_{\text{\textgreek{j}}_{1}}+\frac{1}{\sin^{2}\text{\textgreek{j}}_{1}}\text{\textgreek{D}}_{\mathbb{S}^{2}}
\end{equation}
 ($\text{\textgreek{D}}_{\mathbb{S}^{2}}$ being the Laplace-Beltrami
of the usual metric of $\mathbb{S}^{2}$ in the coordinates $(\text{\textgreek{j}}_{2},\text{\textgreek{j}}_{3})$). 

For any integer $l\ge0$, we will denote with $Y_{l}=Y_{l}(\text{\textgreek{j}}_{2},\text{\textgreek{j}}_{3})$
a spherical harmonic on $\mathbb{S}^{2}$ of order $l$, that is a
function satisfying $\text{\textgreek{D}}_{\mathbb{S}^{2}}Y_{l}+l^{2}Y_{l}=0$,
normalised so that $\int_{\mathbb{S}^{2}}|Y_{l}|^{2}=1$. Then, the
function $u_{l}:\mathbb{S}^{3}\rightarrow\mathbb{C}$ defined by $u_{l}(\text{\textgreek{j}}_{1},\text{\textgreek{j}}_{2},\text{\textgreek{j}}_{3})=sin^{l}(\text{\textgreek{j}}_{1})\cdot cos(\text{\textgreek{j}}_{1})\cdot Y_{l}(\text{\textgreek{j}}_{2},\text{\textgreek{j}}_{3})$
satisfies the equation

\begin{equation}
\text{\textgreek{D}}_{\mathbb{S}^{3}}u_{l}+\text{\textgreek{l}}^{2}u_{l}=0,
\end{equation}
 where $\text{\textgreek{l}}^{2}\doteq(l+1)\cdot(l+3)$ (and hence
$\text{\textgreek{l}}\sim l$ if $l\gg1$)

It is easy (see \cite{Rodnianski2011}) to obtain the following estimates
for $u_{l}$, which are in fact quantitative expressions of the fact
that the $u_{l}$'s become more and more concentrated around $\text{\textgreek{j}}_{1}=\frac{\text{\textgreek{p}}}{2}$
as $l\rightarrow\infty$ (which can be also deduced in view of the
$l$ exponent in $\sin^{l}(\text{\textgreek{j}}_{1})$ in the definition
of $u_{l}$):

\begin{equation}
\int_{|\text{\textgreek{j}}_{1}-\frac{\text{\textgreek{p}}}{2}|\le\frac{\text{\textgreek{p}}}{4}}|u_{l}|^{2}\sim\text{\textgreek{l}}^{-1}\sim l^{-1}\,,\int_{|\text{\textgreek{j}}_{1}-\frac{\text{\textgreek{p}}}{2}|\le\frac{\text{\textgreek{p}}}{4}}|\nabla u_{l}|^{2}\sim1\, as\, l\rightarrow\infty,\label{eq:UlBounds1}
\end{equation}
 as well as

\begin{equation}
\int_{|\text{\textgreek{j}}_{1}-\frac{\text{\textgreek{p}}}{2}|>\frac{\text{\textgreek{p}}}{4}}\Big(|u_{l}|^{2}+|\nabla u_{l}|^{2}\Big)=O(e^{-c\text{\textgreek{l}}})\label{eq:UlBounds2}
\end{equation}
 for some $c>0$.

We now fix a smooth cut-off function $\text{\textgreek{q}}:\mathbb{S}^{3}\rightarrow[0,1]$
such that $\text{\textgreek{q}}\equiv1$ for $\text{\textgreek{j}}_{1}\le\frac{3\text{\textgreek{p}}}{4}$
and $\text{\textgreek{q}}\equiv0$ for $\text{\textgreek{j}}_{1}\ge\frac{7\text{\textgreek{p}}}{8}$.
We define 
\begin{equation}
\tilde{u}_{l}\doteq\text{\textgreek{q}}\cdot u_{l},
\end{equation}
and we note that the $\tilde{u}_{l}$'s are supported in $\{\text{\textgreek{j}}_{1}\le\frac{7\text{\textgreek{p}}}{8}\}$.
Due to this fact, as well as the fact we can isometrically identify
$\{r\le\frac{7\text{\textgreek{p}}}{8}\}\subseteq\mathbb{R}^{3}$
(equipped with $g_{0}$) with $(\mathbb{S}^{3}\backslash\{\text{\textgreek{j}}_{1}>\frac{7\text{\textgreek{p}}}{8}\},g_{\mathbb{S}^{3}})$
through relabelling our coordinate maps $r\rightarrow\text{\textgreek{j}}_{1}$,
$\text{\textgreek{j}}\rightarrow\text{\textgreek{j}}_{2}$, $\text{\textgreek{f}}\rightarrow\text{\textgreek{j}}_{3}$,
we can consider $\tilde{u}_{l}$ as having been defined on $\mathbb{R}^{3}$
and satisfying:

\begin{equation}
\text{\textgreek{D}}_{g_{0}}\tilde{u}_{l}+\text{\textgreek{l}}^{2}\tilde{u}_{l}=2\partial^{i}\text{\textgreek{q}}\cdot\partial_{i}u_{l}+(\text{\textgreek{D}}_{g_{0}}\text{\textgreek{q}})\cdot u_{l},\label{eq:EquationUl}
\end{equation}
 where we raise indices using $g_{0}$. Note that the right hand side
of (\ref{eq:EquationUl}) is supported in $\{\frac{3\text{\textgreek{p}}}{4}\le r\le\frac{7\text{\textgreek{p}}}{8}\}$.

Finally, we define $\tilde{\text{\textgreek{y}}}_{l}:\mathcal{M}\rightarrow\mathbb{C}$,
\begin{equation}
\text{\ensuremath{\tilde{\lyxmathsym{\textgreek{y}}}}}_{l}(t,x)\doteq e^{i\text{\textgreek{l}}t}\cdot\tilde{u}_{l}(x)
\end{equation}
 ($t\in\mathbb{R}$, $x\in\mathbb{R}^{3}$). Note that $\text{\ensuremath{\tilde{\lyxmathsym{\textgreek{y}}}}}_{l}$
is supported in $\{r\le\frac{7\text{\textgreek{p}}}{8}\}$, and, due
to (\ref{eq:EquationUl}) and (\ref{eq:Counterexample1LorentzianMetric}),
it satisfies the equation 
\begin{equation}
\square\tilde{\text{\textgreek{y}}}_{l}=F_{l},
\end{equation}
where 
\begin{equation}
F_{l}(t,x)=e^{i\text{\textgreek{l}}t}\{2\partial^{i}\text{\textgreek{q}}(x)\cdot\partial_{i}u_{l}(x)+(\text{\textgreek{D}}_{g_{0}}\text{\textgreek{q}}(x))\cdot u_{l}(x)\}
\end{equation}
 is supported in $\{\frac{3\text{\textgreek{p}}}{4}\le r\le\frac{7\text{\textgreek{p}}}{8}\}$.

Due to (\ref{eq:UlBounds1}) and (\ref{eq:UlBounds2}), for any $\text{\textgreek{t}}\ge0$
we can estimate (as $l\rightarrow\infty$): 
\begin{equation}
\int_{t=0}J_{\text{\textgreek{m}}}^{T}(\tilde{\text{\textgreek{y}}}_{l})n^{\text{\textgreek{m}}}\sim1,
\end{equation}
\begin{equation}
\int_{t=0}J_{\text{\textgreek{m}}}^{T}(T\tilde{\text{\textgreek{y}}}_{l})n^{\text{\textgreek{m}}}\sim l^{2},
\end{equation}
\begin{equation}
\int_{\{t=\text{\textgreek{t}}\}\cap\{r\le\frac{3\text{\textgreek{p}}}{4}\}}J_{\text{\textgreek{m}}}^{T}(\tilde{\text{\textgreek{y}}}_{l})n^{\text{\textgreek{m}}}=\int_{\{t=0\}\cap\{r\le\frac{3\text{\textgreek{p}}}{4}\}}J_{\text{\textgreek{m}}}^{T}(\tilde{\text{\textgreek{y}}}_{l})n^{\text{\textgreek{m}}}\sim1\label{eq:LocalisedEnergy}
\end{equation}
 and
\begin{equation}
\int_{0}^{\text{\textgreek{t}}}\big(\int_{\{t=s\}}|F_{l}|^{2}\big)^{1/2}ds=\text{\textgreek{t}}\cdot O(e^{-cl}).
\end{equation}

If we denote with $\text{\textgreek{y}}_{l}$ the unique solution
to the wave equation $\square\text{\textgreek{y}}_{l}=0$ on $(\mathcal{M},g)$,
with initial data on $\{t=0\}$ being the induced initial data there
by $\text{\ensuremath{\tilde{\lyxmathsym{\textgreek{y}}}}}_{l}$,
then we have 
\begin{equation}
\int_{t=0}J_{\text{\textgreek{m}}}^{T}(\text{\textgreek{y}}_{l})n^{\text{\textgreek{m}}}=\int_{t=0}J_{\text{\textgreek{m}}}^{T}(\tilde{\text{\textgreek{y}}}_{l})n^{\text{\textgreek{m}}}\sim1.\label{eq:InitialEnergPsiL}
\end{equation}
 Moreover, since $\square_{g}=-\partial_{t}^{2}+\text{\textgreek{D}}_{g_{0}}$,
we calculate 
\begin{equation}
T^{2}\text{\textgreek{y}}_{l}|_{t=0}=\text{\textgreek{D}}_{g_{0}}\text{\textgreek{y}}_{l}||_{t=0}=\text{\textgreek{D}}_{g_{0}}\tilde{\text{\textgreek{y}}}_{l}|_{t=0}=T^{2}\tilde{\text{\textgreek{y}}}_{l}|_{t=0}+\square_{g}\tilde{\text{\textgreek{y}}}_{l}|_{t=0}=T^{2}\tilde{\text{\textgreek{y}}}_{l}|_{t=0}+F_{l}|_{t=0}.\label{eq:EqualityInitialSlice}
\end{equation}
Due to the fact that for any spatial derivative $\partial_{x^{i}}$
we have $\partial_{x^{i}}\partial_{t}\text{\textgreek{y}}_{l}|_{t=0}=\partial_{x^{i}}\partial_{t}\text{\ensuremath{\tilde{\lyxmathsym{\textgreek{y}}}}}_{l}|_{t=0}$,%
\footnote{since $\partial_{t}\text{\textgreek{y}}_{l}$ and $\partial_{t}\tilde{\text{\textgreek{y}}}_{l}$
are identical on $\{t=0\}$ %
} we can estimate in view of (\ref{eq:EqualityInitialSlice}): 
\begin{equation}
|\int_{t=0}J_{\text{\textgreek{m}}}^{T}(T\text{\textgreek{y}}_{l})n^{\text{\textgreek{m}}}|\le\int_{t=0}J_{\text{\textgreek{m}}}^{T}(T\tilde{\text{\textgreek{y}}}_{l})n^{\text{\textgreek{m}}}+\int_{t=0}|F_{l}|^{2}\sim l^{2}+O(e^{-cl})\sim l^{2}.\label{eq:InitialEnergyTPsiL}
\end{equation}

By an application of Duhamel's principle on $\text{\textgreek{y}}_{l}-\tilde{\text{\textgreek{y}}}_{l}$
(which satisfies $\square_{g}(\text{\textgreek{y}}_{l}-\tilde{\text{\textgreek{y}}}_{l})=-F_{l}$
and has vanishing initial data on $\{t=0\}$), we infer that for any
$\text{\textgreek{t}}\ge0$:

\begin{equation}
\Big(\int_{t=\text{\textgreek{t}}}J_{\text{\textgreek{m}}}^{T}(\text{\textgreek{y}}_{l}-\tilde{\text{\textgreek{y}}}_{l})n^{\text{\textgreek{m}}}\Big)^{1/2}\le C\cdot\int_{0}^{\text{\textgreek{t}}}\big(\int_{\{t=s\}}|F_{l}|^{2}\big)^{1/2}ds\le\text{\textgreek{t}}\cdot O(e^{-cl})\label{eq:BoundEnergyDifference}
\end{equation}
and thus in view of (\ref{eq:LocalisedEnergy}) and (\ref{eq:BoundEnergyDifference}):
\begin{equation}
\int_{\{t=\text{\textgreek{t}}\}\cap\{r\le\frac{3\text{\textgreek{p}}}{4}\}}J_{\text{\textgreek{m}}}^{T}(\text{\textgreek{y}}_{l})n^{\text{\textgreek{m}}}\ge\int_{\{t=\text{\textgreek{t}}\}\cap\{r\le\frac{3\text{\textgreek{p}}}{4}\}}J_{\text{\textgreek{m}}}^{T}(\tilde{\text{\textgreek{y}}}_{l})n^{\text{\textgreek{m}}}-\int_{t=\text{\textgreek{t}}}J_{\text{\textgreek{m}}}^{T}(\text{\textgreek{y}}_{l}-\tilde{\text{\textgreek{y}}}_{l})n^{\text{\textgreek{m}}}\ge1-C\text{\textgreek{t}}\cdot e^{-cl}.\label{eq:AlmostFinalStatement}
\end{equation}

Finally, since $\text{\textgreek{y}}_{l}|_{t=0}$ and $T\text{\textgreek{y}}_{l}|_{t=0}$
are supported in $\{r\le\frac{7\text{\textgreek{p}}}{8}\}$, we readily
see that 
\begin{equation}
\int_{t=0}(1+r)J_{\text{\textgreek{m}}}^{T}(\text{\textgreek{y}}_{l})n^{\text{\textgreek{m}}}\sim\int_{t=0}J_{\text{\textgreek{m}}}^{T}(\text{\textgreek{y}}_{l})n^{\text{\textgreek{m}}}\sim1.\label{eq:BoundenergyWeighted}
\end{equation}

From (\ref{eq:InitialEnergPsiL}), (\ref{eq:InitialEnergyTPsiL}),
(\ref{eq:AlmostFinalStatement}) and (\ref{eq:BoundenergyWeighted})
we conclude that there exists a constant $b>0$ such that for all
$l\gg1$ and for any given $\text{\textgreek{t}}>0$:

\begin{equation}
\int_{\{t=\text{\textgreek{t}}\}\cap\{r\le\frac{3\text{\textgreek{p}}}{4}\}}J_{\text{\textgreek{m}}}^{T}(\text{\textgreek{y}}_{l})n^{\text{\textgreek{m}}}>b\cdot l^{-2}\Big\{1-C\text{\textgreek{t}}\cdot e^{-cl}\Big\}\cdot\Big(\int_{t=0}(1+r)\cdot J_{\text{\textgreek{m}}}^{T}(\text{\textgreek{y}}_{l})n^{\text{\textgreek{m}}}+\int_{t=0}J_{\text{\textgreek{m}}}^{T}(T\text{\textgreek{y}}_{l})n^{\text{\textgreek{m}}}\Big).\label{eq:finalStatement}
\end{equation}
 Thus, for any given function $h(t)=o(\log^{-2}(t+2))$ (as $t\rightarrow\infty$),
we can find a large enough $t_{0}\gg1$, such that for any given $\text{\textgreek{t}}>t_{0}$,
after picking $l\sim\frac{2}{c}\log(\text{\textgreek{t}}+2)$ we can
bound from (\ref{eq:finalStatement}):
\begin{align}
\int_{\{t=\text{\textgreek{t}}\}\cap\{r\le\frac{3\text{\textgreek{p}}}{4}\}}J_{\text{\textgreek{m}}}^{T}(\text{\textgreek{y}}_{l})n^{\text{\textgreek{m}}} & >b\cdot\Big\{\frac{c^{2}}{4\{\log(\text{\textgreek{t}}+2)\}^{2}}-C\text{\textgreek{t}}\cdot\text{\textgreek{t}}^{-2}\Big\}\cdot\Big(\int_{t=0}(1+r)\cdot J_{\text{\textgreek{m}}}^{T}(\text{\textgreek{y}}_{l})n^{\text{\textgreek{m}}}+\int_{t=0}J_{\text{\textgreek{m}}}^{T}(T\text{\textgreek{y}}_{l})n^{\text{\textgreek{m}}}\Big)>\\
 & >h(t)\cdot\Big(\int_{t=0}(1+r)\cdot J_{\text{\textgreek{m}}}^{T}(\text{\textgreek{y}}_{l})n^{\text{\textgreek{m}}}+\int_{t=0}J_{\text{\textgreek{m}}}^{T}(T\text{\textgreek{y}}_{l})n^{\text{\textgreek{m}}}\Big).\nonumber 
\end{align}

\end{proof}
\appendix

\section{\label{sec:ConstructionOfTheDoubleNullCoordinateSystem}Construction
of the $(u,r,\text{\textgreek{sv}})$ coordinate chart}

Assume that an open subset $\mathcal{U}$ of a spacetime $(\mathcal{M},g)$
admits a polar coordinate chart $(t,r,\text{\textgreek{sv}})$ for
$\{r\ge R>0\}$, where, for some integer $m\ge1$, the metric $g$
has the expression:

\begin{equation}
g=-\Big(1-\frac{2M}{r}+O_{m}(r^{-1-a})\Big)dt^{2}+\Big(1+\frac{2M}{r}+O_{m}(r^{-1-a})\Big)dr^{2}+r^{2}\Big(g_{\mathbb{S}^{d-1}}+O_{m}(r^{-1-a})\Big)+O_{m}(r^{-a})dtd\text{\textgreek{sv}},\label{eq:MetricAsFlat-1-1}
\end{equation}
 where $M\in\mathbb{R}$ and no metric coefficient depends on the
$t$ coordinate; this means that the vector field $T=\partial_{t}$
is Killing. In the above, the $O_{m}(r^{b})$ notation is used to
denote functions or tensors $h$ on $\mathbb{S}^{d-1}$ depending
on $(r,\text{\textgreek{sv}})$ and satisfying the bound: 
\begin{equation}
\sum_{j=0}^{m}\Big(\sum_{j_{1}+j_{2}=j}r^{j_{1}}||\partial_{r}^{j_{1}}\partial_{\text{\textgreek{sv}}}^{j_{2}}h(r,\text{\textgreek{sv}})||_{g_{\mathbb{S}^{d-1}}}\Big)\le C\cdot r^{b}.
\end{equation}
We will show that we can find a new coordinate chart $(u,r,\text{\textgreek{sv}})$
on $\mathcal{U}$ in which $g$ will take the form: 
\begin{align}
g= & -4\big(1-\frac{2M}{r}+O_{m-1}(r^{-1-a})\big)du^{2}-4\big(1+O_{m-1}(r^{-1-a})\big)dudr+r^{2}(g_{\mathbb{S}^{d-1}}+O_{m-1}(r^{-1-a}))+\label{eq:MetricUR-1-1}\\
 & +O_{m-1}(r^{-a})dud\text{\textgreek{sv}}+O_{m-1}(r^{-a})drd\text{\textgreek{sv}}.\nonumber 
\end{align}

In order to construct the required coordinate function $u$, we first
introduce an auxiliary coordinate function $r^{*}=r^{*}(r,\text{\textgreek{sv}})$
, the analogue of the Regge-Wheeler coordinate function for the Schwarzschild
spacetime. The main identity that we will need to hold is $g_{r^{*}r^{*}}=-g_{tt}$
in the $(t,r^{*},\text{\textgreek{sv}})$ coordinate system. From
the expression of the metric (\ref{eq:MetricAsFlat-1-1}), we see
that $\sqrt{-\frac{g_{rr}}{g_{tt}}}=(1-\frac{2M}{r})^{-1}+h(r,\text{\textgreek{sv}})$,
for a smooth function $h(r,\text{\textgreek{sv}})=O_{m}(r^{-1-a})$.
Therefore, we have to define:

\[
r^{*}(r,\text{\textgreek{sv}})\doteq r+2M\cdot log(r-2M)-\int_{r}^{+\infty}h(\text{\textgreek{r}},\text{\textgreek{sv}})\, d\text{\textgreek{r}}.
\]
 Due to the fact that $h=O_{m}(r^{-1-a})$, $r^{*}$ is a continuous
function of $(r,\text{\textgreek{sv}})$ satisfying $\frac{\partial r^{*}}{\partial r}=(1-\frac{2M}{r})^{-1}+h(r,\text{\textgreek{sv}})$
as desired. Moreover 
\begin{equation}
\partial_{\text{\textgreek{sv}}}r^{*}(r,\text{\textgreek{sv}})=-\int_{r}^{+\infty}\partial_{\text{\textgreek{sv}}}h(\text{\textgreek{r}},\text{\textgreek{sv}})\, d\text{\textgreek{r}}=O_{m-1}(r^{-a}).
\end{equation}

Hence, in the coordinate system $(t,r^{*},\text{\textgreek{sv}})$
(which is indeed a coordinate system for $r\gg1$ due to the form
of $r^{*}$) we can compute that the metric takes the following form:

\begin{align}
g= & \Big(1-\frac{2M}{r}+O_{m-1}(r^{-1-a})\Big)\Big(-dt^{2}+(dr^{*})^{2}\Big)+r^{2}\Big(g_{\mathbb{S}^{d-1}}+O_{m-1}(r^{-1-a})\Big)+\label{eq:metricR*-1}\\
 & +O_{m-1}(r^{-a})dr^{*}d\text{\textgreek{sv}}+O_{m-1}(r^{-a})dtd\text{\textgreek{sv}}.\nonumber 
\end{align}

We can now introduce the function $u=\frac{1}{2}(t-r^{*})$. In the
$(u,r^{*},\text{\textgreek{sv}})$ coordinates, the metric takes the
form:

\begin{align}
g= & -4(1-\frac{2M}{r}+O_{m-1}(r^{-1-a}))\big(du^{2}+dudr^{*}\big)+r^{2}(g_{\mathbb{S}^{d-1}}+O_{m-1}(r^{-1-a}))+\label{eq:metricUR*-1}\\
 & +O_{m-1}(r^{-a})dud\text{\textgreek{sv}}+O_{m-1}(r^{-a})dr^{*}d\text{\textgreek{sv}}.\nonumber 
\end{align}

Thus, switching finally to the $(u,r,\text{\textgreek{sv}})$ coordinate
system, since $\partial_{r}r^{*}=1+\frac{2M}{r}+O_{m-1}(r^{-1-a})$,
we deduce that the metric has the required expression: 
\begin{align}
g= & -4\big(1-\frac{2M}{r}+O_{m-1}(r^{-1-a})\big)du^{2}-4\big(1+O_{m-1}(r^{-1-a})\big)dudr+r^{2}(g_{\mathbb{S}^{d-1}}+O_{m-1}(r^{-1-a}))+\label{eq:MetricUR-1}\\
 & +O_{m-1}(r^{-a})dud\text{\textgreek{sv}}+O_{m-1}(r^{-a})drd\text{\textgreek{sv}}.\nonumber 
\end{align}

\section{\noindent \label{sec:ProofOfInclusion}Proof of the inclusion $J^{-}(p)\cap\text{\textgreek{S}}\subseteq\text{\textgreek{S}}\cap\mathcal{D}$
for $p\in J^{+}(\text{\textgreek{S}})\cap\mathcal{D}$}

Let $(\mathcal{M}^{d+1},g)$, $d\ge3$ be a globally hyperbolic spacetime
with a Cauchy hypersurface $\text{\textgreek{S}}$. Let $\text{\textgreek{S}}$
be asymptotically flat (in the sense of Assumption \hyperref[Assumption 1]{1}),
and let $\mathcal{D}$ be the domain of outer communications associated
to one asymptotically flat region $\mathcal{I}_{as}$ of $(\mathcal{M},g)$
(see Assumption \hyperref[Assumption 1]{1} for the relevant definitions).
We also define 
\[
\mathcal{H}\doteq\partial\mathcal{D},
\]
 and we set 
\[
\mathcal{H}^{+}\doteq J^{+}(\mathcal{I}_{as})\cap\partial\big(J^{-}(\mathcal{I}_{as})\big)
\]
 and 
\[
\mathcal{H}^{-}\doteq J^{-}(\mathcal{I}_{as})\cap\partial\big(J^{+}(\mathcal{I}_{as})\big).
\]
 We will establish the following inclusion:
\begin{lem}
\label{lem:-for-,Inclusion}With $\text{\textgreek{S}}$, $\mathcal{D}$,
$\mathcal{H}^{+}$ and $\mathcal{H}^{-}$ as above, under the additional
assumption that $\mathcal{H}^{-}\subset I^{-}(\text{\textgreek{S}}\cap\mathcal{D})$,
the following inclusion holds: 
\begin{equation}
J^{-}(p)\cap\text{\textgreek{S}}\subseteq\text{\textgreek{S}}\cap\mathcal{D}\label{eq:Inclusion}
\end{equation}
 for any $p\in J^{+}(\text{\textgreek{S}})\cap\mathcal{D}$.\end{lem}
\begin{proof}
The proof will follow by contradiction: if 
\begin{equation}
J^{-}(p)\cap\text{\textgreek{S}}\subsetneq\text{\textgreek{S}}\cap\mathcal{D},\label{eq:Contradiction}
\end{equation}
then there exists a past directed causal curve $\text{\textgreek{g}}:[0,1)\rightarrow\mathcal{M}$
which is past inextendible with $\text{\textgreek{g}}(0)=p$, and
for which $\text{\textgreek{g}}\cap(\text{\textgreek{S}}\cap\mathcal{D})=\emptyset$.
But then, 
\begin{equation}
\text{\textgreek{g}}\subseteq J^{-}(\mathcal{I}_{as})\label{eq:ContainedPast}
\end{equation}
 since $p\in\mathcal{D}\subseteq J^{-}(\mathcal{I}_{as})$. Furthermore,
$\text{\textgreek{g}}\cap\text{\textgreek{S}}\neq\emptyset$, since
\textgreek{S} is a Cauchy hypersurface of $(\mathcal{M},g)$. Thus,
we have 
\begin{equation}
\text{\textgreek{g}}\cap(\text{\textgreek{S}}\cap J^{-}(\mathcal{I}_{as}))\neq\emptyset.\label{eq:NonEmptyIntersection}
\end{equation}

If, in addition to (\ref{eq:ContainedPast}), $\text{\textgreek{g}}$
is entirely contained also in $J^{+}(\mathcal{I}_{as})$, then we
would deduce from (\ref{eq:NonEmptyIntersection}) that $(\text{\textgreek{g}}\cap J^{+}(\mathcal{I}_{as}))\cap(\text{\textgreek{S}}\cap J^{-}(\mathcal{I}_{as}))\neq\emptyset$,
or equivalently (since $\mathcal{D}=J^{+}(\mathcal{I}_{as})\cap J^{-}(\mathcal{I}_{as})$)
that $\text{\textgreek{g}}\cap(\mathcal{D}\cap\text{\textgreek{S}})\neq\emptyset$,
which is the required contradiction. 

If, on the other hand, $\text{\textgreek{g}}$ is not entirely contained
in $J^{+}(\mathcal{I}_{as})$, then $\text{\textgreek{g}}\cap\partial J^{+}(\mathcal{I}_{as})$
should be non empty, which implies that 
\begin{equation}
\text{\textgreek{g}}\cap\mathcal{H}^{-}\neq\emptyset\label{eq:IntersectionHorizon}
\end{equation}
because of (\ref{eq:ContainedPast}) and the fact that $\mathcal{H}^{-}=J^{-}(\mathcal{I}_{as})\cap\partial J^{+}(\mathcal{I}_{as})$.
Therefore, let 
\begin{equation}
\text{\textgreek{l}}_{0}=\inf\big\{\text{\textgreek{l}}\in[0,1):\,\text{\textgreek{g}}(\text{\textgreek{l}})\in\mathcal{H}^{-}\big\}.
\end{equation}
Since $\mathcal{H}^{-}$ is closed by definition, $\text{\textgreek{g}}(\text{\textgreek{l}}_{0})\in\mathcal{H}^{-}$,
and due to the definition of $\text{\textgreek{l}}_{0}$, we have
\begin{equation}
\text{\textgreek{g}}([0,\text{\textgreek{l}}_{0}])\subset J^{+}(\mathcal{I}_{as})\cap J^{-}(\mathcal{I}_{as})=\mathcal{D}.
\end{equation}
In view of our assumption that $\mathcal{H}^{-}\subset J^{-}(\text{\textgreek{S}}\cap\mathcal{D})\subseteq J^{-}(\text{\textgreek{S}}\cap J^{-}(\mathcal{I}_{as}))$
and the facts that $\text{\textgreek{g}}(\text{\textgreek{l}}_{0})\in\mathcal{H}^{-}$
and $\text{\textgreek{g}}$ is causal and past directed, 
\begin{equation}
\text{\textgreek{g}}((\text{\textgreek{l}}_{0},1))\cap(\text{\textgreek{S}}\cap J^{-}(\mathcal{I}_{as}))=\emptyset.\label{eq:Empty2}
\end{equation}
 can not intersect $\text{\textgreek{S}}\cap J^{-}(\mathcal{I}_{as})$.
However, in view of (\ref{eq:NonEmptyIntersection}), there exists
some $\text{\textgreek{l}}_{1}\in[0,1)$ such that $\text{\textgreek{g}}(\text{\textgreek{l}}_{1})\in\text{\textgreek{S}}\cap J^{-}(\mathcal{I}_{as})$.
Thus, from (\ref{eq:Empty2}) we infer that $\text{\textgreek{l}}_{1}\in[0,\text{\textgreek{l}}_{0}]$.
But since $\text{\textgreek{g}}([0,\text{\textgreek{l}}_{0}])\subset\mathcal{D}$,
we obtain that $\text{\textgreek{g}}\cap(\text{\textgreek{S}}\cap\mathcal{D})\neq\emptyset$,
which again is the required contradiction.

Hence, we conclude that for any $p\in J^{+}(\text{\textgreek{S}})\cap\mathcal{D}$
we have $J^{-}(p)\cap\text{\textgreek{S}}\subseteq\text{\textgreek{S}}\cap\mathcal{D}$.
\end{proof}

\section{\label{sec:Interpolation}Interpolation for $r$-weighted energy
bounds}
\begin{lem}
\label{lem:InterpolationLemma}With the notations as in Section \ref{sec:Polynomial-decay},
if for any smooth solution $\text{\textgreek{y}}$ to $\square_{g}\text{\textgreek{y}}=0$
on $J^{+}(\text{\textgreek{S}})\cap\mathcal{D}$ with compactly supported
initial data on $\text{\textgreek{S}}\cap\mathcal{D}$ and for any
$\text{\textgreek{t}}\in[\frac{1}{4}t^{*},\frac{3}{4}t^{*}]$ we can
bound 
\begin{equation}
\int_{S_{\text{\textgreek{t}}}\cap\mathcal{R}(0,t^{*})}J_{\text{\textgreek{m}}}^{N}(\text{\textgreek{y}}_{\le\text{\textgreek{w}}_{+}})n_{S_{\text{\textgreek{t}}}}^{\text{\textgreek{m}}}\le\frac{C}{\text{\textgreek{t}}-\frac{1}{4}t^{*}}\cdot\Big\{ e^{C\text{\textgreek{w}}_{+}}\int_{t=0}J_{\text{\textgreek{m}}}^{N}(\text{\textgreek{y}})n^{\text{\textgreek{m}}}+\int_{t=0}rJ_{\text{\textgreek{m}}}^{N}(\text{\textgreek{y}})n^{\text{\textgreek{m}}}\Big\}\label{eq:PolynomialDecayLowFrequenciesApp-1}
\end{equation}
 and 
\begin{equation}
\int_{S_{\text{\textgreek{t}}}\cap\mathcal{R}(0,t^{*})}J_{\text{\textgreek{m}}}^{N}(\text{\textgreek{y}}_{\le\text{\textgreek{w}}_{+}})n_{S_{\text{\textgreek{t}}}}^{\text{\textgreek{m}}}\le C\int_{t=0}J_{\text{\textgreek{m}}}^{N}(\text{\textgreek{y}})n^{\text{\textgreek{m}}},\label{eq:PolynomialDecayLowFrequenciesApp}
\end{equation}
then for any $\text{\textgreek{d}}_{0}\in[0,1]$ and any $\text{\textgreek{t}}\in[\frac{1}{4}t^{*},\frac{3}{4}t^{*}]$
we can also bound 
\begin{equation}
\int_{S_{\text{\textgreek{t}}}\cap\mathcal{R}(0,t^{*})}J_{\text{\textgreek{m}}}^{N}(\text{\textgreek{y}}_{\le\text{\textgreek{w}}_{+}})n_{S_{\text{\textgreek{t}}}}^{\text{\textgreek{m}}}\le\frac{C}{\big(\text{\textgreek{t}}-\frac{1}{4}t^{*}\big)^{\text{\textgreek{d}}_{0}}}\cdot\Big\{ e^{C\text{\textgreek{w}}_{+}}\int_{t=0}J_{\text{\textgreek{m}}}^{N}(\text{\textgreek{y}})n^{\text{\textgreek{m}}}+\int_{t=0}r^{\text{\textgreek{d}}_{0}}J_{\text{\textgreek{m}}}^{N}(\text{\textgreek{y}})n^{\text{\textgreek{m}}}\Big\}.\label{eq:PolynomialDecayLowFrequenciesApp-2}
\end{equation}
\end{lem}
\begin{proof}
For any $\text{\textgreek{t}}\in[\frac{1}{4}t^{*},\frac{3}{4}t^{*}]$,
we will define the following energy norm on the smooth and compactly
supported pairs of functions $(\text{\textgreek{f}}_{0},\text{\textgreek{f}}_{1})$
on the hypersurface $S_{\text{\textgreek{t}}}\cap\mathcal{R}(0,t^{*})$:
\begin{equation}
||(\text{\textgreek{f}}_{0},\text{\textgreek{f}}_{1})||_{\mathcal{H}_{en,\text{\textgreek{t}}}^{1}}\doteq\int_{S_{\text{\textgreek{t}}}\cap\mathcal{R}(0,t^{*})}J_{\text{\textgreek{m}}}^{N}(\text{\textgreek{f}})n_{S_{\text{\textgreek{t}}}}^{\text{\textgreek{m}}},\label{eq:EnergyNorm}
\end{equation}
where $\text{\textgreek{f}}$ is the unique solution on $J^{+}(S_{\text{\textgreek{t}}})\cap\mathcal{R}(0,t^{*})$
with $\text{\textgreek{f}}|_{S_{\text{\textgreek{t}}}\cap\mathcal{R}(0,t^{*})}=\text{\textgreek{f}}_{0}$
and $N\text{\textgreek{f}}|_{S_{\text{\textgreek{t}}}\cap\mathcal{R}(0,t^{*})}=\text{\textgreek{f}}_{1}$.
We will also denote with $||\cdot||_{\mathcal{H}_{en,\text{\textgreek{t}}}^{-1}}$
the dual norm of $||\cdot||_{\mathcal{H}_{en,\text{\textgreek{t}}}^{1}}$
on the space of smooth and compactly supported pairs on $S_{\text{\textgreek{t}}}\cap\mathcal{R}(0,t^{*})$,
that is to say: 
\begin{equation}
||(\text{\textgreek{f}}_{0},\text{\textgreek{f}}_{1})||_{\mathcal{H}_{en,\text{\textgreek{t}}}^{-1}}=\sup_{\{(w_{0},w_{1}):\,||(w_{0},w_{1})||_{\mathcal{H}_{en,\text{\textgreek{t}}}^{1}}=1\}}\int_{S_{\text{\textgreek{t}}}\cap\mathcal{R}(0,t^{*})}Re\big\{\text{\textgreek{f}}_{0}\bar{w}_{0}+\text{\textgreek{f}}_{1}\bar{w}_{1}\big\}.
\end{equation}

For any complex number $s$ in the strip $\{0\le Re(s)\le1\}$, we
define the function $\text{\textgreek{y}}^{(s)}:\mathcal{D}(\text{\textgreek{S}}_{0})\rightarrow\mathbb{C}$
(where $\mathcal{D}(\text{\textgreek{S}}_{0})$ is the domain of dependence
of $\text{\textgreek{S}}_{0}$) as the unique solution of the following
initial value problem 
\begin{equation}
\begin{cases}
\square\text{\textgreek{y}}^{(s)}=0\\
\text{\textgreek{y}}^{(s)}|_{\text{\textgreek{S}}_{0}}=(1+r)^{\frac{s}{2}(\text{\textgreek{d}}_{0}-1)+\frac{(1-s)}{2}\text{\textgreek{d}}_{0}}\text{\textgreek{y}}|_{\text{\textgreek{S}}_{0}},\\
N\text{\textgreek{y}}^{(s)}|_{\text{\textgreek{S}}_{0}}=N\big((1+r)^{\frac{s}{2}(\text{\textgreek{d}}_{0}-1)+\frac{(1-s)}{2}\text{\textgreek{d}}_{0}}\text{\textgreek{y}}\big)|_{\text{\textgreek{S}}_{0}}.
\end{cases}
\end{equation}
Notice that $\text{\textgreek{y}}^{(\text{\textgreek{d}}_{0})}\equiv\text{\textgreek{y}}$.
For any pair $w=(w_{0},w_{1})$ of smooth and compactly supported
functions on $S_{\text{\textgreek{t}}}\cap\mathcal{R}(0,t^{*})$ with
$||w||_{\mathcal{H}_{en,\text{\textgreek{t}}}^{-1}}=1$, we also introduce
the following function 
\begin{equation}
F_{w}[\text{\textgreek{y}}](s)\doteq\int_{S_{\text{\textgreek{t}}}}Re\big\{(\text{\textgreek{t}}-\frac{1}{4}t^{*})^{\frac{s}{2}}\text{\textgreek{y}}_{\le\text{\textgreek{w}}_{+}}^{(s)}\cdot\bar{w}_{0}+(\text{\textgreek{t}}-\frac{1}{4}t^{*})^{\frac{s}{2}}N(\text{\textgreek{y}}_{\le\text{\textgreek{w}}_{+}}^{(s)})\cdot\bar{w}_{1}\big\},\label{eq:DefinitionHolomorphicFunction}
\end{equation}
which is holomorhic in $s$ in the strip $\{0\le Re(s)\le1\}$. When
$s=0+\text{\textgreek{r}}i$ with $\text{\textgreek{r}}\in\mathbb{R}$,
an application of (\ref{eq:PolynomialDecayLowFrequenciesApp}) for
$\text{\textgreek{y}}^{(\text{\textgreek{r}}i)}$ in place of $\text{\textgreek{y}}$
readily yields (after also applying a Cauchy--Schwarz inequality on
(\ref{eq:DefinitionHolomorphicFunction}), in view of the fact that
$||w||_{\mathcal{H}_{en,\text{\textgreek{t}}}^{-1}}=1$) that
\begin{align}
|F_{w}[\text{\textgreek{y}}](i\text{\textgreek{r}})| & \le C\int_{t=0}\big((1+r)^{\text{\textgreek{d}}_{0}}J_{\text{\textgreek{m}}}^{N}(\text{\textgreek{y}})n^{\text{\textgreek{m}}}+(1+r)^{\text{\textgreek{d}}_{0}-2}|\text{\textgreek{y}}|^{2}\big)\le\label{eq:BoundFirstLine}\\
 & \le C\int_{t=0}(1+r)^{\text{\textgreek{d}}_{0}}J_{\text{\textgreek{m}}}^{N}(\text{\textgreek{y}})n^{\text{\textgreek{m}}}\nonumber 
\end{align}
(the second line following from the first after an application of
a Hardy-type inequality). When $s=1+\text{\textgreek{r}}i$ with $\text{\textgreek{r}}\in\mathbb{R}$,
an application of (\ref{eq:PolynomialDecayLowFrequenciesApp-1}) for
$\text{\textgreek{y}}^{(1+\text{\textgreek{r}}i)}$ in place of $\text{\textgreek{y}}$
yields (after applying again a Cauchy--Schwarz inequality on (\ref{eq:DefinitionHolomorphicFunction}),
in view of the fact that $||w||_{\mathcal{H}_{en,\text{\textgreek{t}}}^{-1}}=1$),
that 
\begin{align}
|F_{w}[\text{\textgreek{y}}](1+i\text{\textgreek{r}})| & \le C\Big\{ e^{C\text{\textgreek{w}}_{+}}\int_{t=0}\big((1+r)^{\text{\textgreek{d}}_{0}-1}J_{\text{\textgreek{m}}}^{N}(\text{\textgreek{y}})n^{\text{\textgreek{m}}}+(1+r)^{\text{\textgreek{d}}_{0}-3}|\text{\textgreek{y}}|^{2}\big)+\int_{t=0}\big((1+r)^{\text{\textgreek{d}}_{0}}J_{\text{\textgreek{m}}}^{N}(\text{\textgreek{y}})n^{\text{\textgreek{m}}}+(1+r)^{\text{\textgreek{d}}_{0}-2}|\text{\textgreek{y}}|^{2}\big)\Big\}\label{eq:BoundThirdLine}\\
 & \le C\Big\{ e^{C\text{\textgreek{w}}_{+}}\int_{t=0}J_{\text{\textgreek{m}}}^{N}(\text{\textgreek{y}})n^{\text{\textgreek{m}}}+\int_{t=0}r^{\text{\textgreek{d}}_{0}}J_{\text{\textgreek{m}}}^{N}(\text{\textgreek{y}})n^{\text{\textgreek{m}}}\Big\}.\nonumber 
\end{align}

By applying the Phragmen--Lindel\"of maximum principle on the function
$F_{w}[\text{\textgreek{y}}](s)$ on the strip $\{0\le Re(s)\le1\}$
and using (\ref{eq:BoundFirstLine}) and (\ref{eq:BoundThirdLine})
we obtain 
\begin{equation}
|F_{w}[\text{\textgreek{y}}](\text{\textgreek{d}}_{0})|\le C\Big\{ e^{C\text{\textgreek{w}}_{+}}\int_{t=0}J_{\text{\textgreek{m}}}^{N}(\text{\textgreek{y}})n^{\text{\textgreek{m}}}+\int_{t=0}r^{\text{\textgreek{d}}_{0}}J_{\text{\textgreek{m}}}^{N}(\text{\textgreek{y}})n^{\text{\textgreek{m}}}\Big\},
\end{equation}
that is to say 
\begin{equation}
\Big|\int_{S_{\text{\textgreek{t}}}}Re\big\{\text{\textgreek{t}}^{\frac{\text{\textgreek{d}}_{0}}{2}}\text{\textgreek{y}}_{\le\text{\textgreek{w}}_{+}}\cdot\bar{w}_{0}+\text{\textgreek{t}}^{\frac{\text{\textgreek{d}}_{0}}{2}}N(\text{\textgreek{y}}_{\le\text{\textgreek{w}}_{+}})\cdot\bar{w}_{1}\big\}\Big|\le C\Big\{ e^{C\text{\textgreek{w}}_{+}}\int_{t=0}J_{\text{\textgreek{m}}}^{N}(\text{\textgreek{y}})n^{\text{\textgreek{m}}}+\int_{t=0}r^{\text{\textgreek{d}}_{0}}J_{\text{\textgreek{m}}}^{N}(\text{\textgreek{y}})n^{\text{\textgreek{m}}}\Big\}.\label{eq:BoundFromPhragmenLindelof}
\end{equation}
Since (\ref{eq:BoundFromPhragmenLindelof}) holds for all pairs $w=(w_{0},w_{1})$
with $||w||_{\mathcal{H}_{en,\text{\textgreek{t}}}^{-1}}=1$, from
(\ref{eq:BoundFromPhragmenLindelof}) and the definition (\ref{eq:EnergyNorm})
of the $\mathcal{H}_{en,\text{\textgreek{t}}}^{1}$ norm, we finally
obtain (\ref{eq:PolynomialDecayLowFrequenciesApp-2}). 
\end{proof}
\bibliographystyle{plain}
\bibliography{DatabaseExample}

\end{document}